\documentclass[11pt]{amsproc}
\usepackage{fullpage}
\usepackage{graphicx}
\usepackage{amssymb}
\usepackage{amsfonts}
\usepackage{amsthm}
\usepackage{amsmath}
\usepackage{eucal}
\usepackage[shortalphabetic,abbrev]{amsrefs}
\usepackage{xypic}
\usepackage{pdfsync}
\usepackage{hyperref}

\title{Symplectic Cohomology and Duality for the Wrapped Fukaya Category}
\author{Sheel Ganatra}

\def\bd{\partial}
\def\ra{\rightarrow}
\def\lra{\longrightarrow}
\def\Z{{\mathbb Z}}
\def\R{{\mathbb R}}
\def\C{{\mathbb C}}

\def\K{{\mathbb K}}
\def\w{\mathcal{W}(E)}
\def\A{\mathcal{A}}
\def\B{\mathcal{B}}
\def\M{\mathcal{M}}
\def\N{\mathcal{N}}
\def\e{\epsilon}
\def\cc{\mathcal{C}}
\def\dd{\mathcal{D}}
\def\mod{\mathrm{\!-\! mod}}
\def\rmod{\mathrm{mod\!-\!}}
\def\bimod{\mathrm{\!-\! mod\!-\!}}

\def\e{\epsilon}
\def\a{\alpha}

\def\r#1{\mathrm{#1}}

\def\mc#1{\mathcal{#1}}
\def\ol#1{\overline{#1}}
\def\w{\mathcal{W}}

\def\d{\Delta}
\def\M{\mathcal{M}}

\def\ainf{A_\infty}
\def\f{\c{F}}
\def\sh{SH}

\def\bd{\partial}
\def\z2{\Z / 2\Z}
\def\y{\mc{Y}}

\def\id{\mathrm{id}}
\def\mf#1{\mathfrak{#1}}

\def\p{\partial}
\def\A{\mc{A}}
\def\ob{\mathrm{ob\ }}

\def\chtopmaps{\overline{\mc{P}}^2_{d}(\gamma_-; \gamma_+,\vec{x})}
\def\H{\mc{H}}
\def\ocs{\mc{N}_{h,n,\vec{m}}^{\mathbf{I},\vec{\mathbf{K}}}}
\def\cocs{\overline{\mc{N}}_{h,n,\vec{m}}^{\mathbf{I},\vec{\mathbf{K}}}}
\def\yr{\mathbf{Y}_R}
\def\yl{\mathbf{Y}_L}
\def\yyr#1{\mc{Y}^r_{#1}}
\def\yyl#1{\mc{Y}^l_{#1}}
\def\oc{\mc{OC}}
\def\co{\mc{CO}}

\def\aa{\mathbf{a}}
\def\bb{\mathbf{b}}
\def\bigsp{\ \ \ \ \ \ \ \ \ \ \ \ \ \ \ }

\newcommand*{\longhookrightarrow}{\ensuremath{\lhook\joinrel\relbar\joinrel\rightarrow}}

\newtheorem{lem}{Lemma}[section]
\newtheorem{prop}{Proposition}[section]
\newtheorem{thm}{Theorem}[section]
\newtheorem{cor}{Corollary}[section]

\newtheorem{defn}{Definition}[section]
\newtheorem{claim}{Claim}[section]

\newtheorem{rem}{Remark}[section]

\theoremstyle{remark}

\newtheorem{ex}{Example}[section]
\numberwithin{equation}{section}
\begin{document}
\begin{abstract}
Consider the wrapped Fukaya category $\mathcal{W}$ of a collection of exact
Lagrangians in a Liouville manifold. Under a non-degeneracy condition implying
the existence of enough Lagrangians, we show that natural geometric maps from
the Hochschild homology of $\mathcal{W}$ to symplectic cohomology and from
symplectic cohomology to the Hochschild cohomology of $\mathcal{W}$ are
isomorphisms, in a manner compatible with ring and module structures. This
is a consequence of a more general duality for the wrapped Fukaya category,
which should be thought of as a non-compact version of a Calabi-Yau structure.
The new ingredients are:  (1) Fourier-Mukai theory for $\mathcal{W}$ via a
wrapped version of holomorphic quilts, (2) new geometric operations, coming
from discs with two negative punctures and arbitrary many positive punctures,
(3) a generalization of the Cardy condition, and (4) the use of homotopy units
and A-infinity shuffle products to relate non-degeneracy to a resolution of the
diagonal.
\end{abstract}

\maketitle

\section{Introduction}

It is a conjecture of Kontsevich \cite{Kontsevich:1995uq} (inspired by mirror
symmetry) that the quantum cohomology ring of a compact symplectic manifold $M$
should be isomorphic to the {\it Hochschild cohomology} 
\begin{equation}
    \mathrm{HH}^*(\mc{F}(M))
\end{equation} 
of the Fukaya category $\mc{F}(M)$.  There are at least two strong motivations
for understanding this conjecture.  For one, such an isomorphism would allow
one to algebraically recover quantum cohomology along with its ring structure
from computations of the Fukaya category.  In another direction, Hochschild
cohomology measures deformations of a category, so the conjecture has
implications for the deformation theory of Fukaya categories; see e.g.
\cite{Seidel:2002ys}.

We address a non-compact version of Kontsevich's conjecture, in the setting of
exact (non-compact) symplectic manifolds. The relevant symplectic objects are
{\bf Liouville manifolds}, exact symplectic manifolds with a convexity
condition at infinity. Examples include cotangent bundles, affine complex
varieties, and more general Stein manifolds.  In this setting, there is an
enlargement of the Fukaya category known as the {\bf wrapped Fukaya category}
\begin{equation}
    \w:=\w(M),
\end{equation}
whose objects include properly embedded (potentially non-compact) Lagrangians,
and whose morphism spaces include intersection points as well as {\it Reeb
chords} between Lagrangians at infinity.  The wrapped Fukaya category is
expected to be the correct mirror category to coherent sheaves on non-proper
varieties, see e.g. \cite{Abouzaid:2010ly} \cite{Abouzaid:2011fk}.  Moreover,
it is the {\it open-string}, or Lagrangian, counterpart to a relatively
classical invariant of non-compact symplectic manifolds, {\bf symplectic
cohomology}
\begin{equation} 
    SH^*(M),
\end{equation}
first defined by Cieliebak, Floer, and Hofer \cite{Floer:1994uq}
\cite{Cieliebak:1995fk}.

There are also existing geometric maps from the {\it Hochschild homology}
\begin{equation}
\mathrm{HH}_{*}(\w)
\end{equation}
to symplectic cohomology \cite{Abouzaid:2010kx} and from symplectic cohomology
to the Hochschild cohomology \cite{Seidel:2002ys}
\begin{equation}
    \mathrm{HH}^*(\w).
\end{equation}  
Thus, one can posit that a version of Kontsevich's conjecture holds in this
setting.

In this paper, we prove a non-compact version of Kontsevich's
conjecture for a Liouville manifold $M$ of dimension $2n$, assuming a
non-degeneracy condition for $M$ first introduced by Abouzaid
\cite{Abouzaid:2010kx}.  The reason for the non-degeneracy assumption is
essentially this: to have any hope that symplectic cohomology be recoverable
from the wrapped Fukaya category, it is important that the target manifold
contain ``enough Lagrangians.'' 
\begin{defn} \label{nondegeneracydef}
A finite collection of Lagrangians $\{L_i\}$ is said to be {\bf essential} if
the natural map from Hochschild homology of the wrapped Fukaya category
generated by $\{L_i\}$ to symplectic cohomology hits the {\bf identity}
element. Call $M$ {\bf non-degenerate} if it admits any essential collection of
Lagrangians.
\end{defn}
The non-degeneracy condition is explicitly known for cotangent bundles
\cite{Abouzaid:2010fk} and some punctured Riemann surfaces
\cite{AbAurUnpublished}. In general, we expect that work of Bourgeois, Ekholm,
and Eliashberg \cite{BEE1published}, suitably translated into the setting of
the wrapped Fukaya category, would imply that every Stein manifold is
non-degenerate, with essential Lagrangians given by the ascending co-cores of a
plurisubharmonic Morse function. In a related direction, if $M$ is the total
space of a {\it Lefschetz fibration}, upcoming work of Abouzaid and Seidel
\cite{AbSeInProgress} should also imply that $M$ is non-degenerate, with
essential Lagrangians given by the {\it Lefschetz thimbles} of the fibration.

Our main result can then be stated as follows:
\begin{thm}\label{shhh}
If $M$ is non-degenerate, then the natural geometric maps
\begin{equation}
    \mathrm{HH}_{*-n}(\w) \stackrel{[\oc]}{\lra} SH^*(M) \stackrel{[\co]}{\lra} \mathrm{HH}^*(\w)
\end{equation}
are both isomorphisms, compatible with Hochschild ring and module structures.
\end{thm}

One consequence of Theorem \ref{shhh} is that if $M$ is non-degenerate, the
Hochschild homology and cohomology of $\w$ are isomorphic, up to a shift.
In fact, a key step in proving Theorem \ref{shhh} is first constructing a direct
geometric Poincar\'{e} duality isomorphism
\begin{equation} \label{duality1}
    \mathrm{HH}_{*-n}(\w,\mc{B}) \stackrel{\sim}{\lra} \mathrm{HH}^*(\w,\mc{B}),
\end{equation}
for an arbitrary coefficient bimodule $\mc{B}$, which does not pass through
$\mathrm{SH}^*(M)$. Such dualities have appeared before in the context of the
algebraic geometry of smooth varieties. Van den Bergh \cite{bergh:1998uq}
\cite{bergh:2002fk} was the first to observe a duality between Hochschild
homology and cohomology for the coordinate ring of a {\it smooth Calabi-Yau}
affine variety; see also \cite{Krahmer:2007kx}. The relevant notion for us is a
categorical, or non-commutative, version of {\it smooth and Calabi-Yau},
generalizing perfect complexes on a smooth Calabi-Yau variety.  As in the
algebro-geometric setting, {\it smoothness} is the prerequisite property that
must be defined first.

\begin{defn}[Kontsevich-Soibelman \cite{Kontsevich:2006fk}]
    An $\ainf$ category $\mc{C}$ is {\bf (homologically) smooth} if its
    diagonal bimodule is {\bf perfect}, i.e., lies in the category split
    generated by tensor products of left and right Yoneda modules.
\end{defn}
\noindent As a key technical ingredient in Theorem \ref{shhh}, we relate
smoothness to the non-degeneracy condition.
\begin{thm}\label{smoothness}
    If $M$ is non-degenerate, then $\w$ is smooth.
\end{thm}

\begin{defn}
    An  $\ainf$ category $\mc{C}$ is  a {\bf non-compact Calabi-Yau category}
    if it is smooth and there is a Poincar\'{e} duality-type
    natural transformation 
    \begin{equation}
        \mathrm{HH}_{*-n}(\mc{C},\mc{B}) \stackrel{\sim}{\lra} \mathrm{HH}^{*}(\mc{C},\mc{B})
    \end{equation}
    of functors from bimodules to chain complexes, inducing isomorphisms on
    homology. Such a natural transformation should be induced by an equivalence
    of bimodules
    \begin{equation}
        \mc{C} \stackrel{\sim}{\lra} \cc^![n].
    \end{equation}
    Here, $\cc^!$ is a perfect bimodule which represents, via tensoring,
    Hochschild cohomology.
    \begin{equation}
        H^*(\cc^! \otimes_{\cc\!-\!\cc} \mc{B}) \cong \mathrm{HH}^*(\mc{C},\mc{B})
    \end{equation}
\end{defn}
The bimodule $\cc^!$, defined in Section \ref{moduleduality}, is known as the
{\bf inverse dualizing bimodule}.  The {\it non-compact Calabi-Yau} terminology
was introduced by Kontsevich and Soibelman \cite{Kontsevich:2006fk} as a
categorical abstraction of perfect complexes on a smooth, not necessarily
proper Calabi-Yau variety. 

With these definitions in place, we can state our second main result.
\begin{thm}[Duality for the wrapped Fukaya category]  \label{wrapcy}
Suppose $M$ is non-degenerate. Then, $\w$ is smooth and there is
a geometric map 
\begin{equation}\label{cymorphism}
    \mc{CY}:\w \stackrel{\sim}{\lra} \w^![n]
\end{equation}
giving $\w$ the structure of a {\bf non-compact Calabi-Yau category}. 
\end{thm}
The geometric map (\ref{cymorphism}) will be constructed from a new family of
operations on the wrapped Fukaya category, parametrized by discs with multiple
inputs and two (non-adjacent) outputs. The first-order term, for example, is a
map
\begin{equation}
    \hom_{\w}(A,B) \ra \hom_{Vect}(\hom_{\w}(C,D), \hom_{\w}(A,D) \otimes \hom_{\w}(C,B))
\end{equation}
induced by a single disc with four boundary marked points, alternating as
inputs and outputs, with a fixed cross ratio.  A simplified version of these
operations, in which the two outputs are adjacent, is used in a crucial way in
\cite{Abouzaid:2010kx}.
\begin{rem}
In contrast to the compact setting, symplectic cohomology and morphisms in the
wrapped Fukaya category lack non-degenerate pairings, meaning one cannot trade
inputs for outputs in Floer theoretic operations. From the perspective of
topological field theory, the surfaces giving rise to operations on $\w$ must
have at least one output. Thus, the two-output operations considered here and
in \cite{Abouzaid:2010kx} give new operations on $\w$.
\end{rem}


In order to deduce Theorem \ref{shhh} from Theorem \ref{wrapcy},  we invoke a
holomorphic-curve version of the {\it Cardy condition} from topological field
theory:
\begin{thm}[Generalized Cardy Condition]\label{cardythm}
    There is a (homotopy)-commutative diagram
    \begin{equation}\label{cardy}
        \xymatrix{\r{HH}_{*-n}(\w,\w) \ar[r]^{ [\mc{CY}_\#]} \ar[d]^{[\oc]} & \r{HH}_*(\w,\w^!) \ar[d]^{[\bar{\mu}]} \\
        SH^*(M) \ar[r]^{[\co]\ \ \ \ \ }& \r{HH}^*(\w,\w)}
\end{equation}
Here $[\mc{CY}_\#]$ is the map induced from (\ref{cymorphism}) on the level of
Hochschild homology, and $\bar{\mu}$ is a map induced by compositions of
$\ainf$ multiplication. 
\end{thm}
The proof follows from analyzing the degenerations of operations parametrized
by families of annuli with marked points on both boundary components.

The map $[\bar{\mu}]$, a special case of the map appearing in the
identification 
\[
H^*(\w^! \otimes_{\w\!-\!\w} \mc{B}) \simeq \r{HH}^*(\w,\mc{B}), 
\] 
is an isomorphism whenever $\w$ is smooth. So (\ref{cardy}) provides a
commutative diagram in which the top and right maps are isomorphisms. Theorem
\ref{wrapcy}, which concerns the bottom and left maps, then follows from the
observation that
\begin{prop}\label{surjectivity}
    If $M$ is non-degenerate, then $[\oc]$ is surjective and $[\co]$ is injective.
\end{prop}
\noindent Proposition \ref{surjectivity} is a simple consequence of the
compatibility of $[\oc]$ and $[\co]$ with algebraic structures: $[\co]$ is a
morphism of rings, and $[\oc]$ is a morphism of $SH^*(M)$ modules, using the
map $[\co]$ and the module structure of Hochschild homology over cohomology. As
an alternative to Proposition \ref{surjectivity}, our proof of Theorem
\ref{wrapcy} also directly establishes that $[\co]$ is an isomorphism,
which suffices.

As one application of this work, whenever $M$ is non-degenerate, we find two
methods for explicitly recovering the ring structure on $SH^*(M)$ from the
wrapped Floer theory of any essential collection of Lagrangians: one directly
from Hochschild cohomology using its product, and one from Hochschild homology
using $\ainf$ structure maps and the new geometric operations $\mc{CY}$. In
the case $M$ is Stein, formulae formally similar to the latter description have
appeared in \cite{BEE2}, phrased in terms of the symplectic field theory
invariants of attaching Legendrians of a suitable handle decomposition.  Taking
our formulae applied to the (conjecturally essential) collection of ascending
Lagrangian co-cores for the same decomposition, it is natural to expect that
the underlying algebraic structures should be related.

\begin{rem}
Going further than the product, one might hope to recover all field-theoretic
operations on symplectic cohomology from the Hochschild homology of the wrapped
Fukaya category, at least in the sense of \cite{Costello:2006vn},
\cite{Lurie:2009fk}, or \cite{Kontsevich:uq}. We have not attempted to do so
here, but have two remarks. First, the work of \cite{Costello:2006vn} is not
immediately applicable, as it is formulated under different hypotheses---see
instead \cite{Lurie:2009fk}*{Remark 4.2.17} for a version relevant to the
wrapped Fukaya category.  Second, in order to cite a theorem from topological
field theory, there is an additional criterion to check: the Hochschild cycle
whose image is $1 \in SH^*(M)$ should be a homotopy fixed point the $S^1$
action on the Hochschild complex (the one inducing the Connes' $B$ operator).
We leave this for future work.
\end{rem}

Theorem \ref{wrapcy} will follow from the construction and analysis of
a natural functor $\mathbf{M}$, coming from operations parametrized by {\it
quilted strips}, from the wrapped Fukaya category of the product $M^- \times M$
to {\it bimodules} over $\w$.  
  Quilted strips have been used in the
compact setting \cite{Wehrheim:2010fk} \cite{Wehrheim:2010uq} to associate a
functor between the (cohomology-level) Fukaya categories of $M$ and $N$ to a
(monotone) Lagrangian correspondence $L \subset M^- \times N$ (an extension to
the chain level is forthcoming \cite{MWW1}). Using a variant of the original
parameter spaces appearing in \cite{Mau:2010lq}, we give a direct chain-level
construction of a version of this functor for the wrapped Fukaya category.  The
fullness of this functor, which contains both open-closed operations and
morphisms with two output punctures in its definition, will have implications
for the isomorphism (\ref{cymorphism}).  
The need for such a construction partially explains the length of this paper.
The idea that moduli spaces of quilts involving the diagonal $\Delta$ should be
related to open-closed string maps is not new---see e.g.,
\cite{Abouzaid:2010vn}. However, to our knowledge, our exploitation of the
relationship between $\mathbf{M}$ and the morphism $\mc{CY}$ is new. 

In more detail, we define an $\ainf$ category 
\[ 
    \w^2
\]
associated to the product $M^- \times M$, with two types of objects: products
$L_i \times L_j$ of objects in $\w$, and the diagonal $\Delta$. For purely
technical reasons, we do not call this the wrapped Fukaya category of the
product, as compositions in this category are defined using split Hamiltonians,
which are not in the usual class of admissible Hamiltonians on the Liouville
manifold $M^- \times M$, see e.g., \cite{Oancea:2006oq}. However, this category
has the property (reflecting an observation initially due to Floer
\cite{Floer:1988kx}) that 
\begin{equation}\label{deltash}
    H^*(\hom_{\w^2}(\Delta,\Delta)) = \sh^*(M).
\end{equation}

Then, we construct an $\ainf$ functor
\begin{equation}
    \mathbf{M}: \w^2 \lra \w\bimod\w
\end{equation}
from $\w^2$ to the category of bimodules over $\w$, satisfying the following
properties: on the level of objects, a product Lagrangian $K \times L$ is
sent to the corresponding tensor product of left and right Yoneda modules $\mc{Y}^l_K
\otimes \mc{Y}^r_L$, and the diagonal $\Delta$ is sent to the diagonal bimodule
$\w_\Delta$. Moreover, on morphism complexes, 
\begin{equation}
    H^*(\mathbf{M}^1): H^*(\hom_{\w^2}(\Delta,\Delta)) \lra H^*(\hom_{\w\bimod\w}(\w_{\Delta},\w_{\Delta}))
\end{equation}
agrees with the map $[\co]$ between symplectic cohomology and Hochschild
homology. Similarly, we show the first order part of the map (\ref{cymorphism})
appears in the functor $\mathbf{M}$ applied to the morphism complex
$\hom_{\w^2}(\Delta,K \times L)$. The functor $\mathbf{M}$ we construct has the
feature (unlike versions that use a more general class of Lagrangian
correspondences and varying width strips) that no really new analysis is used
in its construction beyond that which appears in the construction of
open-closed operations and (\ref{cymorphism}). Also, by using bimodules as a
target instead of functors, we avoid passing to a wrapped version of the {\it
extended Fukaya category} \cite{Wehrheim:2010fk} whose objects consist of
sequences of Lagrangian correspondences (this was suggested to
us by Ma'u \cite{Mau:2010lq}). Thus, it can be thought of as a lightweight
construction of the quilt functor. 

We establish the following crucial fact about the functor $\mathbf{M}$, which
follows from a variant of the Yoneda lemma for bimodules:
\begin{prop}\label{splitfull}
    $\mathbf{M}$ is full on the full sub-category of $\w^2$ consisting of
    product Lagrangians.  
\end{prop}
Thus, Theorem \ref{smoothness}, the quasi-isomorphism in Theorem
\ref{wrapcy} (and indeed, more directly the statement that $\co$ is an
isomorphism) are consequences of the following:
\begin{thm}[Resolution of the diagonal]\label{resdiag}
    Suppose $M$ is non-degenerate, with essential collection $\mc{L}$. Then in
    the category $\w^2$, $\Delta$ is split-generated by products of objects in
    the collection $\mc{L}$. By Proposition \ref{splitfull}, $\mathbf{M}$ is
    full on $\w^2$.  
\end{thm}

Theorem \ref{resdiag}, the most technical of our results, involves
checking a {\it split-generation} criterion introduced into the symplectic
literature by Abouzaid \cite{Abouzaid:2010kx}: In $\w^2$, $\Delta$ is
split-generated by the subcategory $\w^2_{split}$ consisting of product
Lagrangians if and only if the cohomology-level unit $1 \in
H^*(\hom_{\w^2}(\Delta,\Delta))$ is in the image of a natural map from a
certain bar complex involving $\Delta$ and $\w^2_{split}$: 
\begin{equation}
    \label{hitunitdelta1}
    [\mu]: H^*(\mc{Y}^r_{\Delta} \otimes_{\w^2_{split}} \mc{Y}^l_{\Delta}) \lra H^*(\hom_{\w^2}(\Delta,\Delta)).
\end{equation}
As part of the proof of Theorem \ref{resdiag}, we give a natural chain-level
comparison map \begin{equation} \label{gammamap1}
    [\Gamma]: \rm{HH}_*(\w,\w) \lra H^*(\mc{Y}^r_{\Delta} \otimes_{\w^2_{split}}
    \mc{Y}^l_{\Delta})
\end{equation}
between the Hochschild homology and the bar complex in (\ref{hitunitdelta1}).
We also prove that it intertwines the non-degeneracy condition and the
split-generation condition, so there is a commutative diagram
\begin{equation} 
    \label{shufflediag}
    \xymatrix{\r{HH}_*(\w,\w) \ar[r]^{[\Gamma]\ \ \ }\ar[d]^{[\oc]} & H^*(\y^r_\d \otimes_{\w^2_{split}}\y^l_\d) \ar[d]^{[\mu]} \\
    SH^*(M) \ar[r]^{\sim\ \ \ } & H^*(\hom_{\w^2}(\Delta,\Delta))}
\end{equation}
which translates the non-degeneracy condition into (\ref{hitunitdelta1}).  This
comparison map from morphisms and $\ainf$ structure maps in $\w$ to those in
$\w^2$ involves an $\ainf$ generalization of the {\it Eilenberg-Zilberg shuffle
product} map from algebraic topology (see e.g., \cite{Loday:1992fk}*{\S 4.2}),
which involves summing over all possible ways to insert units. In order for
such a map to be a chain map in the $\ainf$ setting, we are led to
geometrically strictify units in both $\w$ and $\w^2$, using an implementation
of the {\it homotopy units} of \cite{Fukaya:2009ve} \cite{Fukaya:2009qf}. This
construction, which may be of independent interest, is the second reason for
the length of this paper.

\subsection{Contents of paper}

In Section \ref{algebrasection}, we collect necessary facts about $\ainf$
categories, modules and bimodules. These include definitions of these objects,
constructions of differential graded categories associated to modules and
bimodules, and a discussion of various tensor products associated to modules
and bimodules. We recall definitions of the Hochschild co-chain and chain
complexes, along with the ring and module structures on these complexes.  In
fact, we give two different chain-level models for each complex, one that seems
more common in the symplectic literature, and a quasi-isomorphic one
coming from categories of bimodules. We discuss canonical modules and bimodules
coming from the Yoneda embedding, and the notion of {\bf homological
smoothness}. We discuss the notion of {\bf split-generation} by a subcategory
and recall a criterion for split-generation. Finally, we introduce operations
of duality for modules and bimodules, in order to define a {\bf dual bimodule}
\begin{equation}
    \mc{B}^!
\end{equation}
associated to any bimodule $\mc{B}$.  Specializing to the diagonal bimodule
$\cc_{\Delta}$ over an $\ainf$ category
$\cc$, we obtain the so-called {\bf inverse dualizing bimodule}
\begin{equation}
    \cc^! := (\cc_{\Delta})^!.
\end{equation}
We prove that, assuming $\cc$ is smooth, tensoring a bimodule with $\cc^!$
amounts to taking the Hochschild cohomology of the bimodule, via a natural
comparison map
\begin{equation}
    [\mu_{\mc{B}}]: H^*(\cc^! \otimes_{\cc\!-\!\cc} \mc{B}) \stackrel{\sim}{\longrightarrow} \r{HH}^*(\cc,\mc{B}),
\end{equation}
an $\ainf$-categorical generalization of a result due to Van den Bergh
\cite{bergh:1998uq} \cite{bergh:2002fk}.

In Sections \ref{chcwsection}, \ref{ocfloersection}, \ref{openclosedmaps}, and
\ref{unstableoperations} we construct and prove various facts about the main
geometric players in Theorem \ref{shhh}. First, in Section \ref{chcwsection},
we present our geometric setup, recalling the notions of {\bf Liouville
manifolds}, {\bf symplectic cohomology}, and {\bf wrapped Floer cohomology}. In
Section \ref{ocfloersection}, we introduce the moduli space of {\bf genus 0
open-closed strings}, which are genus 0 bordered surfaces with signed boundary
and interior marked points equipped with an additional framing around each
interior point, restricting to at most one interior output or two boundary
outputs. We define Floer data for such moduli spaces, 
and construct Floer-theoretic operations for submanifolds of these moduli
spaces. 
In Section \ref{openclosedmaps}, we consider operations induced by various
families of submanifolds. For spheres with two inputs, we obtain the {\bf
product} on symplectic cohomology and from discs with arbitrarily many inputs
and one output, we construct {\bf $\ainf$ structure maps} for the {\bf wrapped
Fukaya category}.
Then, using families of discs with boundary and interior marked points, we
define the open-closed maps 
\begin{equation}
    [\oc]: \r{HH}_{*-n}(\w,\w) \lra SH^*(M)
\end{equation}
from Hochschild homology to symplectic cohomology and
\begin{equation} 
    [\co]: SH^*(M) \lra \r{HH}^*(\w,\w)
\end{equation}
from symplectic cohomology to Hochschild cohomology. Actually we define two
chain-level versions each of these open-closed maps, one for each of our two
explicit Hochschild chain and co-chain complexes from Section
\ref{algebrasection}, and prove that they are homotopic.
Finally, we prove some basic facts about the open-closed maps: the map $[\co]$ is
a morphism of rings, giving $\r{HH}_{*-n}(\w,\w)$ the structure of a module
over $SH^*(M)$ via the existing module structure of $\r{HH}_{*}(\w,\w)$ over
$\r{HH}^*(\w,\w)$. With respect to this induced module structure, we show that
the map $[\oc]$ is a morphism of $SH^*(M)$ modules. This immediately implies
Proposition \ref{surjectivity}.
\noindent Also, in Section \ref{unstableoperations}, we recall operations
arising from unstable surfaces, the identity morphism and homology unit. 

In Section \ref{pairs}, a core technical section, we introduce Floer-theoretic
operations parametrized by new abstract moduli spaces of {\bf pairs of
partially glued discs} modulo simultaneous automorphisms. We begin with the
moduli space of {\bf pairs of discs}, a moduli space that not identical to the
product of the moduli space of discs with itself, which arises as a further
quotient by relative automorphisms of factors. We construct a compactification
of these spaces, along the way examining natural subspaces where various
boundary points are identified. We define a {\bf partial gluing operation}, on
loci where various boundary components between factors coincide, from pairs of
discs to genus 0 open-closed strings.  Finally, we define Floer data and
operations for families of such partially glued pairs of discs.  This
construction allows us in subsequent sections to rapidly construct models for
Floer theory on the product and the quilt functor, using operations entirely in
$M$.

In Section \ref{productsection}, we move to the product manifold 
\begin{equation}
    M^- \times M,
\end{equation}
a symplectic manifold which contains two natural classes of Lagrangians, {\bf
split Lagrangians} of the form $L_i \times L_j$ and the {\bf diagonal}
$\Delta$. In general, there are technical issues defining wrapped Floer theory
and symplectic cohomology of products, coming from a technical point that in
order to define Floer theoretic invariants for $M$, we have fixed a choice of
contact boundary at infinity, and such a choice does not determine a choice of
contact boundary at infinity for the product (rather, it defines a
``normal-crossings contact manifold''). See for example \cite{Oancea:2006oq}
for a solution in the case of symplectic cohomology.  Instead of solving these
issues, in Section \ref{productsection} we use split Hamiltonians to reduce the
moduli spaces of discs in $M^- \times M$ (for Lagrangians in $\ob \w^2$) to
glued pairs of discs in $M$, producing a model 
\begin{equation}
    \w^2
\end{equation}
for the wrapped Fukaya category of $M^- \times M$. The reason we needed {\it glued}
pairs of discs and not disjoint pairs of discs comes of course from the
presence of the diagonal, beginning with a classical observation that the
self-Lagrangian Floer cohomology of the diagonal $\Delta \subset M^- \times M$
computes the ordinary Floer cohomology of $M$.  

In Section \ref{quiltsection}, we study spaces of {\bf quilted strips}, first
introduced by Ma'u \cite{Mau:2010lq} as a variant of the quilted surfaces of
Wehrheim-Woodward and Ma'u-Wehrheim-Woodward \cite{Wehrheim:2010fk}
\cite{Mau:2010fk}.  Using an embedding from labeled quilted strips with seam
labels in $\ob \w^2$ into glued pairs of discs, we obtain Floer-theoretic
operations, which we show in Proposition
\ref{quiltainffunctor} give an $\ainf$ functor
\begin{equation} \label{mfunctor1}
    \mathbf{M}: \w^2 \lra \w\bimod\w
\end{equation} 
where $\w\bimod\w$ is the category of bimodules over $\w$.  This is a bimodule
version of a functoriality result in the sense of \cite{Wehrheim:2010fk} for
the wrapped Fukaya category (this version formulated by \cite{Mau:2010lq}). The
relation to open-closed operations is this: we see that the diagonal $\Delta$
is sent to the diagonal bimodule $\w_\Delta$, and the first order term
\begin{equation}
    [\mathbf{M}^1]: SH^*(M) = HW^*(\Delta, \Delta) \lra \hom_{\w\bimod\w}(\w_{\Delta},\w_{\Delta}) \simeq \r{HH}^*(\w,\w)
\end{equation}
is exactly the closed-open map $[\co]$ (Proposition \ref{m1equalsco}---in fact,
it agrees on the chain level with the second chain-level closed-open map we
defined). We verify that $\mathbf{M}$ sends split Lagrangians to
tensor products of Yoneda modules, and hence, by a version of the Yoneda lemma
for bimodules, $\mathbf{M}$ is always {\it full} on split Lagrangians
(Proposition \ref{splitfull}).  A related observation was first made by
Abouzaid-Smith \cite{Abouzaid:2010vn}, in the compact setting for an analogous
functor with target category endofunctors of {\it extended Fukaya categories}
(instead of bimodules).

In Section \ref{homotopyunits}, another technical section, we introduce the
tool of {\bf homotopy units}, first constructed geometrically by Fukaya, Oh,
Ohta, and Ono \cite{Fukaya:2009qf} in the compact setting. Homotopy units allow
one to geometrically strictify units in an $\ainf$ category, which otherwise
only exist on the homology level.  A basic consequence of homotopy units is
that one can talk about Floer-theoretic operations on the Fukaya category
induced from families of surfaces by forgetting boundary marked points.  Thus,
we first explore and construct Floer theoretic operations with forgotten
boundary points (Section \ref{forgottensubsection}). Then, we develop
homotopies between such Floer theoretic operations and ones in which we had
glued in geometric units instead of forgetting. In order to obtain a
quasi-isomorphic category, we must, as in \cite{Fukaya:2009qf}, construct
operations corresponding to all possible higher homotopies. When done for glued
pairs of discs, (only allowing ourselves to forget boundary marked points with
asymptotic to morphisms between product Lagrangians), we obtain a category
quasi-isomorphic to $\w^2$. The resulting category, which we call
\begin{equation} \label{tw}
    \tilde{\w}^2
\end{equation}
is identical to $\w^2$, except that its morphism spaces contain additional
formal elements of the form $e^+ \otimes x$, $x \otimes e^+$, which we term
{\bf one-sided homotopy units}, (there are also formal elements $f \otimes x$,
$x\otimes f$, corresponding to the homotopy between the geometric unit and
$e^+$). $\ainf$ operations on $\tilde{\w}^2$ involving $e^+$ elements 
involve forgetful maps applied to boundary marked points, and operations
involving $f$ elements correspond to homotopies between forgetful maps and
geometric units. A crucial observation made in Proposition
\ref{shuffleidentities} asserts that $\ainf$ operations on $\tilde{\w}^2$
reduce to operations on $\w$ when some of the inputs are one-sided homotopy
units.

In Section \ref{splitgendeltasection}, we use the tools developed in the
previous section to prove that if $M$ is non-degenerate, then $\Delta$ is
split-generated by product Lagrangians in $\w^2$ (the contents of Theorem
\ref{resdiag}). This, along with the existence of the functor $\mathbf{M}$
which is full on product Lagrangians (Proposition \ref{splitfull}) implies
$\mathbf{M}$ is full on $\Delta$, and in particular $\w$ is smooth (Theorem
\ref{smoothness}).  As an immediate
consequence, the map from $SH^*(M) \lra \r{HH}^*(\w,\w)$, which we have shown
is a part of the data of $[\mathbf{M}^1]$, is an isomorphism, proving part of
Theorem \ref{shhh}.  To prove split-generation of $\Delta$, we construct the
map $\Gamma$ described in (\ref{gammamap1}) between the Hochschild homology of
of $\w$ and a {\it bar complex} appearing in an algebraic split-generation
criterion for $\Delta$ in $\tilde{\w}^2$ (see (\ref{tw})), discussed in Section
\ref{algebrasection}. The map $\Gamma$, an $\ainf$ version of the {\it shuffle
product}, uses the formal elements $e^+ \otimes x$ and $x \otimes e^+$ in an
essential way, and is a chain map intertwining the natural maps in the {\bf
non-degeneracy} and {\bf split-generation} criteria as in (\ref{shufflediag})
thanks to identities involving operations in $\tilde{\w}^2$ applied to formal
units, Proposition \ref{shuffleidentities}.  As a consequence, we deduce that
product Lagrangians split generate $\Delta$, whenever $M$ is non-degenerate.

Since $\w$ is smooth, by Section \ref{algebrasection}, the
bimodule $\w^!$ is perfect and represents Hochschild cohomology.  In Section
\ref{wtowshrieksection}, we construct a geometric morphism of bimodules
\begin{equation}
    \mc{CY}: \w_{\Delta} \lra \w^![n],
\end{equation}
coming from new operations controlled by discs with two negative punctures and
arbitrary positive punctures.  To first order, we show this map agrees on the
chain level with the first order term of the quilt functor 
\[
\mathbf{M}^1: \hom_{\w^2}(\Delta,K \times L) \lra \hom_{\w\!-\!\w}(\w_{\Delta},\mc{Y}^l_K \otimes \mc{Y}^r_K)
\]
constructed in Section \ref{quiltsection} (Proposition \ref{mcy}).
This identification, which implies Theorem \ref{wrapcy}, relies on the
identities for $\w^2$:
\begin{align}
    \label{fpi} \hom_{\w^2}(\Delta,K\times L) &\simeq \hom_{\w}(K,L)[-n],\\
    \hom_{\w^2}(K\times L,\Delta) &\simeq \hom_{\w}(L,K),
\end{align}
(see Propositions \ref{fundamentalproductidentity} and
\ref{categoricalproductidentity}). The latter identity has a categorical shadow
(Proposition \ref{homintodiagonal}), but the former is very special to
categories expected to be non-compact Calabi-Yau. For example, we do not expect
Fukaya-Seidel categories of products of Lefschetz fibrations to satisfy
(\ref{fpi}).

In Section \ref{cardyconditionsection}, we analyze operations coming from
spaces of annuli with many positive boundary marked points on both boundaries
and one negative boundary marked point on the outer boundary. 
We show that degenerations of a codimension 1 family of these annuli give a
relation between $\mc{CY}$ and standard open-closed maps, proving Theorem
\ref{cardythm}.
In the relevant diagram (\ref{cardy}), we have shown in Section
\ref{algebrasection} that if $\w$ is smooth, then the vertical map $\bar{\mu}$
is always a quasi-isomorphism.  We have also shown in Sections
\ref{wtowshrieksection} and \ref{splitgendeltasection} that $[\co]$ and
$[\mc{CY}_\#]$ are isomorphisms.  This implies that $[\oc]$ is also a
quasi-isomorphism, completing the proof of Theorem \ref{shhh}.

Finally, in Section \ref{consequencessection}, we explore a few basic
consequences of this work. For one, we establish in Section
\ref{converseresult} a converse result that if $\Delta$ is split-generated by
product Lagrangians, then $M$ is non-degenerate, so Theorem \ref{shhh}
continues to apply. Then, furthering the relation between Theorem \ref{wrapcy}
and Poincar\'{e} duality, we demonstrate in Section \ref{fundclass} that the
unique pre-image $\sigma \in \r{HH}_*(\w,\w)$ of $1 \in SH^*(M)$, which we call
a {\bf fundamental class}, satisfies
\begin{equation}
    \cdot \cap \sigma: \r{HH}^*(\w,\w) \stackrel{\sim}{\lra} \r{HH}_{*-n}(\w,\w).
\end{equation}
As an application of this circle of ideas, in Section \ref{ringstructure} we
give an explicit formula for the $SH^*(M)$ pair of pants product on Hochschild
homology, using only the operations given by $\mc{CY}$ and the $\ainf$
structure maps of $\w$.

In Appendix \ref{actionsection}, we prove a compactness result for for
the moduli spaces controlling our Floer-theoretic operations, which may be of
independent interest.  Such a result is necessary as we are considering
spaces of maps into a non-compact target $M$.  In order to apply standard
Gromov compactness results, one first must show that spaces of maps with fixed
asymptotics are {\it a priori bounded} in the target. Such results for the
wrapped Fukaya category have generally used a convexity argument
\cite{Abouzaid:2010ly} \cite{Abouzaid:2010kx} or maximum principles
\cite{Seidel:2010fk}, which rely in a strong way on the Hamiltonian flow used
having a rigid form at infinity.  However, to construct higher operations on
symplectic cohomology, we have found it necessary to introduce small time and
surface dependent perturbations that are not of this rigid form.  Our solution
fuses standard maximum principle and convexity arguments with a
slight strengthening of the maximum principle for cylindrical regions of the
source surface, dating back to work of Floer-Hofer \cite{Floer:1994uq} and
Cieliebak \cite{Cieliebak:1994fk}, and also used successfully by Oancea
\cite{Oancea:2008fk}.

Lastly, in Appendix \ref{orientationsection}, we discuss ingredients necessary
to construct all of our operations with appropriate signs over $\Z$ (or
alternatively a field of characteristic other than $2$). We begin the appendix
with a discussion of orientation lines, recall relevant results that orient
moduli spaces of maps, give orientations for the abstract moduli spaces we use,
and demonstrate the theory with a complete calculation of signs in an example.

\subsection*{Conventions} We will for simplicity work over a field
${\mathbb K}$ of arbitrary characteristic (though our results seem valid more
generally, e.g.,  over $\Z$). If $f$ denotes any chain map, we will use $[f]$
to denote the resulting homology-level map.

\subsection*{Acknowledgements}
I would like to thank Mohammed Abouzaid and Denis Auroux (my Ph.D.
advisor) for all of their help with this work and for suggesting this line of
inquiry. As should be evident, the work here has been greatly influenced by the
paper \cite{Abouzaid:2010kx}.  I am also grateful to Yakov Eliashberg and Paul
Seidel for many enlightening discussions, Sikimeti Ma'u for explaining her work
\cite{Mau:2010lq} on quilted strips, and Jonathan Bloom for comments on an
earlier draft. Lastly, I would like to thank the organizers
and attendees of the 2009 AIM Conference on Cyclic Homology and
Symplectic Topology for a stimulating atmosphere that directly inspired much of
this work.  This work was partly completed with the support of an NSF Graduate
Fellowship.

\section{Algebraic preliminaries} \label{algebrasection}

We give an overview of the algebraic technology appearing in this paper:
$\ainf$ categories, functors, modules, bimodules, and Hochschild homology and
cohomology.  We also recall some useful but slightly more involved algebraic
details: the $\ainf$ Yoneda embedding, the K\"{u}nneth formula for split
bimodules, pullbacks of modules and bimodules along functors, and ring/module
structures on Hochschild groups. We also introduce notions of {\bf module}
and {\bf bimodule duality}, in order to ultimately define a natural bimodule
$\cc^!$ associated to any $\ainf$ category $\cc$. None of the material is
completely new, although some of it does not seem to have appeared in the
$\ainf$ or symplectic context.

\subsection{\texorpdfstring{$\ainf$}{A-infinity} algebras and categories} 

\begin{defn} 
    An {\bf $\ainf$ algebra} $\A$ is a graded vector space $\mc{A}$
    together with maps 
    \begin{equation}
        \mu^s_\A: \mc{A}^{\otimes s} \ra \mc{A},\  s\geq 1
    \end{equation} 
    of degree $2-s$ such that the following quadratic relation holds, for each
    $k$:
\begin{equation} 
    \label{ainfeq} \sum_{i,l} (-1)^{\maltese_i} \mu^{k-l+1}_\A(x_k, \ldots, x_{i+l+1}, \mu^{l}_\A(x_{i+l}, \ldots, x_{i+1}), x_i, \ldots, x_1) = 0.
\end{equation}
where the sign is determined by
\begin{equation}
    \maltese_i := |x_1| + \cdots + |x_i| - i.
\end{equation}
\end{defn}
\begin{rem}
    The parity of $\maltese_i$ is the same as the sum of the {\bf reduced
    degrees} $\sum_{j=1}^i ||x_j||$.  Here $||x_j|| = |x_j| - 1$ is the
    degree of $x_j$ thought of as an element of the shifted vector space
    $\mc{A}[1]$. Thus, $\maltese_i$ can be thought of as a Koszul-type sign
    arising as $\mu^l$ acts from the right.
\end{rem}

The first few $\ainf$ relations are, up to sign:
\begin{align}
\mu^1(\mu^1(x)) &= 0\\
 \mu^1(\mu^2(x_0,x_1)) &= \pm \mu^2(\mu^1(x_0),x_1) \pm \mu^2(x_0,\mu^1(x_1)) \\
 \pm\mu^2(\mu^2(x_0,x_1),x_2) \mp \mu^2(x_0,\mu^2(x_1,x_2)) &= \mu^1(\mu^3(x_0,x_1,x_2)) \nonumber \pm \mu^3(\mu^1(x_0),x_1,x_2) \\
 &\pm \mu^3(x_0,\mu^1(x_1),x_2) \pm \mu^3(x_0,x_1,\mu^1(x_2)) 
\end{align}
In particular, the first few equations above imply that
        $\mu^1$ is a differential, 
        (up to a sign change) $\mu^2$ descends to a product on $H^*(\A,\mu^1)$, and 
        the resulting homology-level product $H^*(\mu^2)$ is associative.
$\mu^3$ can be thought of as the associator, and the other $\mu^k$ are higher
homotopies for associativity.

One can also recast the notion of an $\ainf$ algebra in the following way: Let 
\begin{equation}
    T\A[1] = \bigoplus_{i > 0} \A[1]^{\otimes i}
\end{equation}
be the tensor co-algebra of the shifted $\A[1]$. 

Given any map $\phi: T\A[1] \ra \A[1]$, there is a unique so-called {\bf hat extension}
\begin{equation}
    \hat{\phi}: T\A[1] \ra T \A[1]
\end{equation}
specified as follows:
\begin{equation}
    \label{extendphi}
    \hat{\phi}(x_k\otimes \cdots \otimes x_1) := \sum_{i,j} (-1)^{\maltese_i} x_k
    \otimes \cdots \otimes x_{i+j+1} \otimes \phi^{j}(x_{i+j}, \ldots,
    x_{i+1})\otimes x_{i}\otimes \cdots \otimes x_1.  
\end{equation}
The (shifted) $\ainf$ operations $\mu^i$ fit together to form a map 
\begin{equation}
    \mu: T\A[1] \ra \A[1]
\end{equation}
of total degree 1.  Then the $\ainf$ equations, which can be re-expressed as
one equation
\begin{equation}
    \mu \circ \hat{\mu} = 0,
\end{equation}
are equivalent to the requirement that $\hat{\mu}$ is a differential on $T
\A[1]$ \begin{equation}
    \hat{\mu}^2 = 0.
\end{equation}

\begin{rem}
    Actually, the hat extension $\hat{\phi}$ defined above is the unique
    extension satisfying the graded {\it co-Leibniz} rule with respect to the
    natural co-product $\Delta : T\A[1] \lra T\A[1] \otimes T\A[1]$, given by
    \begin{equation}
        \Delta(x_n \otimes  \cdots \otimes x_1) = \sum_i (x_n \otimes \cdots \otimes x_{i+1}) \otimes (x_{i} \otimes \cdots \otimes x_1).
    \end{equation}
    In this way, the association of $(T\A,\hat{\mu})$ to the $\ainf$ algebra
    $(\A,\mu)$ gives an embedding of $\ainf$ algebra structures on a vector
    space to differential-graded co-algebra structures on the tensor algebra
    over that vector space. The chain complex $(T\A, \hat{\mu})$ is called the
     {\bf bar complex} of $\A$.
\end{rem}

The discussion so far generalizes in a straightforward manner to the
categorical setting.
\begin{defn} 
    An {\bf $\ainf$ category} $\mc{C}$ consists of the following data:
\begin{itemize}
\item a collection of objects $\ob \mc{C}$

\item for each pair of objects $X,X'$, a graded vector space $\hom_\cc (X,X')$

\item for any set of $d+1$ objects $X_0, \ldots, X_d$, higher composition maps
    \begin{equation}\label{compositionmaps}
        \mu^d: \hom_\cc (X_{d-1},X_d) \times \cdots \times \hom_\cc (X_{0},X_1) \ra \hom_\cc (X_0,X_d)
    \end{equation}
    of degree $2-d$, satisfying the same quadratic relations as equation
    (\ref{ainfeq}).
\end{itemize}
\end{defn}

\noindent In this paper, we will work with some $\ainf$ categories $\mc{C}$
with finitely many objects $X_1, \ldots, X_k$. As observed in
\cite{Seidel:2009dq} and \cite{Seidel:2008cr}, any such category $\mc{C}$ is
equivalent to an $\ainf$ algebra over the semi-simple ring \[R = \K e_1 \oplus
\cdots \oplus\K e_k,\] which we also call $\mc{C}$. The correspondence is
as follows: as a graded vector space this algebra is 
\begin{equation} \mc{C}  :=
    \bigoplus_{i,j} \hom (X_i,X_j) \end{equation}
with the idempotents $e_i$ of $R$ acting by 
\begin{equation}
    e_s \cdot \mc{C} \cdot e_t = \hom(X_t,X_s).
\end{equation}
Tensor products are now interpreted as being over $R$ (with respect to
composable morphisms), i.e.  \begin{equation}
    \mc{C}^{\otimes r} := \mc{C}^{\otimes_R r} = \bigoplus_{V_0, \ldots, V_r \in \ob
    \mc{C}} \hom(V_{r-1},V_r) \otimes \cdots \otimes \hom(V_0,V_1).
\end{equation} 
In this picture, the $\ainf$ structure on the category $\mc{C}$ is equivalent
to the data of an $\ainf$ structure over $R$ on the graded vector space
$\mc{C}$. Namely, maps
\begin{equation}
    \mu^d: \mc{C}^{\otimes d} \lra \mc{C}
\end{equation}
are by definition the same data as the higher composition maps
(\ref{compositionmaps}).

\begin{defn}
    Given an $\ainf$ category $\cc$, the {\bf opposite category}
    \begin{equation}
        \cc^{op}
    \end{equation}
    is defined as follows: 
    \begin{itemize}
        \item objects of $\cc^{op}$ are the same as objects of $\cc$,
        \item as graded vector spaces, morphisms of $\cc^{op}$ are morphism of
            $\cc$ with objects reversed: \begin{equation}
                \hom_{\cc^{op}}(X,Y) = \hom_{\cc}(Y,X)
            \end{equation}
        \item $\ainf$ operations are, up to a sign, the reversed $\ainf$ operations of $\cc$:
            \begin{equation}
                \mu^d_{\cc^{op}}(x_1, \ldots, x_d) = (-1)^{\maltese_d}\mu^d_{\cc}(x_d, \ldots, x_1)
            \end{equation}
            where $\maltese_d = \sum_{i=1}^d ||x_i||$ is the usual sign.
    \end{itemize}
\end{defn}
The opposite category $\cc^{op}$ can be thought of as an algebra over the
semi-simple ring $R^{op}$.

\subsection{Morphisms and functors}
\begin{defn}
    A {\bf morphism of $\ainf$ algebras} 
    \begin{equation}
        \mathbf{F}: (\A,\mu_{\A}^{\cdot}) \ra
    (\mc{B},\mu_{B}^{\cdot})
    \end{equation}
    is the data of, for each $d\geq 1$, maps of graded vector spaces
    \begin{equation}
        \mathbf{F}^d : \A^{\otimes d} \ra \mc{B}
    \end{equation}
    of degree $1-d$, satisfying the following equation, for each $k$:
    \begin{equation}\label{functoreqn}
        \begin{split}
        \sum_{j; i_1 + \cdots + i_j = k}
        \mu_{\mc{B}}^j(\mathbf{F}^{i_j}(x_k, \ldots, x_{k-i_j+1}),\ldots,\mathbf{F}^{i_1}(x_{i_1},\ldots, x_{1})) =\\
        \sum_{s \leq k,t} (-1)^{\maltese_t} \mathbf{F}^{k-s+1}(x_k,\ldots, x_{t+s+1}, \mu_{\A}^s(x_{t+s}, \ldots, x_{t+1}), x_t, \ldots, x_1).
    \end{split}
    \end{equation}
    Here,
    \begin{equation}
        \maltese_t = \sum_{i=1}^t ||x_i||
    \end{equation}
    is the same (Koszul) sign as before.
\end{defn}
Suppose $\A$ and $\mc{B}$ are algebras over semi-simple rings, or equivalently
$\ainf$ categories. Then, unwinding the definition above leads to the following
categorical notion of functor:
\begin{defn}
    An {\bf $\ainf$ functor} 
    \begin{equation}
        \mathbf{F}: \mc{C} \lra \mc{C}'\end{equation} 
    consists of the following data:
    \begin{itemize}
        \item For each object $X$ in $\mc{C}$, an object $\mathbf{F}(X)$ in
            $\mc{C}'$, 
        \item for any set of $d+1$ objects $X_0, \ldots, X_d$,
            higher maps
    \begin{equation}\label{functormaps}
        \mathbf{F}^d: \hom_\cc (X_{d-1},X_d) \times \cdots \times \hom_\cc
        (X_{0},X_1) \lra \hom_\cc (\mathbf{F}(X_0),\mathbf{F}(X_d))
    \end{equation}
    of degree $1-d$, satisfying the same relations as equation (\ref{functoreqn}).
    \end{itemize}
\end{defn}
The equations (\ref{functoreqn}) imply that the first-order term of any
morphism or functor descends to a cohomology level functor $[\mathbf{F}^1]$. We
say that a morphism $\mathbf{F}$ is a {\bf quasi-isomorphism} if
$[\mathbf{F}^1]$ is an isomorphism.  Call a functor $\mathbf{F}$ {\bf
quasi-full} if $[\mathbf{F}^1]$ is an isomorphism onto a full subcategory of
the cohomology of the image, and call it a {\bf quasi-equivalence} if
$\mathbf{F}$ is also essentially surjective.

\subsection{Unitality} 
There are three a posteriori equivalent definitions of units in the $\ainf$
setting. We will make definitions for $\ainf$ algebras.
\begin{defn} 
    An $\ainf$ algebra $\A$ is {\bf strictly unital} if there is an element
$e_+\in \A$ such that 
\begin{equation}\label{strictuniteqns}
    \begin{split}
       \mu^1(e_+) &= 0,\\
       \mu^2(e_+,x) &= (-1)^{|x|}\mu^2(x,e_+) = x,\mathrm{\ and}\\
       \mu^k(\ldots, e_+, \ldots) &= 0,\ k \geq 3.
   \end{split}
   \end{equation}
\end{defn}
A weaker version, following \cite{Seidel:2008zr}, is the notion of a
homology-level unit.  
\begin{defn} 
    $\A$ is said to be {\bf homologically unital} if there is an element $e \in
    \A$ of degree 0 that descends to a unit on level of homology; i.e. 
    $\left(H^*(\A),H^*(\mu^2),[e]\right)$ is an associative unital algebra. Any
    such element $e$ is called a {\bf homology unit} of $\A$.  
\end{defn}
Although strict unitality is algebraically a desirable property, most often
one can only geometrically construct a homological unit.  Fukaya-Oh-Ohta-Ono
\cite{Fukaya:2009ve} observed that there is a richer structure which can be
constructed geometrically, that interpolates between these two notions.  
\begin{defn} \label{homotopyunitdef}
    Let $\mc{A}$ be an $\ainf$ algebra with homological unit $e$. A {\bf
    homotopy unit} for $(\mc{A},e)$ is an $\ainf$ structure $\mu_{\A'}$ on the
    graded vector space
    \begin{equation}
        \A' := \A \oplus \K f[1] \oplus {\mathbb K} e^+
    \end{equation}
    restricting to the original $\ainf$ structure on $\A$, satisfying
    \begin{equation}
        \mu^1(f) = e^+ - e,
    \end{equation}
    with $e^+$ is a strict unit for $(\A',\mu_{\A'})$.
\end{defn}
As a sanity check, we note that in the definition above, the inclusion
\begin{equation}
    \A \hookrightarrow \A'
\end{equation}
is a quasi-isomorphism.  Moreover, the condition that $e^+$ be a strict unit
determines all $\ainf$ structure maps $\mu_{\A'}$ involving occurrences of
$e^+$. Thus, additional data involved in constructing the required $\ainf$
structure on $\A'$ is exactly contained in operations with occurrences of
of $f$. Thus, the data of a homotopy unit translates into the data of maps
\begin{equation}
    \mathfrak{h}_k: (T\A)^{\otimes k} \ra \A
\end{equation}
such that the operations
\begin{equation}
    \begin{split}
    \mu^{i_1 + \cdots + i_k + k-1}_{\A'}(x^1_1, \ldots,
x^1_{i_1}, f, x^2_1, &\ldots, x^2_{i_2}, \ldots \ldots, f, x^k_{1}, \ldots,
x^k_{i_k})  := \\
& \mathfrak{h}_k (x^1_1 \otimes \cdots \otimes x^1_{i_1}; \ldots ; x^k_1 \otimes \cdots \otimes x^k_{i_k})
\end{split}
\end{equation} 
satisfy the $\ainf$ relations. The $\ainf$ relations can then also be
translated into equations for the $\mathfrak{h}_k$, which we will omit for the
time being; see \cite{Fukaya:2009ve}*{\S 3.3} for greater detail. 

These definitions all admit fairly straightforward categorical generalizations.
For homological unitality, one mandates that each object $X$ contain a homology
level identity morphism $[e_X] \in H^*(\hom_\cc(X,X))$. For strict unitality,
one requires the existence of morphisms $e_{X}^+ \in \hom_{\cc}(X,X)$
satisfying (\ref{strictuniteqns}). Finally, one defines a homotopy unit structure on
$\cc$ to be the structure of an $\ainf$ category on $\cc'$, defined to be $\cc$
enlarged with additional morphisms generated by formal elements $f_X, e_X^+ \in
\hom_\cc(X,X)$, satisfying the same conditions as Definition
\ref{homotopyunitdef}.

By definition any $\ainf$ algebra with a homotopy unit is quasi-equivalent to a
strictly unital one (namely $\A'$) and is homologically unital (with homological
unit $e$). Conversely, it is shown in \cite{Seidel:2008zr} that any
homologically unital $\ainf$ algebra is quasi-equivalent to a strictly unital
or homotopy unital $\ainf$ algebra. The same holds for $\ainf$ categories.

\subsection{Categories of modules and bimodules}
To an $\ainf$ algebra or category $\cc$ one can associate categories of left
$\ainf$ modules and right $\ainf$ modules over $\cc$. These categories are dg
categories, with explicitly describable morphism spaces and differentials.
Similarly, to a pair of $\ainf$ algebras/categories $(\cc, \mc{D})$, one can
associate a dg category of $\ainf$ $\cc\!-\!\mc{D}$ bimodules. These dg
categories can be thought of as 1-morphisms in a two-category whose objects are
$\ainf$ categories.
\begin{rem}
    The fact that module categories over $\cc$ are dg categories comes from an
    interpretation of left/right module categories over $\cc$ as categories of
    (covariant/contravariant) $\ainf$ functors from $\cc$ into chain complexes.
    The dg structure is then inherited from the dg structure on chain
    complexes.  Similarly, $\cc\!-\mc{D}$ bimodules can be thought of as
    $\ainf$ {\bf bi-functors} from the $\ainf$ {\bf bi-category} $\cc^{op}
    \times \mc{D}$ into chain complexes.  We will not pursue this viewpoint
    further, and instead refer the reader to \cite{Seidel:2008zr}*{\S (1j)}.  
\end{rem}

\begin{defn} \label{leftmoduledef}
    A {\bf left $\cc$-module} $\mc{N}$ consists of the following data:
    \begin{itemize}
        \item For $X \in \ob \cc$, a graded vector space $\mc{N}(X)$.
        \item For $r\geq 0$, and objects $X_0, \ldots, X_r \in \ob \cc$, module structure maps
            \begin{equation}
                \mu^{r|1}_{\mc{N}}: \hom_{\cc}(X_{r-1},X_r) \otimes \cdots \hom_{\cc}(X_0, X_1) \otimes \mc{N}(X_0) \lra \mc{N}(X_r)
            \end{equation}
            of degree $1-r$, satisfying the following analogue of the $\ainf$
            equations, for each $k$: 
            \begin{equation}\label{ainfmodleft}
                \begin{split}
                    \sum (-1)^{\maltese_0^s}\mu^{k-j+1|1}_{\mc{N}}(x_s, \ldots, x_{s+j+1}, \mu^j_{\cc}(x_{s+j}, \ldots, x_{s+1}), x_{s}, \ldots,  x_1, \mathbf{n}) \\
                    + \sum \mu^{s|1}_{\mc{N}}(x_1, \ldots,x_s, \mu^{k-s|1}_{\mc{N}}(x_{s+1}, \ldots,  x_k, \mathbf{n})) = 0.
                \end{split}
            \end{equation}
    \end{itemize}
    Here, the sign
    \begin{equation}
        \maltese_0^s := |\mathbf{n}| + \sum_{i=1}^s ||x_i||
    \end{equation}
    is given by the sum of the degree of $\mathbf{n}$ plus the reduced degrees
    of $x_1, \ldots, x_s$.  
\end{defn}
\noindent The first two equations
\begin{equation}
    \begin{split}
        (\mu^{0|1}_{\mc{N}})^2 &= 0\\
        \mu^{1|1}_{\mc{N}} (a,\mu^{0|1}_{\mc{N}}(\mathbf{m})) 
        \pm \mu^{1|1}_{\mc{N}}(\mu^1_\cc(a),\mathbf{m}) &= 
        \pm \mu^{0|1}_{\mc{N}} (\mu^{1|1}_{\mc{N}}(a,\mathbf{m})) 
    \end{split}
\end{equation}
imply that $\mu^{0|1}_{\mc{N}}$ is a differential and that the first module
multiplication $\mu^{1|1}_{\mc{N}}$ descends to homology.  Right modules have
an essentially identical definition, with a direction reversal and slightly
different signs.
\begin{defn} \label{rightmoduledef}
    A {\bf right $\cc$-module} $\mc{M}$ consists of the following data:
    \begin{itemize}
        \item For $X \in \ob \cc$, a graded vector space $\mc{M}(X)$.
        \item For $r\geq 0$, and objects $X_0, \ldots, X_r \in \ob \cc$, module structure maps
            \begin{equation}
                \mu^{1|r}_{\mc{M}}: \mc{M}(X_0) \otimes \hom_{\cc}(X_{1},X_0) \otimes \cdots \hom_{\cc}(X_{r-1}, X_{r}) \lra \mc{M}(X_r)
            \end{equation}
            of degree $1-r$, satisfying the following analogue of the $\ainf$
            equations, for each $k$: 
            \begin{equation}\label{ainfmodright}
                \begin{split}
                    \sum (-1)^{\maltese_{-k}^{-(s+j+1)}}\mu^{1|k-j+1}_{\mc{M}}(\mathbf{m}, x_1, \ldots, x_{s}, \mu^j_{\cc}(x_{s+1}, \ldots, x_{s+j}), x_{s+j+1}, \ldots,  x_k) \\
                    + \sum (-1)^{\maltese_{-k}^{-(s+1)}}\mu^{1|s}_{\mc{M}}(\mu^{1|k-s}_{\mc{M}}(\mathbf{m}, x_1, \ldots, x_{s}),x_{s+1}, \ldots,  x_k) = 0.
                \end{split}
            \end{equation}
    \end{itemize}
    Here, the signs as usual denote the sum of the reduced degrees of elements to the right:
    \begin{equation}
        \maltese_{-k}^{-a} := \sum_{i=a}^k ||x_i||.
    \end{equation}
\end{defn}
\noindent Again, the first two equations
imply that $\mu^{1|0}_{\mc{M}}$ is a differential and that the first module
multiplication $\mu^{1|1}_{\mc{M}}$ descends to homology. Thus, for right or
left modules, one can talk about unitality.
\begin{defn}[Compare \cite{Seidel:2008zr}*{\S (2f)}] 
    A left (right) module is {\bf homologically-unital} if the underlying
    cohomology left (right) modules are unital; that is, for any $X \in \ob
    \cc$ with homology unit $e_X$, the cohomology level module multiplication
    by $[e_X]$ is the identity.
\end{defn}
\noindent Now, let $\cc$ and $\mc{D}$ be $\ainf$ categories.
\begin{defn}  \label{bimoduledef} An $\ainf$ {\bf
    $\cc\!-\!\dd$  bimodule} $\mc{B}$ consists of the following data:
\begin{itemize}
\item for $V \in \ob \cc$, $V' \in \ob \dd$, a graded vector space
    $\mc{B}(V,V')$

\item for $r,s \geq 0$, and objects $V_0, \ldots, V_r \in \ob \cc$, $W_0,
\ldots, W_s \in \ob \cc'$, bimodule structure maps 
\begin{equation}
\begin{split}
    \mu^{r|1|s}_{\mc{B}}:&\hom_{\cc}(V_{r-1},V_r) \times \cdots
    \hom_{\cc}(V_0,V_1) \times \mc{B}(V_0,W_0) \times \\
    &\times\hom_{\dd}(W_1,W_0) \times \cdots \times \hom_{\dd}(W_{s},W_{s-1})
    \longrightarrow \mc{B}(V_r,W_s) 
\end{split}
\end{equation}
of degree $1-r-s$, 
\end{itemize}
such that the following equations are satisfied, for each $r\geq 0$, $s\geq 0$:
\begin{equation}
    \label{ainfbimod}
    \begin{split}
        &\sum (-1)^{\maltese_{-s}^{-(j+1)}}\mu^{r-i|1|s-j}_{\mc{B}} (v_r, \ldots, v_{i+1}, \mu^{i|1|j}_{\mc{B}}(v_{i}, \ldots, v_{1},\mathbf{b}, w_1, \ldots, w_{j}), w_{j+1} \ldots, w_{s}) \\
        &+\sum (-1)^{\maltese_{-s}^k}\mu^{r-i+1|1|s}_{\mc{B}} (v_r, \ldots, v_{k+i+1}, \mu^{i}_{\cc}(v_{k+i}, \ldots, v_{k+1}), v_k, \ldots, v_1, \mathbf{b}, w_1, \ldots, w_s) \\
        &+\sum (-1)^{\maltese_{-s}^{-(l+j+1)}}\mu^{r|1|s-j+1}_{\mc{B}}(v_r,\ldots, v_1, \mathbf{b},w_1, \ldots, w_l, \mu^{j}_{\dd} (w_{l+1}, \ldots, w_{l+j}), w_{l+j+1} \ldots, w_s) \\
       &\ \ \ \ \ \  = 0.
\end{split} 
\end{equation} 

The signs above are given by the sum of the
degrees of elements to the right of the inner operation, with the convention
that we use {\bf reduced degree} for elements of $\cc$ or $\dd$ and {\bf full
degree} for elements of $\mc{B}$. Thus, 
\begin{align}
    \maltese_{-s}^{-(j+1)} &:= \sum_{i=j+1}^s ||w_i||,\\
    \maltese_{-s}^{k} &:= \sum_{i=1}^s ||w_i|| + |\mathbf{b}| + \sum_{j=1}^k ||v_j||.
\end{align}
\end{defn}
\noindent Once more, the first few equations imply that $\mu^{0|1|0}$ is a
differential, and the left and right multiplications $\mu^{1|1|0}$ and
$\mu^{0|1|1}$ descend to homology.
\begin{defn}
    Let $\cc$ and $\mc{D}$ be homologically-unital $\ainf$ categories, and
    $\mc{B}$ a $\cc\!-\!\mc{D}$ bimodule. $\mc{B}$ is {\bf
    homologically-unital} if the homology level multiplications $[\mu^{1|1|0}]$
    and $[\mu^{0|1|1}]$ are unital, i.e.  homology units in $\cc$ and $\mc{D}$
    act as the identity.
\end{defn}
\noindent We will frequently refer to $\cc\!-\!\cc$ bimodules as simply
$\cc$-bimodules, or bimodules over $\cc$. 

Now, we define the $dg$-category structure on various categories of modules and
bimodules. For the sake of brevity, we assume that $\ainf$ categories $\cc$ and
$\mc{D}$ have finitely many objects, and can thus be thought of as algebras
over semi-simple rings $R$ and $R'$ respectively. In this language, a left
(right) $\cc$-module $\mc{N}$ ($\mc{M}$) is the data of an $R$ ($R^{op}$)
vector space $\mc{N}$ ($\mc{M}$) together with maps
\begin{equation}
    \begin{split}
        \mu^{r|1}_{\mc{N}}: \cc^{\otimes_R r} \otimes_R \mc{N} &\lra \mc{N},\ r \geq 0\\
        \mu^{1|s}_{\mc{M}}:  \mc{M} \otimes_R \cc^{\otimes_R s} &\lra \mc{M},\ s \geq 0
\end{split}
\end{equation}
satisfying equations (\ref{ainfmodleft}) and (\ref{ainfmodright}) respectively. Similarly, a $\cc\!-\!\mc{D}$ bimodule $\B$ is an $R \otimes R'^{op}$ vector space $\B$ together with maps
\begin{equation}
    \mu^{r|1|s}_{\mc{B}}: \cc ^{\otimes_R r} \otimes_R \B \otimes_{R'} \dd^{\otimes_{R'}
s} \lra \mc{B}
\end{equation} 
satisfying (\ref{ainfbimod}). We can combine the structure maps
 $\mu^{r|1|s}_\B$, $\mu^{1|s}_{\mc{M}}$, $\mu^{r|1}_{\mc{N}}$
for all $r,s$ to form total (bi)-module structure maps
\begin{equation} \label{totalmodulestructures}
    \begin{split}
        \mu_\B := \oplus \mu^{r|1|s}_\B: T\cc \otimes \B \otimes T\dd &\lra \B\\
        \mu_{\mc{N}} := \oplus \mu^{r|1}_{\mc{N}}: T\cc \otimes \mc{N} &\lra \mc{N}\\
        \mu_{\mc{M}} := \oplus \mu^{1|s}_{\mc{M}}:  \mc{M} \otimes T\cc &\lra \mc{M}.
    \end{split}
\end{equation}
The {\bf hat extensions} of these maps 
\begin{equation} \label{totalhatmodulestructures}
    \begin{split}
        \hat{\mu}_\B : T\cc \otimes \mc{B} \otimes T\mc{D} &\lra T\cc \otimes \mc{B}\otimes T\mc{D}\\
        \hat{\mu}_{\mc{N}} : T\cc \otimes \mc{N} &\lra T\cc \otimes \mc{N}\\
        \hat{\mu}_{\mc{M}} :  \mc{M} \otimes T\cc &\lra \mc{M} \otimes T\cc.
    \end{split}
\end{equation}
sum over all ways to collapse subsequences with either module/bimodule or
$\ainf$ structure maps, as follows:
\begin{align}
 \hat{\mu}_\B(c_k, &\ldots, c_1, \mathbf{b}, d_1, \ldots, d_l) := \\
 \nonumber &\sum (-1)^{\maltese_{-l}^{-(t+1)}} c_k \otimes \cdots \otimes c_{s+1} \otimes \mu_\B^{s|1|t}(c_{s}, 
\ldots, c_1, \mathbf{b}, d_1, \ldots, d_{t}) \otimes d_{t+1} \otimes \cdots \otimes d_l\\
\nonumber &+ \sum (-1)^{\maltese_{-l}^{s}}c_k\otimes \cdots \otimes c_{s+i+1} \otimes \mu_\cc^i(c_{s+i}, 
\ldots, c_{s+1}) \otimes c_{s} \otimes \cdots \otimes c_1 \otimes \\
\nonumber &\ \ \ \ \ \ \ \ \ \ \ \ \ \mathbf{b} \otimes d_1 \otimes \cdots \otimes d_l\\
\nonumber &+ \sum (-1)^{\maltese_{-l}^{-(j+t+1)}}c_k \otimes \cdots \otimes c_1 \otimes \mathbf{b} \otimes d_1 \otimes \cdots \otimes d_{j} \otimes \\
\nonumber &\ \ \ \ \ \ \ \ \ \ \ \ \ \mu_{\mc{D}}^t(d_{j+1}, \ldots d_{j+t}) \otimes d_{j+t+1} \otimes \cdots \otimes d_l\\
 \hat{\mu}_\N(c_k,& \ldots, c_1, \mathbf{n}) :=\\
 &\nonumber\sum c_k \otimes \cdots \otimes c_{s+1} \otimes \mu_\N^{s|1}(c_{s}, 
\ldots, c_1, \mathbf{n})\\
\nonumber &+ \sum (-1)^{\maltese_0^{s}}c_k\otimes \cdots \otimes c_{s+i+1} \otimes \mu_\cc(c_{s+i}, 
\ldots, c_{s+1}) \otimes c_{s} \otimes \cdots \otimes c_1 \otimes 
\mathbf{n} \\
 \hat{\mu}_\M( \mathbf{m},& d_1, \ldots, d_l) :=\\
 \nonumber &\sum (-1)^{\maltese_{-l}^{-(t+1)}} \mu_\B^{1|t}( \mathbf{m}, d_1, \ldots, d_{t}) \otimes d_{t+1} \otimes \cdots \otimes d_l\\
\nonumber &+ \sum (-1)^{\maltese_{-l}^{-(t+j+1)}}\mathbf{m} \otimes d_l \otimes \cdots d_{t} \otimes \mu_{\cc}^j(d_{t+1}, \ldots,  d_{t+j}) \otimes d_{t+j+1} \otimes \cdots \otimes d_l,
\end{align}
with signs as specified in Definitions \ref{bimoduledef}, \ref{leftmoduledef},
\ref{rightmoduledef}.  Then the $\ainf$ bimodule and module equations, which
can be concisely written as 
\begin{equation}
    \begin{split}
    \mu_\B \circ \hat{\mu}_\B &= 0,\\
    \mu_\N \circ \hat{\mu}_\N &= 0,\\
    \mu_\M \circ \hat{\mu}_\M &= 0,
\end{split}
\end{equation}
are equivalent to requiring that the hat extensions
(\ref{totalhatmodulestructures}) are differentials.  
\begin{rem}
    Actually, the hat extensions of the maps $\mu_{\mc{N}}$, $\mu_{\mc{M}}$,
    $\mu_{\mc{B}}$ are the unique extensions of those maps which are
a bicomodule co-derivation with respect to the structure of $T\cc \otimes \M
\otimes T\dd$ as a bicomodule over differential graded co-algebras
$(T\cc,\hat{\mu}_\cc)$, $(T\dd,\hat{\mu}_{\dd})$. A good reference for this
perspective, which we will not spell out more, is \cite{Tradler:2008fk}.
\end{rem}
\begin{defn}
    A {\bf pre-morphism} of left $\cc$ modules of degree $k$
    \begin{equation}
        \mc{H}: \mc{N} \lra \mc{N}'
    \end{equation}
    is the data of maps
    \begin{equation}
        \mc{H}^{r|1}: \cc^{\otimes r} \otimes \mc{N}  \lra \mc{N}',\ r \geq 0
    \end{equation}
    of degree $k-r$. These can be packaged together into a total pre-morphism map
    \begin{equation}
        \mc{H} = \oplus \mc{H}^{r|1}: T \cc \otimes \mc{N} \lra \mc{N}'.
    \end{equation}
\end{defn}
\begin{defn}
    A {\bf pre-morphism} of right $\cc$ modules of degree $k$
    \begin{equation}
        \mc{G}: \mc{M} \lra \mc{M}'
    \end{equation}
    is the data of maps
    \begin{equation}
        \mc{G}^{1|s}: \mc{M} \otimes \cc^{\otimes s} \lra \mc{M}',\ s \geq 0
    \end{equation}
    of degree $k-s$. 
    These can be packaged together into a total pre-morphism map
    \begin{equation}
        \mc{G} = \oplus \mc{G}^{1|s}: \mc{M} \otimes T \cc \lra \mc{M}'.
    \end{equation}
\end{defn}
\begin{defn}
    A {\bf pre-morphism of $\cc\!-\!\dd$
bimodules} of degree $k$
\begin{equation}
    \mc{F}: \mc{B} \lra \mc{B}'
\end{equation}
is the data of maps 
\begin{equation}
    \mc{F}^{r|1|s}:  \cc^{\otimes r} 
    \otimes \mc{B} \otimes
    \dd^{\otimes s} \longrightarrow \mc{B}',\ r,s \geq 0.  
\end{equation}
of degree $k-r-s$.
These can be packaged together into a total pre-morphism map
    \begin{equation}
        \mc{F} := \oplus \mc{F}^{r|1|s}: T\cc\otimes \mc{B} \otimes T \mc{D} \lra \mc{B}'.
    \end{equation}
\end{defn}
\begin{rem}
    Such morphisms are said to be degree $k$ because the induced map 
    \begin{equation}
        \mc{F}: T \cc[1] \otimes \mc{B} \otimes T \dd[1] \lra \mc{B}'
    \end{equation}
    has graded degree $k$.
\end{rem}
Now, any collapsing maps of the form 
\begin{equation}
    \begin{split}
        \phi: T\cc \otimes \B \otimes T\mc{D} &\lra \B'\\
        \psi: T\cc \otimes \mc{N} &\lra \mc{N}'\\
        \rho: \mc{M} \otimes T\cc &\lra \mc{M}'
    \end{split}
\end{equation}
admit, in the style of (\ref{extendphi}), {\bf hat extensions}
\begin{equation}
    \begin{split}
        \hat{\phi}: T\cc \otimes \B \otimes T\mc{D} &\lra T\cc \otimes \B' \otimes T\mc{D}\\
        \hat{\psi}: T\cc \otimes \mc{N} &\lra T\cc \otimes \mc{N}'\\
        \hat{\rho}: \mc{M} \otimes T\cc &\lra \mc{M}' \otimes T\cc
    \end{split}
\end{equation}
which sum over all ways (with signs) to collapse a subsequence with $\phi$,
$\psi$, and $\rho$ respectively:
\begin{equation}
    \begin{split}
\hat{\mc{\phi}}(c_k,& \ldots, c_1, \mathbf{b}, d_1, \ldots, d_l) := \\
&\sum (-1)^{|\phi|\cdot \maltese_{-l}^{-(t+1)}} c_k \otimes \cdots \otimes
c_{s+1} \otimes \mc{\phi}(c_{s}, \ldots, c_1,\mathbf{b},
d_1, \ldots, d_t) \otimes d_{t+1} \otimes \cdots \otimes d_l.\\
\hat{\mc{\psi}}(c_k,& \ldots, c_1, \mathbf{n}) := \\
&\sum c_k \otimes \cdots \otimes c_{s+1} \otimes \mc{\psi}(c_{s}, \ldots,
c_1,\mathbf{m}).\\
\hat{\mc{\rho}}(\mathbf{m},& c_1, \ldots, c_l) :=\\ 
&\sum (-1)^{|\rho| \cdot \maltese_{-l}^{-(t+1)}}
\mc{\rho}(\mathbf{m}, c_1, \ldots, c_t) \otimes c_{t+1} \otimes \cdots\otimes c_l.
\end{split}
\end{equation} 
\begin{rem}
Once more, the hat extensions are uniquely specified by the requirements that
$\hat{\psi}$ and $\hat{\rho}$ be (left and right) co-module homomorphisms over
the co-algebra $(T\cc, \Delta_{\cc})$, and that $\hat{\phi}$ be a bi-co-module
homomorphism over the co-algebras $(T\cc, \Delta_{\cc})$ and
$(T\mc{D},\Delta_{\mc{D}})$. In this manner, categories of modules and
bimodules over $\ainf$ algebras give categories of dg comodules and dg
bicomodules over the associated dg co-algebras. See \cite{Tradler:2008fk}.
\end{rem}
\noindent It is now easy to define composition of pre-morphisms: 
\begin{defn}If
$\mc{F}_1$ is a pre-morphism of $\ainf$ left modules/right modules/bimodules
from $\mc{M}_0$ to $\mc{M}_1$ and $\mc{F}_2$ is a pre-morphism of $\ainf$ left
modules/right modules/bimodules from $\mc{M}_1$ to $\mc{M}_2$, define the {\bf
composition} $\mc{F}_2 ``\circ" \mc{F}_1$ as:
\begin{equation}
    \mc{F}_2 ``\circ" \mc{F}_1 := \mc{F}_2 \circ \hat{\mc{F}}_1.
\end{equation} 
\end{defn}
\begin{rem}
Observe the hat extension of the composition agrees with the composition of the
hat extensions, e.g. $\widehat{\mc{F}_2 ``\circ" \mc{F}_1} = \hat{\mc{F}}_2
\circ \hat{\mc{F}}_1$.  
i.e. this notion agrees with usual composition of
homomorphisms of comodules/bi-comodules.
\end{rem}
\noindent Similarly, there is a differential on pre-morphisms. 
\begin{defn}
    If $\mc{F}$ is a pre-morphism of left modules/right modules/bimodules from
    $\mc{M}$ to $\mc{N}$ with associated bimodule structure maps $\mu_{\mc{M}}$
    and $\mu_{\mc{N}}$, define the {\bf differential} $\delta \mc{F}$ 
    to be:
    \begin{equation}
        \delta (\mc{F}) := \mu_{\mc{N}} \circ \hat{\mc{F}} -(-1)^{|\mc{F}|}\mc{F} \circ \hat{\mu}_{\mc{M}}.
\end{equation}
\end{defn}
The fact that $\delta^2 = 0$ is a consequence of the $\ainf$ module or bimodule
equations for $\mc{M}$ and $\mc{N}$. 
As one consequence of $\delta(\mc{F}) = 0$, the first order
term $\mc{F}^{0|1}$, $\mc{F}^{1|0}$ or $\mc{F}^{0|1|0}$ descends to a
cohomology level module or bimodule morphism.  Call any pre-morphism $\mc{F}$
of bimodules or modules a {\bf quasi-isomorphism} if $\delta(\mc{F}) = 0$, and
the resulting cohomology level morphism $[\mc{F}]$ is an isomorphism.

\begin{rem}We have developed modules and bimodules in parallel, but note now
    that modules are a special case of bimodules in the following sense: a left
    $\ainf$ module (right $\ainf$ module) over $\cc$ is a $\cc\!-\!\K$
    ($\K\!-\!\cc$) bimodule $\mc{M}$ with structure maps $\mu^{r|1|s}$ trivial for
    $s>0$ ($r>0$). Thus we abbreviate $\mu^{r|1|0}$ by $\mu^{r|1}$ (and
    correspondingly, $\mu^{0|1|s}$
    by $\mu^{1|s}$). 
\end{rem}
Thus, we have seen that $\cc\!-\!\mc{D}$ bimodules, as well as left and right
$\cc$ modules form dg categories which will be denoted 
\begin{equation}
    \begin{split}
        \cc&\bimod \mc{D}\\
        \cc&\mod\\
        &\ \rmod\cc 
    \end{split}
\end{equation} 
respectively.

\subsection{Tensor products}
There are several relevant notions of tensor product for modules and bimodules.
The first notion, that of tensoring two bimodules over a single common side,
can be thought of as composition of 1-morphisms in the 2-category of $\ainf$
categories.  
\begin{defn}
    Given a $\mc{C}\!-\!\mc{D}$ bimodule $\M$ and an $\mc{D}\!-\!\mc{E}$ bimodule $\N$, the
    {\bf (convolution) tensor product over $\mc{D}$} 
    \begin{equation}\M \otimes_{\mc{D}}
        \N\end{equation}
        is the $\mc{C}\!-\!\mc{E}$ bimodule given by
\begin{itemize}
\item underlying graded vector space 
    \begin{equation}
        \M \otimes T\mc{D} \otimes \N;
\end{equation}

\item differential 
    \begin{equation}
        \mu^{0|1|0}_{\M \otimes_{\mc{D}} \N}: \M \otimes T\mc{D} \otimes \N \ra \M \otimes T\mc{D} \otimes \N
\end{equation}
    given by
    \begin{equation}
        \begin{split}
        \mu^{0|1|0}(&\mathbf{m}, d_1, \ldots, d_k,\mathbf{n}) =\\
        &\sum (-1)^{\maltese_{-(k+1)}^{-(t+1)}} \mu_\M^{0|1|t} (\mathbf{m}, d_1, \ldots, d_t) \otimes d_{t+1} 
        \otimes \cdots \otimes d_k \otimes \mathbf{n} \\
        &+ \sum \mathbf{m} \otimes d_1 \otimes \cdots \cdots \otimes d_{k-s} \otimes \mu_\N^{s|1|0}(d_{k-s+1}, \ldots, d_k, \mathbf{n})\\
        &+ \sum (-1)^{\maltese_{-(k+1)}^{-(j+i+1)}}\mathbf{m} \otimes d_1 \otimes \cdots \otimes d_j \otimes \mu^i_{\mc{A}}(d_{j+1}, \ldots, d_{j+i}) \otimes \\
        &\ \ \ \ \ \ \ \ \ \ \ \ \ d_{j+i+1} \otimes \cdots \otimes  d_k \otimes \mathbf{n}.
    \end{split}
    \end{equation}

\item for $r$ or $s$ $>0$, higher bimodule maps 
    \begin{equation}
        \mu^{r|1|s}_{\M\otimes_{\mc{D}}\N} : \cc^{\otimes r} \otimes \M \otimes T\mc{D}
        \otimes \N \otimes \mc{E}^{\otimes s} \lra \M \otimes T\mc{D} \otimes \N
   \end{equation}
    given by: 
    \begin{equation}
        \begin{split}
        \mu^{r|1|0} & (c_1, \ldots, c_r, \mathbf{m}, d_1, \ldots, d_k, 
        \mathbf{n}) = \\
        &\sum_{t} (-1)^{\maltese_{-(k+1)}^{-(t+1)}}\mu^{r|1|t}_\M(c_1, \ldots, c_r,\mathbf{m}, 
        d_1, \ldots, d_t)\otimes d_{t+1} \otimes \cdots \otimes d_k 
        \otimes \mathbf{n}
    \end{split}
    \end{equation}
    \begin{equation}
        \begin{split}
        \mu^{0|1|s}&(\mathbf{m}, d_1, \ldots, d_k, \mathbf{n}, e_1, 
        \ldots, e_s) = \\
        &\sum_j \mathbf{m} \otimes d_1 \otimes \cdots \otimes
        d_{k-j} \otimes \mu^{j|1|s}_\N(d_{k-j+1}, \ldots, d_k, \mathbf{n}, 
        e_1, \ldots, e_s)
    \end{split}
    \end{equation}
    and 
    \begin{equation}
        \mu^{r|1|s} = 0\mathrm{\ if\ } r>0\ \r{and}\ s > 0.
    \end{equation}
\end{itemize} 
In all equations above, the sign is the sum of degrees of all elements to the
right, using reduced degree for elements of $\mc{A}$ and full degree for
elements of $\mc{N}$:  
\begin{equation}
    \maltese_{-(k+1)}^{-t} := |\mathbf{n}| + \sum_{i=t}^{k} ||d_i||.
\end{equation}
\end{defn}
\noindent One can check that these maps indeed give $\M \otimes_{\mc{D}} \N$
the structure of an $\cc\!-\!\mc{E}$ bimodule. As one would expect from a
two-categorical perspective, convolution with $\mc{N}$ gives a dg functor 
\begin{equation}
\cdot \otimes_{\mc{D}} \N: \cc\bimod\mc{D} \longrightarrow \cc\bimod\mc{E}.
\end{equation}
Namely, there is an induced map on morphisms, which we will omit for the time
being.

As a special case, suppose $\M^r$ is a right $\A$ module and $\N^l$ is a left
$\A$ module. Then, thinking of $\M$ and $\N$ as ${\mathbb K}\!-\!\A$ and
$\A\!-\!{\mathbb K}$ modules respectively, there two possible one-sided tensor
products.  The tensor product over $\A$
\begin{equation}\label{tensorbarcomplex}
    \M^r \otimes_\A \N^l
\end{equation}
is by definition the graded vector space
\begin{equation} 
    \M^r \otimes T\A \otimes \N^l
\end{equation}
with differential 
\begin{equation}
    d_{\M \otimes_\A \N}: \M^r \otimes T\A \otimes \N^l \ra \M^r \otimes T\A \otimes \N^l
\end{equation}
given by
    \begin{equation}
        \begin{split}
         d(&\mathbf{m},a_1, \ldots, a_k,\mathbf{n}) =\\
         &\sum (-1)^{\maltese_{-(k+1)}^{-(r+1)}}\mu_\M^{1|r}(\mathbf{m}, a_1, \ldots, a_r) \otimes a_{r+1} 
        \otimes \cdots \otimes a_k \otimes \mathbf{n} \\
         &+ \sum \mathbf{m} \otimes a_1 \otimes \cdots \cdots a_{k-s} \otimes 
         \mu_\N^{s|1}(a_{k-s+1}, \ldots, a_k, \mathbf{n})\\
         &+ \sum (-1)^{\maltese_{-(k+1)}^{-(s+j+1)}}\mathbf{m} \otimes a_1 \otimes \cdots \otimes 
         \mu^j_\A(a_{s+1}, \ldots, a_{s+j})\otimes a_{s+j+1} \otimes \cdots 
         \otimes a_k \otimes \mathbf{n}.
    \end{split}
\end{equation}
In the opposite direction, tensoring over ${\mathbb K}$, we obtain the
{\it product $\A\!-\!\B$ bimodule} 
\begin{equation}
    \N^l \otimes_{\mathbb K} \M^r,
\end{equation}
which equals $\N^l \otimes_{\mathbb K} \M^r$ on the level of graded vector spaces and has 
\begin{equation}
    \begin{split}
        \mu^{r|1|s}_{\N \otimes_{\mathbb K} \M}(a_1, &\ldots, a_r, \mathbf{n} \otimes \mathbf{m}, b_1, \ldots, b_s) := \\
    &\begin{cases}
        (-1)^{|\mathbf{m}|}\mu^{r|1}_\N(a_1, \ldots, a_r, \mathbf{n}) \otimes \mathbf{m} & s = 0, r>0 \\
        \mathbf{n} \otimes \mu^{1|s}_\M(\mathbf{m}, b_1, \ldots, b_s) & r = 0, s> 0 \\
        (-1)^{|\mathbf{m}|}\mu^{1|0}_\N(\mathbf{n}) \otimes \mathbf{m} + \mathbf{n} \otimes \mu^{0|1}_\M (\mathbf{m}) & r = s= 0\\
        0 & \mathrm{otherwise}
    \end{cases}
\end{split}
\end{equation}
The $\ainf$ bimodule equations follow from the $\ainf$ module equations for
$\M$ and $\N$.

Finally, given an $\mc{A}\!-\!\mc{B}$ bimodule $\mc{M}$ and a
$\mc{B}\!-\!\mc{A}$ bimodule $\mc{N}$, we can simultaneously tensor over the
$\mc{A}$ and $\mc{B}$ module structures to obtain a chain complex.
\def\mn{\mc{M} \otimes_{\mc{A}\!-\!\mc{B}} \mc{N}}
\begin{defn}
    The {\bf bimodule tensor product} of $\mc{M}$ and $\mc{N}$ as above, denoted
    \begin{equation}
        \mc{M} \otimes_{\mc{A}\!-\!\mc{B}} \mc{N}
    \end{equation}
    is a chain complex defined as follows: As a vector space,
    \begin{equation}
        \mn := (\mc{M} \otimes T\mc{B} \otimes \mc{N} \otimes T \mc{A})^{diag},
    \end{equation}
    where the $diag$ superscript means to restrict to cyclically composable
    elements. The differential on $\mn$ is 
    \begin{equation}
        \begin{split}
            d_{\mn}:\mathbf{m}&\otimes b_k\otimes \cdots\otimes b_1 \otimes \mathbf{n}\otimes a_1\otimes \cdots \otimes a_l \longmapsto \\
            &\sum_{r,s} (-1)^{\#_{r,s}}\mu^{l-r|1|k-s}_{\mc{M}}(a_{r+1}, \ldots, a_l, \mathbf{m}, b_k, \ldots, b_{s+1}) \otimes b_{s} \otimes \cdots \otimes b_1 \otimes \\
            &\bigsp\mathbf{n} \otimes a_1 \otimes \cdots \otimes a_{r} \\
            &+ \sum_{i,r} (-1)^{\maltese_{-l}^i}\mathbf{m} \otimes b_k \otimes \cdots b_{i+r+1} \otimes \mu^r_{\mc{B}}(b_{i+r}, \ldots, b_{i+1}) \otimes b_i \otimes \cdots \otimes b_1 \otimes \\
            &\bigsp\mathbf{n} \otimes a_1 \otimes \cdots \otimes a_l \\
            &+ \sum_{j,s} (-1)^{\maltese_{-l}^{-(j+s+1)}}\mathbf{m} \otimes b_k
            \otimes \cdots b_1 \otimes \mathbf{n} \otimes a_1\otimes \cdots 
            \otimes a_j \otimes \\
            &\bigsp\mu^s_{\mc{A}}(a_{j+1}, \ldots, a_{j+s}) \otimes a_{j+s+1} 
            \otimes \cdots \otimes a_l \\
            &+ \sum_{r,s} (-1)^{\maltese_{-l}^{-(s+1)}}\mathbf{m} \otimes 
            b_k \otimes \cdots \otimes b_{r+1} \otimes 
            \mu^{r|1|s}_{\mc{N}}(b_{r}, \ldots, b_1, \mathbf{n}, 
            a_1, \ldots, a_s) \otimes\\
            &\bigsp a_{s+1} \otimes \cdots \otimes a_l
    \end{split}
\end{equation}
with signs given by:
\begin{align}
    \maltese_{-l}^{-t} &:= \sum_{n=t}^l ||a_n||\\
    \maltese_{-l}^{i} &:= \sum_{n=1}^l ||a_n|| + |\mathbf{n}| + \sum_{m=1}^i ||b_m||\\
    \label{sharpsign}\#_{i,r} &:= \bigg(\sum_{n=r+1}^l ||a_n|| \bigg) \cdot \bigg(|\mathbf{m}| + \sum_{m=1}^k ||b_m|| + |\mathbf{n}| + \sum_{n=1}^r ||a_n||\bigg) + \maltese_{-r}^s.
\end{align} 
The sign (\ref{sharpsign}) should be thought of as the Koszul sign coming from
moving $a_{r+1}, \ldots, a_l$ past all the other elements, applying
$\mu_{\mc{M}}$ (which acts from the right), and then moving the result to the
left.
\end{defn}
The bimodule tensor product is functorial in the following sense. If 
\begin{equation}
    \mc{F}: \mc{N} \lra \mc{N}'
\end{equation}
is a morphism of $\B\!-\!\A$ bimodules, then there is an induced morphism
\begin{equation}
    \mc{M} \otimes_{\A\!-\!\B} \mc{N} \stackrel{\mc{F}_{\#}}{\lra} \mc{M} \otimes_{\A\!-\!\B} \mc{N}'
\end{equation}
given by summing with signs over all ways to collapse some of the terms around
the element of $\mc{N}$ by the various $\mc{F}^{r|1|s}$, which can be concisely written as
\begin{equation}\label{fsharp}
    \mc{F}_{\#}(\mathbf{m}\otimes b_1\otimes \cdots\otimes b_k \otimes \mathbf{n}\otimes a_1\otimes \cdots \otimes a_l) : = \mathbf{m} \otimes \hat{\mc{F}}(b_1, \ldots, b_k, \mathbf{n}, a_1, \ldots, a_l).
\end{equation}
One can then see that 
\begin{prop}\label{inducedmorphisms}
    Via (\ref{fsharp}), quasi-isomorphisms of bimodules induce quasi-isomorphisms of complexes.
\end{prop}

\begin{rem}
There are identically induced morphisms $\mc{G}_{\#}: \mc{M}
\otimes_{\A\!-\!\B} \mc{N} \ra \mc{M}' \otimes_{\A\!-\!\B}  \mc{N}$ from
morphisms $\mc{G}: \mc{M} \ra \mc{M}'$. One simply needs to add additional
Koszul signs coming from moving elements of $\B$ to the beginning in order to
apply $\mc{G}$. 
\end{rem}

\begin{rem}
    Suppose for a moment that the categories $\mc{A}$ and $\mc{B}$ have one
    object each and no higher products; e.g. $A:= \mc{A}$ and $B:= \mc{B}$ can
    be thought of as ordinary unital associative algebras over $\mathbb{K}$.
    Similarly, let $M$ and $N$ be ordinary $A\!-\! B$ and $B\!-\! A$ bimodules
    respectively. Then, the explicit chain level description we have given
    above for $M \otimes_{A\!-\!B} N$ is a bar-complex model computing the
    derived tensor product or bimodule Tor
    \begin{equation}
        Tor_{A\!-\!B}(M,N):=M \otimes^{\mathbb L}_{A\!-\! B} N.
    \end{equation}
    Now, note that $M$ can be thought of as a right $A^{op} \otimes B$ module and
    $N$ can be thought of as a left $A^{op} \otimes B$ module. Thus, using
    (\ref{tensorbarcomplex}) we can write down an explicit chain complex
    computing the tensor product $M \otimes_{A^{op} \otimes B}^{\mathbb L} N$ as
    \begin{equation}
        M \otimes T(A^{op} \otimes B) \otimes N.
    \end{equation}
    plus a standard differential. This is a second canonical bar complex
    that computes the same Tor group $Tor_{A\!-\!B}(M,N) = Tor_{A^{op} \otimes
    B}(M,N)$; in particular these two complexes have the same homology. There
    are natural intertwining chain maps explicitly realizing this
    quasi-isomorphism, which we would like to emulate in the $\ainf$-category
    setting.  However, we stumble into a substantial initial roadblock: there
    is not a clean notion of the tensor product of $\ainf$ algebras or
    categories $\mc{A}^{op}\otimes \mc{B}$. We defer further discussion of
    these issues to a later point in the paper.
\end{rem}

\subsection{The diagonal bimodule}\label{diagonalbimodulesection}
For any $\ainf$ category $\mc{A}$, there is a natural $\mc{A}\!-\!\mc{A}$
bimodule quasi-representing the identity convolution endofunctor.
\begin{defn}
    The {\bf diagonal bimodule} $\mc{A}_{\Delta}$ is specified by the following
    data:
    \begin{align}
        \mc{A}_{\Delta}(X,Y) & := \hom_{\mc{A}}(Y,X) \\
        \mu_{\mc{A}_{\Delta}}^{r|1|s}(c_r, \ldots, c_1, \mathbf{c}, c_1',
        \ldots, c_s') &:= (-1)^{\maltese_{-s}^{-1} + 1}\mu_{\mc{A}}^{r+1+s}(c_r, \ldots, c_1, \mathbf{c},
        c_1', \ldots, c_s'). 
    \end{align}
    with 
    \begin{equation}
        \maltese_{-s}^{-1} := \sum_{i=1}^s ||c_i'||.
    \end{equation}
\end{defn}
\noindent One of the standard complications in theory of bimodules is that
tensor product with the diagonal is only quasi-isomorphic to the identity.
However, these quasi-isomorphisms are explicit, at least in one direction.
\begin{prop}\label{tensordiagonal}
    Let $\M$ be a homologically unital right $\ainf$ module $\M$ over $\A$.
    Then, there is a quasi-isomorphism of modules
    \begin{equation} 
        \mc{F}_{\Delta,right}: \M \otimes_\A \A_{\Delta} \lra \M 
    \end{equation}
given by the following data:
\begin{equation} \label{tensormapright}
    \begin{split}
        \mc{F}_{\Delta,right}^{1|l}:\M \otimes T \A \otimes \A_{\Delta} \otimes \A^{\otimes l} &\lra \M  \\
    (\mathbf{m}, a_k, \ldots, a_1, \mathbf{a}, a^1_1, \ldots, a^1_l) &\longmapsto
    (-1)^{\circ_{-l}^k}\mu^{1|k+l+1}_{\M} (\mathbf{m}, a_k, \ldots, a_1,\mathbf{a}, a^1_1, \ldots a^1_l),
 \end{split}
 \end{equation}
 where the sign is
 \begin{equation}
     \circ_{-l}^k = \sum_{n=1}^l ||a^1_n|| + | \mathbf{a}| - 1 + \sum_{m=1}^k ||a_m||.
 \end{equation}
There are similar quasi-isomorphisms of homologically unital left-modules
\begin{equation}\label{tensormapleft}
 \begin{split}
     \mc{F}_{\Delta,left}: \A_{\Delta} \otimes_\A \N &\lra \N\\
     \mc{F}^{l|1}: (a_1^1, \ldots, a_l^1, \mathbf{a}, a_k, \ldots a_1, \mathbf{n}) &\longmapsto (-1)^{\bullet_0^k}\mu^{k+l+1|1}(a_1^1, \ldots, a_l^1, \mathbf{a}, a_k, \ldots, a_1, \mathbf{n}).
 \end{split}
\end{equation}
and quasi-isomorphisms of homologically unital bimodules
\begin{equation} \label{tensormapbimodright}
   \begin{split}
       \mc{F}_{\Delta,right}: \B \otimes_\A &\A_{\Delta} \lra \B \\
       \mc{F}^{r|1|s}_{\Delta,right}:\A^{\otimes r} \otimes \B \otimes T\A \otimes &\A_{\Delta} \otimes \A^{\otimes s} \lra \B\\
       (a_1, \ldots, a_r, \mathbf{b}, a_l^1, \ldots, a_1^1, \mathbf{a},& a_1^2, \ldots, a_s^2) \longmapsto \\
       &(-1)^{\circ_{-s}^l}\mu^{r|1|l+s+1}_{\B}(a_1, \ldots, a_r, \mathbf{b}, a_l^1, \ldots, a_1^l, \mathbf{a}, a_1^2, \ldots, a_s^2)
   \end{split}
\end{equation}
\begin{equation} \label{tensormapbimodleft}
    \begin{split}
        \mc{F}_{\Delta,left}:\A_{\Delta} \otimes_\A &\B \lra \B\\
        \mc{F}^{r|1|s}_{\Delta,left}:\A^{\otimes r} \otimes \A_{\Delta} \otimes T\A \otimes &\B \otimes \A^{\otimes s} \lra \B\\
        (a_1, \ldots, a_r, \mathbf{a}, a_l^1, \ldots, a_1^1, \mathbf{b},& a_1^2, \ldots, a_s^2) \longmapsto\\
        &(-1)^{\bigstar_{-s}^l}\mu^{r+l+1|1|s}_\B(a_1, \ldots, a_r, \mathbf{a}, a_l^1, \ldots, a_1^1, \mathbf{b}, a_1^2, \ldots, a_s^2)
   \end{split}
   \end{equation}
   with signs
   \begin{align}
    \bullet_{0}^k &:= |\mathbf{n}|-1 + \sum_{i=1}^k ||a_i||\\
       \circ_{-s}^l & := \sum_{n=1}^s ||a^2_n|| + |\mathbf{a}| - 1 + \sum_{m=1}^l ||a^1_m||\\
       \bigstar_{-s}^l  &:= \sum_{n=1}^s ||a^2_n|| + |\mathbf{b}| - 1 + \sum_{m=1}^l ||a^1_m||.
   \end{align}
\end{prop}
\begin{proof}
We will just establish that (\ref{tensormapbimodright}) is a quasi-isomorphism;
the other bimodule case (\ref{tensormapbimodleft}) is analogous and the module
cases (\ref{tensormapleft}) and (\ref{tensormapright}) are special cases. Also,
we omit signs from the proof, leaving them as an exercise.
First, suppose that $\A$, $\cc$ are ordinary unital associative algebras $A, C$
over semi-simple rings $R$, $R'$ (the case $\mu^1 = 0$, $\mu^k = 0$ for $k>2$).
Similarly,
suppose $\B$ is an ordinary unital $\cc\!-\! \A$-bimodule $B$ ($\mu^{1|0} = 0$,
so $\mu^{0|1|1}$ is multiplication by elements from $A$, $\mu^{1|1|0}$ is
multiplication by $C$, and $\mu^{r|1|s} = 0$ for $r+s \neq 1$). In this special
case, the bimodule
\begin{equation}
    B \otimes_A A_{\Delta}
\end{equation}
has internal differential given by
\begin{equation}
    \begin{split}
        d_{B \otimes_A A_{\Delta}} := \mu^{0|1|0}_{B \otimes_A A_{\Delta}} : \mathbf{b} \otimes a_1 \otimes \cdots &\otimes a_k \otimes \mathbf{a} \longmapsto\\
        &(\mathbf{b} \cdot a_1) \otimes \cdots \otimes a_k \otimes \mathbf{a}   \\
        + &\sum_{i} \mathbf{b} \otimes a_1 \otimes \cdots (a_{i} \cdot a_{i+1}) \otimes \cdots a_k \otimes \mathbf{a}  \\
        + &\mathbf{b} \otimes a_1 \otimes \cdots \otimes a_{k-1} \otimes (a_k \cdot \mathbf{a}),
\end{split}
\end{equation}
where we are not allowed to multiply $\mathbf{b}$ with $\mathbf{a}$, e.g. $d(\mathbf{b} \otimes \mathbf{a}) = 0.$ The morphism of bimodules
\begin{equation}
    \mc{F}_{\Delta,right} = \oplus \mc{F}^{k|1|l}: B \otimes_A  A_{\Delta} \lra M
\end{equation}
has first order term $\mc{F}^{0|1|0}$ given by
\begin{equation}
    \mc{F}^{0|1|0}: \mathbf{b}\otimes a_1\otimes \cdots\otimes a_k\otimes \mathbf{a} \longmapsto \begin{cases}
        \mathbf{b} \cdot \mathbf{a} & k=0\\
        0 & \textrm{otherwise.}
    \end{cases}
\end{equation}
The cone of this morphism is 
\begin{equation} \label{coneassociative}
    \mathrm{Cone}(\mc{F}^{0|1|0}) := (M \otimes T\A \otimes A_{\Delta}) \oplus M[1]
\end{equation}
with differential given by the following matrix
\begin{equation}
    \left(\begin{array}{cc}
        d_{B \otimes_A A_{\Delta}} & 0 \\
        \mc{F}^{0|1|0} & 0,
    \end{array} \right).
\end{equation}
This cone complex is visibly identical to the classical right-sided {\bf bar
complex} for $B$ over $A$, i.e. the chain complex 
\begin{equation}
    B \otimes TA
\end{equation}
with differential
\begin{equation}
    d(\mathbf{b}\otimes a_1 \otimes \cdots \otimes a_k) = (\mathbf{b}\cdot a_1) \otimes a_2 \otimes \cdots \otimes a_k + \sum_i \mathbf{b} \otimes a_1 \otimes \cdots \otimes (a_i\cdot a_{i+1}) \otimes \cdots a_k.
\end{equation}
But for $B$ and $A$ unital, the bar complex is known to be acylic, with
contracting homotopy \begin{equation}
    \mathfrak{h}: \mathbf{b} \otimes a_1 \otimes \cdots \otimes a_k \longmapsto \mathbf{b} \otimes a_1 \otimes \cdots \otimes a_k \otimes e,
\end{equation}
where $e$ is the unit of $A$.

In the general case where $\mc{B}$ and $\mc{A}$ may have differentials and
higher products, the cone of $\mc{F}^{0|1|0}$ is the complex
\begin{equation}
    (\mc{B} \otimes T\A \otimes \A_{\Delta}) \oplus \mc{B}[1]
\end{equation}
with differential
\begin{equation}
    \left(\begin{array}{cc}
    d_{\mc{B} \otimes_\A \A_{\Delta}} & 0 \\
        \mc{F}^{0|1|0} & \mu^{0|1|0}_{\mc{B}},
    \end{array} \right).
\end{equation}
Here $d_{\mc{B} \otimes_\A \A_{\Delta}}$ is the internal bimodule differential.
This differential respects the {\it length filtration} of the complex, and thus
we can look at the associated spectral sequence. The only terms that preserve
length involve the differentials $\mu^{0|1|0}_{\mc{B}}$ and $\mu^1_{\A}$ and thus
the first page of the spectral sequence is the complex
\begin{equation}
    (H^*(\mc{B}) \otimes T (H^*(\A)) \otimes H^*(\A_{\Delta})) \oplus H^*(\mc{B})
\end{equation}
with first page differential given by all of the homology-level terms involving
$\mu^2_{\A}$ and $\mu^{1|1}_{\mc{B}}$. This is exactly the cone complex
considered in (\ref{coneassociative}) for the homology level morphism 
\begin{equation}
    H^*(\mc{F}^{1|1}_{\Delta,right}): H^*(\B)\otimes T (H^*(\A)) \otimes H^*(\A_{\Delta}) \lra H^*(\B);
\end{equation}
hence the first page differential is acylic.
\end{proof}
The case of Proposition \ref{tensordiagonal} in which $\A$ is strictly unital
also appears in \cite{Seidel:2008cr}*{\S 2}. Now, finally suppose we took the
tensor product with respect to the diagonal bimodule on the right and left of a
bimodule $\mc{B}$. Then, as one might expect, one can compose two of the above quasi-isomorphisms 
to obtain a direct 
\begin{equation}\label{fleftright}
    \mc{F}_{\Delta, left, right}:= \mc{F}_{\Delta,left } ``\circ" \mc{F}_{\Delta,right}: \A_{\Delta} \otimes_\A \mc{B} \otimes_\A \A_{\Delta} \stackrel{\sim}{\lra} \mc{B},
\end{equation}
which, explicitly, is given by
\begin{equation}
    \begin{split}
    \mc{F}_{\Delta,left,right}^{r|1|s}& :\\
    a_1\otimes \cdots \otimes & a_r \otimes \mathbf{a}\otimes a_1'\otimes \cdots \otimes a_l'\otimes \mathbf{b}\otimes a_1''\otimes \cdots\otimes a_k''\otimes \mathbf{a}'\otimes a_1'''\otimes \cdots\otimes a_s''' \\
    \longmapsto \sum &\mu_{\mc{B}}^{r+1+i|1|s-j} (a_1, \ldots, a_r, \mathbf{a}, a_1', \ldots, a_i', \mu_{\mc{B}}^{l-i|1|k+1+j}(a_{i+1}', \ldots, a_l',\\
    &\mbox{ }\mbox{ }\mbox{ }\mbox{ }\mbox{ }\mbox{ }\mathbf{b}, a_1'',\ldots, a_k'', \mathbf{a}', a_1''', \ldots, a_j'''), a_{j+1}''', \ldots, a_s''').
\end{split}
\end{equation}
up to signs that have already been discussed.  There is an analogous morphism
$\mc{F}_{\Delta,right,left}$ given by collapsing on the left first before
collapsing to the right.

\subsection{The Yoneda embedding}
Objects in $\cc$ provide a natural source for left and right $\cc$-modules.
\begin{defn} 
    Given an object $X \in \ob \cc$, the {\bf left Yoneda-module} $\mc{Y}^l_X$
    over $\cc$ is defined by the following data: 
    \begin{equation}
    \mc{Y}^l_X(Y) := \hom_\cc(X,Y) \mathrm{\ for\ any\ }Y\in \ob \cc
\end{equation}
\begin{equation} 
    \begin{split}
    \mu^{r|1}: \hom_\cc(Y_{r-1},Y_r) \times \hom_\cc(Y_{r-2},Y_{r-1}) \times &\cdots \times \hom_\cc(Y_0,Y_1) \times \mc{Y}^l_X(Y_0) \lra \mc{Y}^l_X(Y_r) \\
    (y_r,\ldots, y_1,\mathbf{x}) &\longmapsto (-1)^{\maltese_0^r}\mu^{r+1}(y_r, \ldots, y_1, \mathbf{x}),
\end{split}
\end{equation}
with sign
\begin{equation}
    \maltese_0^r = ||\mathbf{x}|| + \sum_{i=1}^r ||y_i||.
\end{equation}
Similarly, the {\bf right Yoneda-module} $\mc{Y}^r_X$ over $\cc$ is defined by the following data:
\begin{equation}
    \mc{Y}^r_X(Y) := \hom_\cc(Y,X) \mathrm{\ for\ any\ }Y\in \ob \cc
\end{equation}
\begin{equation} 
    \begin{split}
    \mu^{1|s}: \mc{Y}^r_X(Y_s)\times \hom_\cc(Y_{s-1},Y_s) \times \hom_\cc(Y_{s-2},Y_{s-1}) \times &\cdots \times \hom_\cc(Y_0,Y_1) \ra \mc{Y}^r_X(Y_0) \\
    (\mathbf{x}, y_s,\ldots, y_1) &\mapsto \mu^{s+1}(\mathbf{x},y_s, \ldots, y_1).
\end{split}
\end{equation}
\end{defn}
\noindent These modules are associated respectively to the {\bf left} and {\bf
right Yoneda embeddings}, $\ainf$ functors which we will now describe.
\begin{defn}
    The {\bf left Yoneda embedding} is a contravariant $\ainf$ functor
    \begin{equation}
        \yl: \cc^{op} \lra \cc\mod
    \end{equation}
    defined as follows: On objects,
    \begin{equation} \yl(X) := \yyl{X}.\end{equation}
        On morphisms
        \begin{equation}\label{leftyonedaphi}
    \begin{split}
    \yl^d: \hom(X_{d-1},X_d) \times \cdots \times \hom(X_0,X_1) &\lra \hom_{\cc\mod}(\yyl{X_d},\yyl{X_0})\\
    (x_d, \ldots, x_1) &\longmapsto \phi_{(x_1, \ldots, x_d)}
\end{split}
\end{equation}
where $\phi_{\vec{x}}:=\phi_{x_1, \ldots, x_d}$ is the morphism given by
\begin{equation}\label{leftyonedaphiexplanation}
    \begin{split}
    \hom(Y_{f-1},Y_f) \times \cdots \times \hom(Y_0,Y_1) \times 
    &\yyl{X_d}(Y_0) \lra \yyl{X_0}(Y_f)\\
    (y_f, \ldots, y_1, \mathbf{m}) &\longmapsto (-1)^{\maltese_{-d}^f}\mu_\cc^{f+d+1}(y_f, \ldots, y_1, \mathbf{m}, x_d, \ldots, x_1)
\end{split}
\end{equation}
with sign
\begin{equation}
    \maltese_{-d}^f = \sum_{i=1}^d ||x_i|| + ||\mathbf{m}|| + \sum_{j=1}^f ||y_j||.
\end{equation}
\end{defn}
\begin{defn}
    The {\bf right Yoneda embedding} is a (covariant) $\ainf$ functor
    \begin{equation}
        \yr: \cc \lra \rmod\cc
    \end{equation}
    defined as follows: On objects,
\begin{equation} \yr(X) := \yyr{X}.\end{equation}
        On morphisms
\begin{equation}
    \begin{split}
    \yr^d: \hom(X_{d-1},X_d) \times \cdots \times \hom(X_0,X_1) &\lra \hom_{\rmod\cc}(\yyr{X_0},\yyr{X_d})\\
    (x_d, \ldots, x_1) &\mapsto \psi_{(x_1, \ldots, x_d)}
\end{split}
\end{equation}
where $\psi_{\vec{x}}:=\psi_{x_1, \ldots, x_d}$ is the morphism given by
\begin{equation}
    \begin{split}
        \yyr{X_0}(Y_f) \times \hom(Y_{f-1},Y_f) \times \cdots \times \hom(Y_0,Y_1) 
     &\lra \yyr{X_d}(Y_0)\\
     (\mathbf{m}, y_f, \ldots, y_1) &\longmapsto \mu_\cc^{f+d+1}(x_d, \ldots, x_1, \mathbf{m}, y_f, \ldots, y_1).
\end{split}
\end{equation}
\end{defn}

An important feature of these modules, justifying the use of module categories,
is that the $\ainf$ Yoneda embedding is full. In fact, a slightly stronger
result is true, which we will need.
\begin{prop}[Seidel
    \cite{Seidel:2008zr}*{Lem. 2.12}] \label{moduleyoneda}
    Let $\cc$ be a homologically unital, and let $\mc{M}$ and $\mc{N}$ be
    homologically unital left and right $\cc$ modules respectively. Then, for
    any object $X$ of $\cc$ there are quasi-isomorphisms of chain complexes
    \begin{align}
        \lambda_{\mc{M},X}: \mc{M}(X) &\stackrel{\sim}{\lra} \hom_{\cc\mod}(\mc{Y}^l_X,\mc{M})\\
        \lambda_{\mc{N},X}: \mc{N}(X) &\stackrel{\sim}{\lra} \hom_{\rmod\cc}(\mc{Y}^r_X,\mc{N}).
    \end{align}
\end{prop}

When $\mc{M} = \mc{Y}^l_Z$ or $\mc{N} = \mc{Y}^r_Z$, the
quasi-isomorphisms defined above 
\begin{align}
    \hom_{\cc} (X,Z) &\stackrel{\sim}{\lra} \hom_{\cc\mod}(\mc{Y}^l_Z,\mc{Y}^l_X)\\
    \hom_{\cc} (X,Z) &\stackrel{\sim}{\lra} \hom_{\cc\mod}(\mc{Y}^r_X,\mc{Y}^r_Z)
\end{align}
are exactly the first order terms of the Yoneda embeddings $\yl^l$ and
$\yr^1$, implying that 
\begin{cor}[\cite{Seidel:2008zr}*{Cor. 2.13}]
    The Yoneda embeddings $\yl$ and $\yr$ are full.
\end{cor}

In Section \ref{moduleduality}, we will prove analogous results for bimodules.

\subsection{Pullbacks of modules and bimodules} \label{pullback}
Given an $\ainf$ functor 
\begin{equation}
\mathfrak{F}: \mc{A} \ra \mc{B}
\end{equation} there is an associated pull-back functor on modules
\begin{equation}
\mf{F}^* : \mc{B}\mod \longrightarrow \mc{A}\mod,
\end{equation}
defined as follows: 
\begin{defn}
    Given a right $\mc{B}$ module $\mc{M}$ with structure maps $\mu^{1|r}_\M$,
    and an $\ainf$ functor $\mathfrak{F}: \mc{A} \ra \mc{B}$, define the {\bf
    pullback of $\mc{M}$ along $\mf{F}$} to be the right $\mc{A}$ module
\begin{equation}
    \mf{F}^* \M (Y):= \M(\mf{F}(Y)),\ Y\in \ob \mc{A}
\end{equation}
with module structure maps
\begin{equation}
    \mu^{1|r}_{\mf{F}^*\M}(\mathbf{a},a_1, \ldots, a_r) := 
    \sum_{k,i_1 + \cdots + i_k = r} \mu^{1|k}_\M\left(\underline{\mathbf{a}},
    \mf{F}^{i_1}(a_1, \ldots), \cdots, \mf{F}^{i_k}(\ldots, a_r)\right).
\end{equation}
Here on the right side, $\underline{\mathbf{a}}$ is simply $\mathbf{a}$ thought
of as living in some $\M(\mf{F}(Y))$ instead of $\mf{F}^* \M(Y)$.
\end{defn}
\begin{ex} \label{yonedasubcategory}
    Given an $\ainf$ category $\mc{C}$ and a collection of objects $\{X_i\}$ in
    $\mc{C}$, let $\mc{X}$ be the full subcategory of $\mc{C}$ with objects
    $\{X_i\}$. Then the naive inclusion functor
    \begin{equation}
        \iota: \mc{X} \longhookrightarrow \mc{C},
    \end{equation}
    induces a pullback on modules
    \begin{equation}
        \begin{split}
        \iota^*: \mc{C}\mod &\lra \mc{X}\mod\\
        \iota^*: \rmod\mc{C} &\lra \rmod\mc{X}
    \end{split}
    \end{equation}
    which is the ordinary restriction. Namely, a $\mc{C}$ module such as
    \begin{equation}
        \mc{Y}^r_Z;\ Z \in \ob \mc{C}
    \end{equation}
    induces a module 
    \begin{equation}
        \iota^*\mc{Y}^r_Z
    \end{equation}
    over $\mc{X}$, in which one pairs $Z$ only with objects in $\mc{X}$. We
    will often refer to this module simply as $\mc{Y}^r_Z$ when the category
    $\mc{X}$ is explicit.  
\end{ex}

\def\f{\mf{F}}
\def\bbf{\bar{\mf{F}}}
We can repeat these definitions for contravariant functors, which we will need.
Namely, let
\begin{equation}
\bar{\mf{F}}: \mc{A}^{op} \ra \mc{B}
\end{equation}
be a contravariant $\ainf$ functor, 
which consists of maps
\begin{equation}
\bar{\mf{F}}^1: \mc{A}(X,Y) \ra \mc{B}(\bbf(Y),\bbf(X))
\end{equation}
and higher order maps
\begin{equation}
\bbf^d: \mc{A}(X_{d-1},X_d) \otimes \mc{A}(X_{d-2},X_{d-1}) \otimes \cdots \otimes \mc{A}(X_0,X_1) \longrightarrow \mc{B}(\bbf(X_d),\bbf(X_0))
\end{equation}
satisfying the following equations
\begin{equation}
\sum \mu_\mc{B}\left(\bbf^{i_k}(\dots a_k) \bbf^{i_{k-1}}(\ldots) 
\cdots \bbf^{i_1}(a_1, \ldots, a_{i_1})\right) = 
\sum \bbf\left(a_1\cdots \mu_\A(\cdots)\cdots a_k\right).
\end{equation}
(notice the order reversal in the $a_i$).
In this case, pull-back changes the direction of the module action
\begin{equation}
\bbf^*: \mc{B}\mod \longrightarrow \rmod\mc{A}
\end{equation}
\begin{defn} 
    Given a left $\mc{B}$ module $\mc{M}$, and a contravariant functor $\bbf:
    \mc{A}^{op} \ra \mc{B}$ as above, the {\bf pulled-back right module} $\bbf^*\M$
    is defined by \begin{equation}
    \begin{split}
    \bbf^*\M(X) &= \M(\bbf(X)),\ X \in \ob \mc{A} \\
    \mu^{1|r}_{\bbf^*\M}(\mathbf{m},a_1, \ldots, a_r) &= \sum_{k,i_1 + \cdots +  i_k = r} \mu^{k|1}(\bbf^{i_k}(\ldots, a_r)\cdots \bbf^{i_1}(a_1, \ldots, a_{i_1}),\underline{\mathbf{m}})
\end{split}
\end{equation}
\end{defn}

\noindent This entire process can be repeated for bimodules, with contravariant
or covariant functors. Given $\ainf$ categories $\mc{A}_1$, $\mc{A}_2$,
$\mc{B}_1$, $\mc{B}_2$ and functors 
\begin{equation}
    \begin{split}
    \mf{F} &: \mc{A}_1 \longrightarrow \mc{A}_2 \\
    \mf{G} &: \mc{B}_1 \longrightarrow \mc{B}_2
\end{split}
\end{equation}
there is an associated pull-back functor
\begin{equation}
    (\mf{F} \otimes \mf{G})^*: \mc{A}_2\bimod\mc{B}_2 \longrightarrow \mc{A}_1\bimod\mc{B}_1
\end{equation}
defined as follows: If $\mc{M} \in \ob \mc{A}_2\bimod\mc{B}_2$, then
\begin{equation}
    (\mf{F}\otimes \mf{G})^* \mc{M}(A,B) := \mc{M}(\mf{F}(A),\mf{G}(B)),
    \ A \in \ob \mc{A}_1,\ B \in \ob \mc{B}_1
\end{equation}
with structure maps 
\begin{equation}
    \begin{split}
     \mu^{r|1|s}_{(\mf{F} \otimes \mf{G})^*\M}(a_r, \ldots, a_1,& 
    \mathbf{m}, b_1, \ldots, b_s) := \\
    \sum_{k,i_1 + \cdots + i_k = r} &\sum_{l, j_1 + \cdots j_l = s} 
    \mu^{k|1|l}_\M (\mf{F}^{i_k}(a_r, \ldots, ), \cdots, \mf{F}^{i_1}(\ldots
    a_1),\\ &\underline{\mathbf{m}}, \mf{G}^{j_1}(b_1,
    \ldots), \cdots, \mf{G}^{j_l}(\ldots, b_s)).
\end{split}
\end{equation}
Once more, $\underline{\mathbf{m}}$ is simply the element $\mathbf{m}$ thought
of as living in $\mc{M}(\mf{F}(A),\mf{G}(B))$ for some $A$, $B$. Finally, abbreviate the pull-back $(\mf{F} \otimes \mf{F})^*$ by simply $\mf{F}^*$.

The Yoneda embeddings $\yl$ and $\yr$ behave compatibly with pullback, in the
sense that for any functor $\mathbf{F}: \cc \lra \mc{D}$
there are natural transformation of functors
\begin{equation}
    \begin{split}
        \mf{T}^{\mathbf{F}}_L: (\yl)_{\cc} &\lra \mathbf{F}^* \circ (\yl)_{\mc{D}} \circ \mathbf{F}\\
    \mf{T}^{\mathbf{F}}_R: (\yr)_{\cc} &\lra \mathbf{F}^* \circ (\yr)_{\mc{D}} \circ \mathbf{F}.
\end{split}
\end{equation}
 We will not need the full data of this
natural transformation (the interested reader is referred to
\cite{Seidel:2008zr}*{eq. (1.23)}), but to first order, we obtain morphisms of
modules
\begin{equation} \label{nattransmodulemorphism}
    \begin{split}
    (\mf{T}^{\mathbf{F}}_L)_X: \mc{Y}^l_X &\lra \mathbf{F}^*\mc{Y}^l_{\mathbf{F}(X)}.\\
    (\mf{T}^{\mathbf{F}}_R)_X: \mc{Y}^r_X &\lra \mathbf{F}^*\mc{Y}^r_{\mathbf{F}(X)}.
\end{split}
\end{equation}
given by
\begin{equation}
    \begin{split}
        (\mf{T}_L^\mathbf{F})_X^{r|1}(a_1, \ldots, a_r, \mathbf{x}) &:= \mathbf{F}^{r+1}(a_1, \ldots, a_r \mathbf{x}).\\
        (\mf{T}_R^\mathbf{F})_X^{1|s}(\mathbf{x}, a_1, \ldots, a_s) &:= \mathbf{F}^{s+1}(\mathbf{x}, a_1, \ldots, a_s).
    \end{split}
\end{equation}
Finally, tensoring $(\mf{T}_L^\mathbf{F})_X$ and $(\mf{T}_R^{\mathbf{F}})_Z$,
we obtain associated morphisms of Yoneda bimodules
\begin{equation}\label{nattransbimodulemorphism}
    (\mf{T}_{LR}^{\mathbf{F}})_{X,Z}:= (\mf{T}^{\mathbf{F}}_L)_X \otimes (\mf{T}^{\mathbf{F}}_R)_Z: \mc{Y}^l_X \otimes \mc{Y}^r_Z \lra \mathbf{F}^*\mc{Y}^l_{\mathbf{F}(X)} \otimes \mathbf{F}^*\mc{Y}^r_{\mathbf{F}(Z)}.
\end{equation}
These maps are quasi-isomorphisms if $\mathbf{F}$ is. There are also analogous
versions of these natural transformations for contravariant functors
$\mathbf{G}$, in which naturally, Yoneda lefts and rights get reversed:
\begin{equation}
    \begin{split}
    (\mf{T}^{\mathbf{G}}_L)_X: \mc{Y}^l_X &\lra \mathbf{G}^*\mc{Y}^r_{\mathbf{G}(X)}.\\
    (\mf{T}^{\mathbf{G}}_R)_X: \mc{Y}^r_X &\lra \mathbf{G}^*\mc{Y}^l_{\mathbf{G}(X)}.
\end{split}
\end{equation}

\subsection{Hochschild invariants}
In what follows, let $\A$ be an $\ainf$ algebra or category, and $\B$ an
$\A\!-\!\A$ bimodule, frequently referred to as simply an $\A$-bimodule.
To such a pair ($\A$,$\B$), one can associate invariants known as {\bf
Hochschild cohomology} and {\bf Hochschild homology}. We momentarily bypass the
more conceptual route of defining these as bimodule Ext or Tor groups, and give
explicit co-chain level models, using the $\ainf$ bar complex.
\begin{defn}
    The {\bf (ordinary) Hochschild co-chain complex} of $\A$ with coefficients in $\B$ is
\begin{equation}
    \r{CC}^*(\A,\B) := \hom_{Vect}(T\A,\B),
\end{equation}
with grading
\begin{equation}\label{hochschildcochaingradings}
    \r{CC}^r(\A,\B) :=  \hom_{grVect}(\oplus \A^{\otimes j},\B[r+j]).
\end{equation}
Given a Hochschild co-chain $\phi \in \r{CC}^l(\A,\B)$, one can consider the
extension 
\begin{equation}
    \hat{\phi}: T\A \lra T\A \otimes \B \otimes T\A.
\end{equation}
given by
\begin{equation}
    \begin{split}
        \hat{\phi}(x_k, \ldots, x_1):=\sum_{i,j}(-1)^{l \cdot \maltese_1^i}
        x_k \otimes \cdots \otimes x_{j+1}  &\otimes \phi(x_j, \ldots, x_{i+1}) \otimes x_i \otimes  \cdots \otimes x_1.
    \end{split}
\end{equation}
with sign given by the degree $l$ of $\phi$ times
\begin{equation}
    \maltese_1^i := \sum_{s=1}^i ||x_s||.
\end{equation}
Then the {\bf differential} is given by:
\begin{equation}
    d(\phi): = \mu_{\B} \circ \hat{\phi} - \phi \circ \hat{\mu}_{\A}.
\end{equation}
\end{defn}
\noindent With respect to the grading, the differential clearly has degree 1.

In an analogous fashion, we give an explicit chain-level model for the 
Hochschild homology complex $(\r{CC}_*(\A,\B),d_{\r{CC}_*})$.
\begin{defn}
    Let $\A$ be an $\ainf$ algebra and $\B$ an $\ainf$ bimodule. The {\bf
    (ordinary) Hochschild homology chain complex} $\r{CC}_*(\A,\B)$ is defined to
    be
    \[ \r{CC}_*(\A,\B) := (\B \otimes_R T\A)^{\r{diag}}, \] where the
    $\r{diag}$ superscript means we
restrict to cyclically composable elements of $(\B \otimes_R T\A)$. Explicitly
this complex is the direct sum of, for any $k$ and any $k+1$-tuple of objects
$X_0, \ldots, X_k\in \ob \A$, the vector spaces
\begin{equation}
\B(X_k,X_0) \times
\hom_\A(X_{k-1},X_k) \times \cdots \times \hom_\A(X_0,X_1).
\end{equation}
\end{defn}
\noindent The differential $d_{\r{CC}_*}$ acts on Hochschild chains as follows:
\begin{equation}
    \begin{split}
    d_{\r{CC}_*} (\mathbf{b} \otimes &x_1 \otimes \cdots \otimes x_k) =\\
    &\sum (-1)^{\#_{j}^i}\mu_\B(x_{k-j+1}, \ldots, x_k, \mathbf{b}, x_1, \ldots, x_i) \otimes x_{i+1} \otimes \cdots \otimes  x_{k-j} \\
    &+ \sum (-1)^{\maltese_{-k}^{-(s+j+1)}}\mathbf{b} \otimes x_1 \otimes \cdots \otimes \mu^j_\A(x_{s+1} \otimes \cdots \otimes x_{s+j}) \otimes x_{s+j} \otimes \cdots \otimes x_k
\end{split}
\end{equation}
with signs
\begin{align}
    \maltese_{-k}^{-t} &:= \sum_{i=t}^k ||x_i||\\
    \sharp_j^i &:= \bigg(\sum_{s=k-j+1}^k ||x_s||\bigg) \cdot \bigg(|\mathbf{b}| + \sum_{t=1}^{k-j} ||x_t||\bigg) + \maltese_{-(k-j)}^{-(i+1)}.
\end{align}
In this complex, Hochschild chains are graded as follows:
\begin{equation} \label{hochschildchaingrading}
    \deg(\mathbf{b} \otimes x_1 \otimes \cdots \otimes x_k) := \deg(\mathbf{b}) + \sum_i \deg(x_i) - k+1.
\end{equation}
\begin{ex} \label{barcomplexobservation}
Let $\mc{M}$ be a right $\cc$ module and $\mc{N}$ a left $\cc$ module, and form the product bimodule $\mc{N} \otimes_{\mathbb K} \mc{M}$. Then the Hochschild chain complex
    \begin{equation}
        \r{CC}_*(\cc,\mc{N} \otimes_{\mathbb K} \mc{M})
    \end{equation}
is exactly the bar complex
    \begin{equation}
        \mc{M} \otimes_{\cc} \mc{N}
    \end{equation}
defined in (\ref{tensorbarcomplex}), up to reordering a term (along with the
accompanying Koszul sign change).  
\end{ex}

Hochschild homology and cohomology are both functorial with respect to bimodule
morphisms. For example, for a morphism of $\mc{A}$ bimodules
\begin{equation}
    \mc{F}: \B \lra \B'
\end{equation}
of degree $|\mc{F}|$, the induced map on Hochschild chain complexes sums over all ways
to apply the various $\mc{F}^{r|1|s}$ to the element of $\B$ along with some
nearby terms:
\begin{equation} \label{functorialmapbimodules}
    \begin{split}
    \mc{F}_{\#}: CC_*(\A,\B) &\lra CC_*(\A,\B')\\
    \mathbf{b} \otimes a_1 \otimes \cdots \otimes a_k &\longmapsto \sum_{i,j} (-1)^{\sharp_j^i}\mc{F}^{i|1|j} (a_{k-j+1}, \ldots, a_k, \mathbf{b}, a_1, \ldots, a_i) \otimes a_{i+1} \otimes \cdots a_{k-j}.
\end{split}
\end{equation}
where
\begin{equation}
    \sharp_j^i = \bigg(\sum_{s=k-j+1}^k ||x_s||\bigg) \cdot \bigg( |\mathbf{b}| + \sum_{t=1}^{k-j} ||x_t||\bigg) + |\mc{F}| \cdot \maltese_{-(k-j)}^{-(i+1)}.
\end{equation}
Moreover, (as one might expect) the induced morphisms $\mc{F}_{\#}$ are
quasi-isomorphisms if $\mc{F}$ are.

There is also a form of functoriality with respect to morphisms of underlying
$\ainf$ algebras. If 
\begin{equation}
    \mathbf{F} : \mc{A} \lra \tilde{\mc{A}}
\end{equation}
is an $\ainf$ morphism (functor), then for any $\tilde{\mc{A}}$ bimodule $\mc{B}$,
there is an induced map 
\begin{equation}
    CC_*(\A,\mathbf{F}^*(\mc{B})) \stackrel{\mathbf{F}_\#}{\lra} CC_*(\tilde{A},\mc{B}) 
\end{equation}
defined as follows:
\begin{equation} \label{functorialmap}
    \begin{split}
    \mathbf{b} \otimes &a_1 \otimes \cdots \otimes a_k \longmapsto \\
    &\sum_s \sum_{i_1+ \cdots + i_s = k} \underline{\mathbf{b}} \otimes \mathbf{F}^{i_1}(a_1, \ldots, a_{i_1}) \otimes \cdots \otimes \mathbf{F}^{i_s}(a_{k-i_s+1}, \ldots, a_k).
\end{split}
\end{equation}
Here as usual $\underline{\mathbf{b}}$ is $\mathbf{b}$ thought of as living in
$\mc{B}$ instead of $\mathbf{F}^*\mc{B}$.  This is functorial and
respects quasi-isomorphisms.

\subsection{Module and ring structures}

\noindent The Hochschild co-chain complex of the bimodule $\A$ has a product,
giving it the structure of a dg algebra.
\begin{defn}
The {\bf Yoneda product} on $\r{CC}^*(\A,\A)$ is given by
\begin{equation}
    \begin{split}
    \phi &\star \psi (x_1, \ldots, x_k) := \\
    &\sum (-1)^{\bigstar}
    \mu^k (x_1, \ldots, x_i, \phi(x_{i+1}, \ldots, x_{i+r}), \ldots, \psi(x_{j+1}, \ldots, x_{j+l}),x_{j+l+1}, \ldots, x_k)
\end{split}
\end{equation}
with sign
\begin{equation}
    \bigstar := |\phi| \cdot \bigg(\sum_{s=j+l+1}^k ||x_s|| + \sum_{t=i+r+1}^{j} ||x_t||\bigg) + |\psi| \cdot \bigg(\sum_{s=j+l+1}^k ||x_s|| \bigg).
\end{equation}
\end{defn}
\begin{rem}
    In categorical language, $\r{CC}^*(\A,\A)$ is the direct sum of, for any
    $k$ and objects $X_0, \ldots, X_k$, 
\begin{equation}
    \hom_\A (X_{k-1},X_k) \times \cdots \hom_\A (X_0,X_1) \ra \hom_\A(X_0,X_k).
\end{equation}
\end{rem}
\begin{rem}
One can check using the $\ainf$ relations and the definition of the Hochschild
differential that the Yoneda product is commutative on the level of homology. 
\end{rem}
\begin{rem}
    Actually, one can define two (almost) commuting products on the Hochschild
    co-chain complex $\r{CC}^*(\A,\A)$, part of the structure it possesses as
    an $E_2$ algebra. The other product comes more directly from the
    2-categorical perspective. Namely, thinking of Hochschild cohomology as
    endomorphisms of the diagonal bimodule, the second product structure arises
    from {\bf composition} of endomorphisms. We will not further develop the
    $E_2$ structure here.
\end{rem}
\noindent If $\A$ is (homologically) unital, then $\r{HH}^*(\A,\A)$ is also
(homologically) unital.

There is a {\bf cap-product map}
\begin{equation}
    \cap: \r{CC}^*(\A,\A) \times \r{CC}_*(\A,\B) \ra \r{CC}_*(\A,\B) \end{equation}
given by 
\begin{equation}
    \begin{split}
        \alpha \cap (\mathbf{b} \otimes & x_1 \otimes  \cdots \otimes  x_n) :=
        \\ \sum &(-1)^\diamond \mu_B\big(x_{j+1}, \ldots, x_n, \mathbf{b}, x_1, \ldots, x_i, \\
        &\alpha (x_{i+1},
        \ldots, x_{i+k}),x_{i+k+1}, \ldots, x_{s}\big)\otimes x_{s+1} \otimes \cdots \otimes
        x_{j},
    \end{split}
    \end{equation} 
where the sign is
\begin{equation}
    \diamond := |\alpha| \cdot \bigg( \sum_{s=i+k+1}^n ||x_i|| \bigg) + \bigg(\sum_{s=j+1}^n ||x_i|| \bigg) \cdot \bigg(|\mathbf{b}| + |\alpha| + \sum_{t=1}^j ||x_i|| \bigg) + \maltese_{-j}^{-(s+1)}.
\end{equation}

\begin{prop}
    The cap product gives $\r{HH}_*(\A,\B)$ the structure of a module over
    $\r{HH}^*(\A,\A)$.  
\end{prop}
\begin{proof}
We need to check that 
\begin{itemize}
    \item cap product is a {\it chain map}, which follows from verifying
        \begin{equation}
            d(\alpha \cap b) = \delta \alpha \cap b + \alpha \cap db.
        \end{equation}
    \item cap product is natural with respect to the product structure, namely
        \begin{equation} (\alpha \star \beta) \cap b \simeq \alpha \cap
            (\beta \cap b)
        \end{equation}
            for Hochschild cocycles $\alpha$, $\beta$ and a Hochschild cycle
            $b$.  
    \end{itemize} 
This is an exercise up to sign, involving an
application of the $\ainf$ bimodule relations for $\B$ and the definitions of
Hochschild cycle and cocycle.
\end{proof}

\subsection{Hochschild invariants from bimodules}\label{twopointed}
There are alternate chain level descriptions of Hochschild invariants
that align more closely with our viewpoint of using $\ainf$ bimodules. 
\begin{defn}
    The {\bf two-pointed complex for Hochschild homology} 
    \begin{equation}
        _2 \r{CC}_*(\A,\B)
    \end{equation}
    is the chain complex computing the bimodule tensor product with the
    diagonal bimodule: 
    \begin{equation}
        _2 \r{CC}_*(\A,\B) := \A_{\Delta} \otimes_{\A\!-\!\A} \B.
    \end{equation}
\end{defn}
\noindent Observe that the complex ${_2}\r{CC}_*(\A,\B)$ can be alternatively described as the ordinary Hochschild complex 
\begin{equation}
    CC_*(\A,\A_{\Delta}\otimes_\A \B).
\end{equation}
Thus, the quasi-isomorphism of bimodules
\begin{equation}
    \A_{\Delta} \otimes_\A \B \stackrel{\sim}{\lra} \B
\end{equation}
functorially induces a quasi-isomorphism of complexes
\begin{equation}\label{Phidef}
    \Phi:\mbox{ } _2 \r{CC}_*(\A,\B) \stackrel{\sim}{\lra} \r{CC}_*(\A,\B).
\end{equation}
explicitly given by
\begin{equation}\label{twopointedhomologyquasi}
    \begin{split}
        \Phi (\mathbf{a}, &a_1, \ldots, a_k, \mathbf{b}, \bar{a}_1, \ldots, \bar{a}_l) := \\
        &\sum (-1)^{\bowtie}\mu^{i + k+1|1|j}_\B(\bar{a}_{l-i+1}, \ldots, \bar{a}_l, \mathbf{a}, a_1, \ldots, a_k, \mathbf{b}, \bar{a}_1, \ldots, \bar{a}_j) \otimes \bar{a}_{j+1}
        \otimes \cdots \otimes  \bar{a}_{l-i}.
\end{split}
\end{equation}
with sign 
\begin{equation}
    \bowtie: = \bigg( \sum_{s=l-i+1}^l ||\bar{a}_s|| \bigg) \cdot \bigg( |\mathbf{a}| + |\mathbf{b}| + \sum_{t=1}^k ||a_t|| + \sum_{s=1}^{l-i} ||\bar{a}_s||\bigg) + \sum_{m=j+1}^{l-i} || \bar{a}_m||.
\end{equation}

\begin{defn}
    The {\bf two-pointed complex for Hochschild cohomology} 
    \begin{equation}_2 \r{CC}^*(\A,\B)\end{equation} 
        is the chain complex of morphisms in the category of bimodules:
    \begin{equation}
        _2 \r{CC}^*(\A,\B):= \hom_{\A\!-\A}(\A_{\Delta},\B).
    \end{equation}
\end{defn}

Similarly, as one natural interpretation of Hochschild cohomology is as
endomorphisms of the identity functor or the (derived) self-ext of the diagonal
bimodule, one expects a quasi-isomorphism of complexes 
\begin{equation}
    \Psi: \r{CC}^*(\A,\B) \stackrel{\sim}{\ra} \hom_{\A\! -\! \A}(\A_{\Delta},\B)
\end{equation} 
Explicitly, if $\phi \in \r{CC}^*(\A,\B)$ is a Hochschild co-chain
 then one such map is given by:  
 \begin{equation} \label{twopointedcohomologyquasi}
     \begin{split}
    \Psi(\phi)(x_1, &\ldots, x_k, \mathbf{a}, y_1, \ldots, y_l) :=\\
    &\sum (-1)^{\blacktriangle}\mu_{\B}^{k+1+i|1|s} (x_1, \ldots, x_k, \mathbf{a}, y_1, \ldots, y_i, \phi(y_{i+1}, \ldots, y_{l-s}), y_{l-s+1}, \ldots, y_l)
\end{split}
\end{equation}
with sign
\begin{equation}
    \blacktriangle:= |\phi| \cdot (\sum_{j=l-s+1}^l ||y_j||).
\end{equation}
\begin{prop}
    $\Psi$ is a quasi-isomorphism when $\A$ is homologically unital.
\end{prop}

\begin{proof}
    As with previous arguments, we will simply exhibit the result for $\mc{A} =
    A$ without differentials and higher products $\mu^k=0$, and similarly for
    the bimodule structure (also disregarding signs). In this case,
\begin{equation}
    \r{CC}^*(A,A) = \hom_{Vect}(TA,A)
\end{equation}
with differential
\begin{equation}
    \begin{split}
    _1 d\phi(x_1, &\ldots, x_k) = x_1 \cdot \phi(x_2, \ldots, x_k) \\
    &+ \phi(x_1, \ldots, x_{k-1})\cdot x_k \\
    &+ \sum \phi(x_1, \ldots,  x_{i} \cdot x_{i+1}, \ldots, x_k).
\end{split}
\end{equation}
The two pointed complex is
\begin{equation}
    {_2}\r{CC}^*(A,A) := \hom_{Vect}(TA \otimes \underline{A} \otimes TA, A)
\end{equation}
with differential given by
\begin{equation}
    \begin{split}
        _2 d(\psi)(y_1, \ldots, y_r, &\mathbf{a}, x_1, \ldots, x_s) = y_1 \cdot \psi(y_2, \ldots, y_r, \mathbf{a}, x_1, \ldots, x_s) \\
        &+ \psi(y_1, \ldots, y_r, \mathbf{a}, x_1, \ldots, x_{s-1}) \cdot x_s \\
        &+ \sum \psi(y_1, \ldots, y_i \cdot y_{i+1}, \ldots, y_r, \mathbf{a}, x_1, \ldots, x_s)\\
        &+ \sum \psi(y_1, \ldots, y_r, \mathbf{a}, x_1, \ldots, x_j \cdot x_{j+1},\ldots, x_s)\\
        &+  \psi(y_1, \ldots, y_{r-1}, y_r\cdot\mathbf{a}, x_1, \ldots,  x_s)\\
        &+  \psi(y_1, \ldots, y_r, \mathbf{a}\cdot x_1,x_2 \ldots,  x_s).
    \end{split}
\end{equation}
Here we have underlined the bimodule $\underline{A}$ and boldfaced the bimodule
element $\mathbf{a}$ for ease of reading; the main difference from the ordinary
Hochschild complex is that one input, the element $\mathbf{a}$ must always be
specified, and moreover this element cannot come outside the Hochschild cochain
in the differential ${_2}d$. In contrast $A:= \hom(A^{\otimes 0},A)$ is
naturally a
sub-complex of the ordinary Hochschild complex $\r{CC}^*(A,A)$.  

Split the $ _2 \r{CC}(A,A)$ differential above into the sum of two
types of terms. The first only involves terms on the left of $\mathbf{a}$
\begin{equation}
    \begin{split}
        d_{left}(\psi)(y_1, \ldots, y_r, &\mathbf{a}, x_1, \ldots, x_s) = 
y_1 \cdot \psi(y_2, \ldots, y_r, \mathbf{a}, x_1, \ldots, x_s) \\
        &+ \sum \psi(y_1, \ldots, y_i \cdot y_{i+1}, \ldots, y_r, \mathbf{a}, x_1, \ldots, x_s)\\
        &+  \psi(y_1, \ldots, y_{r-1}, y_r\cdot\mathbf{a}, x_1, \ldots,  x_s)
    \end{split}
\end{equation}
and the second only involves terms on the right
\begin{equation}
    \begin{split}
        d_{right}(\psi)(y_1, \ldots, y_r, &\mathbf{a}, x_1, \ldots, x_s) = 
        \psi(y_1, \ldots, y_r, \mathbf{a}, x_1, \ldots, x_{s-1}) \cdot x_s \\
        &+ \sum \psi(y_1, \ldots, y_r, \mathbf{a}, x_1, \ldots, x_j \cdot x_{j+1},\ldots, x_s)\\
        &+  \psi(y_1, \ldots, y_r, \mathbf{a}\cdot x_1,x_2 \ldots,  x_s).
    \end{split}
\end{equation}
The map $\mathbf{\Psi}: \r{CC}_*(A,A) \lra {_2}\r{CC}_*(A,A)$ is given by:
\begin{equation}
    \mathbf{\Psi}(\phi)(y_1, \ldots, y_r, \mathbf{a}, x_1, \ldots, x_s) = \begin{cases}
        0 & r > 0\\
        \mathbf{a} \cdot \phi(x_1, \ldots, x_s) & r = 0
    \end{cases}
    \end{equation}
It is an easy verification that this is a chain map. Let us show explicitly
that it induces a quasi-isomorphism. The Cone of $\Psi$ is the complex
\begin{equation} \label{conecc2complex}
    Cone(\mathbf{\Psi}) := \r{CC}^*(A,A) \oplus {_2}\r{CC}^*(A,A)[1]
\end{equation}
with differential given by
\begin{equation}
    d_{Cone(\mathbf{\Psi})} = \left(\begin{array}{cc} _1 d & 0 \\ \mathbf{\Psi} & _2 d \end{array}\right)
\end{equation}

The second Hochschild complex admits a
filtration by ``length on the right,'' i.e.
\begin{equation}
    F_p {_2}\r{CC}(A,A) = \bigoplus_{l \geq p} \hom(TA \otimes \underline{A} \otimes A^{\otimes l}, A).
\end{equation}
This is compatible under $\mathbf{\Psi}$ with the length filtration on $\r{CC}(A,A)$
\begin{equation}
    F'_p \r{CC}(A,A) = \bigoplus_{l\geq p} \hom(A^{\otimes l},A)
\end{equation}
so we obtain an overall filtration $\mc{G}$ on the complex $Cone(\Psi)$.
The zeroth page of the associated spectral sequence, the associated graded
of the filtration, is 
\begin{equation}
    Gr(\mc{G})_l:= \hom_{Vect}(A^{\otimes l}, A) \oplus \bigoplus_{k\geq 0}\hom_{Vect}(A^{\otimes k} \otimes \underline{A} \otimes A^{\otimes l},A).
\end{equation}
where the first piece comes from $\r{CC}^*(A,A)$ and the remainder comes from
${_2}\r{CC}^*(A,A)$. The differential on this page of the spectral
sequence acts as follows: 
\begin{equation}
        d \phi = \begin{cases} 0 & \phi \in \hom_{Vect}(A^{\otimes l}, A)\\
            \mathbf{\Psi}(\phi) & \phi \in \hom_{Vect} (\underline{A} \otimes A^{\otimes l}, A)\\
            d_{left}(\phi) & \phi \in \hom_{Vect}(A^{\otimes k} \otimes \underline{A} \otimes A^{\otimes l}, A)
        \end{cases}
\end{equation}
We can view this as a single complex, for each $l$, of
\begin{equation}
    \bigoplus_{k' \geq 0}\hom_{Vect}(A^{\otimes k'} \otimes A^{\otimes l}, A)
\end{equation}
with differential
\begin{equation}
    \begin{split}
    d\phi(y_1, \ldots, y_{k'}, x_1, \ldots, x_l) &= y_1 \cdot \phi(y_2, \ldots, y_{k'}, x_1, \ldots, x_l) \\
    &+\sum_i \phi(y_1, \ldots, y_i \cdot y_{i+1}, \ldots, y_{k'}, x_1, \ldots, x_l).
\end{split}
\end{equation}
This is visibly a bar complex, admitting for unital $A$ the contracting homotopy 
\begin{equation}
    \mf{h}(\phi)(y_1, \ldots, y_{k'}, x_1, \ldots, x_l) = \phi(y_1, \ldots, y_k', e, x_1, \ldots, x_l).
\end{equation}
Hence, the first page of the spectral sequence vanishes and we see that
$Cone(\mathbf{\Psi})$ is acyclic.

The general case follows from analyzing the associated length filtration on the
complex computing $Cone(\mathbf{\Psi})$, as in Proposition \ref{tensordiagonal}
below. Namely, the first page of the associated spectral sequence gives exactly
the homology-level complex (\ref{conecc2complex}), which we have just shown to
be acyclic.
\end{proof}

\begin{rem}
    Our emphasis on multiple chain-level models for Hochschild invariants, with
    explicit quasi-isomorphisms between them, may seem contrary to the
    ``derived'' or ``Morita-theoretic'' perspective that the resulting
    invariants are abstractly independent of choices. One reason for such an
    exposition is that the open-closed geometric maps constructed in Section
    \ref{openclosedmaps} depend explicitly on choices of chain complexes. As
    one consequence of the discussion here, we will obtain the (unsurprising)
    result that variants of the geometric open-closed maps that with different
    source (two-pointed) Hochschild chain complexes are quasi-isomorphic in an
    explicit fashion.  This is expected but not a priori obvious from
    definitions.
\end{rem}

\subsection{Split-generation}\label{splitgensection}
Let $X \subset C$ be a full subcategory of a triangulated category. We say
that $X$ {\bf split-generates $C$} if every element of $C$ is isomorphic to a
summand of a finite iterated cone of elements in $X$. We say a triangulated
category is {\bf split-closed} if any idempotent endomorphism of an object $Z$
leads to a splitting of that object as a direct sum $X \oplus Y$.

Now recall that, for an $\ainf$ category $\cc$, the category of modules $\rmod
\cc$ is naturally {\bf pre-triangulated} (meaning the cohomology level category
$H^0(\rmod \cc)$ is triangulated)---we can take sums, shifts, and mapping cones
of modules, and hence complexes of modules. 

There is a notion of an {\bf idempotent up to homotopy}
\cite{Seidel:2008zr}*{(4b)} which would allow us to properly extend discussion
of split-generation to the chain level. However, it is also known that a
cohomology level idempotent endomorphism can always be lifted, essentially
uniquely, to an idempotent up to homotopy \cite{Seidel:2008zr}*{Lemma 4.2}.
Thus, for our purposes, it is sufficient to make the following definition:

\begin{defn}
    Let $\cc$ be an $\ainf$ category, and $\mc{X} \subset \mc{C}$ a full
    subcategory. We say that $\mc{X}$ {\bf split-generates} $\mc{C}$ if any
    Yoneda module in $\rmod\cc$ admits a homologically left-invertible morphism
    into a (finite) complex of Yoneda modules of objects of $\mc{X}$.
\end{defn}
If $i$ is the (homology-level) morphism and $p$ is the (homology) left inverse,
then the reverse composition $i \circ p$ is the idempotent that exhibits the
target module as a homological summand of the larger complex.

\begin{defn}
    Call a right module $\mc{M}$ over $\mc{X}$ {\bf perfect} if it admits a
    homologically left-invertible morphism into a finite complex of Yoneda
    modules of objects of $\mc{X}$.
\end{defn}

Given a collection of objects $\{X_i\}$ in an $\ainf$ category $\mc{C}$, it is
then natural to ask when they split-generate another object $Z$. There is a
criterion for split-generation, known to category theorists and first
introduced in the symplectic/$\ainf$ setting by Abouzaid
\cite{Abouzaid:2010kx}, as follows. Denote by 
\begin{equation}\mc{X}\end{equation}
the full sub-category of $\mc{C}$ with objects $\{X_i\}$. Then, one can form
the chain complex 
\begin{equation}\label{yrlbarcomplex}
    \mc{Y}^r_Z \otimes_{\mc{X}} \mc{Y}^l_Z
\end{equation}
where the above notation indicates that the Yoneda modules $\mc{Y}^r_Z$,
$\mc{Y}^l_Z$ are thought of as modules over $\mc{X}$ via the inclusion $\mc{X}
\subset \mc{C}$ as in Example \ref{yonedasubcategory}. Concretely, this bar
complex is given as
\begin{equation}
    \bigoplus_{k\geq 1} \bigoplus_{X_{i_1}, \ldots, X_{i_k} \in \ob \mc{X}}  \hom_{\mc{C}}(X_{i_k},Z) \otimes \hom_{\mc{X}}(X_{i_{k-1}},X_{i_k}) \otimes \cdots \otimes \hom_{\mc{X}}(X_{i_1},X_{i_2})  \otimes  \hom_{\mc{C}} (Z,X_{i_1})
\end{equation}
with differential given by summing over all ways to collapse some (but not all)
of the terms with a $\mu$:
\begin{equation}\label{yrlbardifferential}
    \begin{split}
        d(\mathbf{a} \otimes x_k \otimes \cdots \otimes x_1 \otimes \mathbf{b}) &= \sum_i (-1)^{\maltese_0^i}\mu^i(\mathbf{a},x_k, \ldots, x_{i+1}) \otimes x_{i} \otimes \cdots \otimes x_k \otimes \mathbf{b} \\
        &+ \sum_{j,r}(-1)^{\maltese_0^r} \mathbf{a}\otimes \cdots \mu^j(x_{r+j}, \ldots, x_{r+1}) \otimes x_r \otimes \cdots \otimes x_1 \otimes \mathbf{b}\\
        &+\sum_s (-1)^{\maltese_0^s}\mathbf{a}\otimes  \cdots \otimes x_{s+1} \otimes \mu^{s+1} (x_{s}, \ldots, x_1, \mathbf{b} ),
\end{split}
\end{equation}
where the sign is the usual
\begin{equation}
    \maltese_0^t := |\mathbf{b}| + \sum_{j=1}^t ||x_j||.
\end{equation}
There is a collapsing morphism
\begin{equation} \label{generatingmap}
    \begin{split}
    \mc{Y}^r_Z \otimes_{\mc{X}} \mc{Y}^l_Z &\stackrel{\mu}{\lra} \hom_{\mc{C}}(Z,Z)\\
    \mathbf{a} \otimes x_k \cdots \otimes x_1 \otimes \mathbf{b} &\longmapsto (-1)^{\maltese_0^k}\mu^{k+2}(\mathbf{a},x_k,\ldots,x_1, \mathbf{b}).
\end{split}
\end{equation}
which is a chain map, inducing a homology level morphism 
\begin{equation}
    [\mu]: H^*(\mc{Y}^r_Z \otimes_{\mc{X}} \mc{Y}^l_Z) \lra H^*(\hom_{\mc{C}}(Z,Z)).
\end{equation}
This map can be thought of as the first piece of information involving $Z$ from
the category $\mc{C}$ that is not already contained in the $\mc{X}$ modules
$\mc{Y}^r_Z$ and $\mc{Y}^l_Z$. The following proposition relates a checkable criterion involving the map $[\mu]$ to the split-generation of $Z$.
\begin{prop}[\cite{Abouzaid:2010kx}*{Lemma 1.4}]
The following two statements are equivalent:
\begin{align}
\label{splitgencriterion}&\textrm{The identity element }[e_Z] \in H^*(\hom_{\mc{C}}(Z,Z))\textrm{ is in
the image of }[\mu].  \\
\label{splitgenstatement}&\textrm{The object }Z\textrm{ is split generated by the }\{X_i\}.  
\end{align}
\end{prop}

In fact, the criterion (\ref{splitgencriterion}) in turn implies that we
completely understand the map $[\mu]$.
\begin{prop}\label{splitgeniso}
    If the identity element $[e_Z]$ is in the image of $[\mu]$, then the map
    $[\mu]$ is an isomorphism.
\end{prop}
\begin{proof}
    We will show this up to checking signs. First, assume that there is are no
    differentials $\mu^1 = 0$ or higher products $\mu^k = 0$, $k > 2$. Denote
    the ordinary composition $\mu^2$ by
    $\cdot$. 
    Then, the map (\ref{generatingmap})
    is given by
    \begin{equation}
        \mu: \mathbf{a} \otimes x_1 \cdots \otimes x_k \otimes \mathbf{b} 
        \longmapsto 
        \begin{cases}
            \mathbf{a}\cdot \mathbf{b} & k = 0\\
            0 & \textrm{otherwise}.
        \end{cases}
    \end{equation}
    Thus, we may suppose that there exists elements $\mathbf{a}_k \in
    \hom(X_{i_k},X)$ $\mathbf{b}_k \in \hom(Z,X_{i_k})$ such that 
    \begin{equation} \label{thereexistakbk}
        \sum \mathbf{a}_k \cdot \mathbf{b}_k = e_Z.
    \end{equation}
    Now, the cone of the morphism $\mu$ is a version of the bar complex where
    the outermost elements are both $Z$ and the inner elements range over
    $\mc{X}$:
    \begin{equation}
        Cone(\mu):= \hom(Z,Z)[1] \oplus \mc{Y}^r_Z \otimes_{\mc{X}} \mc{Y}^l_Z
    \end{equation}
    with differential given by the usual bar differential on $\mc{Y}^r_Z
    \otimes_{\mc{X}} \mc{Y}^l_Z$ along with the collapsing multiplication
    $\mathbf{a} \otimes \mathbf{b} \mapsto \mathbf{a} \cdot \mathbf{b}$:
    \begin{equation}
        d (\mathbf{a} \otimes x_1 \cdots x_k \otimes \mathbf{b}) : = 
        \begin{cases}
            \mathbf{a} \cdot \mathbf{b} & k = 0\\
            d_{\mc{Y}^r_Z \otimes_{\mc{X}} \mc{Y}^l_Z}(\mathbf{a} \otimes x_1 \cdots x_k \otimes \mathbf{b}) & \textrm{otherwise}
        \end{cases}
        \end{equation}
        This admits the following contracting homotopy:
        \begin{equation}
            \begin{split}
            \mf{H}: Cone(\mu) &\lra Cone(\mu)\\
            z &\longmapsto \sum (z \cdot \mathbf{a}_k) \otimes \mathbf{b}_k,\ z \in \hom(Z,Z) \\
            \mathbf{a} \otimes x_1 \otimes \cdots \otimes x_k \otimes \mathbf{b} &\longmapsto \sum \mathbf{a}_k \otimes (\mathbf{b}_k \cdot \mathbf{a}) \otimes x_1 \otimes \cdots \otimes x_k \otimes \mathbf{b}.
        \end{split}
        \end{equation}
        One can check that, as (\ref{thereexistakbk}) holds, $\mf{H}$ satisfies 
        \begin{equation}
            d\mf{H} - \mf{H}d = id
        \end{equation}
        as desired.

    Now, suppose there are non-trivial differentials. Then
     the map (\ref{generatingmap})
    is still given by
    \begin{equation}
        \mu: \mathbf{a} \otimes x_1 \cdots \otimes x_k \otimes \mathbf{b} 
        \longmapsto 
        \begin{cases}
            \mathbf{a}\cdot \mathbf{b} & k = 0\\
            0 & \textrm{otherwise}.
        \end{cases}
    \end{equation}
    The condition that $[e_Z]$ is in the image implies that for some
    representative $e_Z$, there exists $\mathbf{a}_k \in
    \hom(Z,X_{i_k})$ $\mathbf{b}_k \in \hom(X_{i_k},Z)$ such that 
    \begin{equation} 
        \begin{split}
        d(\sum \mathbf{a}_k \otimes \mathbf{b}_k) &= 0\\
        \sum \mathbf{a}_k \cdot \mathbf{b}_k &= e_Z.
    \end{split}
    \end{equation}
    But the differential on the length two words in the bar complex $\mc{Y}^r_Z
    \otimes_{\mc{X}} \mc{Y}^l_Z$ only involves $\mu^1$ terms, so the element $\sum
    a_k \otimes b_k$ descends to the homology level bar complex
    $\mc{Y}^r_{H^*(Z)} \otimes_{H^*(\mc{X})} \mc{Y}^l_{H^*(Z)}$. Thus taking
    the Cone of $\mu$, and looking at the first page of the spectral sequence
    associated to the length filtration, we conclude by reduction to the
    previous case.

    Now, we reduce the general case to the dg case described above as follows.
    Let 
    \begin{equation}
        \overline{\mc{X}}
    \end{equation}
    denote the image of the subcategory $\mc{X}$ under a given full and
    faithful functor 
    \begin{equation}
        \mathbf{F}: \cc \stackrel{\sim}{\longhookrightarrow} \mc{D}
    \end{equation}
    into a {\it dg category} $\mc{D}$ (for example, we could take $\mathbf{F}$
    to be the right Yoneda embedding $\yr$).  Denote by
    \begin{equation}
        \bar{Z} := \mathbf{F}(Z)
    \end{equation} 
    the image of the object $Z$ under $\mathbf{F}$.  We claim that $\mathbf{F}$
    functorially induces the following homotopy commutative diagram:
    \begin{equation} \label{yonedacommutativediagram}
        \xymatrix{ \mc{Y}^r_Z \otimes_{\mc{X}} \mc{Y}^l_Z \ar[r]^{(\mathbf{F})_{!}}_{\sim}\ar[d]^{\mu_{\cc}} & \mc{Y}^r_{\bar{Z}} \otimes_{\bar{\mc{X}}} \mc{Y}^l_{\bar{Z}} \ar[d]^{\mu_{\mc{D}}} \\
        \hom_{\cc}(Z,Z) \ar[r]^{(\mathbf{F})^1}_{\sim} & \hom_{\mc{D}}(\bar{Z},\bar{Z})  }
    \end{equation}
    Here $(\mathbf{F})_{!}$ is defined as the sum over all ways to apply terms
    of the functor $\mathbf{F}$ to portions of the bar complex $\mc{Y}^r_Z
    \otimes_{\mc{X}} \mc{Y}^l_Z$ except for collapsing maximally
    \begin{equation}\label{yonedashriek}
    \begin{split}
        (\mathbf{F})_{!}(\mathbf{a} \otimes &x_1 \otimes \cdots \otimes x_k \otimes \mathbf{b}) := \\
        &\sum_{s\geq 2}\sum_{i_1 + \cdots + i_s = k+2} \mathbf{F}^{i_1}(\mathbf{a}, x_1, \ldots, x_{i_1-1}) \otimes \cdots \otimes \mathbf{F}^{i_s}(x_{k-i_s+2},\ldots, x_k, \mathbf{b}).
\end{split}
\end{equation}
If (\ref{yonedacommutativediagram}) held along with the equivalences, then we
would be done by reduction to the dg case, as $\mc{D}$ is a dg category.  Now,
the commutativity of the diagram follows from the $\ainf$ functor
equations for the functor $\mathbf{F}$, which take the following form (as
$\mc{D}$ is dg):
\begin{equation} \label{yonedafunctorequation}
    \begin{split}
        \mu^1_{\mathcal{D}}(\mathbf{F}^{k+2}(\aa, &x_1, \ldots, x_k, \bb)) + \sum_{i_1+i_2=k+2}\mu^2_{\mc{D}}(\mathbf{F}^{i_1}(\aa, x_1, \ldots, x_{i_1}),\mathbf{F}^{i_2}(x_{i_1+1}, \ldots, x_k, \bb)) \\
        &= \mathbf{F}^1(\mu_{\cc}^{k+2}(\aa,x_1, \ldots, x_k, \bb)) +  \mathbf{F} (d_{\mc{Y}^r_Z \otimes_{\mc{X}} \mc{Y}^l_Z}(\aa,x_1, \ldots, x_k, \bb)).
    \end{split}
\end{equation}
where $\mathbf{F}$ above is thought of as the total functor
\begin{equation}
    \mathbf{F} = \oplus \mathbf{F}^d: T\cc \ra \mc{D}
\end{equation}
In other words, (\ref{yonedafunctorequation}) tells us that 
\begin{equation}
    \mathbf{F}^1 \circ \mu_\cc - \mu_{\mc{D}} \circ (\mathbf{F})_{!} = \mu^1_{\mc{D}} \circ \mathbf{F} \pm \mathbf{F} \circ d_{\mc{Y}^r_Z \otimes_{\mc{X}} \mc{Y}^l_Z},
\end{equation} 
so the total functor $\mathbf{F}$ implements the homotopy between the two
composites in the diagram (\ref{yonedacommutativediagram}).

Finally, it remains the check the equivalences in
(\ref{yonedacommutativediagram}). We already know $\mathbf{F}^1$ induces a
quasi-isomorphism by assumption, so we just need to
check $(\mathbf{F})_{!}$.  This could be done by hand, but we will present an
alternate naturality/functoriality argument.  First, note that the bar
complexes in (\ref{yonedacommutativediagram}) can be re-expressed by Example
\ref{barcomplexobservation} as Hochschild homology complexes
\begin{equation}
    \begin{split}
        \r{CC}_*(\mc{X},\mc{Y}^l_Z \otimes_{\mathbb K} \mc{Y}^r_Z)\\
        \r{CC}_*(\overline{\mc{X}},\mc{Y}^l_{\bar{Z}} \otimes_{\mathbb K} \mc{Y}^r_{\bar{Z}}).
    \end{split}
\end{equation}
by reordering the positions of $\mc{Y}^l_Z$ and $\mc{Y}^l_{\bar{Z}}$ (with the
Koszul sign change given by the gradings)
Now, we note by (\ref{nattransbimodulemorphism}) that there is
a natural morphism of bimodules
\begin{equation}
    \mf{T}^{\mathbf{F}}_{LR}: \mc{Y}^l_Z \otimes_{\mathbb K} \mc{Y}^r_Z \lra \mathbf{F}^*\mc{Y}^l_{\bar{Z}} \otimes_{\mc{X}} \mathbf{F}^* \mc{Y}^r_{\bar{Z}}.
\end{equation}
By functoriality with respect to bimodules (\ref{functorialmapbimodules}), we
obtain a map on Hochschild complexes 
\begin{equation}
    (\mf{T}^{\mathbf{F}}_{LR})_\#: \r{CC}_*(\mc{X}, \mc{Y}^l_Z \otimes_{\mathbb K} \mc{Y}^r_Z) \lra \r{CC}_*(\mc{X},\mathbf{F}^*(\mc{Y}^l_{\bar{Z}} \otimes_{ \K} \mc{Y}^r_{\bar{Z}})).
\end{equation}
Now, by (\ref{functorialmap}) the functor $\mathbf{F}$ also gives a functorial
map on Hochschild complexes
\begin{equation}\mathbf{F}_\#: \r{CC}_*(\mc{X},\mathbf{F}^*(\mc{Y}^l_{\bar{Z}} \otimes_{\mathbb K} \mc{Y}^r_{\bar{Z}})) \lra 
        \r{CC}_*(\overline{\mc{X}},\mc{Y}^l_{\bar{Z}} \otimes_{\mathbb K} \mc{Y}^r_{\bar{Z}}).
    \end{equation}
We now observe that the composition
\begin{equation} \label{compositionyoneda}
    \mathbf{F}_\# \circ (\mf{T}^{\mathbf{F}}_{LR})_\#
\end{equation}
is exactly the map $(\mathbf{F})_{!}$. Thus, the fact that $(\mathbf{F})_!$ is
a quasi-isomorphism follows from the fact that $\mathbf{F}$ is full, hence the
individual maps in (\ref{compositionyoneda}) are quasi-isomorphisms.
\end{proof}
\begin{rem}
    Actually, the commutative diagram (\ref{yonedacommutativediagram}) holds
    whether or not the target category $\mc{D}$ was a dg category.  
\end{rem}
\noindent Since the converse of the Proposition \ref{splitgeniso} is trivially
true, we see that
\begin{cor}\label{threeequivalent}
The following three statements are equivalent:
\begin{align}
&\textrm{The identity element }[e_Z] \in H^*(\hom_{\mc{C}}(Z,Z))\textrm{ is in
the image of }[\mu].  \\
&\textrm{The object }Z\textrm{ is split generated by the }\{X_i\}.\\
&\textrm{The map }[\mu]\textrm{ is an isomorphism}.
\end{align}
\end{cor}

\begin{prop}\label{fullsplitgen}
Let $\mc{F} : \mc{C} \lra \mc{D}$ be an $\ainf$ functor, such that for a full
subcategory $\mc{X} \subseteq \mc{C}$, $\mc{F}$ is quasi-full.  Then, if
$\mc{X}$ split generates $\mc{C}$, $\mc{F}$ is quasi-full on $\mc{C}$.
\end{prop}
\begin{proof}
    This result does not have a clean proof in the literature. The case that
    $\mc{X}$ (non-split)-generates $\mc{C}$ can be found in
    \cite{Seidel:2008zr}*{Lemma 3.25}.  Parts of the argument for the
    split-closed case are in \cite{Seidel:quartic}*{Lemma 2.5}.
\end{proof}

There are some final definitions we will need in the next subsection.
\begin{defn}
    A $\cc\!-\!\dd$ bimodule $\mc{B}$ is said to be {\bf perfect} if it is
    split-generated by a finite collection of Yoneda bimodules
    \begin{equation}
        \mc{Y}^l_{X_i} \otimes_{\mathbb K} \mc{Y}^r_{Y_j}.
    \end{equation}
\end{defn}

\begin{defn}[{\it (homological) smoothness}, compare \cite{Kontsevich:2006fk}*{\S 8}]
    An $\ainf$ category $\cc$ is said to be {\bf (homologically) smooth} if the
    diagonal bimodule $\cc_{\Delta}$ is a perfect $\cc\!-\!\cc$ bimodule.
\end{defn}
\begin{rem} 
    It is known that the dg category of perfect complexes on a variety
    $X$ is homologically smooth if and only if $X$ itself is smooth in the
    ordinary sense. This provides some justification for the usage of the term
    ``smooth.'' See for example \cite{Krahmer:2007kx}*{Thm. 3.8} for an
    exposition of this result in the case of affine varieties.
\end{rem}

\subsection{Module and bimodule duality}\label{moduleduality}
We have seen that the convolution tensor product with an $\cc\!-\!\dd$ bimodule
$\mc{B}$ induces dg functors of the form
\begin{equation}
    \begin{split}
        \cdot \otimes_{\cc} \mc{B}: \rmod\cc &\lra \rmod \dd \\
        \mc{B} \otimes_{\dd} \cdot: \dd\mod &\lra \cc\mod 
    \end{split}
\end{equation}
We can also define dual, or adjoint functors
\begin{align}
    \hom_{\cc\mod}(\cdot, \mc{B}): \cc\mod &\lra \rmod \dd \\
    \hom_{\rmod\dd}(\cdot,\mc{B}): \rmod\dd &\lra \cc\mod.
\end{align}
\begin{defn}\label{homrightmodule}
    Let $\mc{M}$ be an $\ainf$ left module over a category $\mc{C}$, and
    $\mc{B}$ a $\mc{C}\!-\!\mc{D}$ bimodule. The right $\mc{D}$ module
    \begin{equation}
        \hom_{\mc{C}\mod}(\mc{M},\mc{B})
    \end{equation}
    is specified by the following data:
    \begin{itemize}
\item For each object $Y$ of $\mc{D}$, a graded vector space
    \begin{equation}
        \hom_{\mc{C}\mod}(\mc{M},\mc{B})(Y) := \hom_{\mc{C}\mod}(\mc{M},\mc{B}(\cdot,Y))
    \end{equation}
    whose elements $\mc{F}$ are collections of maps
    \begin{equation}
        \mc{F} = \oplus_r
        \mc{F}^{r|1}: \oplus_{X_0, \ldots, X_r}\hom_{\mc{C}} (X_{r-1},X_r) \times \cdots \hom_{\mc{C}} (X_0,X_1) \times \mc{M}(X_0) \lra \mc{B}(X_r,Y).
    \end{equation}

\item A differential 
    \begin{equation}
        \mu^{1|0}_{\hom_{\mc{C}\mod}(\mc{M},\mc{B})}: \hom_{\mc{C}\mod}(\mc{M},\mc{B})(Y) \lra \hom_{\mc{C}\mod}(\mc{M},\mc{B})(Y)
    \end{equation}
    given by the differential in the $dg$ category of left $\cc$ modules
    \begin{equation}
        d  \mc{F} = \mc{F} \circ \hat{\mu}_{\mc{M}} - \mu_{\mc{B}} \circ \hat{\mc{F}}.
    \end{equation}
    Above, $\mu_{\mc{B}}$ is the total left-sided bimodule structure map $\oplus \mu^{r|1|0}_{\mc{B}}$. 
\item Higher right multiplications:
    \begin{equation}
        \begin{split}
        \mu^{1|s}: \hom_{\mc{C}\mod}(\mc{M},\mc{B})(Y_0)\times &\hom_{\dd}(Y_1,Y_0) \times \cdots \times \hom_\dd(Y_{s},Y_{s-1})  \lra \\
        &\hom_{\mc{C}\mod}(\mc{M},\mc{B})(Y_s)
    \end{split}
    \end{equation}
    given by
    \begin{equation}
        \mu^{1|s}(\mc{F},y_1, \ldots, y_s) := \mc{F}_{y_1, \ldots, y_s} \in \hom_{\cc\mod}(\mc{M},\mc{B}(\cdot,Y_s))
    \end{equation}
    where $\mc{F}_{y_1, \ldots, y_s}$ is the morphism specified by the following data
    \begin{equation}
        \mc{F}_{y_1, \ldots, y_s}^{k|1}(x_1, \ldots, x_k, \mathbf{m}) = \sum_i \mu_{\mc{B}}^{i|1|s}(x_1, \ldots, x_i, \mc{F}^{k-i|1}(x_{i+1}, \ldots, x_k, \mathbf{m}), y_1, \ldots, y_s).
    \end{equation}
    \end{itemize}
\end{defn}

\begin{rem}
    Note that for any $\cc\!-\!\dd$ bimodule $\mc{B}$, $\mc{B}(\cdot, Y)$ is a
    left $\cc$ module with structure maps $\mu^{r|1|0}_{\mc{B}}$, for any
    object $Y \in \ob \dd$. This is implicit in our construction above.
    Similarly, $\mc{B}(X,\cdot)$ is a right $\dd$ module with structure maps
    $\mu^{0|1|s}_{\mc{B}}$.  
\end{rem}

\begin{defn}\label{homleftmodule}
    Let $\mc{N}$ be an $\ainf$ right module over a category $\mc{D}$, and
    $\mc{B}$ a $\mc{C}\!-\!\mc{D}$ bimodule. The left $\cc$ module
    \begin{equation}
        \hom_{\rmod\mc{D}}(\mc{N},\mc{B})
    \end{equation}
    is specified by the following data:
    \begin{itemize}
\item For each object $X$ of $\mc{C}$, a graded vector space
    \begin{equation}
        \hom_{\rmod\mc{D}}(\mc{N},\mc{B})(X) := \hom_{\rmod\dd}(\mc{N},\mc{B}(X,\cdot))
    \end{equation}
    whose elements $\mc{G}$ are collections of maps
    \begin{equation}
        \mc{G} = \oplus_s
        \mc{G}^{1|s}: \oplus_{Y_0, \ldots, Y_s}\mc{N}(Y_0) \times \hom_{\mc{D}} (Y_{1},Y_0) \times \cdots \hom_{\mc{D}} (Y_{s},Y_{s-1})  \lra \mc{B}(X,Y_s).
    \end{equation}

\item A differential 
    \begin{equation}
        \mu^{1|0}_{\hom_{\rmod\dd}(\mc{N},\mc{B})}: \hom_{\rmod\dd}(\mc{N},\mc{B})(Y) \lra \hom_{\rmod\dd}(\mc{N},\mc{B})(Y)
    \end{equation}
    given by the differential in the $dg$ category of right $\dd$ modules
    \begin{equation}
        d  \mc{G} = \mc{G} \circ \hat{\mu}_{\mc{N}} - \mu_{\mc{B}} \circ \hat{\mc{G}}.
    \end{equation}
    Above, $\mu_{\mc{B}}$ is the total right-sided bimodule structure map $\oplus \mu^{0|1|s}_{\mc{B}}$. 
\item Higher left multiplications:
    \begin{equation}
        \begin{split}
        \mu^{r|1}: \hom_{\cc}(X_{r-1},X_r) \times \cdots &\times \hom_\cc(X_{0},X_1) \times \hom_{\rmod\mc{D}}(\mc{N},\mc{B})(X_0)  \lra \\
        &\hom_{\rmod\mc{D}}(\mc{N},\mc{B})(X_r)
    \end{split}
    \end{equation}
    given by
    \begin{equation}
        \mu^{r|1}(x_r, \ldots, x_1, \mc{G}) := \mc{G}_{x_r, \ldots, x_1} \in \hom_{\rmod\dd}(\mc{N},\mc{B}(X_r,\cdot))
    \end{equation}
    where $\mc{G}_{x_r, \ldots, x_1}$ is the morphism specified by the following data
    \begin{equation}
        \begin{split}
        \mc{G}_{x_r, \ldots, x_1}^{1|l}(&\mathbf{n},y_l, \ldots, y_1) = \\
        &\sum_j (-1)^\star \mu_{\mc{B}}^{r|1|l-j}(x_r, \ldots, x_1, \mc{G}^{1|l-j}(\mathbf{n},y_l, \ldots, y_{j+1}), y_{j}, \ldots, y_1)
    \end{split}
    \end{equation}
    with sign
    \begin{equation}
        \star = |\mc{G}| \cdot (\sum_{i=1}^j ||y_i||).
    \end{equation}
    \end{itemize}
\end{defn}

\begin{defn} 
    When the bimodule in question above is the diagonal bimodule
    $\mc{C}_{\Delta}$, we call the resulting left or right module
    $\hom_{\mc{C}\mod}(\mc{M},\mc{C}_{\Delta})$ or
    $\hom_{\rmod\mc{C}}(\mc{N},\cc_{\Delta})$ the {\bf module dual} of
    $\mc{M}$ or $\mc{N}$ respectively.
\end{defn}

\begin{rem} 
    The terminology {\bf module dual} is in contrast to {\it linear dual},
    another operation that can frequently performed on modules and bimodules
    that are finite rank over $\K$ (see e.g. \cite{Seidel:2008cr}).
\end{rem}
Now suppose our target bimodule splits as a tensor product of a left module with a right module
    \begin{equation}
        \mc{B} = \mc{M} \otimes_\K \mc{N}.
    \end{equation}
    Then, given another left module $\mc{P}$, the definitions imply that there is a natural inclusion
    \begin{equation}\label{homleftintosplitmodule}
 \hom_{\cc\mod}(\mc{P},\mc{M}) \otimes_\K \mc{N} \hookrightarrow  \hom_{\cc\mod}(\mc{P},\mc{M} \otimes_\K \mc{N}) 
    \end{equation}
    \begin{lem}\label{homleftintosplitmodulelemma}
        When $\mc{P}$ and $\mc{N}$ are Yoneda modules (or perfect modules), the
        inclusion (\ref{homleftintosplitmodule}) is a quasi-equivalence.
    \end{lem}
    \begin{proof}
        We suppose that $\mc{P}$ and $\mc{N}$ are Yoneda modules $\mc{Y}^l_X$, $\mc{Y}^r_Z$, and compute
        the underlying chain complexes, for an object $B$:
        \begin{align}
            \hom_{\cc\mod}(\mc{Y}^l_X,\mc{M} \otimes_\K\mc{Y}^r_Z)(B) &:= \hom_{\cc\mod}(\mc{Y}^l_X,\mc{M} \otimes_\K \hom_\cc (B,Z))\\
            &\simeq \mc{M}(X) \otimes_\K \hom_\cc(B,Z) \textrm{ (by Prop. \ref{moduleyoneda})}.
        \end{align}
        and
        \begin{align}
            \hom_{\cc\mod}(\mc{Y}^l_X,\mc{M}) \otimes_\K \mc{Y}^r_Z(B) \simeq \mc{M}(X) \otimes_\K \hom(B,Z) \textrm{ (by Prop. \ref{moduleyoneda})}.
        \end{align}
        The inclusion (\ref{homleftintosplitmodule}) commutes with the quasi-isomorphisms used in Proposition 2.3.

        We deduce the result for more general perfect modules by noting that as
        we vary $\mc{P}$ and $\mc{N}$ in (\ref{homleftintosplitmodule}) we
        obtain natural transformations that commute with finite colimits, hence
        they remain isomorphisms for perfect objects.
    \end{proof}

    Similarly, given a right module $\mc{Q}$, there are natural inclusions
    \begin{equation} \label{homrightintosplitmodule}
 \mc{M} \otimes \hom_{\rmod\dd}(\mc{Q},\mc{N}) \hookrightarrow  \hom_{\rmod\dd}(\mc{Q},\mc{M} \otimes_\K \mc{N}).
    \end{equation}
    \begin{lem}\label{homrightintosplitmodulelemma}
        When $\mc{Q}$ and $\mc{M}$ are Yoneda modules or (perfect modules), the
        inclusion (\ref{homrightintosplitmodule}) is a quasi-equivalence.
    \end{lem}

\begin{rem}
    There are also analogously defined functors on modules given by Hom from a bimodule:
    \begin{align}
        \hom_{\cc\mod}(\mc{B}, \cdot): \cc\mod &\lra \dd\mod \\
        \hom_{\rmod\dd}(\mc{B}, \cdot): \rmod\dd &\lra \rmod\cc.
    \end{align}
    We will not need them here.
\end{rem}

The following proposition in some sense verifies that module duality is a sane
operation for homologically unital $\ainf$ categories.
\begin{prop}\label{moduledualprop}
    Let $X$ be an object of a homologically unital $\ainf$ category $\cc$ and
    $\mc{Y}^r_X$, $\mc{Y}^l_X$ the corresponding Yoneda modules. Then, there is
    a quasi-isomorphism between the module dual of $\mc{Y}^r_X$ and
    $\mc{Y}^l_X$, and vice versa:
    \begin{align}
        \hom_{\cc\mod}(\mc{Y}^l_X, \cc_{\Delta}) &\simeq \mc{Y}^r_X\\
        \hom_{\rmod\cc}(\mc{Y}^r_X, \cc_{\Delta}) &\simeq \mc{Y}^l_X.
    \end{align}
\end{prop}
\begin{proof}
    We verify first that the module dual
    \begin{equation} \label{moduledual1}
        \hom_{\cc\mod}(\mc{Y}^l_X,\cc_{\Delta})
    \end{equation}
    is identical in definition to the pulled back right module
    \begin{equation}\label{yonedaleftstartyonedaright}
        \mc{N}:= \yl^* \mc{Y}^l_{\yl(X)},
    \end{equation}
    using the definition of pullback in Section \ref{pullback}.  By the
    definitions in that section, (\ref{yonedaleftstartyonedaright}) is the
    right module is given by the following data: 
    \begin{itemize}
    \item a graded vector space
        \begin{equation} \label{leftmodulehomyoneda}
            \mc{N}(Y) := \hom_{\cc\mod}(\yl(X),\yl(Y)) = \hom_{\cc\mod}(\yl(X), \cc_{\Delta}(\cdot,Y))
        \end{equation}
    \item differential
        \begin{equation}
            \mu^{1|0}_{\mc{N}}
        \end{equation}
        given by the standard differential in the dg category of modules in (\ref{leftmodulehomyoneda}).
    \item bimodule structure maps given by
        \begin{equation} \label{ylpullbackyrstructuremap}
            \begin{split}
                \mu^{1|s}_{\mc{N}}: \mc{N}(Y_0) \times &\hom_{\cc}(Y_1, Y_0) \times \cdots \hom_{\cc}(Y_{s},Y_{s-1}) \lra \mc{N}(Y_s)\\
                \mu^{1|s}_{\mc{N}}(\mathbf{N},y_1, \ldots, y_s) &:= \sum_k\mu^{k|1}_{\mc{Y}^l_{\yl(X)}}(  \yl^{i_k}(y_{s-i_k+1}, \ldots, y_s),\ldots, \yl^{i_1}(y_1, \ldots, y_{i_1}), \underline{\mathbf{N}})\\
            &= \mu^{2}_{\cc\mod}(\yl^{s}(y_1, \ldots, y_s),\underline{\mathbf{N}}).
        \end{split}
    \end{equation}
\end{itemize}
    The last equality in (\ref{ylpullbackyrstructuremap}) used the fact
    that since $\cc\mod$ is a dg category, $\mu^{1|k}_{\mc{Y}^r_{\yl(X)}}$ is
    $\mu^1$ in the category of modules when $k=0$, $\mu^2$ when $k=1$, and 0
    otherwise.
    Now, recall from (\ref{leftyonedaphi})-(\ref{leftyonedaphiexplanation}) that
    \begin{equation}
        \yl^s(y_1, \ldots, y_s) := \phi_{y_1, \ldots, y_s} \in \hom_{\cc\mod}(\yl(Y_0),\yl(Y_s))
    \end{equation}
    is the morphism given by the data
    \begin{equation}
        \phi_{y_1, \ldots, y_s}^{r|1}(x_1, \ldots, x_r, \mathbf{a}) = \mu^{r+1+s}(x_1, \ldots, x_r, \mathbf{a}, y_1, \ldots, y_s)
    \end{equation}
    so by definition the composition
    \begin{equation}
        \mu^{2}_{\cc\mod}(\phi_{y_1, \ldots, y_s},\underline{\mathbf{N}}) \in \mc{N}(Y_s) = \hom_{\cc\mod}(\mc{Y}^l_X,\mc{Y}^l_{Y_s})
    \end{equation}
    is $\phi_{y_1, \ldots, y_s} \circ \hat{\underline{\mathbf{N}}}$, i.e.
    \begin{equation}
        \begin{split}
        \mu^{2}_{\cc\mod}(\phi_{y_1, \ldots, y_s},&\underline{\mathbf{N}})^{r|1}(x_1, \ldots, x_r, \mathbf{n}) =\\
        &\sum_k \mu_{\cc}^{r-k+s+1}\big( x_1, \ldots, x_{r-k}, \underline{\mathbf{N}}^{k|1}(x_{r-k+1}, \ldots, x_r, \mathbf{n}), y_1, \ldots, y_s\big).
    \end{split}
    \end{equation}
    This is evidently the same as the definition of
    $\hom_{\cc\mod}(\mc{Y}^l_X,\cc_{\Delta})$. 
    
    Thus, the first order term of the contravariant natural transformation
    \begin{equation}
    (\mf{T}^{\yl}_R)_X: \mc{Y}^r_X \lra \mathbf{\yl}^*\mc{Y}^l_{\mathbf{\yl}(X)}.
    \end{equation}
    defined by an order reversal of (\ref{nattransmodulemorphism}) provides the
    desired quasi-isomorphism when $\cc$ is homologically unital.

    An analogous check verifies that
    $\hom_{\rmod\cc}(\mc{Y}^r_X,\mc{C}_{\Delta})$ is exactly
    $\mathbf{\yr}^*\mc{Y}^l_{\mathbf{\yr}(X)}$. In a similar fashion, the first
    order term of the natural transformation defined in (\ref{nattransmodulemorphism})
    \begin{equation}
        (\mf{T}^{\yr}_L)_X: \mc{Y}^l_X \lra \mathbf{\yr}^*\mc{Y}^l_{\mathbf{\yr}(X)}
    \end{equation}
    gives the desired quasi-isomorphism.
\end{proof}

\begin{prop}[Hom-tensor adjunction]
    Let $\mc{M}$ and $\mc{N}$ be left and right $\cc$ modules, and $\mc{B}$ a
    $\cc$ bimodule. Then there are natural adjunction isomorphisms, as chain
    complexes 
    \begin{align}
        \label{lefttorightadjunction}\hom_{\cc\!-\!\cc}(\mc{M} \otimes_\K \mc{N},\mc{B}) &=  \hom_{\cc\mod}(\mc{M}, \hom_{\rmod\cc}(\mc{N},\mc{B}))\\
        \label{righttoleftadjunction}\hom_{\cc\!-\!\cc}(\mc{M} \otimes_\K \mc{N},\mc{B}) &=  \hom_{\rmod\cc}(\mc{N}, \hom_{\cc\mod}(\mc{M},\mc{B})).
    \end{align}
\end{prop}
\begin{proof}
    Up to a sign check, we will show that the two expressions in
    (\ref{lefttorightadjunction}) contain manifestly the same amount of data as
    chain complexes; the case
    (\ref{righttoleftadjunction}) is the same.
    A premorphism
    \begin{equation}
        \mc{F}: \mc{M} \lra \hom_{\rmod\cc}(\mc{N},\mc{B})
    \end{equation}
    is the data of morphisms
    \begin{equation}
        \mc{F}^{r|1}: \hom_{\cc}(X_{r-1},X_{r}) \times \cdots \times  \mc{M}(X_0) \lra \hom_{\rmod\cc}(\mc{N}, \mc{B})
    \end{equation}
    sending 
    \begin{equation}
        c_1 \otimes \cdots \otimes c_k \otimes \mathbf{x} \longmapsto \mc{F}^{r|1}(c_1, \ldots, c_k, \mathbf{x}) \in \hom_{\rmod\cc}(\mc{N},\mc{B}(X_r,\cdot)).
    \end{equation}
    Associate to this the morphism
    \begin{equation}
        \underline{\mc{F}} \in \hom_{\cc\!-\!\cc}(\mc{M} \otimes_\K \mc{N},\mc{B})
    \end{equation}
    specified by
    \begin{equation}
        \underline{\mc{F}}^{r|1|s}(a_1, \ldots, a_r, (\mathbf{x} \otimes \mathbf{y}), b_1, \ldots, b_s) := \bigg(\mc{F}^{r|1}(a_1,\ldots, a_r, \mathbf{x})\bigg)^{1|s}(\mathbf{y}, b_1, \ldots, b_s).
    \end{equation}
    This identification is clearly reversible, so it will suffice to quickly
    check that the differential agrees. We compute that
    \begin{equation} \label{termstocompute1}
        \delta \mc{F} = \mc{F} \circ \hat{\mu}_{\mc{M}} - \mu_{\hom(\mc{N},\mc{B})} \circ \hat{\mc{F}}.
    \end{equation}
    By the correspondence given above, $\mc{F} \circ \hat{\mu}_{\mc{M}}$ is
    the morphism whose $1|s$ terms correspond to
    \begin{equation}
        \begin{split}
        \sum &\underline{\mc{F}}^{r-r'|1|s}( a_1, \ldots, a_{r'}, \mu^{r-r'|1}_{\mc{M}}(a_{r'+1}, \ldots, a_r, \mathbf{x}) \otimes \mathbf{y}, b_1, \ldots, b_s)\\
        &+ \sum \underline{\mc{F}}^{r-k'+1|1|s}(a_1, \ldots, a_i, \mu^k_\cc(a_{i+1}, \ldots, a_{i+k}), a_{i+k+1}, \ldots, a_r, \mathbf{x} \otimes \mathbf{y}, b_1, \ldots, b_s).
    \end{split}
    \end{equation}
    In the second term of (\ref{termstocompute1}), there are two cases. First, 
    $\mu_{\hom(\mc{N},\mc{B})}^{0|1} (\mc{F}(a_1, \ldots, a_r, \mathbf{x}))$ is the differential
    \begin{equation}
        \mc{F}(a_1, \ldots, a_r,\mathbf{x}) \circ \hat{\mu}_{\mc{N}} - \mu_{\mc{B}} \circ \hat{\mc{F}}(a_1, \ldots, a_r, \mathbf{x}),
    \end{equation}
    a morphism whose $1|s$ terms correspond to
    \begin{equation}
        \begin{split}
        \sum&\underline{\mc{F}}^{r|1|s-s'}( a_1, \ldots, a_r, \mathbf{x} \otimes \mu_{\mc{N}}^{1|s'}(\mathbf{y}, b_1, \ldots, b_{s'}), b_{s'+1} \ldots, b_s)\\
        &+ \sum \underline{\mc{F}}^{r|1|l-l'+1}(a_1, \ldots, a_r, \mathbf{x} \otimes \mathbf{y}, b_1, \ldots, b_j, \mu^l_{\cc}(b_{j+1}, \ldots, b_{j+l}), b_{j+l+1}, \ldots, b_s)
    \end{split}
    \end{equation}
    and 
    \begin{equation}
        \sum \mu_{\mc{B}}^{0|1|s-s'}(\underline{\mc{F}}^{r|1|s'}(a_1, \ldots, a_r, \mathbf{x} \otimes \mathbf{y}, b_1, \ldots, b_{s'}), b_{s'+1}, \ldots, b_s)
    \end{equation}
    respectively.  Finally, there are higher terms
    \begin{equation}
        \sum\mu_{\hom(\mc{N},\mc{B})}^{r'|1}(a_1, \ldots, a_r', \mc{F}^{r-r'|1}(a_{r'+1}, \ldots, a_r, \mathbf{x}))
    \end{equation}
    whose $1|s$ terms correspond exactly to
    \begin{equation}
        \sum \mu_{\mc{B}}^{r'|1|s'}(a_1, \ldots, a_r', \underline{\mc{F}}^{r-r'|1|s'}(a_{r'+1}, \ldots, a_r, \mathbf{x} \otimes \mathbf{y}, b_1, \ldots, b_s'), b_{s'+1}, \ldots, b_s).
    \end{equation}
    Thus, the differentials agree.
\end{proof}

Using adjunction, we can rapidly prove a few facts about bimodules.
\begin{prop}\label{homintodiagonal}
    There is a quasi-isomorphism of chain complexes
    \begin{equation}
        \hom_{\cc\!-\cc}(\mc{Y}^l_X \otimes \mc{Y}^r_Z, \cc_{\Delta}) \simeq \hom_{\cc}(Z,X).
    \end{equation}
\end{prop}
\begin{proof}
    By adjunction and module duality, we have that
    \begin{align}
        \hom_{\cc\!-\cc}(\mc{Y}^l_X \otimes \mc{Y}^r_Z, \cc_{\Delta}) &\simeq \hom_{\cc\mod} (\mc{Y}^l_X, \hom_{\rmod\cc}(\mc{Y}^r_Z,\cc_{\Delta}))\\
        &\simeq \hom_{\cc\mod} (\mc{Y}^l_X, \mc{Y}^l_Z) \\
        &\simeq \hom_{\cc}(Z,X).
    \end{align}
\end{proof}
Strictly speaking, the next fact is not about bimodules, but it will be useful
for what follows.
\begin{prop}\label{barcomplexhom}
    Let $A$ and $B$ be objects of a homologically unital category $\cc$. Then,
    the collapse map 
    \begin{equation}
        \mu: \mc{Y}^r_A \otimes_\cc \mc{Y}^l_B \simeq \hom(B,A)
    \end{equation}
    defined by 
    \begin{equation}
        \mathbf{a} \otimes c_1 \otimes \cdots \otimes c_k \otimes \mathbf{b} \longmapsto \mu^{k+2}_\cc(\mathbf{a}, c_1, \ldots, c_k, \mathbf{b}).
    \end{equation}
    is a quasi-isomorphism.
\end{prop}
\begin{proof} 
    One can see this result as a consequence of Corollary
    \ref{threeequivalent}, as $\cc$ split-generates itself. More, one could
    examine the cone of $\mu$ and note that it is exactly the usual $\ainf$ bar
    complex for $\cc$. Alternatively, here is a conceptual computation
    using module duality that the chain complexes compute the same homology:
    \begin{align}
        \mc{Y}^r_{A} \otimes_\cc \mc{Y}^l_B &\simeq \hom_{\cc\mod}(\mc{Y}^l_A, \cc_{\Delta}) \otimes_\cc \mc{Y}^l_B\\
        \label{bringin}&\simeq \hom_{\cc\mod}(\mc{Y}^l_A, \cc_{\Delta} \otimes_\cc \mc{Y}^l_B)\\
        &\simeq \hom_{\cc\mod}(\mc{Y}^l_A, \mc{Y}^l_B)\\
        &\simeq \hom_\cc(B,A).
    \end{align}
    Here, the justification for our ability to bring in $\mc{Y}^l_B$ in
    (\ref{bringin}) is analogous to Lemma \ref{homleftintosplitmodulelemma}.
\end{proof}

\begin{prop}[K\"{u}nneth Formula for Bimodules] \label{bimodulekunneth}
    There are quasi-isomorphisms
    \begin{equation}
        \begin{split}
        \hom_{\cc}(X',X) \otimes \hom_{\dd}(Z,Z') &\simeq \hom_{\cc\mod}(\mc{Y}^l_X,\mc{Y}^l_{X'}) \otimes_{\mathbb K} \hom_{\rmod \dd} (\mc{Y}^r_Z,\mc{Y}^r_{Z'})\\ 
        &\simeq \hom_{\cc\!-\!\dd}(\mc{Y}^l_X \otimes \mc{Y}^r_Z, \mc{Y}^l_{X'} \otimes \mc{Y}^r_{Z'}).
    \end{split}
    \end{equation}
\end{prop}
\begin{proof}
    Using adjunction and the Yoneda lemma, we compute
    \begin{align}
        \hom_{\cc\!-\!\dd}(\mc{Y}^l_X \otimes \mc{Y}^r_Z, \mc{Y}^l_{X'} \otimes \mc{Y}^r_{Z'}) &= 
        \hom_{\cc\mod}(\mc{Y}^l_X \hom_{\rmod\dd}( \mc{Y}^r_Z, \mc{Y}^l_{X'} \otimes \mc{Y}^r_{Z'})) \\
        &\simeq 
        \hom_{\cc\mod}(\mc{Y}^l_X, \mc{Y}^l_{X'} \otimes \hom_{\rmod\dd}( \mc{Y}^r_Z,  \mc{Y}^r_{Z'})) \textrm{ (Lemma \ref{homrightintosplitmodulelemma})}\\
        &\simeq \hom_{\cc} (X',X) \otimes \hom_{\rmod\dd}( \mc{Y}^r_Z, \mc{Y}^r_{Z'})\\
        &\simeq \hom_{\cc} (X',X) \otimes \hom_{\dd}(Z,Z').
    \end{align}
    One can check that the maps in this computation are compatible with the natural inclusions
    \begin{equation}
\hom_{\cc\mod}(\mc{Y}^l_X,\mc{Y}^l_{X'}) \otimes_{\mathbb K} \hom_{\rmod \dd} (\mc{Y}^r_Z,\mc{Y}^r_{Z'})
        \longhookrightarrow \hom_{\cc\!-\!\dd}(\mc{Y}^l_X \otimes \mc{Y}^r_Z, \mc{Y}^l_{X'} \otimes \mc{Y}^r_{Z'}).
    \end{equation}
    given by sending a pair of morphisms $\mc{F}, \mc{G}$ to the morphism
    \begin{equation}
        (\mc{F} \otimes \mc{G})^{r|1|s} := \mc{F}^{r|1} \otimes \mc{G}^{1|s}.
    \end{equation}
\end{proof}

We can now attempt to replicate the duality procedure for bimodules. Suppose
first we were in the setting of an ordinary bimodule $B$ over an associative
algebra $A$. An $A$ bimodule structure on $B$ is equivalent to a right $A^e:=A
\otimes A^{op}$ module structure, meaning that one can emulate the above
process and define the dual of $B$ to arise as Hom into some $A^e$ bimodule
$M$.  The basic example is the case $M  = $ the {\bf diagonal $A^e$
bimodule}, $A \otimes_\K A$ as a graded vector space.

Rephrasing everything in the language of bimodules over $A$, an {\it
$A^e$-bimodule structure} on the graded vector space 
\begin{equation}
    A \otimes_\K A
\end{equation}
is the datum of a left and right $A^e$ module structure, i.e. the data of two
$A$ bimodule structures, an {\it outer} structure and an {\it inner} structure.
We naturally arrive at a definition that seems to have been first studied by
Van Den Bergh \cite{bergh:1998uq}.
\begin{defn}[Van Den Bergh \cite{bergh:1998uq},
    Kontsevich-Soibelman \cite{Kontsevich:2006fk}, Ginzburg
    \cite{Ginzburg:2007fk}] 
    The {\bf inverse dualizing bimodule} of $A$ is, as a graded vector space
    \begin{equation}
        A^! := \bigoplus_i\mathrm{Ext}^i_{A^e} (A, A \otimes_\K A)[-i]
    \end{equation}
    where $\mathrm{Ext}$ is taken with respect to the outer bimodule structure
    on $A \otimes A$. The inner bimodule structure of $A \otimes A$ survives
    and gives the bimodule structure on $A^!$.
\end{defn}
More generally, one can define the {\bf dual bimodule} to $B$ to be 
\begin{equation}
    B^! := \bigoplus_i \r{Ext}^i_{A^e}(B, A \otimes_\K A)[i],
\end{equation}
with the same bimodule structure as in the definition.

We would 
like to emulate the definition of $B^!$ in the $\ainf$ setting and define, for
a bimodule $\mc{B}$ over an $\ainf$ category $\cc$, a {\bf bimodule dual}
\begin{equation}
    \mc{B}^! \textrm{``}=\textrm{''} \hom_{\cc\!-\!\cc}(\mc{B},\cc_{\Delta} \otimes_\K \cc_{\Delta}).
\end{equation}
\begin{rem}
    The reason we have put the above equality in quotes is that
    $\textrm{``}\cc_{\Delta} \otimes_\K \cc_{\Delta}\textrm{''}$ is not a
    bimodule or even a space with commuting
    outer and inner bimodule structures, unlike the associative case. It can,
    however, be thought of as a {\bf 4-module}, a special case of a theory of
    $\ainf$ {\bf $n$-modules} recently introduced by Ma'u \cite{Mau:2010lq}. A
    4-module associates to any four-tuple of objects $(X,Y,Z,W)$ a chain
    complex, and to any four-tuples of composable sequences of objects 
    \begin{equation}\bigg( (X_0, \ldots, X_k), (Y_0,
        \ldots, Y_l), (Z_0, \ldots, Z_s), (W_0, \ldots, W_t) \bigg)
    \end{equation} 
    operations $\mu^{k|s|t|l}$ satisfying a generalization of the $\ainf$
    bimodule equations. In the same way that one can tensor/hom modules with
    bimodules to obtain new modules,  one can tensor/hom bimodules with
    4-modules to obtain new bimodules.  The process we are about to describe is
    a special case of a general such theory, which has not been completely
    described.
\end{rem}
    
\begin{defn}\label{bimoduledual}
    Let $\mc{B}$ be an $\ainf$ bimodule over an $\ainf$ category $\cc$. 
    The {\bf bimodule dual} of $\mc{B}$ is the bimodule 
    \begin{equation}
        \mc{B}^!:= \hom_{\cc\!-\!\cc}(\mc{B},\cc_{\Delta} \otimes_\K \cc_{\Delta})
    \end{equation}
    over $\cc$ defined by the following data: 
    \begin{itemize}
        \item For pairs of objects, $(X,Y)$, $\mc{B}^!(X,Y)$ is the chain complex
            \begin{equation}
                \mc{B}^!(X,Y) := \hom_{\cc\!-\!\cc}(\mc{B},\mc{Y}^l_{X} \otimes_\K
                \mc{Y}^r_{Y})
            \end{equation}
            which we recall is the data of, for $k,l \geq 0$ and objects $A_0,
            \ldots, A_k, B_0, \ldots, B_l$, maps
            \begin{equation}
                \begin{split}
                    \mc{F}^{k|1|s}: \hom_\cc(A_{k-1}, A_k) &\otimes \cdots 
                    \otimes \hom_\cc(A_0, A_1) \otimes \mc{B} (A_0, B_0)
                     \\
                    \otimes \hom_\cc (B_1, B_0) \otimes &\cdots \otimes 
                    \hom_\cc(B_{l},B_{l-1})\\
                     &\lra \mc{Y}^l_{X}(A_k) \otimes_\K \mc{Y}^r_{Y}(B_l).
                \end{split}
            \end{equation}
            As usual, package this into a single map
            \begin{equation}
                \mc{F}: T\cc \otimes \mc{B} \otimes T\cc \lra \mc{Y}^l_X \otimes_k \mc{Y}^r_Y.
            \end{equation}
            Then, the differential is given by the usual bimodule hom
            differential 
            \begin{equation}
                \mu^{0|1|0}_{\mc{B}^!}(\mc{F}) = \mc{F} \circ \hat{\mu}_{\mc{B}} \pm \mu_{\mc{Y}^l_X \otimes_K \mc{Y}^r_Y} \circ \hat{\mc{F}}.
            \end{equation}
                
        \item for collections of objects $(X_0, \ldots, X_r, Y_0, \ldots, Y_s)$, maps
            \begin{equation}
                \begin{split}
                    \mu^{r|1|s}_{\mc{B}^!}:\hom_{\mc{B}}(X_{r-1},X_r) \otimes \cdots \otimes
                    &\hom_{\cc}(X_0,X_1) \otimes \mc{B}^!(X_0,Y_0) \\
                &\otimes \hom_{\cc}(Y_1,
                Y_0) \otimes \cdots \hom_{\cc}(Y_{s},Y_{s-1}) \\ &\mbox{ }\ \ \ \lra
                \mc{B}^!(X_r,Y_s).
                \end{split}
            \end{equation}
            defined as follows: 
            \begin{equation}
                \mu^{r|1|s}_{\mc{B}^!} = 0\textrm{ if both }r, s > 0.
            \end{equation}
            \begin{equation}
                \mu^{r|1|0}_{\mc{B}^!}(x_r, \ldots, x_1, \mathbf{\phi}) = \Phi_{(x_r, \ldots, x_1, \phi)} \in \mc{B}^!(X_r,Y_0)
            \end{equation}
            \def\xvec{(x_r, \ldots, x_1, \phi)}
            where $\Phi_{\xvec}$ is the bimodule map whose $k|1|l$ term is:
            \begin{equation}
                \begin{split}
                    \Phi_{\xvec}^{k|1|l}&(a_k, \ldots, a_1, \mathbf{b}, b_l, \ldots, b_1) := \\
                    \sum_{l'\leq l}&(-1)^{\maltese_1^{l'}}\big(\mu_{\cc_\Delta}^{r|1|l'}(x_r, \ldots, x_1, \cdot, b_{l'}, \ldots, b_{1}) \otimes id \big) \\
                    &\circ \phi^{k|1|l-l'}(a_{k}, \ldots, a_1, \mathbf{b}, b_l, \ldots, b_{l'+1}).
            \end{split}
            \end{equation}
            with sign 
            \begin{equation}
                \maltese_1^{l'} := \sum_{i=1}^{l'} ||b_i||.
            \end{equation}
        Also,
            \begin{equation}
                \mu^{0|1|s}_{\mc{B}^!}(\mathbf{\phi}, y_1, \ldots, y_s) = \Phi_{( \phi, y_1, \ldots, y_s)} \in \mc{B}^!(X_0,Y_s)
            \end{equation}
            \def\yvec{(\phi, y_1, \ldots, y_s)}
            where $\Phi_{\yvec}$ is the bimodule map whose $k|1|l$ term is:
            \begin{equation}
                \begin{split}
                    \Phi_{\yvec}^{k|1|l}(a_k, &\ldots, a_1, \mathbf{b}, b_l, \ldots, b_1) := \\
                    \sum_{k'\leq k}&\big(id \otimes 
                    \mu_{\cc_\Delta}^{k'|1|s}(a_k, \ldots, a_{k'+1}, \cdot, 
                    y_1, \ldots, y_s) \big) \\
                    &\circ \phi^{k-k'|1|l}(a_{k'}, \ldots, a_1, 
                    \mathbf{b}, b_1, \ldots, b_{l}).
            \end{split}
            \end{equation}
    \end{itemize}
\end{defn}

There is a more intrinsic definition of $\mc{B}^!$, analogous to the
relationship described in
(\ref{moduledual1})-(\ref{yonedaleftstartyonedaright}) in terms of Yoneda
pullbacks.
Let $\mc{B} \in \ob \cc\bimod\cc$ be a specified bimodule.
Take the right Yoneda module over this bimodule
\begin{equation}
\mc{Y}^{r}_{\mc{B}} \in \ob \rmod (\cc\bimod\cc).
\end{equation}
This is a right module over bimodules. By restricting via the natural embedding
\begin{equation}
    \cc\mod \otimes \rmod\cc \longhookrightarrow \cc\bimod\cc
\end{equation}
we think of $\mc{Y}^r_{\mc{B}}$ as a right {\it dg} module over the
category of {\it split bimodules}, e.g. the tensor product
\begin{equation}
    \cc\mod \otimes \rmod\cc.
\end{equation}
Since a right dg module $\mc{M}$ over a tensor product of
dg categories $\mc{C} \otimes \mc{D}$ is tautologically a
$\mc{C}^{op}-\mc{D}$ bimodule, 
\begin{equation}\label{bimoduletopullback}
    \mc{Y}^r_{\mc{B}}  \in   \ob (\cc\mod)^{op}\bimod (\rmod\cc).
\end{equation}
Now, recall in Section \ref{pullback} that for functors 
\begin{align}
    \mf{F}: \mc{A} &\lra \mc{C}^{op} \\
    \mf{G}: \mc{B} &\lra \mc{D}
\end{align}
we defined a {\it pullback} functor
\begin{equation}
    (\mf{F} \otimes \mf{G})^*: 
    \mc{C}^{op}\bimod\mc{D} \lra \mc{A}\bimod\mc{B}.
\end{equation}
We can now apply this construction to (\ref{bimoduletopullback}), using the left and right Yoneda embeddings
\begin{equation}
    \begin{split}
        \mathbf{Y}_L: \cc &\lra (\cc\mod)^{op} \\
        \mathbf{Y}_R: \cc &\lra (\rmod\cc) \\
    \end{split}
\end{equation}

\begin{prop}
The {\bf bimodule dual} of $\mc{B}$ is equivalent to the $\cc\!-\!\cc$ bimodule
\begin{equation}
\mc{B}^!: = (\mathbf{Y}_L \otimes \mathbf{Y}_R)^* (\mc{Y}^r_{\mc{B}}).
\end{equation}
\end{prop}
\begin{proof}
    We omit a proof of this fact for the time being, which essentially follows
    by comparing definitions as in Proposition \ref{moduledualprop}.
\end{proof}

As a first step, we can take the bimodule dual of a Yoneda bimodule.
\begin{prop}\label{bimoduledualyoneda}
    If $\mc{B}$ is the Yoneda bimodule
    \begin{equation}
        \mc{Y}^l_X \otimes_\K \mc{Y}^r_Z
    \end{equation}
    then $\mc{B}^!$ is quasi-isomorphic to the Yoneda bimodule
    \begin{equation}
        \mc{Y}^l_Z \otimes_\K \mc{Y}^r_X.
    \end{equation}
\end{prop}
\begin{proof}
    By definitions, there is a natural inclusion 
    \begin{equation}
        \hom_{\rmod\cc}(\mc{Y}^r_Z, \cc_{\Delta}) \otimes_\K \hom_{\cc\mod}(\mc{Y}^l_X, \cc_{\Delta}) \hookrightarrow \hom_{\cc\bimod\cc} (\mc{Y}^l_X\otimes \mc{Y}^r_Z,\cc_{\Delta} \otimes_\K \cc_{\Delta}).
    \end{equation}
    inducing a quasi-isomorphism by Proposition \ref{bimodulekunneth}. It
    follows immediately from inspection of Definition \ref{bimoduledual} that
    this inclusion can be extended to a morphism of $\ainf$ bimodules. Thus, by
    Proposition \ref{moduledualprop}, we conclude.
\end{proof}

If a vector space $V$ is finite dimensional, the linear dual $V^\vee$ satisfies
the property that 
\begin{equation}
    V^{\vee} \otimes W \cong \hom_{Vect}(V,W).
\end{equation}
One expects, under suitable finiteness conditions, a similar fact involving the
bimodule dual. The precise statement is
\begin{prop}\label{homrepresent}
    If $\mc{Q}$ is a perfect $\cc\!-\!\cc$ bimodule, then $\mc{Q}^!$ is also
    perfect and for perfect $\mc{B}$ there is a natural quasi-isomorphism
    \begin{equation}
        \mc{Q}^! \otimes_{\cc\!-\!\cc} \mc{B} \simeq \hom_{\cc\!-\!\cc}(\mc{Q},\mc{B}).
    \end{equation}
\end{prop}
\begin{proof}
The perfectness of $\mc{Q}^!$ follows from the fact (Proposition
\ref{bimoduledualyoneda})  that duals of Yoneda bimodules are Yoneda
bimodules; hence we see that if $\mc{Q}$ is a summand
of a complex of Yoneda bimodules, $\mc{Q}^!$ is too. (Implicitly, we are
using the fact that the duality functor commutes with finite cones and
summands.)

   Now, there is a natural transformation of functors
   \begin{equation}
       \mf{C}: \hom_{\cc\!-\!\cc}(\cdot, \cc_{\Delta} \otimes_{\mathbb K} \cc_{\Delta}) \otimes_{\cc\!-\!\cc} \mc{B} \lra \hom_{\cc\!-\!\cc}(\cdot,\cc_{\Delta} \otimes_{\cc} \mc{B} \otimes_{\cc} \cc_{\Delta})
   \end{equation}
   given by, for a bimodule $\mc{Q}$, the natural inclusion of chain
   complexes, 
   \begin{equation}
       \hom_{\cc\!-\!\cc}(\mc{Q}, \cc_{\Delta} \otimes_{\mathbb K}
        \cc_{\Delta}) \otimes_{\cc\!-\!\cc} \mc{B} 
        \longhookrightarrow
        \hom_{\cc\!-\!\cc}(\mc{Q},\cc_{\Delta} \otimes_{\cc} 
        \mc{B} \otimes_{\cc}
       \cc_{\Delta}).
   \end{equation}
   Concretely, this is the map 
   \begin{equation}
       \mf{C}_{\mc{Q}}: \underline{\phi} \otimes x_1 \otimes \cdots \otimes x_k \otimes \mathbf{b} \otimes y_1 \otimes \cdots \otimes y_l \longmapsto \hat{\phi}_{x_1, \ldots, x_k, \mathbf{b}, y_1, \ldots, y_l}
   \end{equation}
   where $\hat{\phi}_{x_1, \ldots, x_k \mathbf{b}, y_1, \ldots, y_l} \in
   \hom_{\cc\!-\!\cc}(\mc{Q},\cc_{\Delta} \otimes_{\cc} \mc{B} \otimes_{\cc}
   \cc_{\Delta})$ is specified by the following data:
   \begin{equation}
       \begin{split}
           \hat{\phi}_{x_1, \ldots, x_k, \mathbf{b}, y_1, \ldots, y_l} (a_1, \ldots, a_r, &\mathbf{q}, b_1, \ldots, b_s):=\\
           &\phi(a_1, \ldots, a_r, \mathbf{q}, b_1, \ldots, b_s) \bar{\otimes} (x_1 \otimes \cdots x_k \otimes \mathbf{b} \otimes y_1 \otimes \cdots \otimes y_l).
   \end{split}
   \end{equation}
   Here the operation $\bar{\otimes}$ is the reversed tensor product
   \begin{equation} \label{reversetensor}
       (a\otimes b) \bar{\otimes} (c_1 \otimes \cdots \otimes c_k) := b \otimes c_1 \otimes \cdots \otimes c_k \otimes a,
   \end{equation}
   extended linearly. For $\mc{Q}$ and $\mc{B}$ both Yoneda bimodules of the
   form $\mc{Y}_{XZ}:=\yyl{X} \otimes \yyr{Z}$ and $\yyl{X'} \otimes \yyr{Z'}$,
   we claim the natural transformation $\mf{C}$ is a quasi-isomorphism. This
   follows from the computations
   \begin{align} 
       \hom_{\cc\!-\!\cc}(\yyl{X} \otimes \yyr{Z}, \cc_{\Delta} &\otimes_{\mathbb K} \cc_{\Delta}) 
       \otimes_{\cc\!-\!\cc} (\yyl{X'} \otimes \yyr{Z'}) \\ 
&\simeq (\yyr{X} \otimes_\K \yyl{Z}) \otimes_{\cc\!-\!\cc} (\yyl{X'} \otimes_\K \yyr{Z'}) \\
&\simeq (\mc{Y}^r_X \otimes_{\cc} \mc{Y}^l_{X'}) \otimes_\K (\mc{Y}^r_{Z'} \otimes_\cc \mc{Y}^l_Z)\\
&\simeq \hom(X',X) \otimes_\K \hom(Z,Z') \textrm{ (Proposition \ref{barcomplexhom})}.
   \end{align}
      and 
   \begin{align}
       \hom_{\cc\!-\!\cc}(\yyl{X} \otimes \yyr{Z},\cc_{\Delta} &\otimes_{\cc} (\yyl{X'} \otimes \yyr{Z'}) \otimes_{\cc} \cc_{\Delta})\\
       &\simeq \hom_{\cc\mod}(\yyl{X}, \cc_{\Delta} \otimes_\cc \yyl{X'}) \otimes_\K \hom_{\rmod\cc}(\yyr{Z},\yyr{Z'} \otimes_\cc \cc_{\Delta})\\
       &\simeq \hom_{\cc\mod}(\yyl{X}, \yyl{X'}) \otimes_\K \hom_{\rmod\cc}(\yyr{Z},\yyr{Z'})\\
       &\simeq \hom(X',X) \otimes_\K \hom(Z,Z'),
    \end{align}
   which are compatible with the morphism $\mf{C}$.

   Since the natural transformation $\mf{C}$ commutes with finite cones and
   summands, we see that for perfect $\mc{Q}$ and $\mc{B}$, there must be a
   quasi-isomorphism 
   \begin{equation}
       \mf{C}_{\mc{Q}}: \hom_{\cc\!-\!\cc}(\mc{Q}, \cc_{\Delta} \otimes_{\mathbb K} \cc_{\Delta}) \otimes_{\cc\!-\!\cc} \mc{B} \stackrel{\sim}{\lra} \hom_{\cc\!-\!\cc}(\mc{Q},\cc_{\Delta} \otimes_{\cc} \mc{B} \otimes_{\cc} \cc_{\Delta}).
   \end{equation}
   Now, postcomposing with, e.g. the quasi-isomorphism 
   \begin{equation}
       \mc{F}_{\Delta, left, right}: \cc_{\Delta} \otimes_\cc \mc{B} \otimes_\cc \cc_\Delta \lra \mc{B}
   \end{equation}
   defined in (\ref{fleftright}) gives the desired quasi-isomorphism.
   
\end{proof}

We now specialize to the case $\mc{B} = \cc_{\Delta}$.
\begin{defn}
    The {\bf inverse dualizing bimodule}
    \begin{equation}
        \cc^!
    \end{equation}
    is by definition the bimodule dual of the diagonal bimodule $\cc_{\Delta}$.
\end{defn}
As an immediate corollary of Proposition \ref{homrepresent},
\begin{cor}[$\cc^!$ represents Hochschild cohomology] \label{representhh}
    If $\cc$ is homologically smooth, then the complex
    \begin{equation}
        \cc^! \otimes_{\cc\!-\!\cc} \mc{B}
    \end{equation}
    computes the Hochschild cohomology $\r{HH}^*(\cc,\mc{B})$.
\end{cor}
As described above, an explicit quasi-isomorphism between this complex and the
complex ${_2}\r{CC}_*(\cc,\mc{B})$ is given by
\begin{equation}\begin{split}\label{explicitquasiiso}
    \bar{\mu}: \cc^! \otimes_{\cc\!-\!\cc} \mc{B} &\lra {_2} \r{CC}^*(\cc,\mc{B})\\
    \bar{\mu}: \phi \otimes x_1 \otimes \cdots \otimes x_k \otimes \mathbf{b} \otimes y_1 \otimes \cdots \otimes y_l &\longmapsto \Psi_{\phi, x_1, \ldots, x_k, \mathbf{b} y_1, \ldots, y_l} \in \hom_{\cc\!-\!\cc}(\cc_{\Delta},\mc{B})
\end{split} \end{equation}
where $\Psi:=\Psi_{\phi, x_1, \ldots, x_k, \mathbf{b} y_1, \ldots, y_l}$ is the
morphism given by 
\begin{equation}
    \begin{split}
        \Psi(a_1, \ldots, a_r, \mathbf{c},& b_1, \ldots, b_s) := \sum_{r',s'}\mc{F}_{\Delta,left,right}^{r'|1|s'}(a_1,  \ldots,  a_r', \\
        \big( \phi^{r-r'|1|s-s'}(&a_{r'+1}, \ldots, a_r, \mathbf{c}, b_1, \ldots, b_{s'})
    \bar{\otimes} (x_1 \otimes \cdots \otimes x_k \otimes \mathbf{b} \otimes y_1 \otimes \cdots \otimes y_l)\big),\\
    &b_{s'+1}, \ldots,  b_s).
\end{split}
\end{equation}
Here $\bar{\otimes}$ is the reverse tensor product defined in
(\ref{reversetensor}), and $\mc{F}_{\Delta,left,right}$ is the bimodule
morphism given in (\ref{fleftright}).

\section{Symplectic cohomology and wrapped Floer cohomology} \label{chcwsection}

\subsection{Liouville manifolds}

Our basic object of study will be a {\bf Liouville manifold}, a manifold
$M^{2n}$ equipped with a one form $\theta$ 
called the {\bf Liouville form},
such that
\begin{equation}
    d\theta = \omega\textrm{ is a symplectic form}.
\end{equation}
The {\bf Liouville vector field} $Z$
is defined to be the symplectic dual to $\theta$
\begin{equation}
    i_{Z}\omega = \theta.
\end{equation}
We further require $M$ to have a {\bf cylindrical} (or {\bf conical}) {\bf
end}. That is, away from a compact region $\bar{M}$, $M$ has the structure of
the semi-infinite symplectization of a contact manifold
\begin{equation} \label{contactization}
    M = \bar{M} \cup_{\partial \bar{M}}\partial \bar{M} \times [1,+\infty)_r,
\end{equation}
such that the flow $Z$ is transverse to $\partial \bar{M} \times \{1\}$ and
acts on the cylindrical region by translation proportional to $r$, the
symplectization coordinate:
\begin{equation}
    Z = r \partial_r.
\end{equation}
The flow of the vector field $Z$ is called the {\bf Liouville flow} and denoted
\begin{equation}
    \psi^{\rho},
\end{equation}
where the time flowed is $\log (\rho)$.
We henceforth fix a representation of $M$ of the form (\ref{contactization}).
\begin{rem}
    One could have instead begun with a {\bf Liouville domain}, an exact
    compact symplectic manifold $\bar{M}$ with contact boundary $\partial
    \bar{M}$, such
    that the Liouville vector field $Z$ is outward pointing along $\partial
    \bar{M}$. One then integrates the flow of $Z$ in a small neighborhood of the
    boundary to obtain a collar neighborhood $\partial \bar{M} \times
    (1-\epsilon,1]$ and then attaches the infinite cone (\ref{contactization})
    to get a Liouville manifold. This process is known as {\bf completion}.
\end{rem}
On the boundary of the compact region $\bar{M}$
\begin{equation}
    \partial \bar{M} := \partial \bar{M} \times \{1\},
\end{equation}
\begin{equation} \label{contactform}
    \bar{\theta}:= \theta|_{\partial \bar{M}}\textrm{ is a contact form}.
\end{equation}
On the {\bf conical end}
\begin{equation}
    \partial \bar{M} \times [1,+\infty),
\end{equation} 
the Liouville form is given by rescaling the contact form
\begin{equation}
    \theta = r \bar{\theta}.
\end{equation}
Moreover, there is an associated {\bf Reeb vector field} on $\partial \bar{M}$
\begin{equation}
    R
\end{equation}
defined in the usual fashion by the requirements that
\begin{equation}
    \begin{cases}
        d\bar{\theta}(R,\cdot) = 0.\\
        \bar{\theta}(R) =1.
    \end{cases}
\end{equation}
Via the product identification (\ref{contactization}), we view $R$ as a
vector field defined on the entire conical end.

\begin{figure}
    \caption{A Liouville manifold with cylindrical end.}
    \centering
    \includegraphics[width=0.5\textwidth]{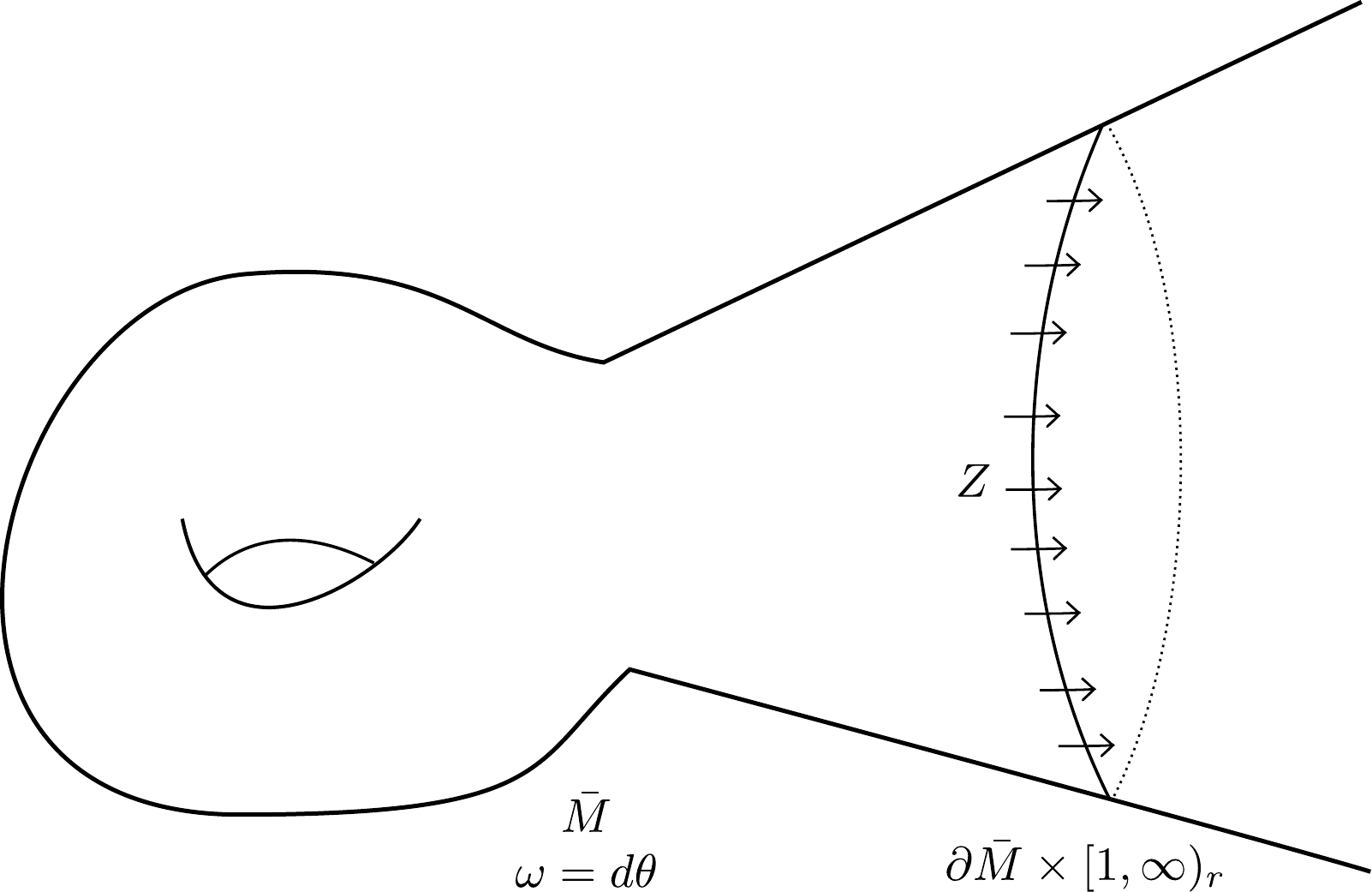}
\end{figure}

Now, we consider a finite collection $\ob \w$ of exact properly embedded
Lagrangian submanifolds in $M$, such that for each $L \in \ob \w$,
\begin{equation}
 \theta \textrm{ vanishes on }L \cap \partial \bar{M} \times [1,+\infty).
\end{equation}
Namely, the intersection $\partial L$ of $L$ with $\partial \bar{M}$ is
Legendrian, and $L$ is obtained by attaching an infinite cylindrical end
$\partial L \times [1,+\infty)$ to $L^{in} = L \cap \bar{M}$.
In addition, for each $L \in \ob \w$,
\begin{equation}
        \textrm{ choose and fix a primitive }f_L: L \ra \R\textrm{ for }\theta|_L. 
\end{equation}
By the above condition, $f_L$ is locally constant on the cylindrical end of
$L$.

To fix an integer grading on symplectic cohomology and $\w$, we require each $L
\in \ob \w$ to be spin, and have vanishing relative first Chern class
$2c_1(M,L) \in H^2(M,L)$. We additionally need to
\begin{align}
    &\textrm{Fix a spin structure (and orientation) on each }L.\\
    &\textrm{Fix a trivialization of }(\Lambda_\C^nT^*M)^{\otimes 2},\textrm{ and a grading on each }L.
\end{align}
We will implicitly fix all of this data whenever referring to a given
Lagrangian.

We restrict to a class of Hamiltonians 
\begin{equation}\label{hamiltonians}
    \mc{H}(M) \subset C^\infty (M,\R), 
\end{equation}
functions $H$ that, away from some compact subset of $M$ satisfy
\begin{equation} 
    H(r,y) = r^2.  
\end{equation}
Consider a class of almost-complex structures $\mc{J}_1(M)$ that are {\bf
rescaled contact type} on the conical end, meaning that 
\begin{equation}
    -\frac{1}{r} \theta \circ J = dr.
\end{equation}
This implies in particular that $J$ intertwines the Reeb and $r$ directions:
\begin{equation}
    \begin{split}
    J(R) &= \partial_r\\
    J(\partial_r) &= -R.
    \end{split}
\end{equation}
\begin{rem} 
    Our class of complex structures differs from those used by Abouzaid
    \cite{Abouzaid:2010kx} and Abouzaid-Seidel \cite{Abouzaid:2010ly}, who
    consider almost complex structures satisfying $\theta \circ J = dr$. 
    The difference will allow us to prove compactness for operations that
    involve a general class of perturbations that differ from functions of
    $r$ by a bounded term.
\end{rem}
We assume that $\theta$ has been chosen generically so that 
\begin{equation}
    \begin{split}
        &\textrm{all Reeb orbits of }\bar{\theta} \textrm{ are
        non-degenerate, and} \\ 
        &\textrm{all Reeb chords between Lagrangians in }\ob \w\textrm{ are
        non-degenerate}.
    \end{split}
\end{equation}

\subsection{Wrapped Floer cohomology}
Fixing a choice of $H \in \mc{H}(M)$ define \[\chi(L_0, L_1)\] to be the set of
time 1 Hamiltonian flows of $H$ between $L_0$ and $L_1$. Given the data
specified in the previous section, the {\bf Maslov index} defines an absolute
grading on $\chi(L_0,L_1)$, which we will denote by 
\begin{equation}
\deg: \chi(L_0, L_1) \ra
\Z. \end{equation}
Then, given a family $J_t \in \mc{J}_1(M)$ parametrized by $t\in [0,1]$, define
the {\bf wrapped Floer co-chain complex} over $\K$ to be, as a graded vector
space, 
\begin{equation}
    CW^i(L_0,L_1,H,J_t) = \bigoplus_{x\in \chi(L_0,L_1), deg(x)=i} |o_x|_{\K}.
\end{equation}
Here $|o_x|_{\K}$, henceforth abbreviated $|o_x|$, is the one-dimensional
$\K$-vector space associated to the one-dimensional real {\bf orientation line}
$o_x$ of $x$. The definition of $o_x$ as the determinant line of a
linearization of Floer's equation is given in Appendix
\ref{orientationsection}.

Now, consider maps
\begin{equation} u: (-\infty, \infty) \times [0,1] \ra M \end{equation}
converging exponentially at each end to time-1 chords of $H$, satisfying boundary conditions 
\begin{align}
\nonumber u(s,0) &\in L_0\\
\nonumber u(s,1) &\in L_1
\end{align}
and satisfying Floer's equation
\begin{equation}\label{floerequation}
(du - X\otimes dt)^{0,1} = 0.
\end{equation}
Above, $X$ is the Hamiltonian vector field of $H$ and we think of the strip 
\begin{equation} 
    Z =
(-\infty,\infty) \times [0,1]
\end{equation}
as equipped with coordinates $s,t$ and the canonical complex structure $j$ ( $j
(\partial_s) = \partial_t$). With this prescription one can rewrite the above
equation in coordinates in the more familiar form 
\begin{equation} \label{floersequation}
\partial_s u = -J_t (\partial_t u - X).
\end{equation}
Given time 1 chords $x_0, x_1 \in \chi(L_0,L_1)$, denote by
\begin{equation}
    \tilde{\mc{R}}^1(x_0;x_1)
\end{equation} 
the set of maps $u$ converging to $x_0$ when $s \ra
-\infty$ and $x_1$ when $s \ra +\infty$. As a component of the zero-locus of an
elliptic operator on the space of smooth functions from $Z$ into $M$, this set
carries a natural topology. Moreover, the natural $\R$ action on
$\tilde{\mc{R}}^1(x_0; x_1)$, coming from translation in the $s$ direction, is
continuous with respect to this topology. Following standard arguments, we
conclude:
\begin{lem} 
    For generic $J_t$, the moduli space $\tilde{\mc{R}}^1(x_0;x_1)$ is a
    compact manifold of dimension $\deg(x_0) - \deg(x_1)$.  The action of
    $\R$ is smooth and free unless $\deg(x_0) = \deg(x_1)$.
\end{lem}
\begin{proof}
    See \cite{Abouzaid:2010kx}*{Lemma 2.3}.
\end{proof}
\begin{defn}
    Define \begin{equation}\mc{R}(x_0; x_1)\end{equation} to be the quotient of
        $\tilde{\mc{R}}^1(x_0;x_1)$ by the $\R$ action whenever it is free, and
        the empty set when the $\R$ action is not free.  
\end{defn}
Also following now-standard arguments, one may construct a bordification
$\overline{\mc{R}}(x_0; x_1)$ by adding {\bf broken strips} 
\begin{equation}
\overline{\mc{R}}(x_0;x_1) = 
\coprod \mc{R}(x_0;y_1) \times 
\mc{R}(y_1;y_2) \times 
\cdots \times \mc{R}(y_k;x_1)
\end{equation}
\begin{lem}
    For generic $J_t$, the moduli space $\overline{\mc{R}}(x_0, x_1)$ is a
    compact manifold with boundary of dimension $\deg(x_0) - \deg(x_1) -1$.
    The boundary is covered by the closure of the images of natural
    inclusions 
    \begin{equation}
    \mc{R}(x_0;y) \times \mc{R}(y; x_1) \ra \overline{\mc{R}}(x_0;x_1).
\end{equation}
\end{lem}
\begin{proof}
    See \cite{Abouzaid:2010kx}*{Lemma 2.4}.
\end{proof}
\begin{lem}
    Moreover, for each $x_1$, the $\overline{\mc{R}}(x_0;x_1)$ is empty for all but finitely many $x_0$.
\end{lem}
\begin{proof}
    A proof of this is given in \cite{Abouzaid:2010kx}*{Lemma 2.5} but it is not
    quite applicable as it involves a general compactness result proven for
    complex structures $J$ satisfying $\theta \circ J = dr$, see
    \cite{Abouzaid:2010kx}*{Lemma B.1-2}. In fact, the arguments from this
    general compactness result directly carry over for our $J_t$ but we can
    alternately apply Theorem \ref{c0bounds}.
\end{proof}

Now, for regular $u \in \mc{R}(x_0; x_1)$, if $\deg(x_0) = \deg(x_1) + 1$, the
orientation on $\mc{R}(x_0;x_1)$ gives, by Lemma \ref{orientationlem} and
Remark \ref{trivializingR}, an isomorphism
\begin{equation}
    \mu_u: o_{x_1} \lra o_{x_0}.
\end{equation}

Thus we can define a differential 
\begin{equation}
    \begin{split}
    d: CW^*(L_0,L_1; H,J_t) &\lra CW^*(L_0,L_1; H,J_t)\\
    d([x_1]) &= \sum_{x_0; \deg(x_0) = \deg(x_1) + 1} \sum_{u \in \mc{R}(x_0;x_1)} (-1)^{\deg(x_0)} \mu_u([x_1]).
    \end{split}
\end{equation} 

\begin{lem}
    \[ d^2 = 0. \]
\end{lem}
\noindent Call the resulting group $HW^*(L_0,L_1)$.
\subsection{Symplectic cohomology} 
To define symplectic cohomology, we break the $S^1$ symmetry that
occurs for non-trivial time 1 orbits of our autonomous Hamiltonian $H$.
Choose $F: S^1 \times M \ra \R$ a smooth non-negative function, with
\begin{itemize}
\item $F$ and $\theta(X_F)$ uniformly bounded in absolute value, and 
\item all time-1 periodic orbits of $X_{S^1}$, the (time-dependent) Hamiltonian
    vector field corresponding to $H_{S^1}(t,m) = H(m) + F(t,m)$, are
    non-degenerate. This is possible for generic
    choices of $F$ \cite{Abouzaid:2010kx}. 
\end{itemize} 
Fixing such a choice, define \[\mc{O}\] to be the set of (time-1)
periodic orbits of $H_{S^1}$. Given an element $y \in \mc{O}$, define the {\it
degree} of $y$ to be \begin{equation}
    \deg(y) := n - CZ(y)
\end{equation}
where $CZ$ is the Conley-Zehnder index of $y$.
Now, define the {\bf symplectic co-chain complex} over $\K$ to be 
\begin{equation} 
    CH^i(M;H,F,J_t) = \bigoplus_{y\in \mc{O}, deg(y)=i} |o_y|_{\K},
\end{equation}
where the {\bf orientation line} $o_y$ is again defined using the determinant
line of a linearization of Floer's equation in Appendix
\ref{orientationsection}.

Given an $S^1$ dependent family $J_{t} \in \mc{J}_1(M)$, consider maps 
\begin{equation}
u:
(-\infty,\infty) \times S^1 \ra M
\end{equation}
converging exponentially at each end to a
time-1 periodic orbit of $H_{S^1}$ and satisfying Floer's equation 
\begin{equation}
(du - X_{S^1} \otimes dt)^{0,1} = 0.
\end{equation}
Here, as above the cylinder $A = (-\infty, \infty) \times [0,1]$ is equipped with
coordinates $s,t$ and a complex structure $j$ with $j (\partial s) = \partial_t$.
As before, this means the above equation in coordinates is the usual
\begin{equation}\label{floersequation2}
\partial_s u = -J_t (\partial_t u - X).
\end{equation}
Given time 1 orbits $y_0, y_1 \in \mc{O}$, denote by
$\tilde{\mc{M}}(y_0;y_1)$ the set of maps $u$ converging to $y_0$ when $s \ra -\infty$
and $y_1$ when $s \ra +\infty$. In analogy with the maps defining wrapped Floer
cohomology, this set is equipped with a topology and a natural $\R$ action
coming from translation in the $s$ direction. We can similarly conclude that
for generic $J_t$, the moduli space is smooth of dimension $\deg(y_0) -
\deg(y_1)$ with free $\R$ action unless it is of dimension 0.

\begin{defn}
    Define \begin{equation}\mc{M}(y_0; y_1)\end{equation} to be the quotient of
        $\tilde{\mc{M}}(y_0;y_1)$ by the $\R$ action whenever it is free, and
        the empty set when the $\R$ action is not free.  
\end{defn}

Construct the analogous bordification $\overline{\mc{M}}(y_0; y_1)$ by
 adding {\bf broken cylinders}
 \begin{equation}
\overline{\mc{M}}(y_0;y_1) = \coprod \mc{M}(y_0;x_1) \times \mc{M}(x_1;x_2) \times \cdots \times \mc{M}(x_k;y_1)
\end{equation}
\begin{lem}
    For generic $J_t$, the moduli space $\overline{\mc{M}}(y_0, y_1)$ is a
    compact manifold with boundary of dimension $\deg(y_0) - \deg(y_1) -1$.
    The boundary is covered by the closure of the images of natural
    inclusions 
    \begin{equation}
    \mc{M}(y_0;y) \times \mc{M}(y; y_1) \ra \overline{\mc{M}}(y_0;y_1).
\end{equation}
    Moreover, for each $y_1$, $\overline{\mc{M}}(y_0; y_1)$ is empty for all
    but finitely many choices of $y_0$.
\end{lem}
\begin{proof}
    Perturbed Hamiltonians of the form $H + F_t$ cease to satisfy a maximum
    principle, by some bounded error term. For {\it rescaled contact-type
    complex structures}, we show in Theorem \ref{c0bounds} that it is still
    possible to ensure that solutions with fixed asymptotics stay within a
    compact set, and a corresponding finiteness result.
\end{proof}
For a regular $u \in \mc{M}(y_0; y_1)$ with $\deg(y_0) = \deg(y_1) +1$, Lemma
\ref{orientationlem} and Remark \ref{trivializingR} give us an isomorphism of
orientation lines
\begin{equation}
    \mu_u: o_{y_1} \lra o_{y_0}.
\end{equation}

Thus we can define a differential 
\begin{equation}
    d: CH^*(M; H,F_t,J_t) \lra CH^*(M; H,F_t,J_t)
\end{equation}
\begin{equation}
    d([y_1]) = \sum_{y_0; \deg(y_0) = \deg(y_1) +1} \sum_{u \in \mc{M}(y_0;y_1)} (-1)^{\deg(y_1)} \mu_u ([y_1]).
\end{equation} 
\begin{lem}
    \[ d^2 = 0. \]
\end{lem}
\noindent Call the resulting group $SH^*(M)$.

\begin{rem}
    Our grading conventions for symplectic cohomology follow Seidel
    \cite{Seidel:2010fk}, Abouzaid \cite{Abouzaid:2010kx}, and Ritter
    \cite{Ritter:2010nx}. These conventions are essentially determined by the
    fact that the identity element lives in degree zero, and the product map is
    also a degree zero operation, making $SH^*(M)$ a graded ring. 
    See the sections that follow for more details.
\end{rem}

\section{Open-closed moduli spaces and Floer data} \label{ocfloersection}
We recall definitions of abstract moduli spaces of genus 0 bordered Riemann
surfaces with interior and boundary marked points, which we will call {\bf
genus-0 open-closed strings}. Then, we define Floer data for such spaces, and
use these Floer data to construct chain-level open-closed operations in the
wrapped setting.  In the next section, we will specialize to examples such as
discs, spheres, and discs with interior and boundary punctures to obtain
$\ainf$ structure maps, the pair of pants product, and various open-closed operations.

\begin{defn} 
A {\bf genus-0 open-closed string} of type $h$ with $n,\vec{m} = (m^1, \ldots,
m^h)$ marked points $\Sigma$ is a sphere with $h$ disjoint discs removed, with $n$
interior marked points and $m^i$ boundary marked points on the $i$th boundary
component $\partial^i \Sigma$.  Fix some subset $\mathbf{I} \subset \{1,
\ldots, n\}$ and a vector of subsets $\vec{\mathbf{K}} = (K^1, \ldots, K^h)$
with $K^i \subset \{1, \ldots, m^i\}$. $\Sigma$ has {\bf sign-type
$(\mathbf{I},\vec{\mathbf{K}})$} if 
    \begin{itemize}
        \item interior marked points $p_i$, with $i \in I$ are negative,
        \item boundary marked points $z_{j,k} \in \partial^j \Sigma$, $k \in K^j$ are negative, and
    \item all other marked points are positive.
\end{itemize}
Also, a genus-0 open-closed string comes equipped with the data of
\begin{itemize}
    \item a choice of {\bf normal vector} or {\bf asymptotic marker} at each
        interior marked point.
\end{itemize}
\end{defn}
For our applications, we explicitly restrict to considering at most one
negative interior marked point or at most two negative boundary marked points,
i.e.  the cases
\begin{equation} \label{fewoutputs}
    \begin{cases}
        |I| =1\mathrm{\ and \ } \sum |K^i| = 0 \\
        |I| = 0\mathrm{\ and\ } \sum |K^i| = 1\mathrm{\ or\ 2}.
    \end{cases}
\end{equation}

\begin{defn}
The {\bf (non-compactified) moduli space of genus-0 open-closed strings} of
type $h$ with $n,\vec{m}$ marked points and sign-type
$(\mathbf{I},\vec{\mathbf{K}})$ is denoted
$\mc{N}_{h,n,\vec{m}}^{\mathbf{I},\vec{\mathbf{K}}}$.  
\end{defn}
Denote by $\overline{\mc{N}}_{h,n,\vec{m}}^{\mathbf{I},\vec{\mathbf{K}}}$ 
the Deligne-Mumford compactification of this space, a real blow-up of the space
as described by Liu \cite{Liu:2004fk}. Note that for $\mathbf{I}$ and
$\vec{\mathbf{K}}$ as in
(\ref{fewoutputs}), the lower-dimensional strata of
$\overline{\mc{N}}_{h,n,\vec{m}}^{\mathbf{I},\vec{\mathbf{K}}}$  consist of
nodal bordered surfaces, with each component genus 0 and also satisfying
(\ref{fewoutputs}).

Fix a collection of strip-like and cylindrical ends near every marked boundary
and interior point of a stable open-closed string, with the cylindrical ends
chosen to have $1 \in S^1$ asymptotic to our chosen marker. Then, at a nodal
surface consisting of $k$ interior nodes and $l$ boundary nodes, there is a
chart 
\begin{equation} \label{standardchart}
[0,1)^{k+l} \ra \overline{\mc{N}}_{h,n,\vec{m}}^{\mathbf{I},\vec{\mathbf{K}}}
\end{equation}
 given by assigning to the coordinate $(\rho_1, \ldots, \rho_k, \eta_1, \ldots,
 \eta_l)$ the glued surface where the $i$th interior node and $j$th boundary
 nodes have been glued with gluing parameters $\rho_i$ and $\eta_j$
 respectively, in a manner so that asymptotic markers line up for interior
 gluings. We refer the reader to Section \ref{homotopyunits} for more details
 on gluing in the case of strip-like ends.
 We see in this way that
 $\overline{\mc{N}}_{h,n,\vec{m}}^{\mathbf{I},\vec{\mathbf{K}}}$ inherits the
 structure of a manifold with corners. Moreover, from the corner charts
 described above every open-string $S \in
 \overline{\mc{N}}_{h,n,\vec{m}}^{\mathbf{I},\vec{\mathbf{K}}}$ inherits a {\bf
 thick-thin decomposition}, where the {\it thin parts} of the surface $S$ are
 by definition the finite cylinders and strips in $S$ that are inherited from
 the gluing parameters if $S$ lies in one of the above such charts. If $S$
 does not lie in such a chart, then $S$ has no thin parts.

\begin{rem} 
    In all the moduli spaces we will actually consider, the marked direction is
    always determined uniquely (and somewhat arbitrarily) by requiring it to
    point towards one particular distinguished boundary point. This works because
    \begin{itemize}
        \item it is consistent with Deligne-Mumford compactifications: when disc components break off and separate the interior puncture from the preferred boundary puncture, the new preferred boundary puncture is the node connecting the components;

        \item it is consistent with the choices of cylindrical ends made at interior punctures created when later we glue pairs of discs across $\Delta$ labels.
    \end{itemize}

    Given this, we will always omit these asymptotic markers from the
    discussion.
\end{rem}
    
\begin{rem}
    By considering moduli spaces where these asymptotic markers vary in $S^1$
    families, one can endow symplectic cohomology and open-closed maps with a
    larger set of operations, e.g. the BV operator. We will not do so here. For
    some additional details on such operations, 
    see \cite{Seidel:2010fk} or \cite{Seidel:2010uq}.
\end{rem}

\subsection{Floer data} \label{openclosedfloer}

First, we note that pullback of solutions to
(\ref{floerequation}) by the Liouville flow for time $\log(\rho)$ defines a
canonical isomorphism
\begin{equation}\label{rescaling}
    CW^*(L_0,L_1;H,J_t) \simeq 
    CW^*\left(\psi^\rho L_0, \psi^\rho L_1; 
    \frac{H}{\rho} \circ \psi^\rho, (\psi^\rho)^* J_t\right)
\end{equation}
We have two main observations which will help us define operations on
the complexes $CW^*(L_0,L_1)$ and $SH^*(M)$: 
\begin{lem}\label{rescalinglemma} 
    The function $\frac{H}{\rho^2}\circ \psi^\rho$ lies in $\mc{H}(M)$.
\end{lem}
\begin{proof}
    The Liouville flow is given on the collar by 
    \begin{equation}
        \psi^\rho (r,y) = (\rho \cdot r, y)
    \end{equation}
    so $r^2 \circ \psi^\rho = \rho^2 r^2$.
\end{proof}
Note however that $(\psi^\rho)^* J_t \notin \mc{J}_1(M)$. In fact, a
computation shows that 
\begin{equation}
    \frac{\rho}{r}\theta \circ (\psi^\rho)^*J_t = dr.
\end{equation}
Motivated by this, 
\begin{defn}
    Define $\mc{J}_c(M)$ to be the space of almost-complex structures $J$ that
    are {\bf c-rescaled contact type}, i.e.
\begin{equation}
    \frac{c}{r} \theta \circ J = dr.
\end{equation} 
Also, define $\mc{J}(M)$ to be the space of almost-complex structures $J$ that
are $c$-rescaled contact type for some $c$.\end{defn}  
To simplify terminology later, we will introduce some new notation for Floer
data. 

\begin{defn}\label{stripcylinderdata}
    A {\bf collection of strip and cylinder data} for a surface $S$ with some
    boundary and interior marked pointed removed is a choice of 
    \begin{itemize}
        \item {\bf strip-like ends} $\e_{\pm}^k : Z_{\pm} \ra S$,
        \item {\bf finite strips} $\e^l : [a^l,b^l] \times [0,1] \ra S$,
        \item {\bf cylindrical ends} $\delta_{\pm}^j: A_{\pm} \times S^1 \ra S$, and
        \item {\bf finite cylinders} $\delta^r : [a_r, b_r] \times S^1,
            \ra S$ 
    \end{itemize}
    all with disjoint image in $S$. Such a collection is said to be {\bf
    weighted} if each cylinder and strip above comes equipped with a choice
    of positive real number, called a weight. Label these weights as
    follows:
    \begin{itemize}
        \item $w_{S,k}^\pm$ is the weight associated to the strip-like end $\e_{\pm}^k$,
        \item $w_{S,l}$ is associated to the finite strip $\e^l$,
        \item $v_{S,j}^\pm$ is associated to the cylindrical end
            $\delta_{\pm}^j$, and 
        \item $v_{S,r}$ is associated to the finite cylinder $\delta^r$.
    \end{itemize}
    Finally, such a collection is said to be {\bf $\delta$-bounded} if 
    \begin{itemize}
        \item the length of each finite cylinder ($b_r - a_r$) is larger than
            $3\delta$.  
    \end{itemize}
\end{defn}
\begin{defn}\label{deltacollar}
    Let $\mathfrak{S}$ be a $\delta$-bounded collection of strip and cylinder
    data for $S$. The {\bf associated  $\delta$-collar} of $S$ is the following
    collection of finite cylinders:
    \begin{itemize}
        \item the restriction $\tilde{\delta}_+^j$ of each positive cylindrical
            end $\delta_+^j : [0,\infty) \times S^1 \ra S$ to the domain
            $[0,\delta] \times S^1$, 
        
        \item the restriction $\tilde{\delta}_-^j$ of each negative cylindrical
            end $\delta_-^j : (-\infty,0] \times S^1 \ra S$ to the domain
            $[-\delta,0] \times S^1$, and 

        \item the restrictions $\tilde{\delta}^r_{in}$ and
            $\tilde{\delta}^r_{out}$ of each finite cylinder $\delta^r: [a_r,
            b_r] \times S^1 \ra S$ to the domains $[a_r,a_r + \delta] \times
            S^1$ and $[b_r - \delta, b_r] \times S^1$ respectively.
    \end{itemize}
    We will often refer to this as the {\bf associated collar} if $\delta$ is
    implicit.  
\end{defn}
\noindent Let $(S,\mathfrak{S})$ be a surface $S$ with a $\delta$-bounded
collection of weighted strip and cylinder data $\mathfrak{S}$.
\begin{defn}
    A one-form $\a_S$ on $S$ is said to be {\bf compatible with the weighted
    strip and cylinder data} $\mathfrak{S}$ if, for each finite or semi-infinite
    cylinder or strip $\kappa$ of $S$, with associated weight $\nu_\kappa$,
    \begin{equation}
        \kappa^* \a_S = \nu_\kappa dt.
    \end{equation}
    Above, $t$ is the coordinate of the second component of the associated strip
    or cylinder.
\end{defn}
\begin{defn} 
    Fix a Hamiltonian $H \in \mc{H}(M)$. An $S$-dependent Hamiltonian $H_S: S
    \ra \mc{H}(M)$ is said to be {\bf H-compatible with the weighted strip and
    cylinder data}
    $\mathfrak{S}$ if, for each cylinder or strip $\kappa$ with associated
    weight $\nu_\kappa$, 
    \begin{equation}
        \kappa^* H_S = \frac{H \circ \psi^{\nu_\kappa}}{\nu_\kappa^2}.
    \end{equation}
\end{defn}
\begin{defn}
    An {\bf $\mathfrak{S}$-adapted rescaling function} is a map $a_S: S \ra
    [1,\infty)$ that is constant on each cylinder and strip of $\mathfrak{S}$,
    equal to the associated weight of that cylinder or strip.  
\end{defn}
\begin{defn}
    Fix a time-dependent almost-complex structure $J_t: S^1 \ra \in \mc{J}_1(M)$,
    and an adapted rescaling function $a_S$. An {\bf
    $(\mathfrak{S},a_S,J_t)$-adapted complex structure} is a map
    $J_S:S\ra\mc{J}(M)$ such that 
    \begin{itemize}
        \item at each point $p \in S$, $J_p \in \mc{J}_{a_S(p)}(M)$,
        \item at each cylinder or strip $\kappa$ with associated weight $\nu_\kappa$, 
            \begin{equation}
                \kappa^* J_S = (\psi^{\nu_\kappa})^* J_t.
            \end{equation}
            Here, if $\kappa$ is a strip, we mean the $[0,1]$ dependent complex
            structure given by pulling back $J_t$ by the projection map
            $[0,1] \ra S^1 = \R / \Z$.
    \end{itemize}
\end{defn}

We will need to introduce Hamiltonian perturbation terms supported on the
cylinders of $(S,\mathfrak{S})$, in order to break the $S^1$ symmetry of
orbits. Since we will be gluing nodal cylindrical punctures together, these
perturbation terms need to possibly have support on the thin-parts of our
gluing as well. The next definition will give us very explicit control over
these perturbation terms.  Let $F_T : S^1 \ra C^\infty(E)$ be a time-dependent
function that is absolutely bounded, with all derivatives absolutely bounded.
Also, let $\phi_\e(s): [0,1] \ra [0,1]$ be a smooth function that is 0 in an
$\e$-neighborhood of 0, 1 in an $\e$-neighborhood of 1, with all derivatives
bounded.
\begin{defn}\label{adaptedperturbation}
    For $(S,\mathfrak{S})$ as above, an {\bf $S^1$-perturbation adapted to
    $(F_T,\phi_\e)$} is a function $F_S : S \ra C^\infty(E)$ satisfying the
    following properties:     
    \begin{itemize}
        \item $F_S$ is locally constant on the complement of the images of all
            cylinders,
        \item on each cylindrical end
            $\kappa^{\pm}$ with associated weight $\nu_\kappa$, outside the
            associated collar, 
            \begin{equation}
                (\kappa^{\pm})^* F_S = \frac{F_T \circ
                \psi^{\nu_\kappa}}{\nu_\kappa^2} + C_{\kappa}, 
            \end{equation} 
            where $C_\kappa$ is a constant depending on the
            cylinder $\kappa^\pm$.
        \item on each finite cylinder $\kappa^r$, outside the associated collar,
            \begin{equation}
                (\kappa^r)^* F_S = m_\kappa \frac{F_T \circ
                \psi^{\nu_\kappa}}{\nu_\kappa^2}+ C_\kappa,
            \end{equation}
            where $C_\kappa$ and $m_\kappa$ are constants depending on the
            cylinder $\kappa^r$.
        \item On each associated $\delta$-collar, $\kappa: [0,\delta] \times
            S^1 \ra S$,  
    \begin{equation}
        \kappa^*F_S = (\kappa^*F_S)|_{0\times S^1} + \phi_\e( s/\delta) (
        (\kappa^*F_S)|_{\delta\times S^1} - (\kappa^*F_{S})|_{0 \times S^1})
    \end{equation}

        \item $F_S$ is weakly monotonic on each cylinder $\kappa$,  i.e.
            \begin{equation}
                \partial_s \kappa^* F_S \leq 0.
            \end{equation} 
    \end{itemize}
\end{defn}
Putting all of these together, we can make the following definition:
\begin{defn}\label{floeropenclosed}
    A {\bf Floer datum} $\mathbf{F}_S$ on a stable genus zero open-closed string
    string $S$ consists of the following choices on each component:
    \begin{enumerate} 
        \item A collection of {\bf weighted strip and cylinder
            data} $\mathfrak{S}$ that is
            $\delta$-bounded;

        \item {\bf sub-closed 1-form}: a one-form $\a_S$ with 
            \begin{equation*}
                d \a_S \leq 0,
            \end{equation*}
        compatible with the weighted strip and cylinder data;

    \item A {\bf primary Hamiltonian} $H_S : S \ra \mc{H}(M)$ that is
        $H$-compatible with the weighted strip and cylinder data $\mathfrak{S}$
        for some fixed $H$;

    \item An {\bf $\mathfrak{S}$-adapted rescaling function} $a_S$;

    \item An {\bf almost-complex structure $J_S$} that is
        $(\mathfrak{S},a_S,J_t)$-adapted for some $J_t$.

    \item An {\bf $S^1$-perturbation $F_S$ adapted to $(F_T,\phi_\e)$} for some
        $F_T$, $\phi_\e$ as above.
\end{enumerate}
\end{defn}

There is a notion of equivalence of Floer data, weaker than strict equality,
which will imply by the rescaling correspondence (\ref{rescaling}) that the
resulting operations are identical.  
\begin{defn} 
    Say that Floer data $D_S^1$ and $D_S^2$ are {\bf conformally equivalent} if
    there exist constants $C,K,K'$ such that 
    \begin{equation}
        \begin{split} 
            a_S^2 &= C\cdot a_S^1,\\
            \alpha_S^2 &= C \cdot \alpha_S^1,\\  
            J_S^2 &= (\psi^C)^* J_S^1,\\
            H_S^2 &= \frac{H_S^1 \circ \psi^C}{C^2} + K,\textrm{ and}\\
            F_S^2 &= \frac{F_S^1 \circ \psi^C}{C^2} + K'.
        \end{split}
    \end{equation}
In other words, the Floer $D_S^2$ is a rescaling by Liouville flow of the
Floer data $D_S^1$, up to a constant ambiguity in the Hamiltonian terms.
\end{defn}

\begin{defn}
    A {\bf universal and consistent choice of Floer data} for genus 0
    open-closed strings is a choice $\bf{D}_S$ of Floer data for every $h$,
    $n$, $\vec{m}$, $\mathbf{I}$, $\vec{\mathbf{K}}$ and every representative
    $S$ of $\overline{\mc{N}}_{h,n,\vec{m}}^{\mathbf{I},\vec{\mathbf{K}}}$,
    varying smoothly over $\mc{N}_{h,n,\vec{m}}^{\mathbf{I},\vec{\mathbf{K}}}$,
    whose restriction to a boundary stratum is conformally equivalent to the
    product of Floer data coming from lower dimensional moduli spaces.
    Moreover, with regards to the coordinates, Floer data agree to infinite
    order at the boundary stratum with the Floer data obtained by gluing.
\end{defn}
\begin{rem}
    By varying smoothly, we mean that the data of $H_S$, $F_S$, $a_S$, $\a_S$,
    and $J_S$, along with the cylindrical and strip-like ends vary smoothly.
    Over given charts of our moduli space, finite cylinders and strips need to
    vary smoothly as well, but they may be different across charts (for
    example, some charts that stay away from lower-dimensional strata may have
    no finite cylinder or strip-like regions).
\end{rem}
All of the choices involved in the definition of a Floer datum above
are contractible, so one can inductively over strata prove that 
\begin{lem} 
    The restriction map from the space of universal and consistent
    Floer data to the space of Floer data for a fixed surface $S$ is
    surjective.
\end{lem} 
\begin{defn} 
    Let $\mathbf{L}$ be a set of Lagrangians. A {\bf Lagrangian labeling from
    $\mathbf{L}$} for a genus-0 open-closed string $S \in
    \mathcal{N}^{\mathbf{I},\vec{\mathbf{K}}}_{h,n,\vec{m}}$ is a choice, for
    each $j=1, \ldots, h$ and for each connected component
    $\partial^j_i S$ of the $j$th boundary disc, of a Lagrangian $L_i^j \in
    \mathbf{L}$. The {\bf space of genus-0 open-closed strings with a fixed
    labeling} $\vec{L} = \{\{L_i^j\}_i\}_j$ is denoted $(\cocs)_{\vec{L}}$. The
    {\bf space of all labeled open-closed strings} is denoted
    $(\cocs)_{\mathbf{L}}$.
\end{defn}
Clearly, $(\cocs)_{\mathbf{L}}$ is a disconnected cover of $\cocs$.  There is
a notion of a {\bf labeled Floer datum}, namely a Floer datum for
the space of open-closed strings equipped with labels
$(\overline{\mathcal{N}}^{\mathbf{I},\vec{\mathbf{K}}}_{h,n,\vec{m}})_{\mathbf{L}}$.
This is simply a choice of Floer data as above in a manner also depending
coherently on the particular Lagrangian labels. We will use this notion in
later sections, along with the following definition. 
\begin{defn} 
    Let $\mathbf{D}_S$ be a Floer datum on a surface $S$. The {\bf induced
    labeled Floer datum} on a labeled surface $S_{\vec{L}}$ is the Floer datum
    $\mathbf{D}_S$ coming from forgetting the labels.  
\end{defn}
\subsection{Floer-theoretic operations} \label{openclosedfloeroperations}
Now, fix a compact oriented submanifold with corners of dimension $d$,
\begin{equation}
    \overline{\mc{Q}}^d\hookrightarrow \cocs.
\end{equation}
Fix a Lagrangian labeling 
\begin{equation}
\{\{L_0^1, \ldots, L_{m_1}^1\}, \{L_0^2, \ldots, L_{m_2}^2\}, \ldots, \{L_0^h,
\ldots, L_{m_h}^h\}\}.
\end{equation}
Also, fix chords 
\begin{equation}
\vec{x} = \{\{x_1^1, \ldots, x_{m_1}^1\}, \ldots,\{ x_1^h, \ldots, x_{m_h}^h\}\}
\end{equation}
and orbits $\vec{y} = \{y_1, \ldots, y_n\}$ with
\begin{equation}
x_i^j \in \begin{cases}
    \chi(L_{i+1}^j,L_i^j) & i \in K^j \\
    \chi(L_i^j, L_{i+1}^j) & \mathrm{otherwise.}
\end{cases}
\end{equation}
Above, the index $i$ in $L_i^j$ is counted mod $m_{j}$. The {\bf
outputs} $\vec{x}_{out}$, $\vec{y}_{out}$ are by definition those $x_i^j$ and
$y_s$ for which $i \in K^j$ and $s \in \mathbf{I}$, corresponding to negative marked points.
The {\bf inputs} $\vec{x}_{in}$, $\vec{y}_{in}$ are the remaining chords and
orbits from $\vec{x}$, $\vec{y}$.  Fixing a chosen universal and consistent
Floer datum, denote $\e^{i,j}_\pm$ and
$\delta^l_\pm$ the strip-like and cylindrical ends corresponding to $x_i^j$ and
$y_l$ respectively. 

Define
\begin{equation}\label{basicopenclosedmodulispace}
\overline{\mc{Q}}^d(\vec{x}_{out},\vec{y}_{out};
\vec{x}_{in},\vec{y}_{in})
\end{equation}
to be the space of maps 
\begin{equation}
    \{ u: S \longrightarrow M:\ S \in \overline{\mc{Q}}^d \}
\end{equation}
satisfying the inhomogenous Cauchy-Riemann equation with respect to the complex
structure $J_S$:
\begin{equation}
    (du - X_S \otimes \alpha_S)^{0,1} = 0
\end{equation}
and asymptotic and boundary conditions:  
\begin{equation}
    \begin{cases}
        \lim_{s \ra \pm\infty} u \circ \e^{i,j}_\pm(s,\cdot) = x_i^j,\\
    \lim_{s \ra \pm\infty} u \circ \delta^l_\pm(s,\cdot) = y_l,\\
    \ \ \ u(z) \in \psi^{a_S(z)} L_i^j, & z\in \partial_i^j S.
\end{cases}
\end{equation}
Above, $X_S$ is the (surface-dependent) Hamiltonian vector field corresponding
to $H_S + F_S$.

\begin{lem} \label{openclosedcompactdimensiontransversality}
    The moduli spaces $\overline{\mc{Q}}^d(\vec{x}_{out},\vec{y}_{out};
    \vec{x}_{in},\vec{y}_{in})$ are compact and there are only finitely many
    collections $\vec{x}_{out}$, $\vec{y}_{out}$ for which they are non-empty
    given input $\vec{x}_{in}$, $\vec{y}_{in}$. For a generic universal and
    conformally consistent Floer data they form manifolds of dimension 
\begin{equation}
    \begin{split}
    \dim
    \mc{Q}^d(\vec{x}_{out},\vec{y}_{out};\vec{x}_{in},\vec{y}_{in}):=\sum_{x_-
    \in \vec{x}_{out}} &\deg (x_-) + \sum_{y_- \in \vec{y}_{out}} \deg (y_-) +\\
    (2-h-|\vec{x}_{out}| - 2|\vec{y}_{out}|) n
    + d 
    &-\sum_{x_+ \in \vec{x}_{in}} \deg(x_+) - \sum_{y_+ \in \vec{y}_{in}}
    \deg(y_+).  
\end{split}
\end{equation}
\end{lem}
\begin{proof}
The dimension calculation follows from a computation of the index of the
associated linearized Fredholm operator. Via a gluing theorem for indices
\cite{Seidel:2008zr}*{(11c)} \cite{Schwarzthesis}*{Thm. 3.2.12}, there is a
contribution coming from the index of the linearized Cauchy Riemann operator on
the compactified surface $\hat{S}$, equal to $n\chi(\hat{S})$, where
$\chi(\hat{S})= (2-h)$ is the Euler characteristic of a genus-0 open-closed
string of type $h$. The other contributions come from the tangent space of
$\mc{Q}$ (contributing $d$), and spectral-flow type calculations on the
striplike and cylindrical ends. This calculation is essentially a fusion of
\cite{Seidel:2008zr}*{Proposition 11.13} and \cite{Ritter:2010nx}*{Lemma 10}.

The proof of transversality for generic perturbation data is a standard
application of Sard-Smale, following identical arguments in
\cite{Seidel:2008zr}*{(9k)} or alternatively \cite{Floer:1995fk}. The usual
proof Gromov compactness also applies, assuming that solutions to Floer's
equation with given asymptotic boundary conditions are a priori bounded in the
non-compact target $M$. This is the content of Theorem \ref{c0bounds}.
\end{proof}
When $\mc{Q}^d(\vec{x}_{out},\vec{y}_{out};\vec{x}_{in},\vec{y}_{in})$ has
dimension zero, we conclude that its elements are rigid.
For any such element $u \in \mc{Q}^d(\vec{x}_{out},
\vec{y}_{out};\vec{x}_{in},\vec{y}_{in})$, we obtain an isomorphism of
orientation lines, by Lemma \ref{orientationlem} 
\begin{equation} \label{orientationiso}
    \mc{Q}_u: 
    \bigotimes_{x \in \vec{x}_{in}} o_x \otimes \bigotimes_{y \in \vec{y}_{in}} o_y \lra \bigotimes_{x \in \vec{x}_{out}} o_x \otimes \bigotimes_{y \in \vec{y}_{out}} o_y.
\end{equation}

Thus, we can define a map 
\begin{equation}
    \begin{split}
         \mathbf{F}_{\overline{\mc{Q}}^d}: \bigotimes_{(i,j);1\leq i\leq m_j;i
        \notin K^j} &CW^*(L_i^j,L_{i+1}^j) \otimes \bigotimes_{1 \leq k \leq n;
        k \notin \mathbf{I}} CH^*(M) \longrightarrow  \\
         &\bigotimes_{(i,j);1\leq i\leq m_j;i
         \in K^j} CW^*(L_{i+1}^j,L_{i}^j) \otimes \bigotimes_{1 \leq k \leq n;
        k \in \mathbf{I}} CH^*(M) \mbox{   }
    \end{split}
\end{equation}
    given by:
    \begin{equation} \label{operation1}
        \begin{split}
            &\mathbf{F}_{\overline{\mc{Q}}^d}([y_t], \ldots, [y_1], 
            [x_s], \ldots, [x_1]) := \\ 
        & \sum_{\dim \mc{Q}^d(\vec{x}_{out},\vec{y}_{out};
        \{x_1, \ldots, x_s\},\{y_1, \ldots, y_t\}) = 0}  
        \sum_{u \in \mc{Q}^d(\vec{x}_{out},\vec{y}_{out}; \{x_1, \ldots, x_s\}, \{y_1, \ldots, y_t\})}
        \mc{Q}_u([x_s], \ldots, [x_1], [y_t], \ldots, [y_1]).
    \end{split}
\end{equation}

This construction naturally associates, to any submanifold $\mc{Q}^d \in
\ocs$, a chain-level map $\mathbf{F}_{\mc{Q}^d}$, depending on a
sufficiently generic choice of Floer data for open-closed strings.  We need to
modify this construction by signs depending on the relative
positions and degrees of the inputs.
\begin{defn}\label{signtwistingdatum}
    Given such a submanifold $\mc{Q}$, a {\bf sign twisting datum} $\vec{t}$
    for $\mc{Q}$ is a vector of integers, one for each input boundary or
    interior marked point on an element of $\mc{Q}$. 
\end{defn}
\noindent 
To a pair $(\mc{Q}, \vec{t})$ one can associate a twisted operation
\begin{equation} 
    \label{signtwistoperation} (-1)^{\vec{t}}
    \mathbf{F}_{\overline{\mc{Q}}^d}, 
\end{equation}
defined as follows. If $\{\vec{x},\vec{y}\} = \{x_1, \ldots, x_s,
y_1, \ldots, y_t\}$ is a set of asymptotic inputs, the {\bf vector of degrees}
is denoted
\begin{equation}
    \vec{\deg} (\vec{x},\vec{y}): = \{\{\deg(x_1), \ldots, \deg(x_s)\},
    \{\deg(y_1), \ldots, \deg(y_t)\}\}.  
\end{equation}
The corresponding sign twisting datum $\vec{t}$ is of the form
\begin{equation}
    \vec{t} := \{\{v_1, \ldots, v_s\}; \{w_1, \ldots, w_t\}\}.
\end{equation}
Then, the operation (\ref{signtwistoperation}) is defined to be
\begin{equation} \label{twistoperation1}
    \begin{split}
        &(-1)^{\vec{t}}\mathbf{F}_{\overline{\mc{Q}}^d}([y_t], \ldots, [y_1], 
        [x_s], \ldots, [x_1]) := \\ 
    & \sum_{\dim \mc{Q}^d(\vec{x}_{out},\vec{y}_{out};
    \vec{x},\vec{y}) = 0}  
    \sum_{u \in \mc{Q}^d(\vec{x}_{out},\vec{y}_{out}; ,\vec{x}, \vec{y})}
    (-1)^{\vec{t} \cdot \vec{\deg}(\vec{x}, \vec{y})} \mc{Q}_u([x_1], \ldots, [x_s], [y_1], \ldots, [y_t]).
\end{split}
\end{equation}
The zero vector $\vec{t} = (0, \ldots, 0)$ recovers the original operation
$\mathbf{F}_{\overline{\mc{Q}}^d}$.

Now, suppose instead that we are given a submanifold $\mc{Q}^d_{\vec{L}}$
of the labeled space $(\cocs)_{\vec{L}}$.  Then, we obtain a chain-level
operation
     \begin{equation} \label{operationlabeled1}
         (-1)^{\vec{t}}\mathbf{F}_{\mc{Q}^d_{\vec{L}}}
     \end{equation}
 that {\it is only defined for the fixed labeling $\vec{L}$}.  In the
 sections that follow, we will use this definition to construct associated
 chain-level map for specific families $\{\mc{Q}^d_{\vec{L}}\}_d$.

 \begin{rem}
     Strictly speaking, when there are two boundary outputs on the same
     component, one only obtains the isomorphism of orientation lines
     (\ref{orientationiso}) after choosing orientations of Lagrangians along
     that boundary component. Since we are working with oriented Lagrangians,
     we are implicitly making such choices. See Appendix
     \ref{orientationsection} for more details.
 \end{rem}

In a different direction, we will make repeated use of the following standard
{\it codimension 1 boundary principle} for Floer-theoretic operations: suppose
that the boundary $\partial \overline{\mc{Q}}^d$ is covered by the images of
natural inclusions of $(d-1)$-dimensional (potentially nodal) orientable
submanifolds
\begin{equation}
    \mc{T}_i \hookrightarrow \partial \overline{\mc{Q}}^d,\ i = 1,\ldots, k.
\end{equation}
Then, standard results tell us that 
\begin{lem}\label{bordificationlem}
In the situation above, the Gromov bordification of the moduli space of maps
$\overline{\mc{Q}}^d(\vec{x}_{out},\vec{y}_{out};\vec{x}_{in},\vec{y}_{in})$
has codimension 1 boundary covered by the images of natural inclusions of the
following spaces:
\begin{align}
    \mc{T}_i(\vec{x}_{out}, \vec{y}_{out}; \vec{x}_{in}, \vec{y}_{in}),\ i = 1, \ldots, k\\
    \label{inx1}\overline{\mc{R}}(\tilde{x}; x_a) \times \overline{\mc{Q}}^d(\vec{x}_{out},\vec{y}_{out};\tilde{\vec{x}}_{in},\vec{y}_{in}) \\
    \label{iny1}\overline{\mc{M}}(\tilde{y}; y_b)\times \overline{\mc{Q}}^d(\vec{x}_{out},\vec{y}_{out};\vec{x}_{in},\tilde{\vec{y}}_{in}) \\
    \label{outx1}\overline{\mc{Q}}^d(\tilde{\vec{x}}_{out},\vec{y}_{out};\vec{x}_{in},\vec{y}_{in}) \times \overline{\mc{R}}(x_c; \tilde{x}) \\
    \label{outy1}\overline{\mc{Q}}^d(\vec{x}_{out},\tilde{\vec{y}}_{out};\vec{x}_{in},\vec{y}_{in}) \times \overline{\mc{M}}(y_d; \tilde{y}).
\end{align}
Here, \begin{itemize}
    \item in (\ref{inx1}), $x_a \in \vec{x}_{in}$ and $\tilde{\vec{x}}_{in}$ is
$\vec{x}_{out}$ with the element $x_a$ replaced by $\tilde{x}$; 
\item in (\ref{iny1}), $y_b \in \vec{y}_{in}$ and $\tilde{\vec{y}}_{in}$ is
$\vec{y}_{in}$ with the element $y_b$ replaced by $\tilde{y}$; 
\item in (\ref{outx1}), $x_c \in \vec{x}_{out}$ and $\tilde{\vec{x}}_{out}$ is
    $\vec{x}_{out}$ with the element $x_c$ replaced by $\tilde{x}$; and
\item in (\ref{outy1}), $y_d \in \vec{y}_{out}$ and $\tilde{\vec{y}}_{out}$ is
    $\vec{y}_{out}$ with the element $y_d$ replaced by $\tilde{y}$.
    \end{itemize}
    The strata (\ref{inx1}) - (\ref{outy1}) range over all
    $\tilde{x}$, $\tilde{y}$ and all possible choices of $x_a, y_b, x_c, x_d$.
\end{lem}
In words, this Lemma says that the boundary of space of maps from $\mc{Q}$ is
covered by maps from the various $\mc{T}_i$ plus all possible
semi-stable strip or cylinder breakings.

The manifolds $\mc{T}_i$, which may live on the boundary strata of $\cocs$,
inherit orientations and Floer data from the choice of $\mc{Q}^d$, via the
convention of orienting relative to the normal vector pointing towards the
boundary. Thus, there are associated signed operations 
\begin{equation}
    (-1)^{\vec{t}} \mathbf{F}_{\mc{T}_i}.
\end{equation}

By looking at the boundary of one-dimensional elements of the moduli space
$\mc{Q}^d(\vec{x}_{out},\vec{y}_{out};\vec{x}_{in},\vec{y}_{in})$, one
concludes that 
\begin{cor}
    In the situation described above, for any $\vec{t}$,
\begin{equation}
    \begin{split}
        \sum_{i=1}^k (-1)^{\vec{t}} &\mathbf{F}_{\mc{T}_i}  + 
        \sum_{i=1}^{s+t} (-1)^{*}(-1)^{\vec{t}}\mathbf{F}_{\overline{\mc{Q}}^d} \circ 
        (id \otimes \cdots \otimes id \otimes \mu^1 \otimes id \otimes 
        \cdots id) \\ 
        &+ \sum_j (-1)^{\dagger}(id \otimes \cdots id \otimes \mu^1 \otimes id
        \otimes \cdots id) \circ \bigg( (-1)^{\vec{t}}\mathbf{F}_{\overline{\mc{Q}}^d} \bigg)= 0,
\end{split}
\end{equation}
where we have used $\mu^1$ to indicate the both the differential on wrapped
Floer homology or symplectic cohomology depending on the input.  The signs
$(-1)^{*}$ and $(-1)^{\dagger}$ will be calculated in Appendix
\ref{orientationsection}.  
\end{cor}
In order to obtain equations such as the $\ainf$ equations, bimodule equations,
various morphisms are chain homotopies, etc. with the correct signs, one needs
to compare signs between the operators $(-1)^{\vec{t}}\mathbf{F}_{\mc{T}_i}$
and the composition of operators arising from $\mc{T}_i$ viewed of as a
(potentially nodal) surface using the consistency condition imposed on our
Floer data. This, plus appropriate choices of sign twisting data for these
strata, will yield all of the relevant signs.
The relevant calculations are performed in Appendix \ref{orientationsection}.

\section{Open-closed maps}\label{openclosedmaps}

\subsection{The product in symplectic cohomology} 
Symplectic cohomology is known to admit a range of TQFT-like operations,
parametrized by surfaces with $I$ incoming and $J$ outgoing ends, for $J>0$,
see e.g. \cite{Ritter:2010nx}. In this section, we will focus on the case with
one output and two inputs ($I = 2$, $J=1$), but one can imagine the following
construction applies more generally to other surfaces and {\it families} of
surfaces.

Denote by $\mc{S}_2$ the configuration space of spheres with two positive and
one negative punctures, with asymptotic markers pointing in the
tangent direction to the unique great circle containing all three points.

\begin{defn} A {\bf Floer datum} $D_T$ on a stable sphere $T\in \mc{S}_{2,1}$ consists of a Floer datum of $T$ thought of as an open-closed string.
\end{defn}

From the previous section, considering the entire manifold $\mc{S}_2$ defines
an operation of degree zero \begin{equation} \mc{F}_2: CH^*(M)^{\otimes
2} \lra CH^*(M).  \end{equation} This is known as the {\it pair of pants
product}.
\begin{figure}[h] 
    \caption{A representative of the one-point space $\mc{S}_{2,1}$ giving the pair of pants product. \label{pop_product}}
    \centering
    \includegraphics[scale=0.7]{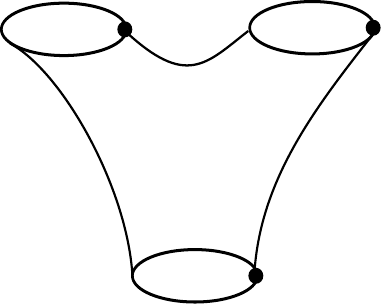}
\end{figure}

\begin{rem}
    For families of spheres with more than two inputs, there ceases to be a
    preferred direction in which to point the asymptotic markers, a situation
    not considered in our work. However, if one mandates that all marked points
    lie on a single great circle, then one recovers the required preferred
    direction. The end result, which gives {\it Massey products} on $SH^*(M)$, will be
    constructed as a special case of the constructions in Sections \ref{pairs}
    and \ref{productsection}.
\end{rem}

\subsection{\texorpdfstring{$\ainf$}{A-infinity} structure maps and the wrapped Fukaya category}

Here we define the higher structure maps $\mu^k$ on $\w$ (including
the product $\mu^2$). We will recall and apply with some detail our
construction of Floer-theoretic operations from Section
\ref{openclosedfloeroperations}, though the reader is warned that subsequent
constructions will be more terse.
Define
\begin{equation}
    \mc{R}^d
\end{equation} 
to be the (Stasheff) moduli space of discs with one negative marked point
$z_0^-$ and $d$ positive marked points $z_1^+, \ldots, z_d^+$ removed from the
boundary, labeled in {\it counterclockwise} order from $z_0^-$.  $\mc{R}^d$ is
a special case of our general construction of open-closed strings.  Denote by
$\overline{\mc{R}}^d$ its natural (Deligne-Mumford) compactification,
consisting of trees of stable discs with a total of $d$ exterior positive
marked points and 1 exterior negative marked point, modulo compatible
reparametrization of each disc in the tree. Recall from the discussion in
Section \ref{ocfloersection} that $\overline{\mc{R}}^d$ inherits the structure
of a manifold with corners, coming from standard gluing charts \begin{equation}
    \label{gluingcoords} (0,+\infty]^k \times \sigma \ra \overline{\mc{R}}^d.
\end{equation}
 near (nodal) strata of codimension $k$.

\begin{figure}[h]
    \caption{Two drawings of a representative of an element of the moduli space
    $\mc{R}^5$. The drawing on the right emphasizes the choices of strip-like
    ends.} 
    \centering
    \includegraphics[width=0.5\textwidth]{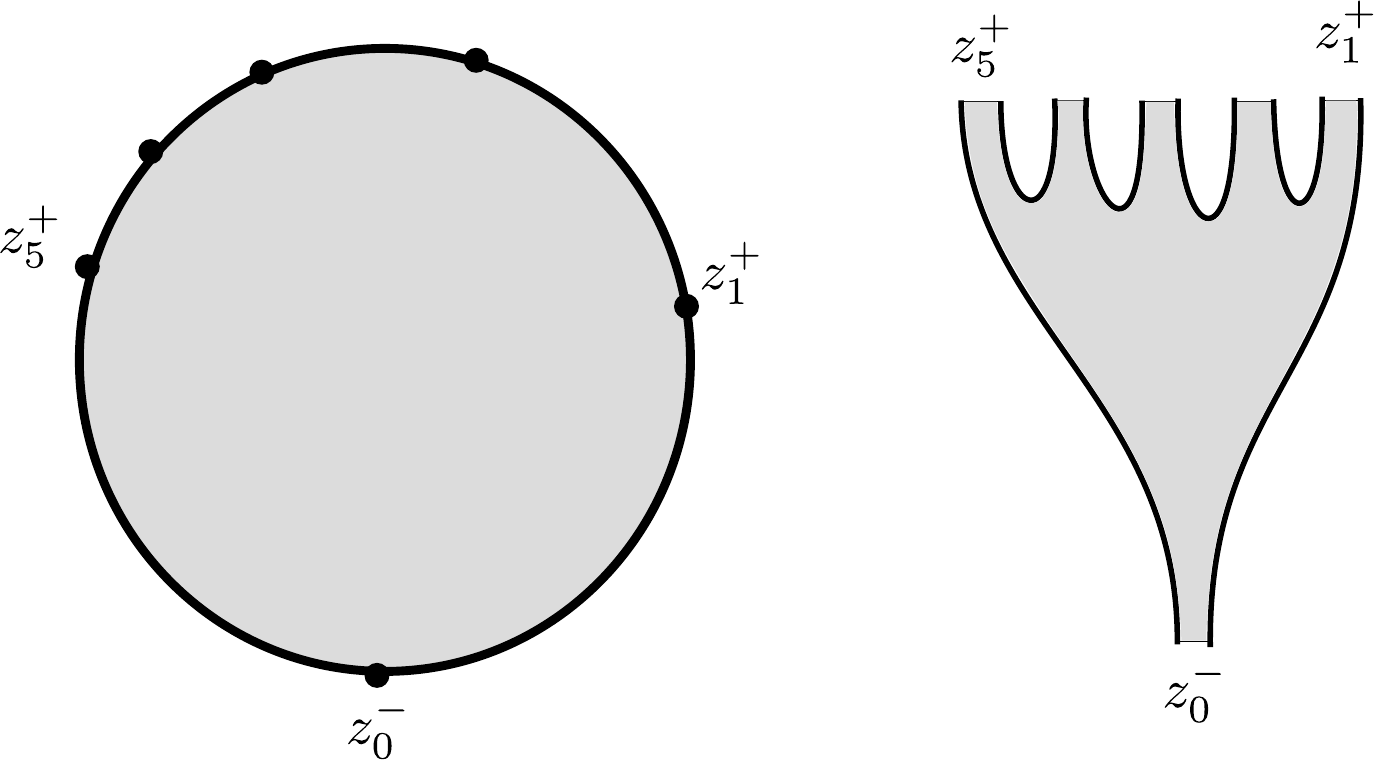}
\end{figure}

Now, in the terminology of Definition \ref{floeropenclosed}, pick a universal
and consistent choice of Floer data $\mathbf{D}_{\mu}$ for the spaces
$\mc{R}^d$, $d > 2$. Also, fix an orientation of the space $\mc{R}^d$,
discussed in Appendix \ref{ainforientation}.

\begin{defn}
    The {\bf $d$th order $\ainf$ operation} is by definition the operation
    \begin{equation}
        \mu^d:=(-1)^{\vec{t}}\mathbf{F}_{\overline{\mc{R}}^d}, 
    \end{equation}
    in the sense of (\ref{signtwistoperation}), where $\vec{t}$ is the sign
    twisting datum given by $(1, 2, \ldots, k)$.
\end{defn}
We step through this construction for clarity.  Let $L_0, \ldots,
L_d$ be objects of $\w$, and consider a sequence of chords $\vec{x} = \{x_k \in
\chi(L_{k-1},L_k)\}$ as well as another chord $x_0 \in \chi(L_0,L_d)$. Given
a fixed universal and consistent Floer data $\mathbf{D}_\mu$, write
$\mc{R}^d(x_0; \vec{x})$ for the space of maps 
\[u: S \ra M \] 
with source an arbitrary element $S \in \mc{R}^d$, with marked points
$(z^0, \ldots, z^d)$ satisfying the boundary asymptotic conditions 
\begin{equation}
    \begin{cases}
        u(z) \in \psi^{a_S(z)} L_k & \mathrm{if\ }z\in \partial S\mathrm{\ lies\ between\ }z^k\ \mathrm{and\ }z^{k+1} \\
        \lim_{s\ra \pm \infty} u \circ \e^k(s,\cdot) = \psi^{a_S(z)}x_k
    \end{cases}
\end{equation}
and differential equation 
\begin{equation}
    (du - X_S \otimes \alpha_S)^{0,1} = 0
\end{equation}
with respect to the complex structure $J_S$ and total Hamiltonian $H_S + F_S$.
Using the consistency of our Floer data and the codimension one boundary of the
abstract moduli spaces $\overline{\mc{R}}^d$, Lemma \ref{bordificationlem}
implies that the Gromov bordification $\overline{\mc{R}}^d(x_0; \vec{x})$ is
obtained by adding the images of the natural inclusions 
\begin{equation}
    \overline{\mc{R}}^{d_1}(x_0; \vec{x}_1) \times \overline{\mc{R}}^{d_2}(y; \vec{x}_2) \ra \overline{\mc{R}}^d(x_0; \vec{x})
\end{equation}
where $y$ agrees with one of the elements of $\vec{x}_1$ and $\vec{x}$ is
obtained by removing $y$ from $\vec{x}_1$ and replacing it with the sequence
$\vec{x}_2$. Here, we let $d_1$ range from 1 to $d$, with $d_2 = d-d_1 + 1$, with the stipulation that $d_1=$ or $d_2 = 1$ is the semistable case:
\begin{equation}
    \overline{\mc{R}}^{1} (x_0; x_1) := \overline{\mc{R}}(x_0; x_1)
\end{equation}
Thanks to Lemma \ref{openclosedcompactdimensiontransversality}, for generically
chosen Floer data $\mathbf{D}_{\mu}$
\begin{cor}
    The moduli spaces $\overline{\mc{R}}^d(x_0;\vec{x})$ are smooth compact
    manifolds of dimension 
    \[
    \deg (x_0) + d - 2 - \sum_{1\leq k \leq d} \deg(x_k).
    \]
\end{cor}
In particular, if $\deg(x_0) = 2 - d + \sum_1^d \deg (x_k)$, then the elements
of $\overline{\mc{R}}^d(x_0; \vec{x})$ are rigid, and for any such rigid $u \in
\overline{\mc{R}}^d(x_0; \vec{x})$, we obtain by Lemma \ref{orientationlem}, an
isomorphism
\begin{equation}
    \mc{R}^d_{u}: o_{x_d} \otimes \cdots \otimes o_{x_1} \lra o_{x_0}.
\end{equation}
Thus, taking into account the sign twisting $\vec{t}$, we define the operation
\begin{equation}
    \mu^d: CW^*(L_{d-1}, L_d) \otimes \cdots \otimes CW^*(L_0,L_1) \lra CW^*(L_0,L_d)
\end{equation}
as a sum
\begin{equation}
    \mu^d([x_d], \ldots, [x_1]) := \sum_{\deg(x_0) = 2- d + \sum \deg(x_k)} \sum_{u \in \overline{\mc{R}}^d(x_0; \vec{x})} (-1)^{\bigstar_d}\mc{R}^d_{u} ([x_d], \ldots, [x_1])
\end{equation}
where 
\begin{equation}
    \bigstar_d = \vec{t} \cdot \vec{\deg}(\vec{x}):=
        \sum_{i=1}^d i \cdot \deg (x_i).
\end{equation}
By looking at the codimension 1 boundary of 1-dimensional families of such
maps, and performing a tedious sign comparison analogous to those Appendix
\ref{orientationsection} (discussed in \cite{Seidel:2008zr}*{Prop. 12.3}), we
conclude that
\begin{lem}
    The maps $\mu^d$ satisfy the $\ainf$ relations.
\end{lem}
\noindent The sign twisting datum used here will reappear with variations
later, so it is convenient to fix notation.
\begin{defn}
    The {\bf incremental sign twisting datum of length $d$}, denoted
    $\vec{t}_d$, is the vector $(1, 2, \ldots, d-1, d)$.  
\end{defn}

\subsection{From the open sector to the closed sector}

As in \cite{Abouzaid:2010kx}, define $\mc{R}^1_d$ to be the abstract moduli
space of discs with $d$ boundary positive punctures $z_1, \ldots, z_d$ labeled in counterclockwise order and 1 interior negative
puncture $z_{out}$, with the last positive puncture $x_d$ marked as
distinguished. Its Deligne-Mumford compactification inherits the structure of a
manifold with corners via the inclusion $\overline{\mc{R}}^1_d \hookrightarrow
\cocs$ where $h=1$, $n=1$, $\vec{m} = (d)$, $\mathbf{I} = \{1\}$,
$\vec{\mathbf{K}} = (\{\})$. 

\begin{figure}[h]
    \caption{Two drawings of representative of an element of the moduli space $\mc{R}^1_4$. The drawing on the left emphasizes the choices of strip-like and cylindrical ends. The distinguished boundary marked point is the one set at $-i$ on the right.}
    \centering
    \includegraphics[scale=0.7]{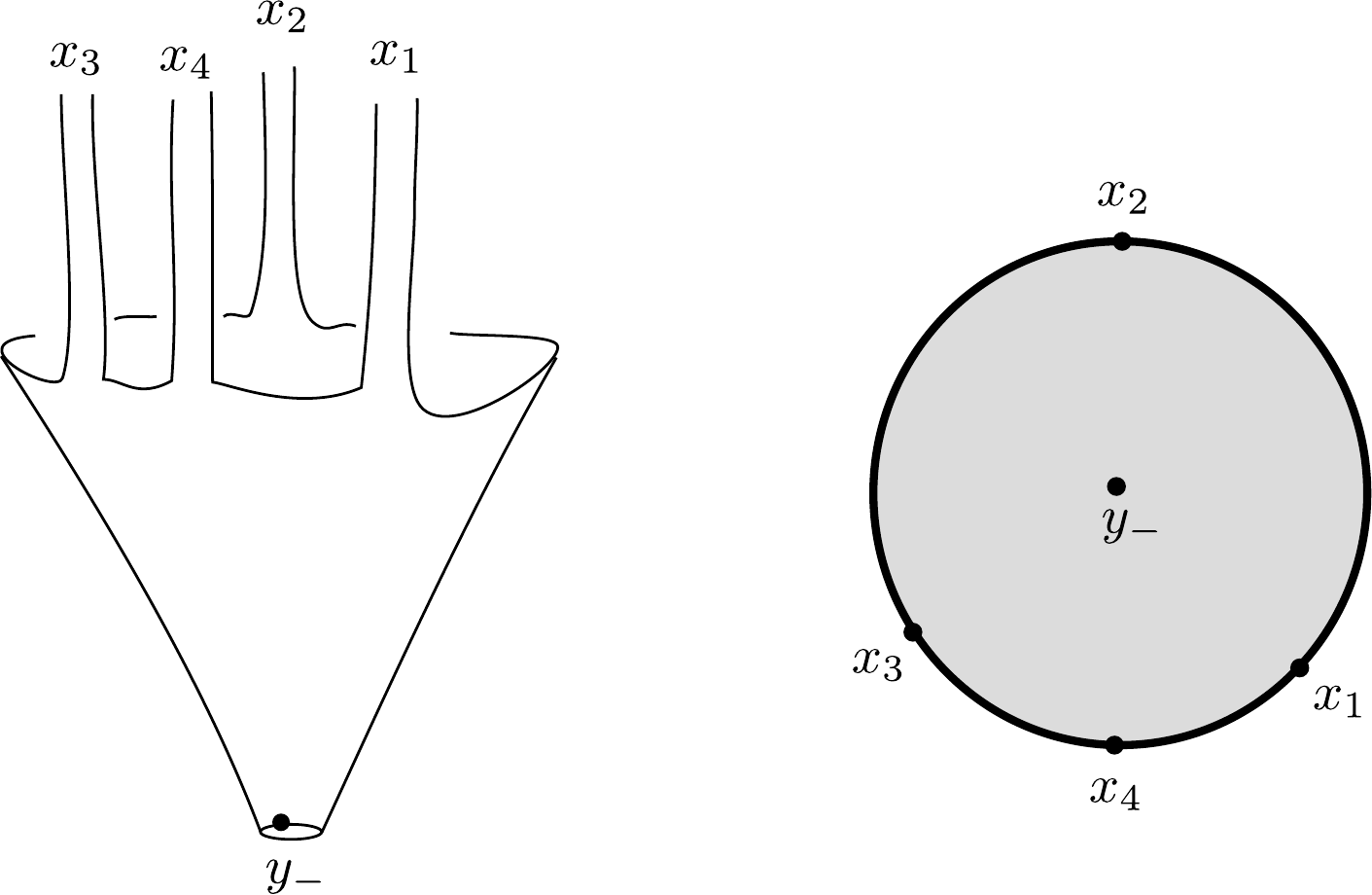}
\end{figure}

In the manner of (\ref{basicopenclosedmodulispace}), using our fixed generic
universal and consistent Floer data and an orientation for $\mc{R}^1_d$ fixed
in Appendix \ref{oc1orientation}, we obtain, for every Lagrangian labeling
$L_1, \ldots,
L_d$, and asymptotic conditions $\{x_1, \ldots, x_d, y_{out}\}$ moduli
spaces
\begin{equation}
    \mc{R}^1_d(y_{out};\{x_i\})
\end{equation}
which are compact smooth manifolds of dimension
\begin{equation}
    \deg(y_{out}) - n + d - 1 - \sum_{k=0}^{d-1} \deg(x_k).
\end{equation}
Then, fixing sign twisting datum
\begin{equation}
    \vec{t}_{\oc,d} := (1, 2, \ldots, d-1, d+1) = \vec{t}_d + (0, \ldots, 0,
    1), 
\end{equation}
we obtain associated operations
\begin{equation}
    \begin{split}
    \oc_d:= &(-1)^{\vec{t}_{\oc,d}}\mathbf{F}_{\mc{R}^1_d}: \\
    &\hom_{\w}(L_d, L_0) \otimes 
    \hom_\w(L_{d-1},L_d) \otimes \cdots \otimes \hom_\w(L_{0},L_1) 
    \lra CH^*(M);
\end{split}
\end{equation}
in other words, operations
\begin{equation}
    \oc_d: (\w_{\Delta} \otimes \w^{\otimes d-1})^{diag} \lra CH^*(M)
\end{equation}
of degree $n-d +1$.  The composite map 
\begin{equation}
    \oc := \sum_d \oc_d
\end{equation} 
therefore gives a map 
\begin{equation} 
    \oc: \r{CC}_*(\w,\w) \lra CH^*(M)
\end{equation} 
of degree $n$ (using the grading conventions for Hochschild homology
(\ref{hochschildchaingrading}).
Recall that the codimension-1 boundary of the Deligne-Mumford compactification
$\overline{\mc{R}}^1_d$
is covered by the following strata: 
\begin{align}
    \label{ocstrata1}\overline{\mc{R}^{m}} &\times_{i} \overline{\mc{R}}^1_{d-m+1} \ \ \ 1 \leq i < d-m+1\\
    \label{ocstrata2}\overline{\mc{R}^m}&\times_{d-m+1} \overline{\mc{R}}^1_{d-m+1} \ \ \ 1 \leq j \leq m
\end{align}
where the notation $\times_j$ means that the output of the first component is
identified with the $j$th boundary input of the second. In the second type of
stratum (\ref{ocstrata2}), the $j$th copy correspond to the stratum in which
the $j$th input point on $\mc{R}^m$ becomes the distinguished boundary marked
point on $\mc{R}_d^1$ after gluing.

The consistency condition imposed on Floer data implies that the Gromov
bordification $\overline{R}^1_d(y_0, \vec{x})$ is obtained by adding the images
of natural inclusions of moduli spaces of maps coming from the boundary strata
(\ref{ocstrata1})-(\ref{ocstrata2}) along with the following semi-stable
breakings
\begin{align} 
     \mc{R}^1_d(y_1,\vec{x}) \times \mc{M}(y_{out};y_1)&\ra
    \partial \overline{\mc{R}}^1_d(y_{out}; \vec{x})\\
    \overline{\mc{R}}^{1}(x_1; x) \times\overline{\mc{R}}^1_{d}(y_{out};
    \tilde{\vec{x}}) &\ra \partial
    \overline{\mc{R}}^1_d(y_{out}; \vec{x}) 
\end{align} 
where in the second type of stratum, $x$ is one of the elements of $\vec{x}$
and $\tilde{\vec{x}}$ is the sequence obtained by replacing $x$ in $\vec{x}$ by
$x_1$.
Thus, modulo a tedious sign verification whose details are discussed in Appendix
\ref{orientationsection}, we obtain that 
\begin{prop}
$\oc$ is a chain map.
\end{prop}

\subsection{From the closed sector to the open sector} 
In a similar fashion,
define $\mc{R}^{1,1}_d$ to be the moduli space of discs with 
\begin{itemize}
    \item $d+1$ boundary marked points removed, $1$ of which is negative and
        labeled $z_0^-$, and $d$ of which are positive and labeled $(z_1,
        \ldots, z_d)$ in counterclockwise order from $z_0^-$; and
    \item one interior positive marked point $y_{in}$ removed.
\end{itemize}
Its Deligne-Mumford compactification inherits the structure of a manifold with
corners via the inclusion $\overline{\mc{R}}^{1,1}_d \hookrightarrow \cocs$,
where $h=1$, $n=1$, $\vec{m} = (d+1)$, $\mathbf{I} = \{\}$, $\vec{\mathbf{K}} =
(\{1\})$.
Thus, let us fix a universal and conformally consistent choice of Floer datum
$\mathbf{D}_{\co}$ on $\mc{R}^{1,1}_d$ for every $d \geq 1$.
Given a Lagrangian labeling and a set of compatible asymptotic conditions,
along with an orientation of $\mc{R}^{1,1}_d$ discussed in Appendix
\ref{co1orientation}, we obtain the moduli space $\mc{R}_d^{1,1}(x_{out};
y_{in},\vec{x})$ as in (\ref{basicopenclosedmodulispace}), which are compact
smooth manifolds of dimension
\begin{equation}
    \deg(x_{out}) + d - \deg(y_{in})- \sum_{k=1}^d \deg(x_k).
\end{equation}
Fix sign twisting datum 
\begin{equation}
    \vec{t}_{\co,d} = (0,1,2, \ldots, d)
\end{equation}
with respect to the ordering of inputs $(y_{in}, x_1, \ldots, x_d)$. Then,
define
\begin{equation}
    \co_d: CH^*(M) \lra \hom_{Vect}(\w^{\otimes d},\w)
\end{equation}
as 
\begin{equation}
    \co_d(y_{in})(x_d, \ldots, x_1) := (-1)^{\vec{t}_{\co,d}} \mathbf{F}_{\overline{\mc{R}}^{1,1}_d}(y_{in},x_d, \ldots, x_1)
\end{equation}
The composite map 
\begin{equation}
    \co = \sum_d \co_d
\end{equation}
gives a map $CH^*(M) \ra \r{CC}^*(\w,\w)$. 

The codimension-1 boundary of the Deligne-Mumford compactification
$\overline{\mc{R}}^{1,1}_d$ is covered by the natural images of the following
products:
\begin{align}
    \overline{\mc{R}}^{m} &\times_{i} \overline{\mc{R}}^{1,1}_{d-m+1}\ \ \ 1 < m < d-1,\ 1 \leq i < d-m+1\\
    \overline{\mc{R}}^{1,1}_{d-(k+l+1)+1} &\times_{k+1} \overline{\mc{R}}^{k+l+1}\ \ 0 < k+l < d-1.
\end{align}
using the notation $\times_j$ as in the last section to indicate gluing the
distinguished output of the first component to the $j$th (boundary) input of
the second component.

\begin{figure}[h]
    \caption{Two drawings of representative of an element of the moduli space $\mc{R}^{1,1}_5$. The drawing on the right emphasizes the choices of strip-like and cylindrical ends.}
    \centering
    \includegraphics[scale=0.7]{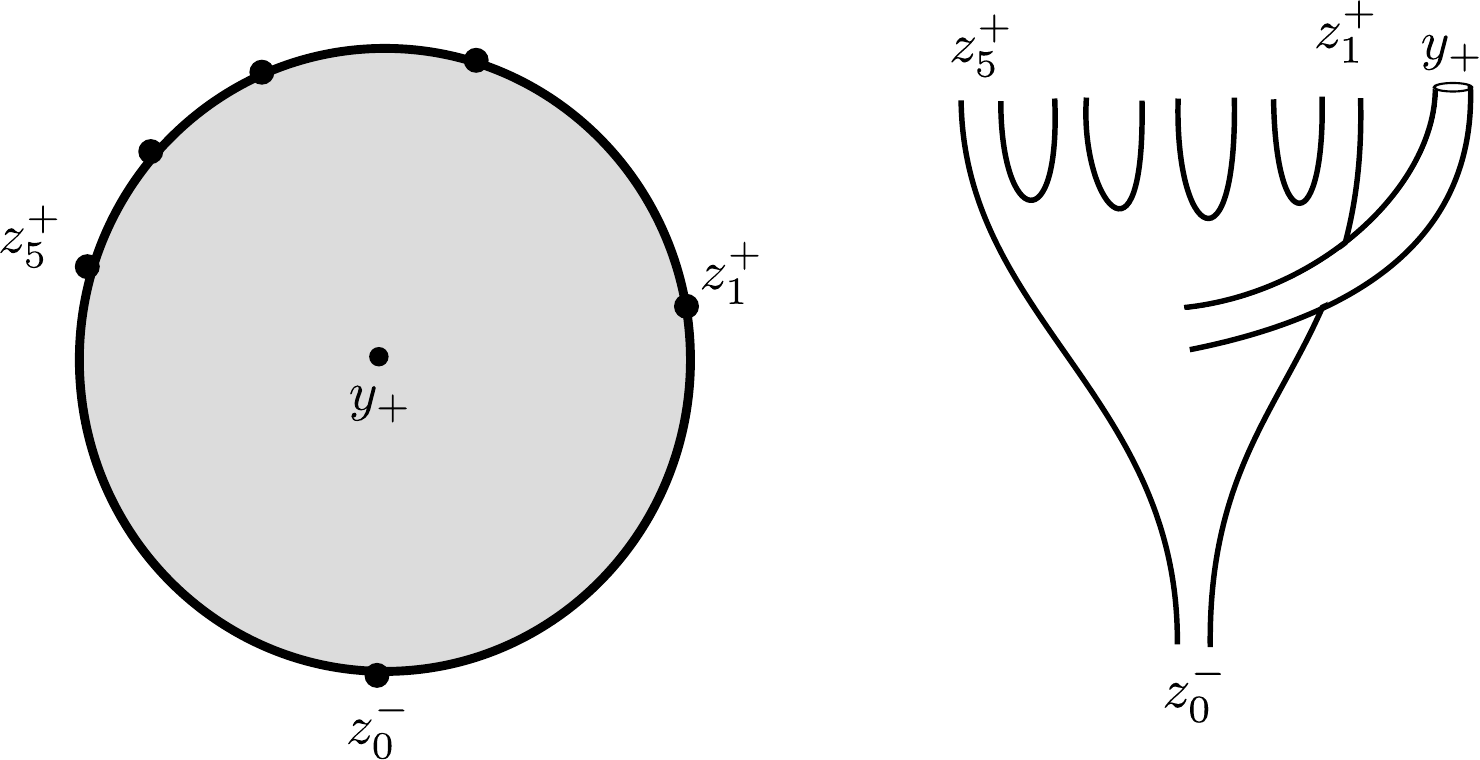}
\end{figure}

The consistency condition implies that the Gromov
bordification $\mc{\ol{R}}^{1,1}_d(x_{out};y_{in},\vec{x})$ is obtained by
adding images of the natural inclusions
\begin{align}
    \mc{M}(y_1;y_{in}) \times \mc{R}^{1,1}_d(x_{out}; y_1,\vec{x})  &\ra
    \bd \ol{\mc{R}}^{1,1}_d(x_{out}; y_{in},\vec{x}) \\
    \mc{R}^{1,1}_{d_2}(x^a;y_{in},\vec{x}^2)\times 
    \mc{R}^{d_1}(x_{out},\vec{x}^1) 
     &\ra \bd
    \ol{\mc{R}}^{1,1}_d(x_{out}; y_{in},\vec{x}) \\
    \mc{R}^{d_2}(x^a;\vec{x}^2) \times 
    \mc{R}^{1,1}_{d_1}(x_{out};y_{in},\vec{x}^1) 
    &\ra 
    \bd \ol{\mc{R}}^{1,1}_d(x_{out};
    y_{in},\vec{x})
\end{align}
where $d_1 + d_2 - 1 = d$, $\vec{x}^2$ is any consecutive sub-vector of size
$d_2$ and $\vec{x}^1$ is obtained by replacing $\vec{x}^2$ in $\vec{x}$ by
$x^a$.
By the above result about Gromov bordification, and a sign verification
discussed in Appendix \ref{orientationsection}, we see that:
\begin{prop}
    $\co$ is a chain map.
\end{prop}

\subsection{Ring and module structure compatibility} 
We will make two
assertions about the maps $\co$ and $\oc$, both of which follow from an
analysis of similar-looking moduli spaces.
\begin{prop}\label{ring}
    $[\co]$ is a ring homomorphism.
\end{prop} 

Via its module structure over Hochschild cohomology and the ring homomorphism
$[\co]$, Hochschild homology $\r{HH}_*(\w,\w)$ obtains the structure of a module
over $SH^*(M)$. With respect to this structure, using a similar argument, we
can prove the following:
\begin{prop}\label{module}
    $[\oc]$ is a map of $SH^*(M)$-modules.
\end{prop}
\noindent The chain-level statement is that the following diagram homotopy commutes:
\begin{equation}\label{modulehomotopy}
    \xymatrix{\r{CC}_*(\w) \times CH^*(M) \ar[r]^{(\oc,\id)}\ar[d]^{(\id,\co)} & CH^*(M) \times CH^*(M) \ar[d]^{*} \\
    \r{CC}_*(\w) \times \r{CC}^*(\w) \ar[r]^{\mbox{ }\mbox{ }\mbox{ }\oc \circ \cap} & CH^*(M)},
\end{equation}
\noindent To prove Proposition \ref{modulehomotopy}, for $r \in (0,1)$, define
the auxiliary moduli space \begin{equation}\mc{P}^2_{d}(r)
\end{equation}
to consist of the unit disc in $\C$ with the following data:
\begin{itemize}
    \item $d+1$ positive boundary marked points $(z_0, \ldots, z_d)$, with
       $z_d$ marked as distinguished, and

    \item two interior marked points $(\kappa_+,\kappa_-)$, one positive and
        one negative 
\end{itemize}
such that,
\begin{equation}
    \textrm{after automorphism, the points $z_d$, $\kappa_+$, $\kappa_-$ lie at $-i$, $-r$, and $r$ respectively}.
\end{equation}
These spaces vary smoothly with $r$ and their union 
\begin{equation}
\mc{P}^2_d := \bigcup_{r\in (0,1)}\mc{P}^2_d (r)
\end{equation}
is naturally a codimension 1 submanifold of genus 0 open-closed
strings consisting of a single disc with $d$ positive boundary punctures and
two interior punctures, one positive and one negative.  Compactifying, we
obtain a family that submerses over $r \in [0,1]$ and see that with
codimension-1 boundary of
$\overline{\mc{P}}^2_d$ is covered by the images of the natural inclusions of
the following products (some living over the endpoints $r \in \{0,1\}$ and some
living over the entire interval):
\begin{align}
    \ol{\mc{R}}^{1,1}_{d_3} \times_{n+k+1} \ol{\mc{R}}^{k+1 + d_2} &\times_m \ol{\mc{R}}^{1}_{d +1 - d_3 - d_2 - k }, \ \ \  d_1 + d_2 + d_3 - 2 = d \ \ \ (r=1) \\
     \ol{\mc{R}}^1_{d+1} &\times_{2} \ol{\mc{S}}_2  \ \ \ \ (r=0) \\
     \ol{\mc{R}}^{m} &\times_n \ol{\mc{P}}^2_{d-m+1}  \\
     \ol{\mc{R}}^{m} &\times_{d-m+1} \ol{\mc{P}}^2_{d-m+1}.
\end{align}

\begin{figure}[h] 
    \caption{The space $\mc{P}^2_d$ and its $r\ra \{0,1\}$ degenerations. Not shown in the degeneration: the special input point is constrained to remain on the middle $\mu$ bubble.  \label{modulestructurefig1}}
    \centering
    \includegraphics[scale = 0.7]{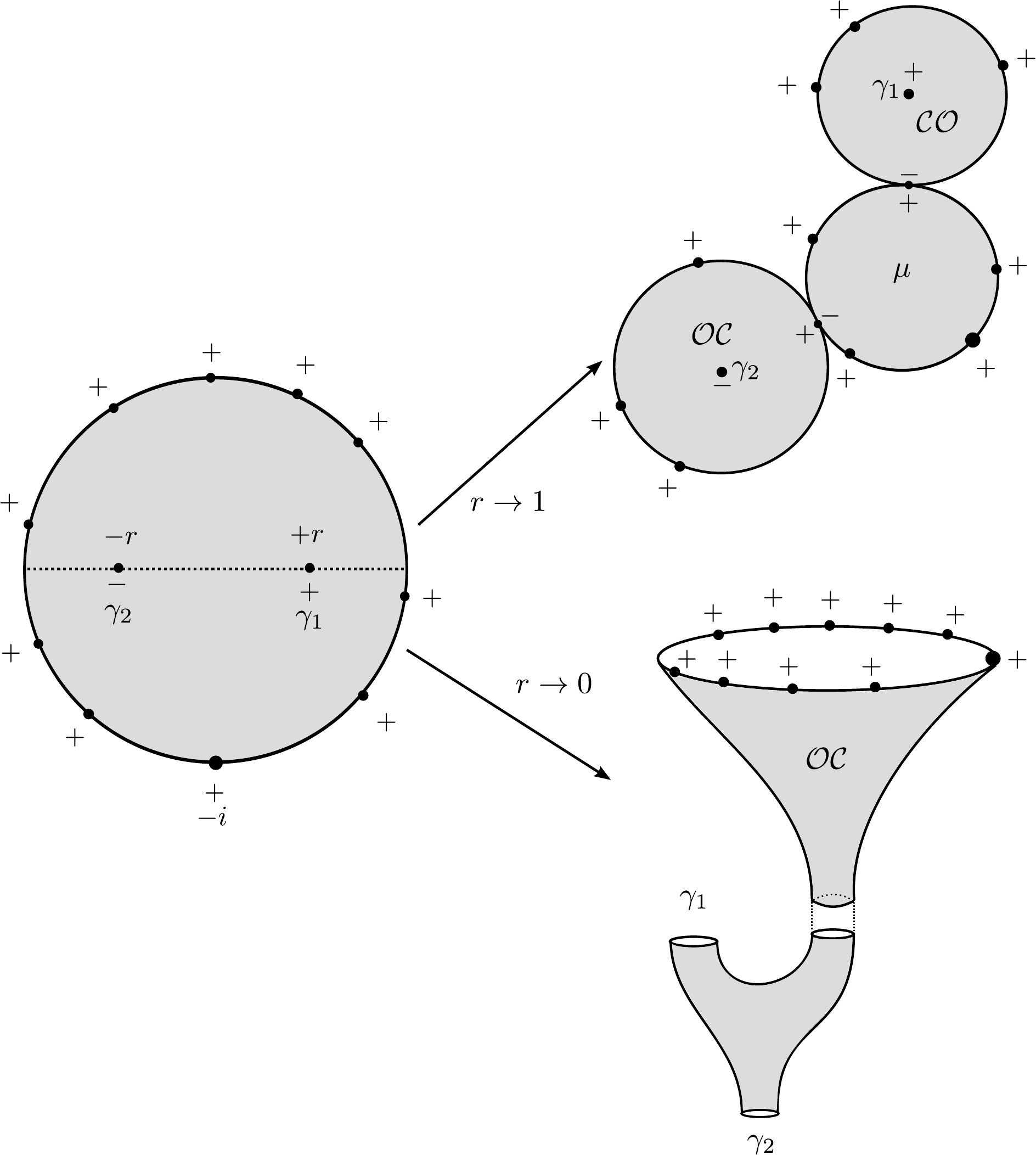}
\end{figure}

Fix a universal and consistent Floer data for all $\overline{\mc{P}}^2_d$.
Given a set of Lagrangian labels \begin{equation}
    L_0, \ldots, L_d, L_{d+1} = L_0\end{equation} and
compatible asymptotic conditions 
$\vec{x} = \{x_k \in
\mc{X}(L_{k},L_{k+1})\}_{k=0}^{d}$ and $\gamma_{-}, \gamma_{+}$,
we obtain a moduli space of maps
\begin{equation}
    \mc{P}^2_{d}(\gamma_-; \gamma_+,\vec{x}).
\end{equation} 
which are smooth compact manifolds of dimension
\begin{equation}
    \deg (\gamma_-) -n + d+1 - \deg (\gamma_+) - \sum_{k=0}^{d} \deg (x_k).
\end{equation}
Consistency of our Floer data implies that the Gromov bordification
$\chtopmaps$ is obtained by adding the images of the natural inclusions
\begin{align}
\mc{M}(\gamma_0;\gamma_+) \times \overline{\mc{P}}^2_{d}(\gamma_-; \gamma_0,\vec{x}) &\ra \bd \chtopmaps\\
\overline{\mc{P}}^2_d(\gamma_0; \gamma_+,\vec{x}) \times \mc{M}(\gamma_-;\gamma_0) &\ra \bd \chtopmaps\\
\label{stratum1}\ol{\mc{R}}^{d_1}(x_a;\vec{x}^2) \times \overline{\mc{P}}^2_{d_2}(\gamma_-;\gamma_+,\vec{x}^1)  &\ra \bd\chtopmaps\\
\label{stratum2}\overline{\mc{R}}^1_d(\gamma_1; \vec{x}) \times \overline{\mc{S}}_2(\gamma_-; \gamma_+,\gamma_1) &\ra \bd\chtopmaps \\
\label{stratum3}\overline{\mc{R}}^{1,1}_{d_3}(x_b; \gamma_+,\vec{x}^3) \times 
\overline{\mc{R}}_{d_2}(x_a; \vec{x}^2)\times 
\overline{\mc{R}}^1_{d_1}(\gamma_-; \vec{x}^1) &\ra\bd\chtopmaps
\end{align}
where 
\begin{itemize}
    \item in (\ref{stratum1}), $\vec{x}^2$ is a subvector of $\vec{x}$, and
        $\vec{x}^1$ is obtained from $\vec{x}$ by replacing $\vec{x}^2$ by
        $x_a$

    \item in (\ref{stratum3}), $\vec{x}^3$ is a subvector of $\vec{x}$, not including the distinguished input $x_0$.
        $\widehat{\vec{x}^2}$ is obtained from $\vec{x}$ by replacing $\vec{x}^3$
        by $x_b$, and then performing a cyclic permutation that brings the
        distinguished output $x_0$ to the right of $x_b$. $\vec{x}^2$ is a
        subvector of $\widehat{\vec{x}}^2$ that includes both $x_b$ and $x_0$, and $\vec{x}^1$ is obtained from
        $\widehat{\vec{x}}^2$ by replacing $\vec{x}^2$
        with $x_a$.
\end{itemize}
Now, define the map
\begin{equation}
    \mc{H}_d: CH^*(M) \otimes (\w_{\Delta} \otimes \w^{\otimes d})^{diag} \ra CH^*(M)
\end{equation}
as
\begin{equation}
    \mc{H}_d := (-1)^{\vec{t}_{\mc{P}^2_d}}\mathbf{F}_{\overline{\mc{P}}^2_d},
\end{equation}
where we use sign twisting datum
\begin{equation}
    \vec{t}_{\mc{P}^2_d} := (-1, 0, 1, \ldots, d)
\end{equation}
corresponding to the ordering of inputs $(\kappa_+,z_0, \ldots, z_d)$.
The composite map $\mc{H} = \sum_d \mc{H}_d$ gives a map 
\begin{equation}\mc{H}: CH^*(M) \times
    \r{CC}_*(\w,\w) \lra CH^*(M).\end{equation} 
By the above result about Gromov bordifications and a sign verification
discussed in Appendix \ref{orientationsection} we conclude that
\begin{prop}
    For any $\a,s \in \r{CC}_*(\w),CH^*(M)$, 
\begin{equation}
d_{SH} \circ \mc{H}(\alpha,s) \pm \mc{H}(\delta(\alpha),s) \pm
\mc{H}(\alpha,d_{SH}(s)) = \oc(\alpha) * s - \oc(\co(s) \cap \alpha).
\end{equation}
\end{prop}
Thus, $\mc{H}$ is the desired chain homotopy for (\ref{modulehomotopy}),
concluding the proof of Proposition \ref{module}.  We briefly indicate how to
change this argument to prove Proposition \ref{ring}. One considers operation
associated to the same abstract moduli space as $\mc{P}^2_d$, where both
interior punctures are marked as positive points and the distinguished boundary
input is now marked as an output. The associated Floer theoretic operation with
a similar sign twist gives a homotopy between the Yoneda product applied to
elements of $\co$ (the degenerate limit $r \ra 1$) and the pair of pants
product applied before applying $\co$ (the limit $r \ra 0$).

\begin{proof}[Proof of Proposition \ref{surjectivity}]
    Suppose $M$ is non-degenerate, and let $\sigma$ be any homology level
    pre-image of $1 \in SH^*(M)$ via the map $\oc$. Then, if $s$ is another
    element of $SH^*(M)$, we see that by Proposition \ref{module}, on homology
    \begin{equation}
        \oc(\co(s) \cap \sigma) = s \cdot \oc(\sigma) = s \cdot 1 = 1.
    \end{equation}
    In particular this implies that $\co(s) \cap \sigma$ is a preimage of $s$,
    and $\co(s)$ cannot be zero unless $s$ is.
\end{proof}

\subsection{Two-pointed open-closed maps}
Since the two-pointed complexes 
\[{_2}\r{CC}_*(\w,\w),\ {_2}\r{CC}^*(\w,\w)\] 
arise naturally from a bimodule perspective, we will define variants of the
chain-level map $\oc$ and $\co$ between $SH^*(M)$ and the respective
two-pointed complexes: 
\begin{align}
    {_2}\oc: {_2}\r{CC}_*(\w,\w) &\lra SH^*(M) \\
    {_2}\co: SH^*(M) & \lra {_2}\r{CC}^*(\w,\w)
\end{align}
To ensure consistency with existing arguments, we prove that the
resulting maps are in fact quasi-isomorphic to $\oc$ and $\co$.

\begin{defn} \label{twopointedocmod}
    The {\bf two-pointed open-closed moduli space} with $(k,l)$ marked points
    \begin{equation}
        \mc{R}^1_{k,l}
    \end{equation}
    is the space of discs with one interior negative puncture labeled
    $y_{out}$, and $k+l+2$ boundary punctures, labeled in counterclockwise order
    $z_0,z_1,\ldots,z_k, z_0', z_1', \ldots, z_l'$, such that:
    \begin{equation}
        \textrm{up to automorphism, }z_0,\ z_0',\ \textrm{and }y_{out}\ \textrm{are constrained to lie at }-i,\ i\ \textrm{and }0\textrm{ respectively}.
    \end{equation}
    Call $z_0$ and $z_0'$ the {\bf special inputs} of any such disc.
\end{defn}
\begin{rem}
    The moduli space $\mc{R}^1_{k,l}$ is a codimension one submanifold of
    $\mc{R}^1_{k+l+2}$, and thus has dimension $k+l$.  
\end{rem}

\begin{figure}[h]
    \caption{A representative of an element of the moduli space $\mc{R}^1_{3,2}$ with special points at 0 (output), $-i$, and $i$. }
    \centering
    \includegraphics[scale=0.7]{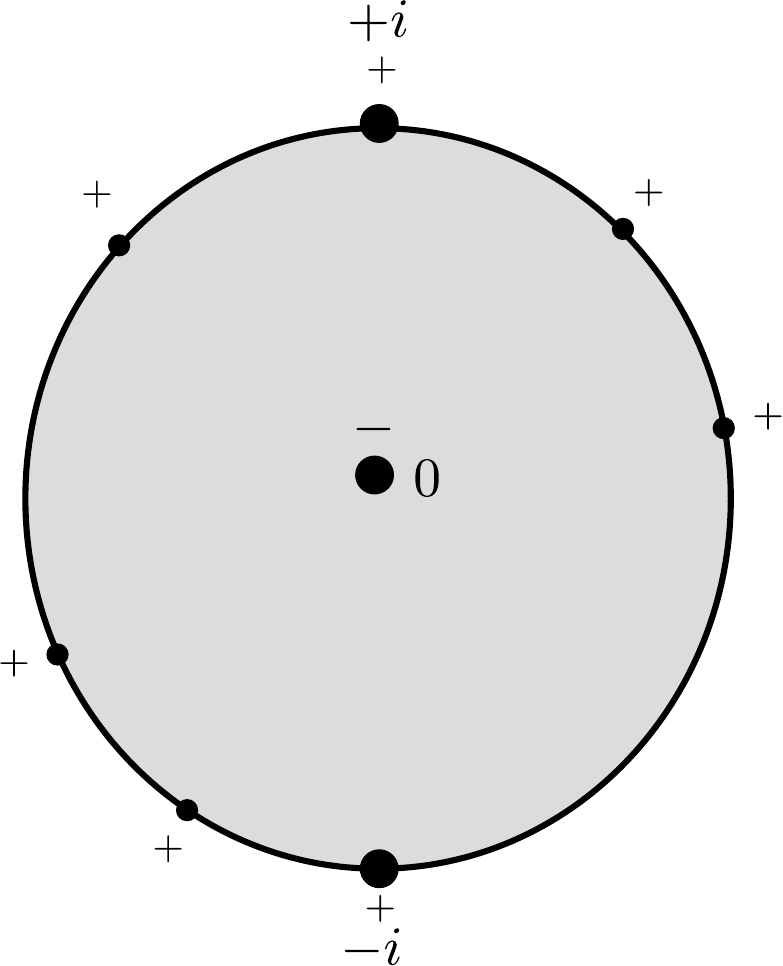}
\end{figure}

The boundary strata of Deligne-Mumford compactification
$\overline{\mc{R}}^1_{k,l}$ is covered by the images of the natural inclusions
of the following products:
\begin{align}
    \label{tocspecialpoint0}    \overline{\mc{R}}^{k'} &\times_{n+1} \overline{\mc{R}}^{1}_{k-k'+1,l},\ \ \ 0 \leq n < k-k'+1 \\
    \label{tocspecialpoint1}\overline{\mc{R}}^{l'} &\times_{(n+1)'} \overline{\mc{R}}^{1}_{k,l-l'+1},\ \ \ 0 \leq n' < l - l'+1 \\
    \label{tocspecialpoint2}\overline{\mc{R}}^{k'+l'+1} &\times_{0} \overline{\mc{R}}^{1}_{k-k',l-l'} \\
    \label{tocspecialpoint3}\overline{\mc{R}}^{l'+k'+1} &\times_{0'} \overline{\mc{R}}^{1}_{k-k',l-l'}.
\end{align}
Here the notation $\times_j$ indicates that one glues the distinguished output
of the first factor to the input $z_j$, and the notation $\times_{j'}$
indicates that one glues the distinguished output of the first factor to the
input $z_j'$.  Moreover, in (\ref{tocspecialpoint2}), after gluing the output
of the first disc to to the first special point $z_0$, the $k'+1$st input
becomes the new special point $z_0$. Similarly in (\ref{tocspecialpoint3}),
after gluing the output of the first stable disc to the second special point
$z_0'$, the $l'+1$st input becomes the new special point $z_0'$.
Thinking of $\mc{R}^{1}_{k,l}$ as a submanifold of open-closed strings, we
obtain, given a compatible Lagrangian labeling $\{L_0, \ldots, L_k, L_0',
\ldots, L_l'\}$ asymptotic input chords $\{x_0, x_1, \ldots, x_k, x_0', x_1',
\ldots,
x_l'\}$ and output orbit $y$, Floer theoretic moduli spaces
\begin{equation}
    \overline{\mc{R}}^{1}_{k,l}(y; x_0, x_1, \ldots, x_k, x_0', x_1', \ldots, x_l')
\end{equation}
of dimension
\begin{equation}
    k+l -n + \deg(y)- \deg(x_0)- \deg(x_0') - \sum_{i=1}^k \deg(x_i) - \sum_{j=1}^l \deg(x_j').
\end{equation}
Here, $L_k, L_0'$ are adjacent to the second special point $z_0'$ and $L_l',
L_0$ are adjacent to $z_0$ (with corresponding inputs $x_0', x_0$).  Using the sign twisting datum
\begin{equation}
    \vec{t}_{_2 \oc_{k,l}} = (1, 2, \ldots, k+1, k+3, k+4, \ldots, k+2+l)
\end{equation}
with respect to the ordering of inputs $(z_0, \ldots, z_k, z_0', \ldots,
z_l')$, define associated Floer-theoretic operations
\begin{equation}
    _2 \oc_{k,l} := (-1)^{\vec{t}_{_2 \oc}} \mathbf{F}_{\overline{\mc{R}}^1_{k,l}} : (\w_{\Delta}
    \otimes \w^{\otimes l} \otimes \w_{\Delta} \otimes \w^{\otimes k})^{diag} \lra
    CH^*(M).
\end{equation}
The two-pointed open-closed map is defined to be the sum of these operations:
\begin{equation}
    _2 \oc = \sum_{k,l} {_2} \oc_{k,l} : (\w_{\Delta} \otimes T\w \otimes \w_{\Delta} \otimes T\w)^{diag} \lra CH^*(M).
\end{equation}
With respect to the grading on the 2-pointed Hochschild complex, $_2 \oc$ is
once more a map of degree $n$. By analyzing the boundary of the one-dimensional
components of $\overline{\mc{R}}^1_{k,l}$, seeing that the relevant
boundary behavior is governed by the codimension-1 boundary of the abstract
moduli space $\overline{\mc{R}}^1_{k,l}$ described from
(\ref{tocspecialpoint0})-(\ref{tocspecialpoint3}) and strip-breaking, and
performing a sign verification in Appendix \ref{orientationsection}, we conclude
that
\begin{cor} \label{2occhain}
    The map $_2 \oc: {_2} \r{CC}_*(\w,\w) \lra CH^*(M)$ is a chain map.
\end{cor}

\begin{defn}
    The {\bf two-pointed closed-open moduli space} with $(r,s)$ marked points
    \begin{equation}
        \mc{R}^{1,1}_{r,s}
    \end{equation}
    is the space of discs with one interior positive puncture labeled $y_{in}$,
    one negative boundary puncture $z_{out}$, and $r+s+1$ positive boundary
    punctures, labeled in clockwise order from $z_{out}$ as $z_1, \ldots, z_r,
    z_{fixed}, z_1', \ldots, z_s'$, subject to the following constraint:
    \begin{equation}
        \textrm{up to automorphism, }z_{out},\ z_{fixed},\ \textrm{and }y_{in}\
        \textrm{lie at }-i,\ i\ \textrm{and }0\textrm{
        respectively}.  
    \end{equation}
\end{defn}

The boundary strata of the Deligne-Mumford compactification
$\overline{\mc{R}}^{1,1}_{r,s}$ is covered by the natural inclusions of the
following products:
\begin{align}
    \label{tcospecialpoint0}\overline{\mc{R}}^{r'} &\times_{n+1} \overline{\mc{R}}^{1,1}_{r-r'+1,s},\ \ \  0 \leq n < r-r'+1\\
    \label{tcospecialpoint1}\overline{\mc{R}}^{s'} &\times_{(m+1)'} \overline{\mc{R}}^{1,1}_{r,s-s'+1},\ \ \ 0 \leq m < s - s' + 1\\
    \label{tcospecialpoint2}\overline{\mc{R}}^{r'+s'+1} &\times_0 \overline{\mc{R}}^{1,1}_{r-r',s-s'}\\
    \label{tcospecialpoint3}\overline{\mc{R}}^{1,1}_{r-a',s-b'} &\times_{a'+1} \overline{\mc{R}}^{a'+b'+1}.
\end{align}
Here in (\ref{tcospecialpoint2}), the output of the stable disc is glued to
the special input $z_{fixed}$ with the $r'+1$st point becoming the new
distinguished $z_{fixed}$. Similarly, in (\ref{tcospecialpoint3}), the output
of the two-pointed closed-open disc $z_{out}$ is glued to the $a+1$st input of
the stable disc.

Thinking of $\mc{R}^{1,1}_{r,s}$ as a submanifold of
open-closed strings, we obtain, given a compatible Lagrangian labeling 
\begin{equation}
    \{L_0,
    \ldots, L_r, L_0', \ldots, L_s'\}
\end{equation}
and input chords $\{x_1, \ldots, x_r, x_{fixed}, 
x_1', \ldots, x_s'\}$ input orbit $y$, and output chord $x_{out}$, a moduli space 
\begin{equation}
    \overline{\mc{R}}^{1,1}_{r,s}(x_{out};y, x_1, \ldots, x_r, x_{fixed}, x_1', \ldots, x_s')
\end{equation}
of dimension
\begin{equation}
    r+s + \deg(x_{out})- \deg(y)- \deg(x_{fixed}) - \sum_{i=1}^r \deg(x_i) - \sum_{j=1}^s \deg(x_j').
\end{equation}
Here, $L_r, L_0'$ are adjacent to the second special output $z_{out}$ and $L_s',
L_0$ are adjacent to $z_{fixed}$ (with corresponding asymptotic conditions $x_{out}, x_{fixed}$).  We also
obtain associated Floer-theoretic operations
\begin{equation}
   \mathbf{F}_{\overline{\mc{R}}^{1,1}_{r,s}} : CH^*(M) \otimes (\w^{\otimes s} \otimes \w_{\Delta} \otimes \w^{\otimes r}) \lra \w_{\Delta}.
\end{equation}
Now, define
\begin{equation}
    _2 \co_{r,s}: CH^*(M) \lra \hom_{Vect}(\w^{\otimes s} \otimes \w_{\Delta} \otimes \w^{\otimes r}, \w_{\Delta})
\end{equation}
as
\begin{equation}
    _2 \co_{r,s}(y)(y_s, \ldots, y_1, \mathbf{b}, x_s, \ldots, x_1) := 
    (-1)^{\vec{t}_{_2 \co_{r,s}}} \mathbf{F}_{\overline{\mc{R}}^{1,1}_{r,s}}(y, y_s, \ldots, x_1, \mathbf{b}, x_r, \ldots, x_1)
\end{equation}
where $\vec{t}_{_2 \co_{r,s}}$ is the sign twisting datum
\begin{equation}
    \vec{t}_{_2 \co_{r,s}} := (-1, 0, \ldots, r-1, r+1, r+2, \ldots, r+s + 1)
\end{equation}
with respect to the input ordering $(y_{in}, z_1, \ldots, z_r, z_{fixed}, z_1', \ldots, z_s')$.
Define the two-pointed closed-open map to be the sum of these operations
\begin{equation}
    {_2}\co = \sum_{r,s} {_2} \co_{r,s} : CH^*(M) \lra \hom_{\w\!-\!\w}(\w_{\Delta},\w_{\Delta})
\end{equation}
With respect to the grading on the 2-pointed Hochschild co-chain complex, $_2 \co$ is
once more a map of degree $0$. An analysis of the boundary of the one-dimensional
components of $\overline{\mc{R}}^{1,1}_{k,l}$ coming from strip-breaking and the
codimension-1 boundary of the abstract
moduli space $\overline{\mc{R}}^{1,1}_{k,l}$ described in
(\ref{tcospecialpoint0})-(\ref{tcospecialpoint3}), along with a sign
verification discussed in Appendix \ref{orientationsection}, we conclude that
\begin{cor}
    The map $_2 \co: CH^*(M) \lra {_2}\r{CC}^*(\w,\w)$ is a chain map.
\end{cor}

We remark that the quasi-isomorphisms of chain complexes
\begin{equation}\nonumber
    \begin{split}
        \mathbf{\Phi}: {_2}\r{CC}_*(\w,\w) &\stackrel{\sim}{\lra} \r{CC}_*(\w,\w) \\
        \mathbf{\Psi}: \r{CC}^*(\w,\w) &\stackrel{\sim}{\lra} {_2}\r{CC}^*(\w,\w)
    \end{split}
\end{equation}
defined in (\ref{twopointedhomologyquasi}) and
(\ref{twopointedcohomologyquasi}) induce identifications of the two-pointed
open-closed maps with the usual open-closed maps. The precise statement is:
\begin{prop}\label{twopointhomotopy1}
    There are homotopy-commutative diagrams
    \begin{equation}\label{homotopyhomology}
        \xymatrix{ {_2}\r{CC}_*(\w,\w) \ar[dr]^{ {_2}\oc} \ar[d]^{\mathbf{\Phi}} & \\
        \r{CC}_*(\w,\w) \ar[r]^{\oc\ \ }& SH^*(M)}
    \end{equation}
    and
    \begin{equation}\label{homotopycohomology}
        \xymatrix{ SH^*(M) \ar[dr]^{ {_2}\co} \ar[d]^{\co} & \\
        \r{CC}^*(\w,\w) \ar[r]^{\mathbf{\Psi} \ }& {_2}\r{CC}^*(\w,\w)}.
    \end{equation}
\end{prop}
\begin{cor}
    The maps $({_2}\oc, {_2}\co)$ are equal in homology to the
    maps $(\oc, \co)$.
\end{cor}
\noindent The homotopies (\ref{homotopyhomology}) and
(\ref{homotopycohomology}) are controlled by the following moduli spaces.
\begin{defn}
    The moduli space
    \begin{equation}
        \mc{S}^1_{k,l}
    \end{equation}
    is the space of discs with one interior negative puncture labeled $y_{out}$
    and $k+l+2$ positive boundary punctures, labeled in clockwise order
    $z_0,z_1,\ldots,z_k, z_0', z_1', \ldots, z_l'$, such that:
    \begin{equation}
        \begin{split}
            \textrm{up to automorphism, }& z_0,\ z_0',\ \textrm{and }y_{out}\ \textrm{are constrained to lie at }-i,\ e^{-i\frac{\pi}{2}(1 -2  t)} \ \textrm{and }0\\
        &\textrm{ respectively, for\ some\ }t \in (0,1).
    \end{split}
    \end{equation}
\end{defn}
The space $\mc{S}^1_{k,l}$ fibers over the open interval $(0,1)$, by the value
of $t$ above.  Compactifying, we see that $\overline{\mc{S}}^1_{k,l}$ submerses
over $[0,1]$ and its codimension 1 boundary strata are covered by the images of
the natural inclusions of the following products (some corresponding to the
limits $t = 0, 1$ and some occurring over the entire interval):
\begin{align}
    \label{tochomotopyspecialpoint0}\overline{\mc{R}}^{k+2+l'+l''} \times_{l-l'-l''+1} \overline{\mc{R}}^{1}_{l-l'-l'' +1} &\textrm{ }(t = 0) \\
     \label{tochomotopyspecialpoint1}  \overline{\mc{R}}^{1}_{k,l}&\textrm{ }(t = 1) \\
       \label{tochomotopyspecialpoint2} \overline{\mc{R}}^{k'} \times_{n+1} \overline{\mc{S}}^{1}_{k-k'+1,l}&\ \ \ 0 \leq n < k-k'+1 \\
\label{tochomotopyspecialpoint3} \overline{\mc{R}}^{l'} \times_{(m+1)'} \overline{\mc{S}}^{1}_{k,l-l'+1}& \ \ \ 0 \leq m < l-l'+1 \\
    \label{tochomotopyspecialpoint4}\overline{\mc{R}}^{k'+l'+1} \times_0 \overline{\mc{S}}^{1}_{k-k',l-l'}& \\
    \label{tochomotopyspecialpoint5}\overline{\mc{R}}^{l'+k'+1} \times_0' \overline{\mc{S}}^{1}_{k-k',l-l'}&
\end{align}
where \begin{itemize}
    \item in (\ref{tochomotopyspecialpoint0}), the $k+1$st and $k+l'+2$nd marked points of the stable disc become the special points $z_0$ and $z_0'$ after gluing; and
    \item the products
        (\ref{tochomotopyspecialpoint2})-(\ref{tochomotopyspecialpoint2}) are
        as in (\ref{tocspecialpoint0})-(\ref{tocspecialpoint3}).  
\end{itemize}
Fixing sign twisting datum
\begin{equation}
    \vec{t}_{_2 \oc\ra \oc, k, l} := (1, \ldots, k+1, k+3, \ldots, k+l+2),
\end{equation}
we obtain an associated operation
\begin{equation}
    \mc{H} := \bigoplus_{k,l} (-1)^{\vec{t}_{_2\oc\ra \oc,k,l}}\mathbf{F}_{\overline{\mc{S}}^1_{k,l}} : {_2}\r{CC}_*(\w,\w) \lra CH^*(M)
\end{equation}
of degree $n-1$. By analyzing the boundary of the 1-dimensional Floer moduli
spaces associated to $\mc{S}^1_{k,l}$ coming from
(\ref{tochomotopyspecialpoint0})-(\ref{tochomotopyspecialpoint5}) and
strip-breaking, as well as verifying signs (see Appendix \ref{orientationsection}), we see that
\begin{equation}
    d_{CH} \circ \mc{H} \pm \mc{H} \circ d_{{_2}\r{CC}} = \oc \circ \mathbf{\Phi} - {_2}\oc,
\end{equation}
verifying the first homotopy commutative diagram.

\begin{figure}[h]
    \caption{A schematic of $\mc{S}^1_{k,l}$ and its $t = 0,1$ degenerations. }
    \centering
    \includegraphics[scale=0.7]{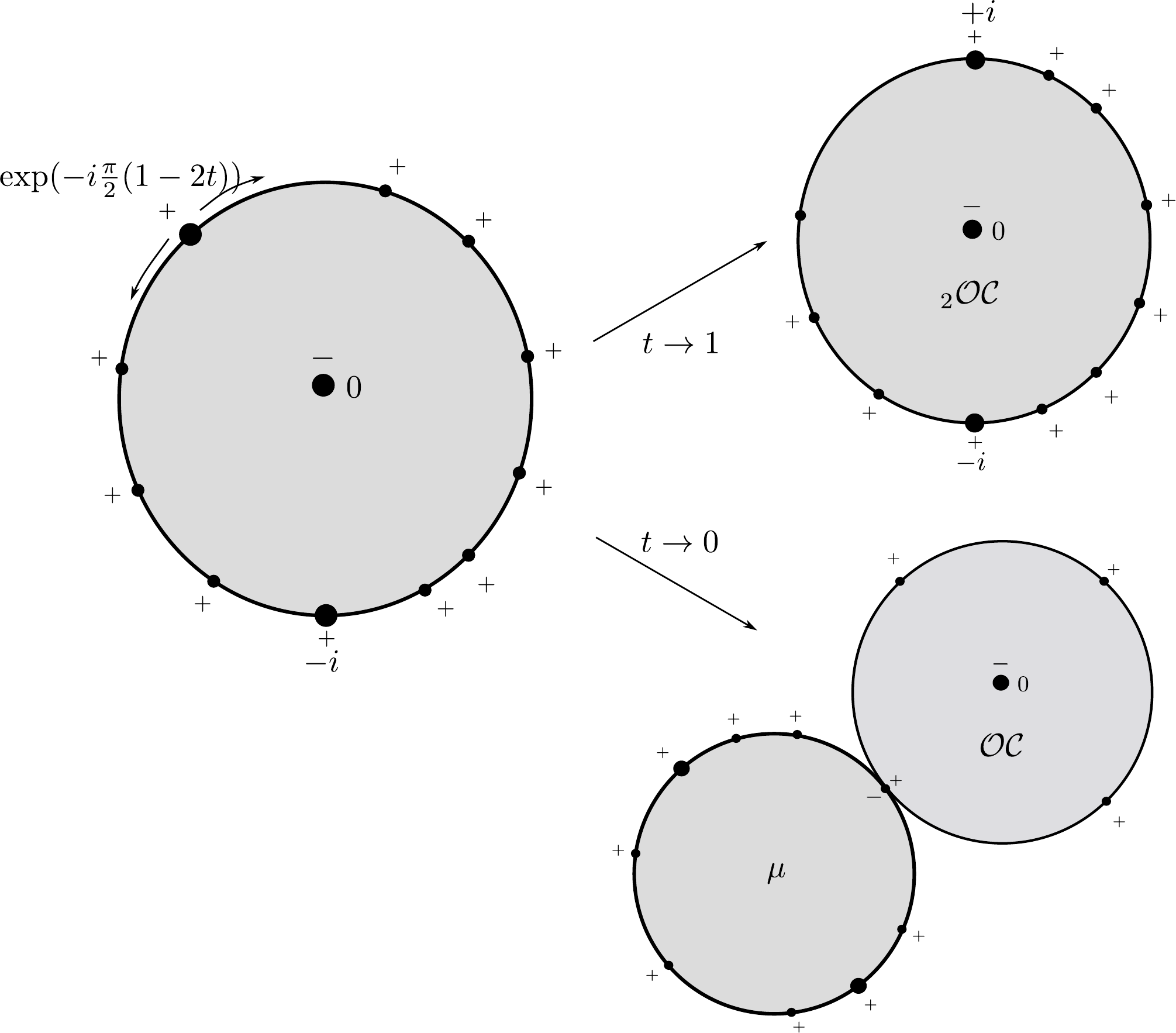}
\end{figure}

\begin{defn}
    The moduli space
    \begin{equation}
        \mc{S}^{1,1}_{r,s}
    \end{equation}
    is the space of discs with one interior positive puncture labeled $y_{in}$,
    one negative boundary puncture $z_{out}$, and $r+s+1$ positive boundary
    punctures, labeled in clockwise order from $z_{out}$ as $z_1, \ldots, z_r,
    z_{fixed}, z_1', \ldots, z_s'$, subject to the following constraint:
    \begin{equation}
        \begin{split}
        \textrm{up to automorphism, }& z_{out},\ z_{fixed},\ \textrm{and }y_{in}\
        \textrm{lie at }-i,\ e^{i(-\pi/2 + \pi\cdot t)}\ \textrm{and }0\textrm{
        respectively},\\
        &\textrm{for\ some\ }t \in (0,1).
    \end{split}
    \end{equation}
\end{defn}
The space $\mc{S}^{1,1}_{r,s}$ again fibers over the open interval $(0,1)$, by
the value of $t$ above.  Compactifying, we see that the codimension 1 boundary
strata of the space $\overline{\mc{S}}^{1,1}_{r,s}$ submerses over $[0,1]$ and
is covered by the images of the natural inclusions of the following products
(some
corresponding to the limits $t = 0, 1$ and some occurring over the entire
interval): 
\begin{align}
    \label{tcohomotopyspecialpoint0} \overline{\mc{R}}^{1,1}_{s-s'-s'' +1} \times_{r+2+s'} \overline{\mc{R}}^{r+2+s'+s''} &\textrm{ }(t = 0) \\
    \label{tcohomotopyspecialpoint1} \overline{\mc{R}}^{1,1}_{r,s}&\textrm{ }(t = 1) \\
    \label{tcohomotopyspecialpoint2}\overline{\mc{R}}^{r'} \times_{n+1} \overline{\mc{S}}^{1,1}_{r-r'+1,s}&\ \ \ 0 \leq n < r-r'+1\\
    \label{tcohomotopyspecialpoint3}\overline{\mc{R}}^{s'} \times_{(m+1)'} \overline{\mc{S}}^{1,1}_{r,s-s'+1}&\ \ \ 0 \leq m < s - s' + 1\\
    \label{tcohomotopyspecialpoint4}\overline{\mc{R}}^{r'+s'+1} \times_0 \overline{\mc{S}}^{1,1}_{r-r',s-s'}&\\
    \label{tcohomotopyspecialpoint5}\overline{\mc{S}}^{1,1}_{r-a',s-b'} \times_{a'+1} \overline{\mc{R}}^{a'+b'+1}&
\end{align}
Using sign twisting datum
\begin{equation}
    \vec{t}_{_2 \co \ra \co, r, s} = (-1, 0,\ldots, r-1, r+1, r+2, \ldots, r+s+1),
\end{equation}
the associated operation
\begin{equation}
    \mc{G} := \bigoplus_{r,s} (-1)^{\vec{t}_{_2 \co \ra \co, r, s}}\mathbf{F}_{\overline{\mc{S}}^{1,1}_{r,s}} : CH^*(M) \lra {_2}\r{CC}^*(\w,\w)
\end{equation}
has degree $-1$. By analyzing the boundary of the 1-dimensional Floer moduli
spaces associated to $\mc{S}^{1,1}_{r,s}$ coming from
(\ref{tochomotopyspecialpoint0})-(\ref{tochomotopyspecialpoint5}) and
strip-breaking, and verifying signs (Appendix \ref{orientationsection}), we see that
\begin{equation}
    \mc{G}\circ d_{CH} \pm d_{{_2}\r{CC}^*} \circ \mc{G} = \mathbf{\Psi} \circ \co - {_2}\co,
\end{equation}
verifying the second homotopy commutative diagram.

\section{Unstable operations} \label{unstableoperations}
Some of the operations we would like to consider are parametrized not by
underlying moduli spaces but instead a single surface.

\subsection{Strips}\label{stripidentity}
Let $\Sigma_1$ denote a disc with two boundary punctures removed, thought of as
a strip $ (\infty,\infty) \times [0,1]$.  We have already defined a
Floer-theoretic operation using $\Sigma_1$, namely the differential $\mu^1$.
Let us recast this operation in terms of Floer data. 
\begin{defn}
    A {\bf Floer datum} for $\Sigma_1$ can be thought of a Floer datum in the sense
    of Definition \ref{floeropenclosed} with the following additional constraints: 
    \begin{itemize}
        \item The strip-like ends $\e_+$ and $\e_-$ are given by inclusion of
            the positive and negative semi-infinite strips respectively.
        \item The incoming and outgoing weights are both equal to a single number $w$.
        \item the one-form $\a$ is $w\cdot dt$ everywhere, as is the rescaling map $a_S$.
        \item the Hamiltonian $H_{\Sigma_1}$ is equal everywhere to $\frac{H \circ \psi^{w}}{w^2}$
        \item the almost complex structure $J_{\Sigma_1}$ is equal everywhere to $(\psi^w)^*J_t$.
    \end{itemize}
\end{defn}
\begin{rem}
    Upon fixing $H$ and $J_t$, the Floer datum above only depends on $w$.
    Moreover, the data defined by any two different weights $w$ and $w'$ are
    conformally equivalent.  
\end{rem}
Fix the Floer datum for $\Sigma_1$ with $w = 1$. This induces, for Lagrangians
$L_0, L_1 \in \ob \w$, and chords $x_0, x_1 \in \chi(L_0, L_1)$, a space of
maps \[\Sigma_1(x_0; x_1)\] satisfying the usual asymptotic and boundary
conditions, and solving the relevant version of Floer's equation for the Floer
datum.  
Instead of dividing by $\R$-translation, we can also consider the
operation induced by the space  $\Sigma_1(x_0; x_1)$ itself, which has dimension
\[
\deg(x_0) - \deg(x_1)
\]
We get a map
\begin{equation} \label{imap}
    I: CW^*(L_0, L_1) \lra CW^*(L_0, L_1)
\end{equation}
defined by 
\begin{equation}
    I([x_0]) := \sum_{x_1: \deg(x_1) = \deg(x_0)} \sum_{u \in \Sigma_1(x_0; x_1)} (-1)^{\deg(x_0)}(\Sigma_1)_u([x_0])
\end{equation}
where $(\Sigma_1)_u: o_{x_0} \ra o_{x_1}$ is the induced map on orientation
lines (using Lemma \ref{orientationlem}).  
\begin{prop}
    \label{identity}
    $I$ is the identity map.
\end{prop}
\begin{proof}
    If $u$ is any non-constant strip mapping into $M$, composing with the $\R$
    action on $\Sigma_1$ gives other maps into $M$ solving the same equation by
    $\R$-invariance of our Floer data; hence $u$ is not rigid. Therefore,
    dimension 0 strips must all be constant, concluding the proof.  
\end{proof}

\subsection{The unit}\label{theunitelement}
Let $\Sigma_0$ denote a once-punctured disc thought of as the upper half plane 
${\mathbb H} \subset \C$ with puncture at $\infty$, thought of as a negative puncture.
\begin{defn}
    A {\bf Floer datum} for $\Sigma_0$ is a Floer datum in the sense of
    Definition \ref{floeropenclosed}. Concretely, this consists of
    \begin{itemize}
        \item A {\bf strip-like end} $\e: (-\infty,0] \times [0,1] \ra \Sigma_0$ around the puncture
        \item A {\bf choice of weight} $w \in [1,\infty)$
        \item A {\bf rescaling map} $a_{\Sigma_0}: \Sigma_0 \ra [1,+\infty)$ equal to $w$ on the strip-like end
        \item {\bf Hamiltonian perturbation}: A map $H_{\Sigma_0}: \Sigma_0 \ra
            \mc{H}(M)$ such that $\e^* H_{\Sigma_0} = \frac{H \circ \psi^w}{w^2}$.
        \item {\bf basic 1-form}: a sub-closed 1-form $\a_{\Sigma_0}$, whose
            restriction to $\partial \Sigma_0$ vanishes, such that $\e^*
            \a_{\Sigma_0} = w\cdot dt$.
        \item {\bf Almost complex structure}: A map $J_{\Sigma_0}: \Sigma_0 \ra
            \mc{J}(M)$ such that $J_{\Sigma_0} \in \mc{J}_{a_{\Sigma_0}}$ and
            $\e^* J_{\Sigma_0} = (\psi^w)^*J_t$.
    \end{itemize}
\end{defn}
\begin{rem}
    Note by Stokes' theorem that in the definition above, the one form
    $\a_{\Sigma_0}$ cannot be closed everywhere, i.e. there are points with $d
    \a_{\Sigma_0} < 0$.  
\end{rem}
\begin{rem}
    Up to conformal equivalence, it suffices to take a Floer datum for
    $\Sigma_0$ with weight $w = 1$.
\end{rem}
Let $L$ be an object of $\mc{W}$, and consider a chord $x_0 \in \chi(L,L)$.
Fixing a Floer datum for $\Sigma_0$, write 
\begin{equation}\Sigma_0(x_0;)\end{equation} 
for the space of
maps $u: \Sigma_0 \ra E$ satisfying boundary and asymptotic conditions
\begin{equation} 
    \begin{cases}
        u(z) \in \psi^{a_S(z)} L & z \in \bd \Sigma_0 \\
        \lim_{s \ra -\infty} u \circ \e(s, \cdot) = x
    \end{cases}
\end{equation}
and differential equation
\begin{equation}
    (du - X_{\Sigma_0} \otimes \a_{\Sigma_0})^{0,1} = 0
\end{equation}
with respect to $J_{\Sigma_0}$.
\begin{lem}
    The space of maps $\Sigma_0(x_0;)$ is compact and forms a manifold of
    dimension $\deg(x_0)$.  
\end{lem}
Thus, we can define the element $e_L \in CW^*(L,L)$ to be the sum
\begin{equation}
    e_L := \sum_{\deg(x_0) = 0} \sum_{u \in \Sigma_0(x_0;)} (\Sigma_0)_u(1)
\end{equation}
where $(\Sigma_0)_u: \R \ra o_{x_0}$ is the induced map on orientation lines
(using Lemma \ref{orientationlem})

\begin{prop}
    The resulting elements $e_{L_i} \in CW^*(L_i,L_i)$ give the identity
    element on homology.
\end{prop}
\begin{proof}
    This is a classical result, but we briefly sketch a proof for completeness;
    see e.g. \cite{Ritter:2010nx} for more details. One first checks via
    analyzing the boundary of the one dimensional moduli space of
    $\Sigma_0(x_0;)$ that $d(e_L) = 0$, so $e_L$ descends to homology.
Then, one needs to check that, up to sign
\begin{align} \label{mu2id}
         \mu^2([x],[e_{L_i}]) &= [x]  \\
         \mu^2([e_{L_i}],[x]) &= [x] \label{mu2idleft},
 \end{align}
where the brackets denote homology classes.
Since the arguments to establish (\ref{mu2id}) and (\ref{mu2idleft}) are
identical, it suffices to construct a geometric chain homotopy between the maps
\begin{equation}
        \mu^2(\cdot,e_{L_i})
    \end{equation}
    and 
    \begin{equation}
        I(\cdot)
    \end{equation}
    where $I$ is as in (\ref{imap}), which can be described as follows. Let
    $\Sigma_2$ be a disc with two incoming boundary marked points $x_1$, $x_2$,
    and one outgoing point $x_{out}$, with
    $x_1$ marked as ``forgotten'' (see Section \ref{homotopyunits} for how to
    do this). Fix a strip-like end around $x_1$ and consider a one parameter
    family of Floer data on $\Sigma_2$ with $x_2$ and
    $x_{out}$ removed, over the interval $[0,1)_t$, such that
    \begin{itemize}
        \item at $t=0$, the Floer data agrees with the translation-invariant one on the $\Sigma_1$ arising by forgetting $x_1$,
        \item for general $t$, the Floer data is modeled on the connect sum of
            the Floer data for $\mu_2$ with the Floer data for $e_{L_i}$ over
            the strip-like ends at the output of $\Sigma_0$ with the one around
            $x_1$, with connect sum length approaching $\infty$ as $t \ra 1$.
    \end{itemize}
    Compactifying and looking at the associated Floer operation, one obtains a
    chain homotopy between the degenerate curve corresponding to
    $\mu(\cdot,e_{L_i})$ and the operation $I(\cdot)$. Finally, one performs a
    sign verification analogous to those in Appendix \ref{orientationsection}.
\end{proof}

\section{Operations from glued pairs of discs}\label{pairs}
In this section, we define a broad class of abstract moduli spaces and their
associated Floer theoretic operations, corresponding to a pair of discs glued
together along some boundary components. This class will arise when defining
operations in the product $M^- \times M$, and in setting up the theory of
quilts. 

\subsection{Connect sums}\label{connectsumsubsection}
We give a short aside on the notation we use for connect sums.  Recall first
the notion of a boundary connect sum between two Riemann surfaces with
boundary, a notion already implicit in our constructions of Deligne-Mumford
compactifications of moduli spaces.
\begin{defn}
    Let $\Sigma_1, \Sigma_2$ be two Riemann surfaces with boundary, with marked
    points $z_1 \in \bd \Sigma_1$, $z_2 \in \bd \Sigma_2$ removed. Let $\e_1:
    Z_+ \ra \Sigma_1$ be a positive strip-like end for $z_1$ and $\e_2: Z_-:
    \ra \Sigma_2$ a negative strip-like end for $z_2$ and let 
$\lambda = \frac{\log \rho}{1 + \log \rho} \in (0,1)$ (correspondingly $\rho
\in [1,\infty)$). The {\bf $\lambda$-connect sum} 
\begin{equation}
    \Sigma_1\#_{(\e_1,z_1),(\e_2,z_2)}^\lambda \Sigma_2
\end{equation}
    is
    \begin{equation}
    \left(\Sigma_1 - \e_1( [\rho,\infty) \times [0,1])\right) 
    \cup_{\varphi_\rho} \left(\Sigma_2 - \e_2( (-\infty,-\rho] \times
    [0,1])\right) \end{equation}
    where 
    \begin{equation}
        \varphi_\rho: \e_1( (0,\rho) \times [0,1]) \ra \e_2( (-\rho,0) \times
        [0,1]).  
    \end{equation}
    is the composition
    \begin{equation}
    \e_1((0,\rho)\times [0,1] ) \stackrel{\e_1^{-1}}{\ra}  (0,\rho) \times [0,1]
    \stackrel{(\cdot - \rho,id)}{\ra} (-\rho,0) \times [0,1] 
    \stackrel{\e_2}{\ra} \e_2( (-\rho,0) \times [0,1]).  
\end{equation}
    We will often write $\Sigma_1 \#^\lambda_{z_1,z_2} \Sigma_2$ when the
    choice of strip-like ends is implicit.
\end{defn}
\begin{defn}
    In the notation of above, the {\bf associated thin part} of a
    $\lambda$-connect sum $\Sigma_1 \#^\lambda_{z_1,z_2} \Sigma_2$ is the
    finite strip parametrization 
    \begin{equation}
    \e_1: [0,\rho] \times [0,1] \ra \Sigma_1 \#^\lambda_{z_1,z_2} \Sigma_2.
\end{equation}
    The {\bf associated thick parts} of this connect sum are the regions
    $\Sigma_1 - \e_1([0,\infty)\times [0,1])$, $\Sigma_2 -
    \e_2([-\infty,0]\times[0,1])$ respectively, thought of as living in the
    connect sum.  
\end{defn}
\noindent The notion of $\lambda$ connect sum extends continuously to the nodal case
$\lambda = 1$.

\subsection{Pairs of discs}
\begin{defn} 
    The {\bf moduli space of pairs of discs with $(k,l)$ marked points}, denoted
\begin{equation}
\mc{R}_{k,l}
\end{equation}
    is the moduli space of pairs of discs with $k$ and $l$ positive marked
    points and one negative marked point each in the same position, modulo {\it
    simultaneous automorphisms}.
\end{defn}
\begin{rem}
    This definition is {\bf not} identical to the product of associahedra
    $\mc{R}^k \times \mc{R}^l$. The latter space is a further quotient of the
    former space by automorphisms of the right or left disc, at least when both
    $k$ and $l$ are in the stable range. Operations at the level of the moduli
    space $\mc{R}_{k,l}$ will arise via quilted strips and homotopy units.  
\end{rem}
\begin{rem} 
    To construct moduli spaces, we require a pair of discs with $(k,l)$ marked
    points to be {\bf stable}: one of $k$ or $l$ must be at least two.
\end{rem}
The Stasheff associahedron embeds in $\mc{R}_{k,l}$ in several ways. 
There is the {\bf diagonal embedding}
\begin{equation}
    \mc{R}^d \stackrel{\Delta_d}{\longhookrightarrow} \mc{R}_{d,d}.
\end{equation}
which is self-explanatory.
When $l = 1$ and $k \geq 2$, there is a one-sided embedding
\begin{equation}\label{onesided1} 
    \mc{R}^k \stackrel{\mc{I}_k}{\longhookrightarrow} \mc{R}_{k,1},
\end{equation} 
where 
$\mc{I}_k =
(id,For_{k-1})$ is the pair of maps corresponding to inclusion and
forgetting the first $k-1$ boundary marked points respectively. Since the
right factor in the image has only one incoming marked
point, we call $\mc{I}_k$ the {\bf right semi-stable embedding}.
Similarly when $k=1$ and $l \geq 2$, forgetting $l-1$
marked points and inclusion gives us the {\bf left semi-stable
embedding} 
\begin{equation} 
    \label{onesided2}
    \mc{R}^k \stackrel{\mc{J}_k}\longhookrightarrow \mc{R}_{1,k}.
\end{equation} 
When $l=0$ or $k=0$, there are also equivalences
    \begin{equation}\label{ghostembeddings}
        \begin{split}
            \mc{R}^k &\stackrel{\sim}{\longhookrightarrow} \mc{R}_{k,0}\\
            \mc{R}^l &\stackrel{\sim}{\longhookrightarrow} \mc{R}_{0,l},
        \end{split}
    \end{equation}
which we call the {\bf right} and {\bf left} {\bf ghost embeddings}
respectively, corresponding to the fact that the right or left
component of $\mc{R}_{k,l}$ is a ghost disc.  In fact, we will never
consider operations with either $k$ or $l$ equal to zero, but these
equivalences help us explain appearances of ordinary associahedra in the
compactification $\overline{\mc{R}}_{k,l}$. Henceforth, let us restrict to $k,l
\geq 1$ and one of $k,l \geq 2$.

The open moduli space $\mc{R}_{k,l}$ admits a stratification by {\it coincident
points} between factors, which we will find useful to explicitly describe. 
\begin{defn}\label{coincidentpointstratification}
    A {\bf $(k,l)$-point identification} $\mf{P}$ is a sequence of tuples
    \begin{equation}
        \{(i_1, j_1), \ldots, (i_s,j_s)\} \subset \{1, \ldots, k\}\times \{1, \ldots, l\}
    \end{equation}
    which are strictly increasing, i.e.
    \begin{equation}
        \begin{split}
        i_{r} &< i_{r+1}\\
        j_r &< j_{r+1}
    \end{split}
\end{equation}
The {\bf number of coincidences} of $\mf{P}$ is the size $|\mf{P}|$.
\end{defn}
\begin{defn} 
    Take a representative $(S_1,S_2)$ of a point in $\mc{R}_{k,l}$. A
    boundary input marked point $p_1$ on $S_1$ is said to {\bf coincide} with a
    boundary marked input marked point $p_2$ on $S_2$ if they are at the same
    position when $S_1$ is superimposed upon $S_2$. This notion is independent
    of the representative $(S_1,S_2)$, as we act by simultaneous automorphism.
\end{defn}

\begin{defn}
    The space of {\bf $\mf{P}$-coincident pairs of discs with $(k,l)$ marked
    points} 
    \begin{equation}
        _\mf{P}\mc{R}_{k,l}
    \end{equation}
    is the subspace of $\mc{R}_{k,l}$ where pairs of input marked points on
    each factor specified by $\mf{P}$ are required to coincide, and no other
    input marked points are allowed to coincide. Here the indices in $\mf{P}$
    coincide with the {\bf counter-clockwise ordering} of input marked points
    on each factor.
\end{defn}
\begin{ex}\label{twocolor}
    When $\mf{P} = \emptyset$, $ _\mf{P} \mc{R}_{k,l}$ is the space of pairs of
    discs where none of the inputs are allowed to coincide. This is a
    disconnected space, with connected components determined by the relative
    ordering of the $k$ inputs on the first disc with the $l$ inputs on the
    second disc. The number of connected components is exactly the number of
    $(k,l)$ {\bf shuffles}, i.e. re-orderings of the sequence
    \begin{equation}
        \{a_1, \ldots, a_k, b_1, \ldots, b_l\}
    \end{equation}
    that preserve relative ordering of the $a_i$, and the relative ordering of
    the $b_j$. Given a fixed $(k,l)$ shuffle, the subspace of pairs of discs
    with appropriate relatively ordered inputs is a copy of the $k+l$
    associahedron $\mc{R}^{k+l}$. Equivalently, there is one copy of
    $\mc{R}^{k+l}$ for each 
    $(k,l)$ {\bf two-coloring} of the combined set of points, that is a
    collection of subsets 
    \begin{equation}
        I,J \subset [k+l],\ |I|= k,\ |J| = l,\ I \cup J = [k+l].
    \end{equation}
    where $[k+l] := \{1, \ldots, k+l\}$.
\end{ex}
\begin{ex}
    When $k =l$ and $|\mf{P}|=k$ is maximal, the associated space $_\mf{P}
    \mc{R}_{k,l}$ is just the diagonal associahedron. $\Delta_{k}(\mc{R}^k)$.
\end{ex}
\noindent We often group these spaces  $ _\mf{P} \mc{R}_{k,l}$ by the number of
coincident points.  
\begin{defn}
    The {\bf space of pairs of discs with $(k,l)$ marked points and $i$
    coincident points} is defined to be 
    \begin{equation}
        _i \mc{R}_{k,l} := \coprod_{|\mf{P}|=i} {_{\mf{P}}} \mc{R}_{k,l}.
    \end{equation}
\end{defn}
The closure of a stratum $ _i \mc{R}_{k,l}$ in $\mc{R}_{k,l}$ is
$\coprod_{j\geq i} {_i}\mc{R}_{k,l}$. Moreover, each stratum $ _i
\mc{R}_{k,l}$ can be explicitly described as a union of associahedra.
\begin{defn}
    Fix disjoint subsets $I$,$J$,$K$ of $[d] = \{1, \ldots, d\}$ such that
    \begin{equation}
        I \cup J \cup K = [d].
    \end{equation}
    The {\bf space of $(I,J,K)$ tricolored discs with $d$ inputs}
    \begin{equation}
        _{I,J,K}\mc{R}^d
    \end{equation}
    is exactly the ordinary associahedron, with inputs labeled by the elements
    $\{L,R,LR\}$ according to whether they are in the set $I$,$J$,or $K$. The
    {\bf space of $(i,j,k)$ tricolored discs with $d$ inputs}, where $i+j+k =
    d$, is the disjoint union over all possible tricolorings of cardinality
    $i,j,k$:
    \begin{equation}
        {_{i,j,k}}\mc{R}^d := \coprod_{|I|=i,|J|=j,|K|=k}{_{I,J,K}}\mc{R}^d.
    \end{equation}
\end{defn}
\noindent There is a canonical identification
\begin{equation}
    _i \mc{R}_{k,l} \simeq {_{k-i,l-i,i}}\mc{R}^{k+l-i}
\end{equation}
given as follows:
To a pair of discs $(S_1,S_2)$ with $i$ coincidences, consider the {\it
overlay} (superimposition) of $S_1$ and $S_2$ along with their marked points.
This is a disc with $k+l-i$ input marked points and one output. Color a
marked point $L$ if the marked point came only from $S_1$, $R$ if the marked
point came only from $S_2$, and $LR$ if the marked point came from both
factors.  Similarly, given a tricolored disc, one can reconstruct a pair of
discs with $i$ coincidences by reversing the above procedure.

\begin{figure}[h] 
    \caption{An example of the correspondence between pairs of discs and tricolored discs. \label{pairstricolor1}}
    \centering
    \includegraphics[scale=0.6]{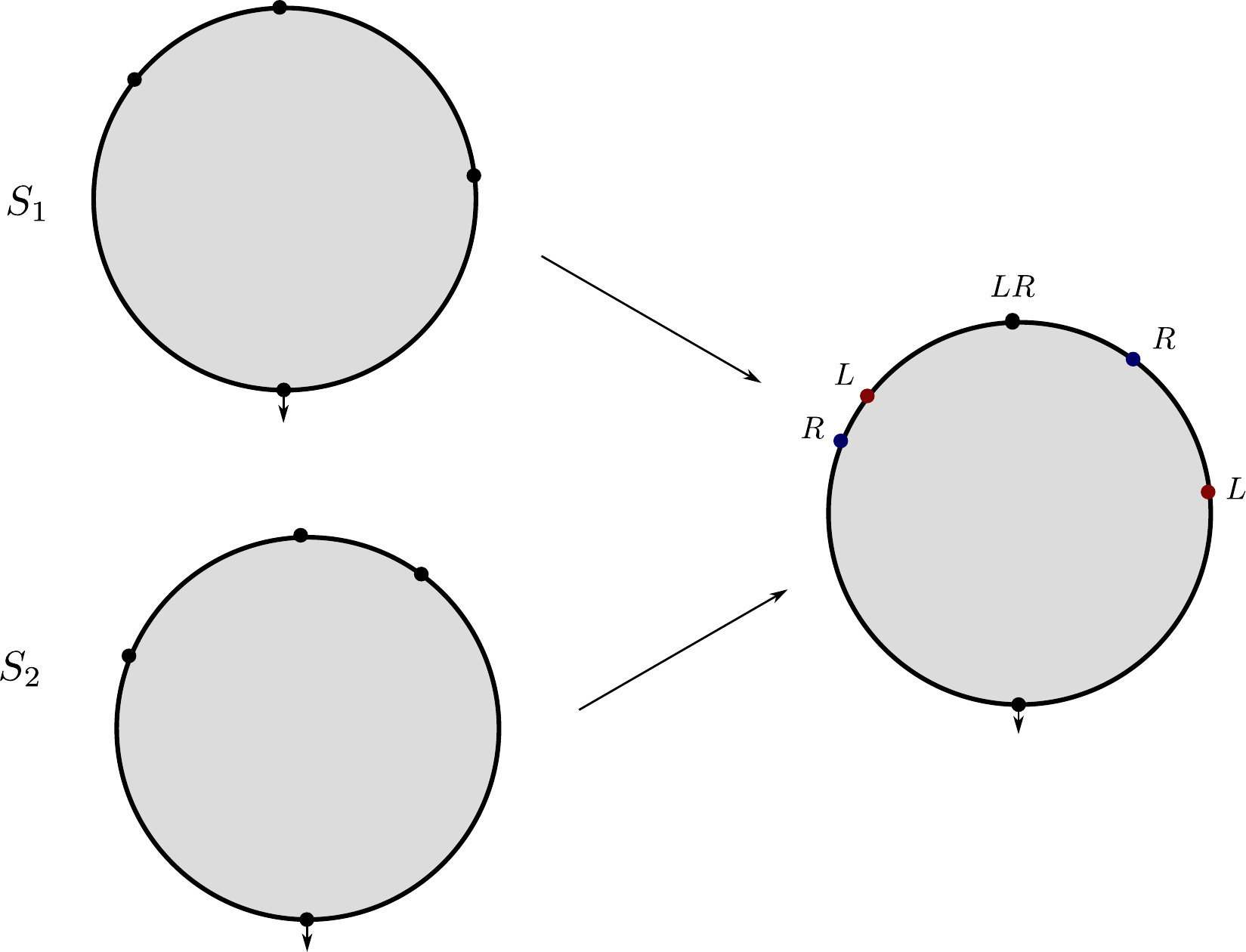}
\end{figure}

The disjoint union of spaces 
\begin{equation}
    \coprod_{i} {_{k-i,l-i,i}}\mc{R}^{k+l-i}
\end{equation}
is set theoretically the same as $\mc{R}_{k,l}$, but has forgotten some of the
topology. Namely, points colored $L$ and $R$ are not allowed to coincide, and
coincident points (those colored $LR$) are not allowed to separate
arbitrarily. 

We now construct a model for the Deligne-Mumford compactification
\begin{equation}\label{compactifiedpairsofdiscs}
    \overline{\mc{R}}_{k,l}.
\end{equation} 
The main idea in our construction is 
to recover this compactification from the Deligne-Mumford compactifications of
the spaces $ _{k-i,l-i,i}
\overline{\mc{R}}^{k+l-i}$ by reconstructing the topology with which points
colored $L$ and $R$ are allowed to coincide, and points colored $LR$ are
allowed to separate.
Note that the compactification of tricolored spaces
\begin{equation}
    _{I,J,K} \overline{\mc{R}}^d
\end{equation}
is exactly the usual Deligne-Mumford compactification, where boundary marked
points on components of nodal discs are colored in a manner induced by the
gluing charts (\ref{gluingcoords}). Internal positive marked points are colored
in the following induced fashion: If the subtree of nodal discs lying above a
given positive marked point is a tree of discs with all $L$ or all $R$ labels,
then color this marked point $L$ or $R$ respectively. If the subtree contains
two out of the three colors ($R$, $L$, $RL$) then color the input $LR$.

Let \begin{equation}D_{LR}^{+1}\end{equation} be a representative of the
    one-point moduli space $ _{\{1\}, \{2\},\emptyset } \mc{R}^2$,
i.e. a disc with inputs labeled $L, R$ in clockwise order.  Similarly, let
\begin{equation}D_{LR}^{-1}\end{equation} be a representative of
$_{\{2\},\{1\},\emptyset} \mc{R}^2$,
i.e. a disc with inputs labeled $R, L$ in clockwise order. Fix a choice of
strip-like ends on $D_{LR}^{\pm 1}$, and let $z_{LR}$ denote the output of each
of these discs. Also, suppose we have fixed a universal and consistent choice
of strip like ends on the various $ _{k-i,l-i,i} \overline{\mc{R}}^{k+l-i}$.

Now, take a (potentially nodal) representative $S$ of a point of $ _{k-i,l-i,i}
\overline{\mc{R}}^{k+l-i}$. Let 
\begin{equation}
    \vec{p} = p_{j_1}, \ldots p_{j_s},\ s\leq i
\end{equation}
be a subset of the points colored $LR$, and $\e_{j_1}, \ldots, \e_{j_s}$ the
associated strip-like ends. Given a
vector
\begin{equation}
    \vec{v} = (v_1, \ldots, v_{j_s}) \in [(-\epsilon,\epsilon)^*]^i,
\end{equation}
where the $*$ means none of the $v_{r}$ are allowed to be zero,
define an element
\begin{equation}
    \mathbf{\Pi}^{\vec{p}}_{\vec{v}} (S) \in 
    {_{k-i+s,l-i+s,i-s}}\overline{\mc{R}}^{k+l-i+s} 
\end{equation}
by the iterated connect sum
\begin{equation}
    \mathbf{\Pi}_{\vec{v}}^{\vec{p}} (S) := S \#_{p_{j_1},z_{LR}}^{1-|v_{1}|} D^{\mathrm{sign}(v_{1})}_{LR} \# \cdots \#_{p_{j_k},z_{LR}}^{1-|v_{j}|} D^{\mathrm{sign}(v_{j})}_{LR}.
\end{equation}
Here $\#^{1-|v_r|}$ is the operation of connect sum with gluing parameter
$1-|v_r|$ in the notation of Section \ref{connectsumsubsection}, which as
$|v_R|$ approaches zero is very close to nodal.  Also, $\mathrm{sign}(v_r)$ is
$+1$ if $v_r$ is positive and $-1$ if $v_r$ is negative. In other words, at
input point $p_{j_r}$ on $S$, we are taking a connect sum with the disc
$D^{\mathrm{sign}(v_r)}_{LR}$ at the point $z_{LR}$, i.e. gluing in two points
labeled $L$ and $R$ in clockwise or counterclockwise order depending on the
sign of $v_r$.

\begin{ex}\label{LRexample}
    It is useful before proceeding to describe the map $\Pi^{\vec{p}}_{\vec{v}}(S)$ in a
simple example. Suppose we are in $ _{1,1,1} \mc{R}^{3}$, the moduli space of
discs with three inputs, one with each color. Pick a non-nodal representative
of an element of this space, without loss of generality one in which the points
are colored $L$, $LR$ and $R$ in clockwise order from the output. Let $p_2$ be
the point colored $LR$ with associated vector $\vec{v} = (v)$, and let us
examine the representative of $\Pi_{(v)}^{p_2}(S)$ for different values of $v \in
(-\epsilon, 0) \cup (0,\epsilon)$. For
$v$ positive, $\Pi^{p_2}_{(v)}(S)$ corresponds to resolving the $LR$ point by
two points, one labeled $L$ and one labeled $R$, with the $L$ point to the left
of the $R$ point, which lives in $ _{\{1,2\},\{3,4\},\emptyset}\mc{R}^4$.
For $v$ negative, we $\Pi^{p_2}_{(v)}(S)$ resolves the $LR$ point in the opposite
direction, giving an element of a different associahedron $ _{\{1,3\},\{2,4\},\emptyset} \mc{R}^4$. 
As $v$ approaches zero from either
direction, the newly created $L$ and $R$ points come together and bubble off,
giving a nodal element in each of these respective associahedra $\Pi^{p_2}_{(0+)}(S)$
and $\Pi^{p_2}_{(0-)}(S)$ with bubble component a $D_{LR}^{+1}$ or $D_{LR}^{-1}$
respectively. To partially recover the topology of $\mc{R}_{2,2}$, we would
like to identify the points $\Pi^{p_2}_{0-}(S)$ and $\Pi^{p_2}_{0+}(S)$, in a
manner preserving the manifold structure near the
identification. See Figure \ref{LRstrataid1}.
\end{ex}

\begin{figure}[h] 
    \caption{The strata we would like to identify. \label{LRstrataid1}}
    \centering
    \includegraphics[scale = 0.5]{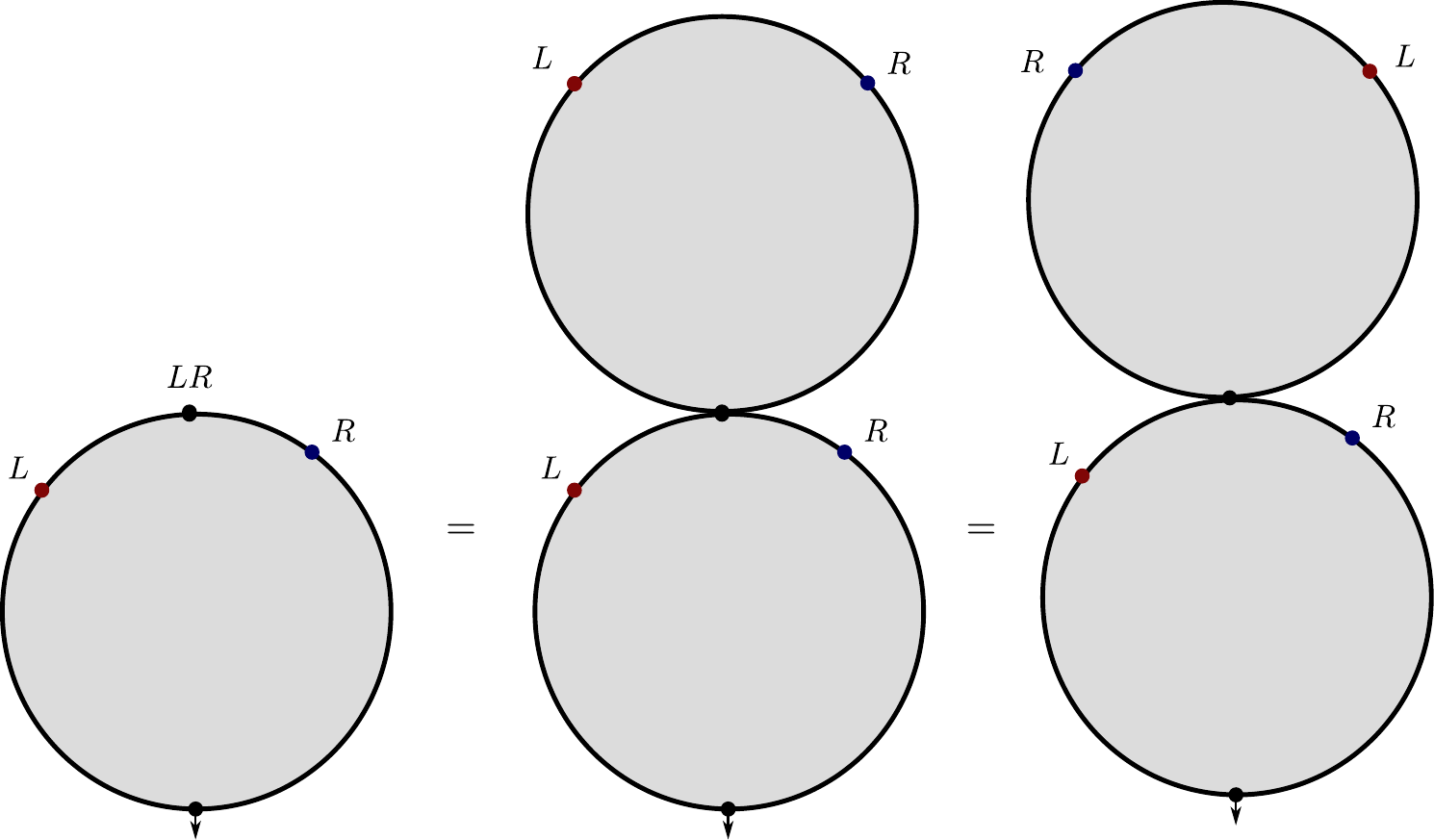}
\end{figure}
The above example illustrates the following properties: $\Pi_{\vec{v}}$ varies
smoothly in $S$ and the parameters $\vec{v} \in [(-\epsilon,\epsilon)^*]^i$,
and there are well defined, but different, nodal limits of the curve
$\Pi_{\vec{v}}(S)$ as components $v_r$ approach 0 from the left or right,
corresponding to gluing on a nodal $D_{LR}^{\pm 1}$ respectively at input
$p_{j_r}$. Indicate these different limits by values $0_+$ and $0_-$
respectively. Then $\Pi_{\vec{v}}(S)$ extends to a map 
\begin{equation}
    \bar{\Pi}_{\vec{v}}
\end{equation}
defined over domain
\begin{equation}
    \vec{v} \in \left[(-\epsilon,0_-] \cup [0_+,\epsilon)\right]^i.
\end{equation}
Define 
\begin{equation}
    _{k-i,l-i,i} (\overline{\mc{R}}^{k+l-i})^*
\end{equation}
to be the locus of the compactifications $ _{I,J,K} \overline{\mc{R}}^{k+l-i}$
with $|I| = k-i$, $|J| = l-i$, $|K| = i$ where there are no leaf bubbles with
one $L$ and one $R$. Put another way, remove the images of
$\bar{\Pi}^{\vec{p}}_{(\cdots, 0_+, \cdots)}$,
$\bar{\Pi}^{\vec{p}}_{(\cdots, 0_-, \cdots)}$ for all relevant domains of
definition of $\Pi$.  Now, define the manifold structure on
$\overline{\mc{R}}_{k,l}$ as follows.  To simplify notation, denote 
\begin{equation}
    {\mathbb I}_{\epsilon}:= (-\epsilon, \epsilon).
\end{equation}
Charts consist of 
\begin{equation} \label{basicpairschart}
    \mc{U} \times {\mathbb I}_\epsilon^i,\ \mc{U} \subset 
    {_{k-i,l-i,i}}(\overline{\mc{R}}^{k+l-i})^*.
\end{equation}
where $\epsilon$ may depend on $\mc{U}$.

For every such $\mc{U}$ above take any subset of indices $\vec{j} = \{j_1,
\ldots, j_s\}$ of the $i$ points colored $LR$, indexed for now from $1$ to $i$.
Let $\vec{p} =\{p_{j_1}, \ldots, p_{j_s}\}$ be the associated points. 
Given any such subset $\vec{j}$ of $\{1, \ldots, i\}$, define
\begin{equation}
    \mc{P}_{\vec{j}}: {\mathbb I}_{\epsilon}^i \lra {\mathbb I}_{\epsilon}^{|\vec{j}|}
\end{equation}
to be the projection onto the coordinates with indices in $\vec{j}$
and
\begin{equation}
    \mc{P}_{\vec{j}^c}: {\mathbb I}_{\epsilon}^i \lra {\mathbb I}_{\epsilon}^{i - |\vec{j}|} 
\end{equation}
to be the projection onto the complementary coordinates.

Then shrinking $\mc{U}$ and $\epsilon$ if necessary, perform a smooth
identification of the restriction of the basic chart (\ref{basicpairschart}),
where the coordinates for indices in $\vec{j}$ are all non-zero
\begin{equation}
    \mc{U} \times {\mathbb I}_{\epsilon}^i|_{\mc{P}_{\vec{j}}^{-1}( ({\mathbb I}_{\epsilon}^*)^{|\vec{j}|} )}
\end{equation}
onto its image under the smooth map
\begin{equation} \label{chartidentification}
    (S, \vec{v}) \longmapsto (\Pi^{\vec{p}}_{\mc{P}_{\vec{j}}(\vec{v})}(S), \mc{P}_{\vec{j}^c}(\vec{v})).
\end{equation} 
All such identifications are manifestly compatible with each
other, and along with the identifications of charts within each $ _{k-i,l-i,i}
\mc{R}^{k+l-i}$, give $\overline{\mc{R}}_{k,l}$ the structure of a smooth
manifold with corners of dimension $k+l-2$. The result is moreover compact, as
topologically it is the quotient of the compact space $
_{k,l,0}\overline{\mc{R}}^{k+l}$ by identifications between otherwise identical
nodal surfaces containing $D_{LR}^{\pm}$ components.
\begin{ex}
    Let us examine the resulting manifold structure on
    $\overline{\mc{R}}_{2,2}$, in a neighborhood of the one point coincidence
    discussed in Example \ref{LRexample}. Consider once more the element $q$ of
    ${_{1,1,1}}\mc{R}^3$ we discussed there, which is represented by
    some disc with inputs colored $L$, $LR$, and $R$ in clockwise order. Then,
    in a neighborhood $U_q$ of $q$, we have a chart
    \begin{equation}\label{distinguishedchart}
        U_{q} \times (-\epsilon, \epsilon)
    \end{equation}
    for some small value of $\epsilon$. 
    There are two distinguished smooth identifications of subsets of
    (\ref{distinguishedchart}). In the first, one resolves the $LR$ point $p_2$
    by an $L$ followed by $R$ 
    \begin{equation}
    \begin{split}
        U_{q} \times (0,\epsilon) &\lra {_{2,2,0}}\mc{R}^4\\
        (S,v) &\longmapsto \Pi^{p_2}_{(v)}(S):= S \#_{p_2}^{1-v} D_{LR}^{+1}
    \end{split}
    \end{equation}
    and in the second, one resolves $p_2$ by an $R$ followed by an $L$
    \begin{equation}
    \begin{split}
        U_{q} \times (-\epsilon,0) &\lra {_{2,2,0}}\mc{R}^4\\
        (S,v) &\longmapsto \Pi^{p_2}_{(v)}(S):= S \#_{p_2}^{1-(-v)} D_{LR}^{-1}.
    \end{split}
\end{equation}
These identifications, along with the existing manifold structure on 
$({_{2,2,0}} \overline{\mc{R}}^4)^*$
determine the manifold structure in a neighborhood of $p_2$. See Figure
\ref{LRstratachart1} for an illustration of the manifold structure near $p_2$.
\end{ex}

\begin{figure}[h] 
    \caption{ The manifold structure near a coincident point on $\mc{R}_{2,2}$. \label{LRstratachart1}}
    \centering
    \includegraphics[scale=0.6]{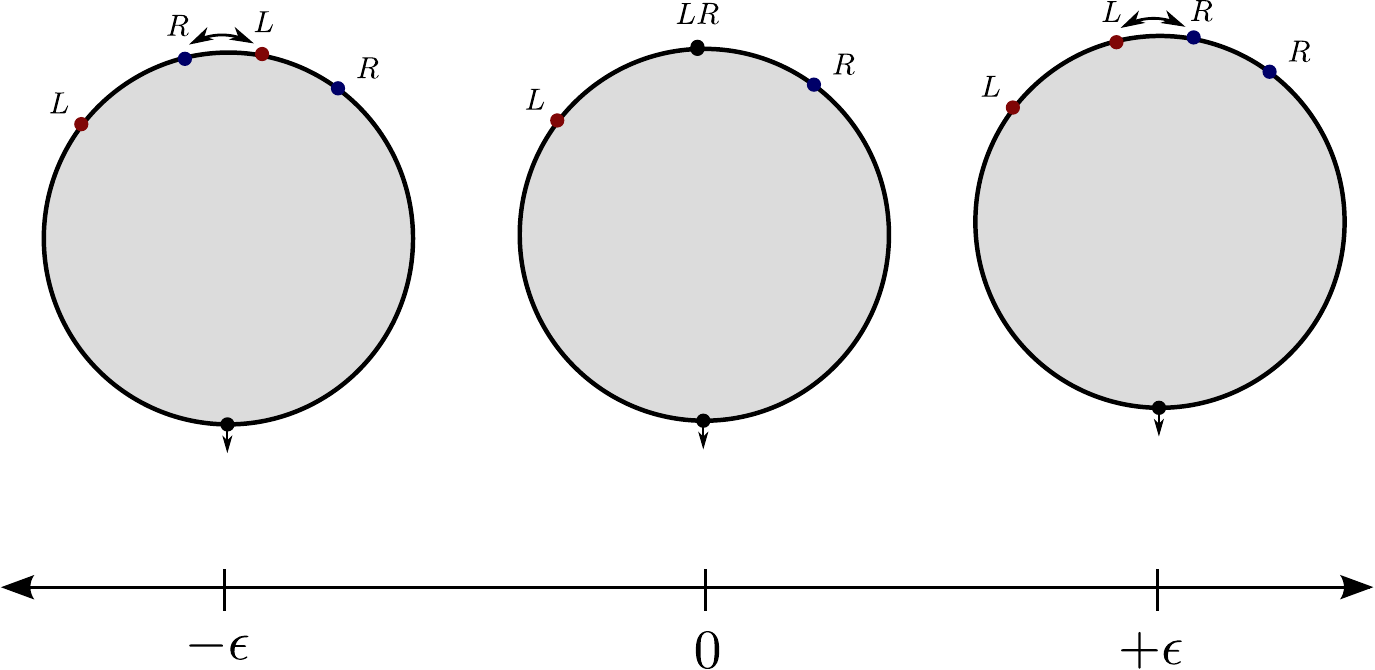}
\end{figure}

This definition of the compactification $\overline{\mc{R}}_{k,l}$ remembers the
structure of the left and right components. Namely, a (potentially nodal)
element $S$ in $\overline{\mc{R}}_{k,l}$ can be thought of as an element of $
_{k-i,l-i,i} \overline{\mc{R}}^{k+l-i}$ for some $i$, with no $D_{LR}^{\pm}$
leaf bubbles. Define the associated {\bf unreduced left disc} $\tilde{S}_1$ to
be obtained from $S$ by deleting all points colored $R$, and forgetting the
$L$, $LR$ colorings. Similarly, define the associated {\bf unreduced right
disc} $\tilde{S}_2$ to be obtained by deleting all points colored $L$ and
forgetting the $R$, $RL$ colorings. The associated {\bf reduced discs} $(S_1,
S_2)$ are given by stabilizing $\tilde{S}_1$ and $\tilde{S}_2$, and there are
well defined inputs on these stabilizations corresponding to the inputs colored
$(L,LR)$ and $(LR,R)$ on $S$ respectively.

In this way, we obtain a {\bf projection map}
\begin{equation}
    \pi_{reduce}: \overline{\mc{R}}_{k,l} \lra \bar{\mc{R}}^k \times \bar{\mc{R}}^l
\end{equation}


Since the manifolds $ _{k-i,l-i,i}\mc{R}^{k+l-i}$ are codimension $i$ in the
space $\mc{R}_{k,l}$, we see that the codimension one boundary of the
compactification $\overline{\mc{R}}_{k,l}$ is contained in the image of the
codimension 1 boundary of the top stratum $ _{k,l,0}\overline{\mc{R}}^{k+l}$.
The chart identifications (\ref{chartidentification}) show that points where an
$L$ and $R$ would have bubbled off in codimension 1 now cease to be boundary
points; thus any codimension 1 bubble must contain at least two $L/LR$ points
or two $R/LR$ points.

The result of this discussion is as follows: we see that the codimension one
boundary of the Deligne-Mumford compactification
(\ref{compactifiedpairsofdiscs}) is covered by the images of the natural
inclusions of the following products:

    \begin{align}
\label{zerocodimonestratum}        \overline{\mc{R}}_{k',l'} &\times {_1}\overline{\mc{R}}_{k-k'+1,l-l'+1},\ k', l', k-k'+1, l-l'+1 \geq 2\\
\label{mbox1}\overline{\mc{R}}_{k',1} &\times {_1}\overline{\mc{R}}_{k-k'+1,l},\ k', k-k'+1, \geq 2\\
\label{mbox2}\overline{\mc{R}}_{1,l'} &\times {_1}\overline{\mc{R}}_{k,l-l'+1},\ l',l-l'+1 \geq 2\\
\label{realcompactification}\overline{\mc{R}}_{k',0} &\times \overline{\mc{R}}_{k-k'+1,l},\ k',k-k'+1 \geq 2\\
\label{realcompactification2}\overline{\mc{R}}_{0,l'} &\times \overline{\mc{R}}_{k,l-l'+1},\ l',l-l'+1 \geq 2
    \end{align}
    The appearance of 1-coincident spaces ${_1}\mc{R}_{k-k'+1,l-l'+1}$
in (\ref{zerocodimonestratum})-(\ref{mbox2}) has a simple explanation.  For
simultaneous bubbling of $L$ and $R$'s to occur, the bubble point must be
coincident, i.e. colored $LR$.

Strictly speaking, the compactification we have described is somewhat larger
than we would ideally like; the strata of most interest to us are
(\ref{realcompactification}) and (\ref{realcompactification2}). However, we
will be able to make arguments showing that operations coming from the other
strata must all be zero.

\begin{rem}
    We could further reduce the compactification of $\mc{R}_{k,l}$ by
    collapsing strata whose subtrees are monochromatic $L$ or $R$ except for a
    single off-color point.  The construction would make points colored $R$
    point view entire subtrees colored $L$ as invisible and vice versa.  Having
    made this construction, one can then check that the resulting operations we
    construct will not change.  We have thus opted for a simpler construction
    at the expense of having a larger compactification.
\end{rem}
\subsection{Sequential point identifications}
We study a particular classes of submanifolds of pairs of discs, and examine
its compactifications carefully. This compactification is the one that will
arise when defining quilts and Floer theoretic operations in the product.
\begin{defn}
    A point identification $\mf{P}$ is said to be {\bf sequential} if it is of
    the form 
    \begin{equation}
        \mf{S} = \{(i_1, j_1), (i_1 +1, j_1+1), \ldots, (i_1 + s, j_1 + s)\}.
    \end{equation}
    It is further said to be {\bf initial} if $(i_1,j_1) = (1,1)$.
\end{defn}

\begin{defn}
    A {\bf cyclic sequential point identification} of type ($r,s$) is one of
    the form
    \begin{equation}
        \mf{S} = \{(1,1), (2,2), \ldots, (r,r), (k-s, l-s), (k-s+1,l-s+1), \ldots, (k,l)\}.
    \end{equation}
    In other words, it is a sequential point identification where we need to
    take indices mod $(k,l)$.
\end{defn}
\begin{prop} \label{sequentialpointidentifications}
    Let $\mf{P}$ be a $(k,l)$ initial sequential point identification of length $s$.
    Then the codimension-1 boundary of the compactification of
    $\mf{P}$-identified pairs of discs 
    \begin{equation}
        _{\mf{P}} \overline{\mc{R}}_{k,l}
    \end{equation}
    is covered by the natural inclusions of the following products: 
    \begin{align}
        \label{bubblinginP} {_{\mf{P}_{max}}} \overline{\mc{R}}_{d,d} &\times {_{\mf{P}'}}\overline{\mc{R}}_{k-d+1,l-d+1}\\
        \label{bubblingoverlapsP} {_{\mf{P}'}}\overline{\mc{R}}_{k',l'} &\times {_{\mf{P}''}} \overline{\mc{R}}_{k-k'+1,l-l'+1}\\
        \label{bubblingoutofP1}\overline{\mc{R}}_{k',l'} &\times {_{\mf{P} \cup (s,t)}}\overline{\mc{R}}_{k-k'+1,l-l'+1}\\
        \overline{\mc{R}}_{k',1} &\times {_{\mf{P} \cup (s,t)}}\overline{\mc{R}}_{k-k'+1,l}\\
        \label{bubblingLR} \overline{\mc{R}}_{1,l'} &\times {_{\mf{P} \cup (s,t)}}\overline{\mc{R}}_{k,l-l'+1}\\
        \label{bubblingoutofPright}\overline{\mc{R}}_{0,l'} &\times {_{\mf{P}}} \overline{\mc{R}}_{k,l-l'+1}\\
        \label{bubblingoutofPend} \overline{\mc{R}}_{k',0} &\times {_{\mf{P}}}\overline{\mc{R}}_{k-k'+1,l}.
    \end{align}
\end{prop}
\begin{proof}
We will only say a few words about the Proposition. Note that under the
point-coincidence stratification, the top stratum of the non-compactified space
$_\mf{P} \mc{R}_{k,l}$, in which there are no other coincidences, correspond to
all tricolorings of discs with $k+l-i$ marked points with $LR$ colorings
specified by $\mf{P}$ and $L$, $R$ colorings arbitrary
\begin{equation}
    _{\mf{P}} \mc{R}_{k,l} = \coprod_{I \cup J \cup \mf{P} = [k+l-i]}{_{I,J,\mf{P}}} \mc{R}^{k+l - |\mf{P}|}.
\end{equation}
Thus, we can determine the codimension-one boundary of the compactification by
looking at the boundary components of the compactifications $_{I,J,\mf{P}}
\overline{\mc{R}}_{k,l}$ which survive our chart maps (\ref{chartidentification}).

The possible strata that arise fall into three different cases: bubbling occurs
entirely within the coincident points (\ref{bubblinginP}), bubbling overlaps
somewhat with the coincident points (\ref{bubblingoverlapsP}), and bubbling
stays entirely away from the coincident points (\ref{bubblingoutofP1}) -
(\ref{bubblingoutofPend}). Note once more that when $L$ and $R$ points
simultaneously bubble, there is an additional coincident point created, hence
the need to add various $(s,t)$ to the coincident set in
(\ref{bubblingoutofP1})-(\ref{bubblingLR}).
\end{proof}

\subsection{Gluing discs}

We now make precise the notion of {\it gluing} pairs of discs along some
identified boundary components, a construction that will arise from
incorporating the diagonal Lagrangian $\Delta$ as an admissible Lagrangian in
the product $M^-\times M$. We begin by discussing the combinatorial type of a
{\bf boundary identifications} of a pair of discs.
\begin{defn}
A {\bf $(k,l)$ boundary identification} is a subset $\mf{S}$ of the set of pairs
$\{0, \ldots, k\} \times \{0, \ldots, l\}$ satisfying the following conditions:
\begin{itemize}
    \item $(0,0)$ and $(k,l)$ are the only admissible pairs in $\mf{S}$
        containing extrema.

    \item ({\bf monotonicity}) 
        $\mf{S}$ can be written as $\{(i_1, j_1), \ldots, (i_s,j_s)\}$ with $i_r < i_{r+1}$ and $j_r < j_{r+1}$.
\end{itemize}
\end{defn}

\begin{defn}\label{scompatible}
Let $S$ and $T$ be unit discs in $\C$ with $k$ and $l$ incoming boundary marked
points respectively, and one outgoing boundary point each. Assume further that
the outgoing boundary points of $S$ and $T$ are in the same position.
Label the boundary components of $S$ 
\begin{equation}
    \{\partial^0 S, \ldots, \partial^k S\}
\end{equation}
in {\bf counterclockwise}
order from the outgoing point, and label the components of $T$ 
\begin{equation}
    \{\partial^0 T,
    \ldots, \partial^l T\}\end{equation} 
in {\bf counterclockwise} order from the outgoing
point. Let $\mf{S}$ be a $(k,l)$ boundary identification. $S$ and $T$ are said
to be {\bf $\mf{S}$-compatible} if 
\begin{itemize}
    \item  the outgoing points of $S$ and $T$ are at the same position.
    \item 
        The identity map induces a one-to-one identification of $\partial^x S$
        with $\partial^y T$ for each $(x,y) \in \mf{S}$.  
\end{itemize}
\end{defn}
\noindent The notion of $\mf{S}$-compatibility is manifestly invariant under
simultaneous automorphism of the pair $(S,T)$. $\mf{S}$-compatibility of a pair
$(S,T)$ also implies certain point coincidences, in the sense of the previous
section.
\begin{defn}
    Let $\mf{S}$ be a $(k,l)$ boundary identification. The {\bf associated
    $(k,l)$ point identification}
    \begin{equation}
        p(\mf{S})
    \end{equation}
    is defined as follows:
    \begin{equation}
        p(\mf{S}) := \{(i,j) | (i,j) \in \mf{S}\textrm{ or } (i-1,j-1) \in \mf{S}\}.
    \end{equation}
\end{defn}
Moreover, if $S$ and $T$ have coincident points consistent with the induced
point-identification $p(\mf{S})$, then $S$ and $T$ are also $\mf{S}$
compatible. Hence, we can make the following definition:
\begin{defn}
    A $(k,l)$ boundary identification $\mf{S}$ is said to be {\bf compatible
    with a $(k,l)$ point identification} $\mf{T}$ if 
    \begin{equation}
        p(\mf{S}) \subseteq \mf{T}
    \end{equation}
    where $p(\mf{S})$ is the associated point identification.
\end{defn}
Now, let $\mf{S}$ be a $(k,l)$ boundary identification with compatible point
identification $\mf{T}$. Given any pair of discs with $\mf{T}$ point
coincidences, there is an associated tricolored disc in the manner described in
the previous section. We see that a boundary identification can be thought of
as a binary $\{$``identified'', ``not identified''$\}$ labeling of the boundary
components between $\mf{T}$-coincident points, which were colored $LR$. Thought
of in this manner, we see that a boundary identification $\mf{S}$ induces a
boundary identification on nodal elements in the compactification
\begin{equation}
    _{\mf{T}} \overline{\mc{R}}_{k,l};
\end{equation}
We can see this as follows. On any nodal component of this space, there are
induced point coincidences coming from gluing maps. Label a boundary component
between two coincident points as ``identified,'' if, after gluing, the boundary
component corresponds to one labeled ``identified.''

Thus, in the same manner that we have already spoken about boundary-labeled
moduli spaces, we can define the {\bf moduli space of $\mf{S}$ identified pairs
of discs with $\mf{T}$ point identifications and $(k,l)$ marked points}
\begin{equation}
    _{\mf{T},\mf{S}}\overline{\mc{R}}_{k,l}
\end{equation}
to be exactly $ _\mf{T} \overline{\mc{R}}_{k,l}$ with the additional boundary
labellings that we described above.

Our reason for defining boundary identification is so that we can speak more
easily about gluings.
\begin{defn}
    Let $S$ and $T$ be compatible with a $(k,l)$ boundary identification datum
    $\mf{S}$. The {\bf $\mf{S}$-gluing}
    \begin{equation}
        \pi_{\mf{S}}:= S \coprod_{\mf{S}} T
    \end{equation}
    is the genus 0 open-closed string defined as follows: view $-S$, i.e. $S$ with the {\bf opposite complex structure} as being the
    south half of a sphere bounding the equator via the complex doubling
    procedure, one of the methods of constructing the moduli of bordered
    surfaces \cite{Liu:2004fk}*{\S 3.1}.  Similarly, view $T$
as the north half of the sphere. Then \begin{equation} (S \coprod_{\mf{S}} T)
    := (-S) \coprod T / \sim \end{equation}
    where $\sim$ identifies $\partial^x (-S)$ to $\partial^y T$ ($\partial^x (-S)$ is the same
    boundary component of $S$ as before, now with the reverse orientation)
    under the identification coming from inclusion into the sphere if and only
    if $(x,y)
    \in \mf{S}$. Boundary marked points are identified as follows: Let $z^i_{-S}$
    be the boundary marked point between $\partial^{i-1} (-S)$ and
    $\partial^{i}(-S)$,
    $z^0_{-S}$ the outgoing marked point, and $z^j_{T}$ similar. Then:
    \begin{itemize}
        \item if $(x-1,y-1), (x,y) \in \mf{S}$, then $z^x_{-S}\sim z^y_{T}$ becomes
            a single interior marked point.  
        \item if $(x-1,y-1) \in \mf{S}$ but $(x,y)$ is not, then $z^x_{-S} \sim
            z^y_{T}$ becomes a single boundary marked point, between $\partial^y
            T$ and $\partial^x (-S)$ 
        \item if $(x,y) \in \mf{S}$ but $(x-1,y-1)$ is not, then $z^x_{-S} \sim
            z^y_{T}$ becomes a single boundary marked point, between $\partial^x
            (-S)$ and $\partial^y T$ 
        \item otherwise, $z^x_{-S}$ and $z^y_T$ are kept distinct, becoming two
            boundary marked points.  
    \end{itemize}
\end{defn}
By $\mf{S}$-compatibility, $S$ and $T$ can be viewed as the south and north
halves of a sphere in a manner preserving the alignment of outgoing marked
points and boundary components specified by $\mf{S}$, so the above definition
is sensible.

One can read off the characteristics of the resulting bordered surface from
$k$, $l$, and $\mf{S}$, which we leave as an exercise. Denote the resulting
number of boundary components of the open-closed string
\begin{equation}
    h(k,l,\mf{S}).
\end{equation}

\begin{figure}[h] 
    \caption{An example of the gluing $\pi_{\mf{S}}$ associated to a $\{(1,1),(k,l)\}$ boundary identification. \label{gluingpairs1}}
    \centering
    \includegraphics[scale=0.6]{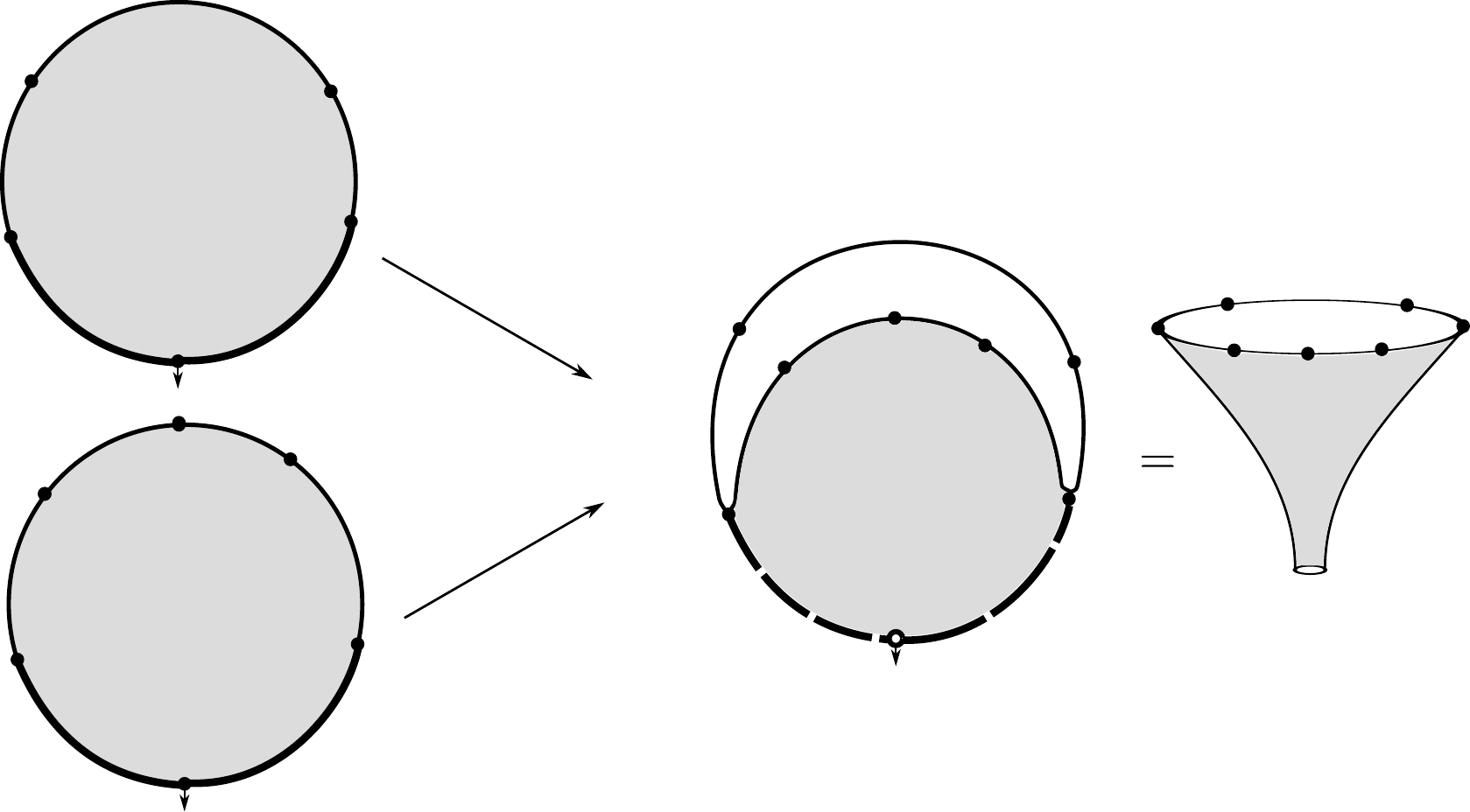}
\end{figure}

\begin{prop}
    Let $\mf{S}$ be a boundary identification, with compatible point identification $\mf{T}$. 
    Then, the gluing operation $S \coprod_{\mf{S}} T$ extends to an operation on the
    Deligne-Mumford compactifications $_{\mf{T}}\overline{\mc{R}}_{k,l}$. 
\end{prop}
\begin{proof}
    Nodal components of an element $P$ in the Deligne-Mumford compactification
    $_{\mf{T}}\overline{\mc{R}}_{k,l}$ can be thought of as nodal tricolored
    discs, with induced boundary identifications in a manner we have already
    described. As we have earlier indicated, the $L/R$ forgetful maps applied to $P$ give us
    {\bf unreduced left and right disc trees} $(\tilde{S}, \tilde{T})$.
    Boundary identifications descend to these trees, because by definition they
    were labellings between points colored $LR$. We now perform the above
    procedure component-wise on this pair to obtain a nodal open-closed string
    $\tilde{S} \coprod_{\mf{S}} (\tilde{T})$, which potentially has
    semi-stable/unstable components. Finally, define
    \begin{equation} \label{gluing}
        \pi_{\mf{S}}(P) := S \coprod_{\mf{S}}T
    \end{equation}
    to be the nodal open-closed string obtained by stabilizing $\tilde{S}
    \coprod_{\mf{S}} \tilde{T}$.
\end{proof}
In the case when the boundary identification $\mf{S}$ is empty, the gluing
operation $\pi_{\mf{S}}$ reduces to the projection we defined earlier, where we
conjugate the first factor:
\begin{equation}
    \pi_{\emptyset} := \pi_{reduce} \circ (-1 \times id).
\end{equation}

\subsection{Floer data and operations}
\begin{defn} \label{floerglued}
    A {\bf Floer datum} for a glued pair of discs $(P,\mf{S})$ is a Floer
    datum for the resulting open-closed string $\pi_{\mf{S}}(P)$, in the
    sense of Definition \ref{floeropenclosed}.
\end{defn}

Now, let us assume that our point identification $\mf{T}$ was sequential or
cyclic sequential.  
\begin{defn}
    A {\bf universal and consistent choice of Floer data for glued
    pairs of discs} $\mathbf{D}_{glued}$ is a choice
    $D_{P,\mf{S}}$ of Floer data in the sense of Definition \ref{floerglued}
    for every $k$, $l$, $(k,l)$ boundary identification $\mf{S}$ and compatible
    sequential point identification $\mf{T}$, and every representative
    $_{\mf{T},\mf{S}}\overline{\mc{R}}_{k,l}$, varying smoothly over
    $_{\mf{T},\mf{S}}\overline{\mc{R}}_{k,l}$, whose restriction to a boundary stratum
    is conformally equivalent to the product of Floer data coming from lower
    dimensional moduli spaces.  Moreover, with regards to the coordinates,
    Floer data agree to infinite order at the boundary stratum with the Floer
    data obtained by gluing. Finally, we require that this choice of Floer
    datum satisfy the following conditions:
    \begin{align} \label{projection}
        &\textrm{The Floer datum only depends on the open-closed string }
        \pi_{\mf{S}}(P);\textrm{ and}\\
        &\textrm{The Floer datum agrees with our previously chosen Floer datum on }
        \pi_{\mf{S}}(P).
    \end{align}
\end{defn}

\begin{defn}
    A {\bf Lagrangian labeling from $\mathbf{L}$} for a glued pair of discs
    $(P, \mf{S})$ is a Lagrangian labeling from $\mathbf{L}$ for the gluing $\pi_{\mf{S}}(P)= S
    \coprod_{\mf{S}} T$, thought of as a (possibly disconnected) open-closed
    string. Given a fixed labeling $\vec{L}$, denote by
        \begin{equation}
(_{\mf{T},\mf{S}}\overline{\mc{R}}_{k,l})_{\vec{L}}
        \end{equation}
        the space of labeled $\mf{S}$-identified pairs of discs with $\mf{T}$ point coincidences.
\end{defn}

Now, fix a compact oriented submanifold with corners of dimension $d$,
\begin{equation}
    \overline{\mc{L}}^d\longhookrightarrow 
    {_{\mf{T},\mf{S}}}\overline{\mc{R}}_{k,l}
\end{equation}
Fix a Lagrangian labeling 
\begin{equation}
    \vec{L} = \{\{L_0^1, \ldots, L_{m_1}^1\}, \{L_0^2, \ldots, L_{m_2}^2\}, \ldots, \{L_0^h,
\ldots, L_{m_h}^h\}\}.
\end{equation}
Also, fix chords 
\begin{equation}
\vec{x} = \{\{x_1^1, \ldots, x_{m_1}^1\}, \ldots,\{ x_1^h, \ldots, x_{m_h}^h\}\}
\end{equation}
and orbits $\vec{y} = \{y_1, \ldots, y_n\}$ with
\begin{equation}
x_i^j \in \begin{cases}
    \chi(L_{i+1}^j,L_i^j) & i \in K^j \\
    \chi(L_i^j, L_{i+1}^j) & \mathrm{otherwise.}
\end{cases}
\end{equation}
Above, the index $i$ in $L_i^j$ is counted mod $m_{j}$. Collectively, the
$\vec{x}, \vec{y}$ are called a set of {\bf asymptotic conditions} for the
labeled moduli space $\mc{L}^d_{\vec{L}}$. The {\bf outputs} $\vec{x}_{out}$,
$\vec{y}_{out}$ are by definition those $x_i^j$ and $y_s$ for which $i \in K^j$
and $s \in \mathbf{I}$, corresponding to negative marked points.
The {\bf inputs} $\vec{x}_{in}$, $\vec{y}_{in}$ are the remaining chords and
orbits from $\vec{x}$, $\vec{y}$.  Fixing a chosen universal and consistent
Floer datum, denote $\e^{i,j}_\pm$ and
$\delta^l_\pm$ the strip-like and cylindrical ends corresponding to $x_i^j$ and
$y_l$ respectively. 

Finally, define
\begin{equation}
\overline{\mc{L}}^d(\vec{x}_{out},\vec{y}_{out};
\vec{x}_{in},\vec{y}_{in})
\end{equation}
to be the space of maps 
\begin{equation}
    \{ u: \pi_{\mf{S}}(P) \longrightarrow M:\ P \in \overline{\mc{L}}^d \}
\end{equation}
satisfying, at each element $P$, {\it Floer's equation for
$(\mathbf{D}_{\mf{S}})_{P}$} with boundary and asymptotic conditions
\begin{equation}
    \begin{cases}
        \lim_{s \ra \pm\infty} u \circ \e^{i,j}_\pm(s,\cdot) = x_i^j,\\
    \lim_{s \ra \pm\infty} u \circ \delta^l_\pm(s,\cdot) = y_l,\\
    \ \ \ u(z) \in \psi^{a_S(z)} L_i^j, & z\in \partial_i^j S.
\end{cases}
\end{equation}

We have the usual transversality and compactness results:
\begin{lem} 
    The moduli spaces $\overline{\mc{L}}^d(\vec{x}_{out},\vec{y}_{out};
    \vec{x}_{in},\vec{y}_{in})$ are compact and there are only finitely many
    collections $\vec{x}_{out}$, $\vec{y}_{out}$ for which they are non-empty
    given input $\vec{x}_{in}$, $\vec{y}_{in}$. For a generic universal and
    conformally consistent Floer data they form manifolds of dimension 
\begin{equation} 
    \begin{split}
    \dim
    \mc{L}^d(\vec{x}_{out},\vec{y}_{out};\vec{x}_{in},\vec{y}_{in}):=\sum_{x_-
    \in \vec{x}_{out}} &\deg (x_-) + \sum_{y_- \in \vec{y}_{out}} \deg (y_-)\\
    + (2-h(k,l,\mf{S}) - |\vec{x}_{out}| - 2 |\vec{y}_{out}|)n
    &+d -\sum_{x_+ \in \vec{x}_{in}} \deg(x_+) - \sum_{y_+ \in \vec{y}_{in}}
    \deg(y_+).  
\end{split}
\end{equation}
\end{lem}
\begin{proof}
The index computation follows from the arguments outlined in the proof of Lemma
\ref{openclosedcompactdimensiontransversality}. The proof of transversality for
generic perturbation data is once more an application of Sard-Smale, following
arguments in \cite{Seidel:2008zr}*{(9k)} or alternatively \cite{Floer:1995fk}.
These arguments show that the extended linearized operator for Floer's
equation, in which one allows deformations of the almost complex structure and
one-form, is surjective. As their arguments are on the level of stabilized
moduli spaces, they imply that transversality can be achieved in our situation
by taking perturbations of Floer data that are constant along the fibers of the
projection map $\pi_{\mf{S}}$. In other words, the class of Floer data
satisfying (\ref{projection}) is large enough to achieve transversality.

The usual Gromov compactness applies once Theorem \ref{c0bounds} is applied to
obtain a priori bounds on maps satisfying Floer's equation with fixed
asymptotics.
\end{proof}

When the dimension of $\mc{L}^d(\vec{x}_{out},\vec{y}_{out};\vec{x}_{in},\vec{y}_{in})$ is 0, we conclude that its elements are rigid.
In particular, any such element  $u \in
\mc{L}^{d}(\vec{x}_{out},\vec{y}_{out};\vec{x}_{in},\vec{y}_{in})$ gives (by
Lemma \ref{orientationlem}) an isomorphism of orientation lines
\begin{equation}
    \mc{L}_u: \bigotimes_{x \in \vec{x}_{in}} o_x \otimes \bigotimes_{y \in \vec{y}_{in}} o_y \lra \bigotimes_{x \in \vec{x}_{out}} o_x \otimes \bigotimes_{y \in \vec{y}_{out}} o_y.
\end{equation}

Using this we define a map 
\begin{equation}
    \begin{split}
         \mathbf{G}_{\overline{\mc{L}}^d}: \bigotimes_{(i,j);1\leq i\leq m_j;i
        \notin K^j} &CW^*(L_i^j,L_{i+1}^j) \otimes \bigotimes_{1 \leq k \leq n;
        k \notin \mathbf{I}} CH^*(M) \longrightarrow  \\
         &\bigotimes_{(i,j);1\leq i\leq m_j;i
         \in K^j} CW^*(L_{i+1}^j,L_{i}^j) \otimes \bigotimes_{1 \leq k \leq n;
        k \in \mathbf{I}} CH^*(M) \mbox{   }
    \end{split}
\end{equation}
given by, as usual (abbreviating $\vec{x}_{in} = \{x_1, \ldots, x_s\}$, $\vec{y}_{in} = \{y_1, \ldots, y_t\}$)
    \begin{equation} 
        \begin{split}
            &\mathbf{G}_{\overline{\mc{L}}^d}([y_t], \ldots, [y_1], 
            [x_s], \ldots, [x_1]) := \\ 
            & \sum_{\dim \mc{L}^d(\vec{x}_{out},\vec{y}_{out};
            \vec{x}_{in},\vec{y}_{in}) = 0} 
            \sum_{u \in \mc{L}^d(\vec{x}_{out},\vec{y}_{out}; 
            \vec{x}_{in}, \vec{y}_{in})}
            \mc{L}_u([x_1], \ldots, [x_s], [y_1], \ldots, [y_t]).
    \end{split}
\end{equation}
This construction naturally associates, to any submanifold $\mc{L}^d \in 
{_{\mf{S},\mf{T}}}\mc{R}_{k,l}$, a map $\mathbf{G}_{\mc{L}^d}$, depending on a
sufficiently generic choice of Floer data for glued pairs of discs. In a
similar fashion, this can be done for a submanifold of the labeled space 
\begin{equation}
    \mc{L}^d_{\vec{L}} \subset (_{\mf{S},\mf{T}}\mc{R}_{k,l})_{\vec{L}},
\end{equation} 
in which case the result is an operation
defined only for a specific labeling, 
\begin{equation} \label{operation2}
    \mathbf{G}_{\overline{\mc{L}}^d_{\vec{L}}},
\end{equation}
This operation can also be constructed with a sign twisting datum to create an
operation 
\begin{equation}
    (-1)^{\vec{t}} \mathbf{G}_{\overline{\mc{L}}^d}
\end{equation}
in an identical fashion to (\ref{twistoperation1}).

\subsection{Examples}
As a first example, consider the case $\mf{S} = \emptyset$ and $\mc{L}$
equal to the full $\overline{\mc{R}}_{k,l}$.
\begin{prop} \label{extraauto}
        The operation associated to $\mc{L} = \overline{\mc{R}}_{k,l}$ with
        arbitrary Lagrangian labeling
        is zero if both $k$ and $l$ are $\geq 1$
        and one of $(k,l)$ is $\geq 2$.
    \end{prop}
    \begin{proof}
        Let $u$ be a rigid element in the associated moduli space
        $\overline{\mc{R}}_{k,l}(\vec{x}_{in}; \vec{x}_{out})$; since we are in
        the transverse situation, we can assume the domain of $u$ is a point in
        the
        interior $p \in \mc{R}_{k,l}$. On the interior, the projection map
        \begin{equation}
            \pi_{\emptyset}: \mc{R}_{k,l} \lra \mc{R}^k \times \mc{R}^l
        \end{equation}
        has fibers of dimension at least 1, parametrized by automorphisms of one
        factor relative to the other. (when $k=1$, we implicitly replace $\mc{R}^k$
        by a point, and same for $l$---stabilization in this case completely
        collapses the left or right component).  Since our Floer data was
        chosen to only depend on $\pi_{\emptyset}(p)$, we conclude that any map
        from an element of the fiber $\pi^{-1}_{\emptyset}(\pi_{\emptyset}(p))$ also satisfies
        Floer's equation; hence $u$ cannot be rigid.
    \end{proof}
    \begin{prop} \label{yonedaquilt}
        The operation associated to the compactification of the inclusion
        $\mc{I}_k$ in (\ref{onesided1}) is $(\mu^k_{\w})^{op} \otimes id$. Similarly the
        operation associated to $\mc{J}_l$ as in (\ref{onesided2}) is $id
        \otimes \mu^l_\w$.
    \end{prop}
    \begin{proof}
A rigid element $u$ of the moduli space associated to
$\overline{\r{im}(\mc{I}_k)}$ has, without loss of generality, domain in the
interior $\r{im}(\mc{I}_k)$.  We note that the projection map 
    \begin{equation}
        \pi_{\emptyset}: \r{im}( \mc{I}_k) \ra
        \mc{R}^k \times \{*\}
    \end{equation}
        is an isomorphism up to conjugation. Since we have chosen Floer
        data compatibly, we obtain an isomorphism
        \begin{equation}
            \overline{\mc{I}_k(\mc{R}^k)}( (x_0, x_0'); ( (x_1, \ldots, x_k),
            x_0')) \simeq \mc{R}^k(x_0; x_k, \ldots, x_1) \times \{*\},
        \end{equation}
        implying the result for $\mc{I}_k$. The result for $\mc{J}_l$ is
        analogous.  
    \end{proof}
More generally, one can look at the submanifold $\mc{Q}_2^{k,l} \subseteq
\mc{R}_{k,l}$ of pairs of discs where the negative marked points are required
to coincide and the marked points immediately counterclockwise are required to
coincide.  
\begin{prop}\label{twocoincide}
The associated operation is zero unless $k = 1$ or $l = 1$, in which case the
previous proposition applies.  
\end{prop}
\begin{proof}
    In this case, the projection
    \begin{equation}\pi_{\emptyset}: \mc{Q}_2^{k,l} \lra \mc{R}^k \times \mc{R}^l
    \end{equation}
    has one dimensional fibers, parametrized by automorphism of one factor
    relative another. We conclude in the manner of the previous two propositions 
    that elements of the associated moduli spaces can never be rigid.
\end{proof}

We can also look at the submanifold 
\begin{equation}
    _{\emptyset, \mf{T}} \mc{R}_{k,l}
\end{equation}
of pairs of discs corresponding to the point identification
\begin{equation}
    \mf{T}= \{(1,1),(2,2)\},
\end{equation}
i.e. discs where the negative point, and first two positive points, are
required to coincide.
This submanifold can be thought of as the image of an open embedding from the pair
of associahedra
\begin{equation}
    \mc{R}^r \times \mc{R}^s \stackrel{Q}{\longrightarrow} \mc{R}_{r,s}
\end{equation}
which can be described as follows: Take the representative of each disc on the
left $(-S_1,S_2)$ for which the negative marked point and the two marked points
immediately counterclockwise of $(S_1, S_2)$ have been mapped to $-i$, $1$, and $i$
respectively. Define
\begin{equation}
    Q([S_1],[S_2]) := [(S_1,S_2)].
\end{equation}
In other words, associate to a pair of discs mod automorphism the pair mod
simultaneous automorphism in which the output and the first two inputs are
required to coincide. The embedding $Q$ not quite extend to an embedding of
$\overline{\mc{R}^r} \times \overline{\mc{R}^s}$ because, among other
phenomena, the three chosen points on each disc will come together
simultaneously in codimension 1. We can still study the operation determined by
the compactification of the embedding.
\begin{prop}\label{threecoincide}
    The operation associated to the compactification 
    $\overline{Q(\mc{R}^r \times \mc{R}^s)}$ is $(\mu^r)^{op} \otimes \mu^s$.  
\end{prop}
\begin{proof}
    The map $Q$ is a left and right inverse to the projection map
    \begin{equation}
        \pi_{\emptyset}: _{\emptyset,\{(1,1),(2,2)\}} \mc{R}_{k,l} \lra \mc{R}^k \times \mc{R}^l
    \end{equation}
    so $\pi_{\emptyset}$ is an isomorphism, up to direction reversal of first
    factor. We conclude that there is an identification of dimension zero
    moduli spaces
    \begin{equation}
        \overline{\r{im}(Q)} (x^1_{out},x^2_{out}; \vec{x}^1_{in}, \vec{x}^2_{in}) \simeq \overline{\mc{R}}^k(x^1_{out}; (\vec{x}^1_{in})^{op}) \times \overline{\mc{R}}^l(x^2_{out}; \vec{x}^2_{in}).
    \end{equation}
    where the {\it op} superscript indicates an order reversal.
\end{proof}

Now, consider the case of a single gluing adjacent to the outgoing marked
points, i.e. $\mf{S} = \{(1,1)\}$ or $\mf{S} = \{(k,l)\}$ with the induced
point identification.
\begin{prop}\label{mufromglue}
    The resulting operation in either case is $\mu^{k+l+1}$. 
\end{prop}
\begin{proof}
    We will without loss of generality do $\mf{S}= \{(1,1)\}$; the associated
    point identification is also $p(\mf{S}) = \{(1,1)\}$.  The gluing morphism
    is of the form 
    \begin{equation}\label{oneboundarygluing}
        \pi_{\mf{S}}: _{\mf{S},p(\mf{S})}\mc{R}_{k,l} \lra \mc{R}^{k+l+1}.
    \end{equation}
    if $k$ or $l$ is $\geq 1$, the unreduced gluing is automatically stable,
    implying that (\ref{oneboundarygluing}) is an isomorphism. We obtain a
    corresponding identification of moduli spaces.
\end{proof}

Our next example is the case $\mf{S} = \{(1,1),(k,l)\}$ with the induced point
identification.
\begin{prop}\label{ocstringsfromgluing}
    The resulting operation is exactly ${_2}\oc^{k-2,l-2}$.
\end{prop}
\begin{proof}
    The surface obtained by gluing the $(1,1)$ and $(k,l)$ boundary components
    together in $\mc{R}_{k,l}$ is stable, and has one interior output marked
    point. There are also $k+l$ boundary marked points, two of which are
    special. In cyclic order on the boundary, there is the identified point $p_1$
    coming from the $(1,1)$ boundary points, the $k-2$ non-identified points
    from the left disc, the identified point $p_2$ coming from the $(k,l)$
    boundary points, and the $l-2$ non-identified points from the right disc.
    Moreover, the identified points $p_1$, $p_2$, and the interior boundary
    point are required to, up to equivalence, lie at the points $-i$, $0$, and
    $i$ respectively. We conclude that the projection is an isomorphism onto
    \begin{equation}
        \pi_{\mf{S}}:{_{\mf{S}, p(\mf{S})}} \mc{R}_{k,l} \lra \mc{R}^1_{k-2,l-2},
    \end{equation}
    the moduli space controlling ${_2}\oc^{k-2,l-2}$.

    See also Figure \ref{gluingpairs1} for an image of this situation.
\end{proof}

Our final example is the case $\mf{S} = \{(1,1),(2,2)\}$, with the induced
point identification.
\begin{prop}\label{cofromglue}
    The resulting operation is exactly ${_2}\co^{k-2,l-2}$.
\end{prop}
\begin{proof}
    The surface obtained by gluing the $(1,1)$ and $(2,2)$ boundary components
    together in $\mc{R}_{k,l}$ is stable, has one interior input marked point,
    and in counterclockwise order on the boundary has one output boundary
    marked point (which was adjacent to the $(1,1)$ gluing), $l-2$ additional
    boundary inputs, one special boundary input (which was adjacent to the
    $(2,2)$ gluing), and $k-2$ additional boundary inputs. Moreover, the output
    boundary marked point, the interior input, and the special boundary input
    are required to, up to equivalence, lie at the points $-i$, $0$, and $i$
    respectively.
    We conclude that the projection is an isomorphism onto
    \begin{equation}
        \pi_{\mf{S}}: {_{\mf{S}, p(\mf{S})}} \mc{R}_{k,l} \lra \mc{R}^{1,1}_{k-2,l-2},
    \end{equation}
    the moduli space controlling ${_2}\co^{k-2,l-2}$.
\end{proof}

\section{Floer theory in the product} \label{productsection}

The Liouville manifold $M^2 : = M^-\times M$ carries a natural symplectic form,
$(-\omega_M, \omega_{M})$ for which the diagonal is a Lagrangian submanifold.
Let \[\pi_i: M^2 \ra M,\ i=1,2\] be the projection to the $i$th component.
As observed by Oancea \cite{Oancea:2006oq}, there is a natural cylindrical end
on $M^- \times M$, with coordinate given by $r_1 + r_2$, where $r_i =
\pi_i^* r$ is the coordinate on the $i$-th factor.  Thus one could define
symplectic homology and wrapped Floer theory by considering Hamiltonians of the
form $(r_1 + r_2)^2$ at infinity. To obtain the comparisons that we desire, we
must consider Floer theory for split Hamiltonians of the form $\pi_1^*H +
\pi_2^* H$, for $H \in \H(M)$. There are immediately some technical
difficulties: split Hamiltonians are {\it not} admissible in the above sense,
and in general will admit some additional chords near infinity.  Using methods
similar to \cite{Oancea:2006oq}, one could prove that these orbits and chords
have sufficiently negative action, do not contribute to the homology, and thus
such split Hamiltonians are a posteriori admissible (thereby proving a
K\"{u}nneth theorem for wrapped Floer homology).

We bypass this issue and instead define all Floer-theoretic operations on the
product for split Hamiltonians and almost complex structures. In this case,
with suitably chosen Floer data, compactness and transversality follow, via
{\it unfolding,} from compactness and transversality of certain open-closed
moduli spaces of maps into $M$ constructed from glued pairs of discs. The end
result will be a model 
\begin{equation}
    \w^2
\end{equation} 
for the wrapped Fukaya category of split Lagrangians and the diagonal in $M^2$,
using only maps and morphisms in $M$.  

\begin{rem}
 There is another technical difficulty with considering Hamiltonians of the
 form $(r_1 + r_2)^2$ at infinity: Products of admissible Lagrangians $L_i
 \times L_j$ are no longer a priori admissible in the product. Namely, it is
 not guaranteed (and highly unlikely) that the primitive $\pi_1^* f_{L_i} +
 \pi_2^* f_{L_j}$ is constant as $(r_1 + r_2) \ra \infty$. The usual method of
 proving that the relevant moduli spaces are compact does not work in this
 situation, and a more refined argument is needed.  
\end{rem}

\subsection{Floer homology with split Hamiltonians}\label{floersplit}
First, let us examine Floer homology groups in $M^2$ for a class of split
Hamiltonians $\tilde{\mc{H}}(M\times M)$ of the form 
\begin{equation}
    \pi_1^*H + \pi_2^*H,\  H \in \H(M).
\end{equation}
For Floer homology between split Lagrangians we immediately obtain a
K\"{u}nneth decomposition.
\begin{lem} \label{kunnethsplit}
    For $H$ and $J$ generic, there is an identification of complexes 
    \begin{equation}\label{kspliteqn}
        CW^*(L_1 \times L_2,L_1' \times L_2'; \pi_1^* H + \pi_2^* H,
        (-J,J)) = CW^*(L_1',L_1,H,J) \otimes CW^*(L_2,L_2';H,J)
    \end{equation}
    where the differential on the right hand side is $\delta_{L_1',L_1} \otimes
    1 + 1 \otimes \delta_{L_2,L_2'}$.  
\end{lem}
\begin{proof}
    If $X$ is the Hamiltonian vector field corresponding to $H$, note that the
    complex on the left-hand side of \ref{kspliteqn} is generated by time 1
    flows of the vector field $(-X,X)$, so there is a one-to-one correspondence
    of generators. By examining equations we see that there is a one-to-one
    correspondence of strips 
    \begin{equation}
        \mc{R}^1( (x_0, x_0'), (x_1, x_1') ) = \mc{R}^1(x_0; x_1) \times
        \mc{R}^1(x_0';x_1') 
    \end{equation}
    In particular, the dimension of $\mc{R}^1( (x_0, x_0'),  (x_1,x_1'))$ is 
    \[
    \deg x_0 - \deg x_1 + \deg x_0' -  \deg x_1' = \dim \mc{R}^1(x_0; x_1) + \dim \mc{R}^1(x_0';x_1')
    \]
    This implies that the one-dimensional component of the moduli space is the
    union of 
    \begin{itemize}
        \item the one-dimensional component of $\mc{R}^1(x_0; x_1)$ times the
            zero-dimensional component of $\mc{R}^1(x_0';x_1')$.
        \item the one-dimensional component of $\mc{R}^1(x_0'; x_1')$ times the
            zero-dimensional component of $\mc{R}^1(x_0;x_1)$.
    \end{itemize}
    But by Proposition \ref{identity}, the zero-dimensional component of the
    above moduli spaces must be constant maps, thereby implying the Lemma.
\end{proof}

Floer trajectories with the diagonal Lagrangian $\Delta$ unfold and can be
compared to symplectic cohomology trajectories, for appropriate Hamiltonians.
\begin{lem}\label{unfolding}
    \label{shdiag}
    For $H_t$ and $J$ generic,
    \begin{equation}
        CW^*(\Delta,\Delta,
        \pi_1^*\frac{1}{2}H_{1-t/2} + \pi_2^*\frac{1}{2}H_{t/2}, 
        (-J,J)) = CH^*(M,H_t,J)
    \end{equation}
    as relatively graded chain complexes.
\end{lem}
\begin{proof} 

Denote 
$\hat{H}_t = \pi_1^*\frac{1}{2}H_{1-t/2} + \pi_2^*\frac{1}{2}H_{t/2}.$
The correspondence between generators is as follows:
Given a time 1 orbit $x$ of $H_t$, we construct a time 1 chord from
$\Delta$ to $\Delta$ of $\hat{H}_t$ 

\begin{equation}
    \hat{x}(t) = (x(1-t/2),x(t/2)) .
\end{equation}

Conversely, given a time 1 chord $\hat{x} = (x_1,x_2)$ of $\hat{H}_t$ , the
corresponding orbit of $H_t$ is given by: 

\begin{equation} 
    x(t) = 
    \begin{cases}
        x_2(2t) & t \leq 1/2 \\
        x_1(2(1-t)) & 1/2 \leq t \leq 1 
    \end{cases}
\end{equation} 

 Let us now identify the moduli spaces counted by either differential. First,
 suppose we have a map $u: \R \times \R/2\Z \ra M$ satisfying:
\begin{align*}
    \partial_s u + J_t (\partial_t u - X_t) &= 0\\
    \lim_{s \ra -\infty} u(s,\cdot) &= x_-\\
    \lim_{s \ra \infty} u(s,\cdot) &= x_+
\end{align*}
where $X_t$ is the Hamiltonian vector field corresponding to $H_t$. Then, note
that the map 

\begin{equation} 
    \hat{u}(s,t) := (u_1(s,1-t/2),u_2(s,t/2))
\end{equation} 
\noindent satisfies the equation 

\begin{equation} 
\partial_s \hat{u} = - (-J_{1-t/2},J_{1-t/2}) \left(\partial_t
    \hat{u} - (-X_{1-t/2},X_{t/2})\right)
\end{equation} 
\noindent with limits 
\begin{align}
    \lim_{s \ra - \infty} \hat{u}(s,\cdot)&= \hat{x}_-,\\ 
    \lim_{s \ra \infty} \hat{u}(s,\cdot)&= \hat{x}_+.
\end{align}
Conversely, suppose we have a map $\hat{u}: \R \times [0,1] \ra M^-\times M$ satisfying
\begin{align}
    \partial_s u + (-J_{1-t/2},J_{1-t/2}) (\partial_t u - (-X_{1-t/2},X_{t/2})) &= 0\\
    \lim_{s\ra -\infty} u(s,\cdot) &= \hat{x}_-\\
    \lim_{s \ra \infty} u(s,\cdot) &= \hat{x}_+.
\end{align}

\noindent Let $u_i = \pi_i \circ \hat{u}$, and define
\begin{equation} 
    u(s,t) = 
    \begin{cases} 
        u_2(s,2t) & 0\leq t \leq 1/2 \\ 
        u_1(s,2(1-t)) & 1/2 \leq t \leq 1.
    \end{cases}
\end{equation} 
Because $\hat{u}(s,0)$ and $\hat{u}(s,1)$ lie on  $\Delta$, $u(s,t)$ is
continuous across the {\it seams} $t = 0,\ 1/2$.  It is clear that $\partial_s
u$ is continuous along $t=0,\ 1/2$.  Thus, as $u$ solves $\partial_s u + J_t
(\partial_t u - X_t) = 0$ on both sides of $t=0$ and $t=1$, we see that
$\partial_t u$ = $J_t(\partial_s u) + X_t$ is continuous, so $u$ is at least
$C^1$ across the seams. Now inductively use the fact that
$\partial_s^k$ is continuous for all $k$ along with applications of
$\partial_s$ and $\partial_t$ to Floer's equation, to conclude that all other
mixed partials are continuous. Therefore $u$ is $C^\infty$ across the seams.
\end{proof}

The cases of $HW^*(\Delta, L_i \times L_j)$ and $HW^*(L_i \times L_j, \Delta)$
are analogous, so we simply state them.
\begin{prop}\label{categoricalproductidentity} As relatively graded chain complexes, 
    \begin{equation}
        CW^*(L_i \times L_j,
        \Delta,\frac{1}{2}(\pi_1^* H + \pi_2^* H),(-J,J) ) = CW^*(L_j, L_i, H, J)
\end{equation}
\end{prop}

\begin{prop} \label{fundamentalproductidentity}
    As relatively graded chain complexes, 
    \begin{equation}
        CW^*(\Delta, L_i \times L_j,\frac{1}{2}(\pi_1^* H + \pi_2^* H),(-J,J))
        = CW^*(L_i, L_j, H, J).
    \end{equation}
\end{prop}
In the setting of the ordinary Fukaya category, the analogue of Lemma
\ref{unfolding} is the well-known correspondence between $HF^*(\Delta,\Delta)$
with the ordinary Hamiltonian Floer homology, or quantum cohomology, of the
target manifold. 
Instead of continuing this correspondence for higher operations and, e.g. {\it
unfolding} Floer data for discs mapping into $M^2$, we will take the above
correspondence as a starting point for a definition of the category $\w^2$
using operations and Floer data in $M$. Define the {\bf objects} of $\w^2$ as
\begin{equation}
    \ob \w^2 := \{L_i \times L_j| L_i, L_j \in \ob \w\} \cup \{\Delta\}.
\end{equation}
For objects $X_k$, $X_l$, define the {\bf generators of the morphism complexes} 
\begin{equation}\label{homgenerators}
    \chi_{M^2}(X_k,X_l) := 
    \begin{cases}
        \chi(L_j,L_i,H) \times \chi(L_i',L_j',H) & X_k = L_i \times L_i',\ X_l = L_j \times L_j' \\
        \chi(L_j,L_i) & X_k  = L_i \times L_j, \  X_l = \Delta \\
        \chi(L_i,L_j) & X_k  = \Delta, \  X_l = L_i \times L_j \\
        \mc{O} & X_k = X_l = \Delta
    \end{cases}
\end{equation}
Also, define the {\bf differential} $\mu^1_{\w^2}$ to be the differentials
coming from the correspondences in the Lemmas above. It remains to define
gradings and construct $\ainf$ operations, which we do in reverse order.

\subsection{The \texorpdfstring{$\ainf$}{A-infinity} category} \label{unfoldingproduct}
To complete the construction of $\w^2$,  we construct higher $\ainf$ operations
$\mu^d_{\w^2}$, $d \geq 2$. First, suppose we have {\it fixed a
universal and conformally consistent Floer datum for pairs of glued discs and
genus-0 open closed strings.} Now, consider the space of labeled associahedra 
\begin{equation}
    \mc{R}^d_{\mathbf{L}^2}
\end{equation}
with label set the relevant Lagrangians in $M^2$:
\begin{equation} 
    \mathbf{L}^2 = \{\Delta\} \cup \{L_i \times L_j | L_i, L_j \in
    \ob \mc{W}\}.
\end{equation}

Consider first the case where all Lagrangians are split.  Discs in $M^- \times
M$ solving the inhomogenous Cauchy-Riemann equation with respect to a split
Hamiltonian in $\tilde{\mc{H}}(M^-\times M)$ split almost complex structure
$(-J,J)$ and split Lagrangian boundary conditions are exactly pairs of discs
$u_1$, $u_2$ with the {\it same conformal structure} (up to conjugation)
solving the inhomogenous Cauchy-Riemann equation with respect to $\omega$, $J$
and respective Lagrangian boundary conditions. The relevant moduli space of
abstract discs is the {\it diagonal associahedron} 
\begin{equation}
    \overline{\mc{R}}^d \stackrel{\Delta_d}{\longhookrightarrow} _{\emptyset,\mf{T}_{max}}\overline{\mc{R}}_{d,d}.
\end{equation}
For labeling sequences $\vec{L}^2$ from $\mathbf{L}^2$
not containing $\Delta$, we can think of $\Delta_d$ as an embedding of {\it
labeled moduli spaces} 
\begin{equation}
    (\Delta_d)_{\vec{L}^2}: (\mc{R}^d)_{\vec{L}^2} \lra
    (\mc{R}_{d,d})_{\mathbf{L}} 
\end{equation}
in the obvious manner: if a boundary component of $S \in
\overline{\mc{R}}^d_{\mathbf{L}^2}$ was labeled $L_i \times L_j$, applying
$\Delta_d$, label the respective component of the first factor $L_i$ and the
second component $L_j$.  

\begin{defn}
    Define the operation $\mu^d_{\w^2}$, for sequences of Lagrangians
    $\vec{L}^2$ in $\mathbf{L}^2$ not containing $\Delta$, to be the operation
    controlled by the image of $(\Delta_d)_{\vec{L}^2}$ as in 
    Equation (\ref{operation2}), with sign twisting datum given by the image of
    the sequential sign twisting datum 
    \begin{equation}
        \vec{t}_d = (1, \ldots, d)
    \end{equation}
    in the following sense: twist inputs in the image of
    $(\Delta_d)_{\vec{L}^2}$ of a boundary point $z_j$ by weight $j$.
\end{defn}
Now, let us give a more general construction of the operations, for cases
including $\Delta$.  Let $S$ be a disc in $\mc{R}^d$ with labels $\vec{L}^2$
from $\mathbf{L}^2$, with at least one label equal to $\Delta$. Let
\begin{equation}
    D(\vec{L}^2)
\end{equation} 
be the set of indices of boundary components of $S$ labeled
$\Delta$. 
Then, let 
\begin{equation}
    \mf{T}_{max} = \{ (1,1), (2,2), \ldots, (d,d)\} \end{equation}
be the maximal boundary identification data and let 
\begin{equation}
    \mf{S}(\mathbf{L}^2) = \{(i,i)| i \in D(\vec{L}^2)\}
\end{equation}
be the set of boundary components determined by the positions of $\Delta$.
Finally, define 
\begin{equation}
    \mathbf{\Phi}_{\vec{L}^2}(\mc{R}^d) := _{\mf{S}(\mathbf{L}^2),\mf{T}_{max}} \mc{R}_{d,d}
\end{equation}
Label the boundary components of the resulting pair of discs as follows: if $\partial_k S$ was
labeled $L_i \times L_j$, then in $\mathbf{\Phi}_{\vec{L}}(S)$, the left image
of $\partial_k S$ will be labeled $L_i$ and the right of 
$\partial_k S$
will
be labeled $L_j$. If $\partial_k S$ was labeled $\Delta$, then it will become
part of a boundary identification and disappear under gluing so there is
nothing to label. 

\begin{defn}
Define the operation \begin{equation}\mu^d_{\w^2},\end{equation} for
sequences of Lagrangians $\vec{L}^2$ in $\mathbf{L}^2$,  to be the
operation controlled by the image of $\mathbf{\Phi}_{\vec{L}^2}$ as in 
Equation (\ref{operation2}).
\end{defn}

\begin{figure}[h] 
    \caption{An example of the labeled gluing $\mathbf{\Phi}_{\vec{L}^2}$. \label{prodgluing1}}
    \centering
    \includegraphics[scale=0.7]{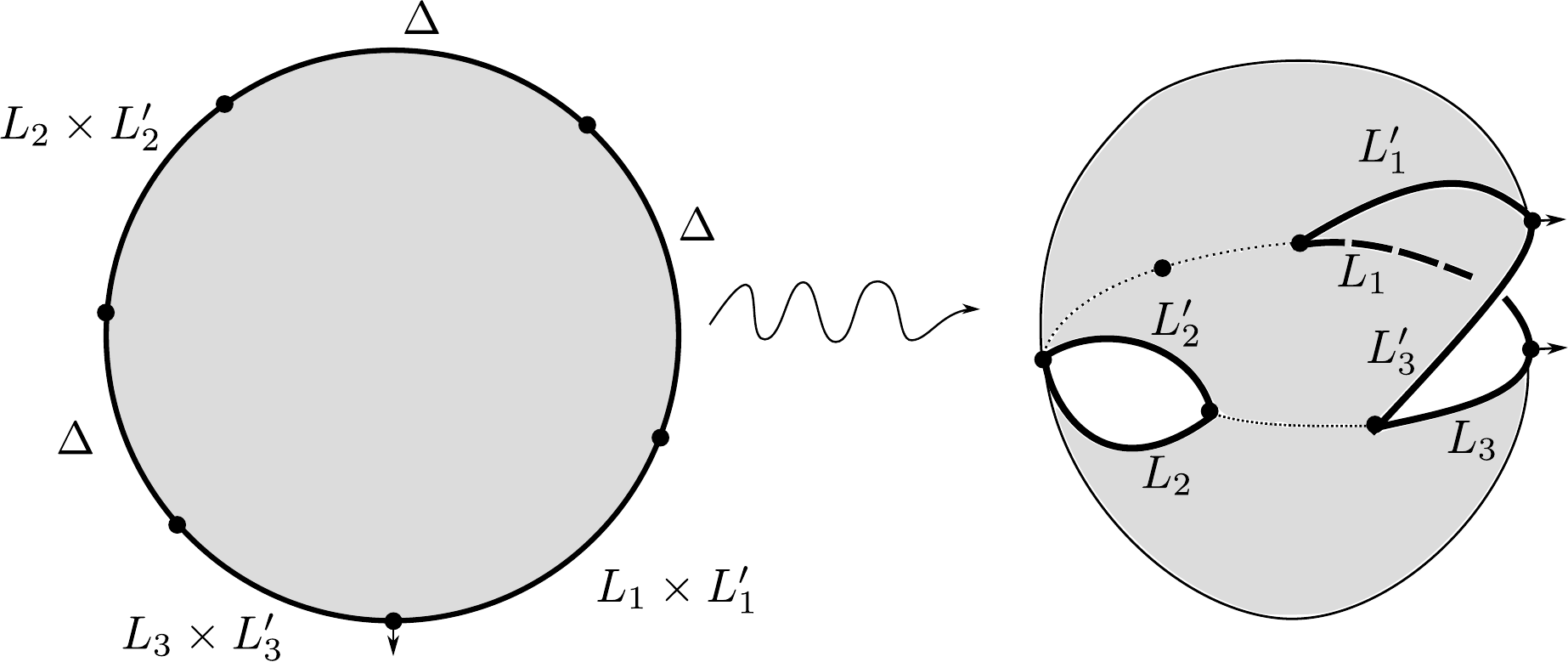}
\end{figure}

Because the unfolding maps $\mathbf{\Phi}_{\vec{L}^2}$ are embeddings of
associahedra,
\begin{prop}
The operations $\mu^d_{\w^2}$ as constructed satisfy the $\ainf$ equations.
\end{prop}

\subsection{Gradings}
The identifications of chain complexes in Section \ref{floersplit} give us
an identification of {\it complexes with relative grading}. In this section, we
would like to make the subtle observation that naively attempting to inherit
the absolute grading from $M$ under these identifications will result in the
$\ainf$ operations having the wrong degree.
\begin{prop}\label{gradings}
    Fix a choice of gradings in $M$. Then, choose gradings for $M^2$ in
    the following manner: given correspondences of generators for morphism complexes
    \def\llr{\longleftrightarrow}
    \begin{equation}
        \begin{split}
        \chi_{M^2}(L_0 \times L_1, L_0' \times L_1') &\longleftrightarrow \chi(L_0' , L_0) \times \chi(L_1,L_1') \\
        x &\llr \hat{x} = (x_1,x_2)
    \end{split}
    \end{equation}
    \begin{equation}
        \begin{split}
            \chi_{M^2}(L_0 \times L_1, \Delta) &\llr \chi(L_1, L_0) \\
            z &\llr \hat{z}
        \end{split}
    \end{equation}
    \begin{equation}
        \begin{split}
            \chi_{M^2}(\Delta, L_0 \times L_1) &\llr \chi(L_0, L_1) \\
            w &\llr \hat{w}
        \end{split}
    \end{equation}
    \begin{equation}
        \begin{split}
            \chi_{M^2}(\Delta, \Delta) &\llr \mc{O} \\
            y &\llr \hat{y}
        \end{split}
    \end{equation}
    assign gradings as follows:
    \begin{align}
        \deg x &= \deg \hat{x} = \deg x_1 + \deg x_2 \\
        \deg z &= \deg \hat{z}\\
        \deg w &= \deg \hat{w} + n\\
        \deg y &= \deg \hat{y}.
    \end{align}
    For this choice, the operations $\mu^d_{\w^2}$ constructed in the previous
    section are of degree $2-d$, thus forming an $\ainf$ structure.
\end{prop}
\begin{proof}
    There are two proofs of this fact. In the first, we can treat the numbers
    $\deg(x)$, $\deg(y)$ as black boxes and verify that the degree assignment
    given above makes the $\ainf$ operations $\w^2$ have correct degree $2-d$
    for any sequence of labeled Lagrangians. This could be done as follows:
    Take an arbitrary labeling $\vec{L}^2$ of Lagrangians, some of which are
    $\Delta$ and some of which are $L_i \times L_j$'s, and calculate the number
    of boundary components, number of boundary outputs, and number of interior
    outputs of the resulting open-closed string
    $\pi_{\mf{S}}(\mathbf{\Phi}_{\vec{L}^2}(S))$, thus arriving at the
    dimension (and therefore degree) of the operation controlled by
    $\mathbf{\Phi}_{\vec{L}^2}$. The main observation here is that,
    inductively, any sequence of consecutive labels of $\Delta$ that do not
    appear at the end of the sequence shift the index by $n$, either 
    by gluing a pair of discs together for the first time,
    by adding an additional boundary component or 
    or by turning a boundary output into an interior output (note that interior
    outputs are only formed if there are $\Delta$ labels on both ends, an edge
    case). Correspondingly, any such sequence contributes a
    term of the form $\hom(\Delta, L_i \times L_j)$.
    Thus in terms of the grading given above, the operation continues to
    have degree $2-d$.

Alternatively, we give a conceptual argument, assuming that $M$ is a compact
manifold. Suppose we had chosen a grading for $\Delta$ such that
$\hom(L_1\times L_2, \Delta) \simeq \hom(L_1,L_2)$ as graded complexes. Then,
by Poincar\'{e} duality on $M^2$ then on $M$, we must have that
\begin{equation}
    \hom(\Delta,L_1\times L_2) \simeq \hom(L_1 \times L_2, \Delta)^{\vee}[2n] \simeq \hom(L_2,L_1)^{\vee}[2n] \simeq \hom(L_1,L_2)[n].
\end{equation}
Of course, Poincar\'{e} duality fails in our situation, but this argument gives
a reasonable sanity check regarding gradings.
\end{proof}

\section{From the product to bimodules} \label{quiltsection}
\subsection{Moduli spaces of quilted strips}
The next three definitions are due to Ma'u \cite{Mau:2010lq}:
\begin{defn} Fix $-\infty < x_1 < x_2 < x_3 < \infty$. A {\bf 3-quilted line}
    consists of the three parallel lines $l_1, l_2, l_3$, each of which is a
    vertical line $\{x_j + i\R\}$ considered as a
    subset of $[x_1,x_3] \times (-\infty,\infty) \subset \C$.
\end{defn}
\begin{defn}
    Let $\mathbf{r} = (r_1, r_2, r_3) \in \Z^{3}_{\geq 0}$. A {\bf 3-quilted line
    with $\mathbf{r}$ markings} consists of the data $(Q, \mathbf{z}_1,
    \mathbf{z}_2, \mathbf{z}_3)$, where $Q$ is a 3-quilted line, and each
    vector $\mathbf{z}_i = (z^1_i, \ldots, z^{r_i}_i)$ is an upwardly ordered
    configuration of points in $l_i$, i.e. $\mathrm{Re}(z_i^k) = l_i$ and
    $-\infty < \mathrm{Im}(z_i^1) < \mathrm{Im}(z_i^2) < \cdots <
    \mathrm{Im}(z_i^{r_i}) < \infty$. 
\end{defn}
\noindent 
There is a free and proper $\R$ action on such quilted lines with markings,
given by simultaneous translation in the $\R$ direction.  
\begin{defn}
    The {\bf moduli space of 3-quilted, $\mathbf{r}$-marked strips}
    $Q(3,\mathbf{r})$ is the set of such 3-quilted lines with $\mathbf{r}$
    markings, modulo translation. 
\end{defn}

\begin{figure}[h] 
    \caption{A quilted strip with $(3,5,2)$ markings. \label{quilt1}}
    \centering
    \includegraphics[scale=0.7]{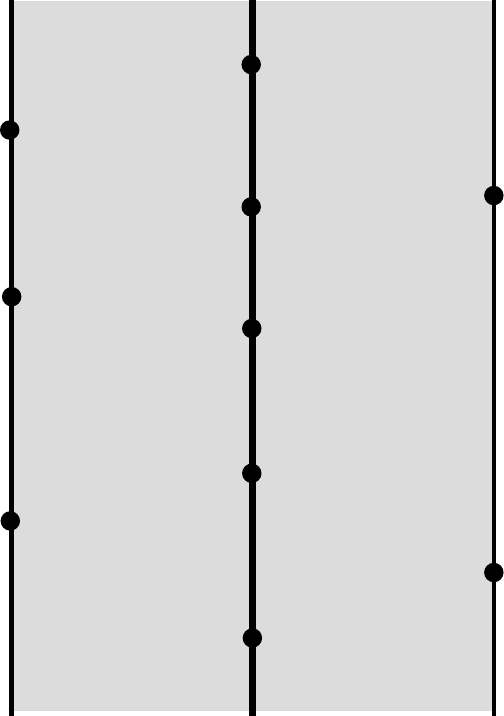}
\end{figure}

 Ma'u also gives a description of  the Deligne-Mumford
 compactification  \[\overline{Q(3,\mathbf{r})}\] of the above moduli space, the
 {\bf moduli space of stable, nodal 3-quilted lines with $\mathbf{r}$
 markings}. Strata consist of multi-level broken 3-quilted lines with stable
 discs glued to marked points on each of the three lines at any level. The
 manifold with corners structure near these strata comes from gluing charts,
 which are similar to ones we have already written down.  We refer the reader
 to \cite{Mau:2010lq}*{\S 2} for more details on this moduli space, but in
 codimension 1
 \begin{prop}\label{boundaryquilt}
     The boundary $\partial \overline{Q(3,\mathbf{r} = (r_1, r_2, r_3))}$ is covered by the
     images of the codimension 1 inclusions
     \begin{equation} \label{boundaryquiltinclusions}
         \begin{split}
             \overline{Q(3,(a,b,c))} \times \overline{Q(3,(r_1-a,r_2-b,r_3-c))} &\ra \partial \overline{Q(3,\mathbf{r})} \\
             \overline{Q(3,(a+1,r_2,r_3))} \times \overline{\mc{R}^{r_1 - a}} &\ra \partial \overline{Q(3,\mathbf{r})} \\
             \overline{Q(3,(r_1,b+1,r_3))} \times \overline{\mc{R}^{r_2 - b}} &\ra \partial \overline{Q(3,\mathbf{r})} \\
             \overline{Q(3,(r_1,r_2,c+1))} \times \overline{\mc{R}^{r_3 - c}} &\ra \partial \overline{Q(3,\mathbf{r})}.
         \end{split}
     \end{equation}
 \end{prop}

We will use the open space $Q(3,\mathbf{r})$ to construct operations controlled
by various glued pairs of discs. The codimension 1 compactification of the
resulting moduli spaces we consider will not quite be
(\ref{boundaryquiltinclusions}), but will only differ by some strata whose
associated operations are zero.
\begin{defn} 
    Let $\mathbf{L}_1$, $\mathbf{L}_2$, $\mathbf{L}_3$ be sets
    of Lagrangians in $M$, $M^2$, and $M$ respectively. A {\bf
    Lagrangian labeling from} $(\mathbf{L}_1, \mathbf{L}_2, \mathbf{L}_3)$ for
    a 3-quilted line with $\mathbf{r}$ markings
    $(Q,\mathbf{z}_1,\mathbf{z}_2,\mathbf{z}_3)$ consists of, for each $i$, an
    assignment of a label in $\mathbf{L}_i$ to each of the $r_i+1$ components
    of the punctured line $l_i - \mathbf{z}_i$. The {\bf space of 3-quilted
    lines} with $\mathbf{r}$-{\bf markings} and $(\mathbf{L}_1,
    \mathbf{L}_2,\mathbf{L}_3)$ {\bf labels} is denoted
    $Q(3,\mathbf{r})_{(\mathbf{L}_1,\mathbf{L}_2, \mathbf{L}_3)}$.
\end{defn}

\subsection{Unfolding labeled quilted strips} 
Fix the label set $\hat{\mathbf{L}} = (\mathbf{L}, \mathbf{L}^2, \mathbf{L})$.
Let $S$ be a stable labeled, 3-quilted strip with $\mathbf{r}$-markings, $S \in
Q(3,\mathbf{r})_{\hat{\mathbf{L}}}$, labeled by $\hat{\vec{L}} = (\vec{L}_0,
\vec{L}^2, \vec{L}_1)$. We associate to $S$ a pair of glued discs
\begin{equation}
    \mathbf{\Psi}_{\hat{\vec{L}}}(S).
\end{equation}
in a manner analogous to the construction of $\mathbf{\Psi}$ in Section
\ref{unfoldingproduct}.
From a 3-quilted line with marked points $S$, consider the substrips $-S_1$ and
$S_2$, where $S_i$ is $(i = 1,2)$ given by the regions in between and
including lines $l_i$ and $l_{i+1}$ ($-S_1$ denotes the reflection of $S_1$
across a vertical axis). 

$-S_1$ and $S_2$ are conformally discs with boundary marked points
$\mathbf{z}_i \cup \mathbf{z}_{i+1} \cup \{a_i^\pm\}$, where $a_i^\pm$ are the
marked points corresponding in the strip-picture to $\pm  \infty$. Denote the
connected components of the line $l_2 - \mathbf{z}_2$ by $\partial_j^2 S$,
$j=1, \ldots, r_2 + 1$, and the images of these boundary components in $S_i$ by
$\partial_j^2 S_i$. Pick some conformal map $\phi$ from the strip to a disc
with marked points at $\pm \infty$ sent to $\pm 1$. Apply this same conformal
map to $-S_1$ and $S_2$ and call the results $-\tilde{S}_1$, $\tilde{S}_2$.
By construction $\tilde{S}_1$ and $\tilde{S}_2$ have $r_2 + 1$ coincident points
\begin{equation}
    \mf{T} = \{(1,1), \ldots, (r_2+1, r_2+1)\}
\end{equation}
coming from the marked points on the strip $l_2$ and the point at $+\infty$.

Now, define the boundary identification
\begin{equation}
    \mf{S}(\hat{\vec{L}}):= \{(i,i) | 1 \leq i \leq r_2+1,\ \partial_i^2
    S\mathrm{\ is\ labeled\ }\Delta\}.  
\end{equation}

Thus, we can define
\begin{equation} \label{unfoldingquiltdefinition}
    \mathbf{\Psi}_{\hat{\vec{L}}} (Q(3,(r_1,r_2,r_3))) := 
{_{\mf{S}(\hat{\vec{L}}),\mf{T}}} \overline{\mc{R}}_{r_2+r_1+1,r_2+r_3+1}.
\end{equation}

The resulting space is labeled as follows: The connected components of $l_i -
z_i$ for $i=1,3$ in the image of $\mathbf{\Psi}$ retain the same
labeling. If $\partial^2_j S$ was labeled by some $L_s \times L_t$, then label
the image of $\partial^2_j S_1$ by $L_s$ and the image of $\partial^2_j S_2$ by
$L_t$.

\begin{figure}[h] 
    \caption{An example of the quilt unfolding $\mathbf{\Psi}_{\hat{\vec{L}}}$. \label{quiltunfoldingimg}}
    \centering
    \includegraphics[scale=0.7]{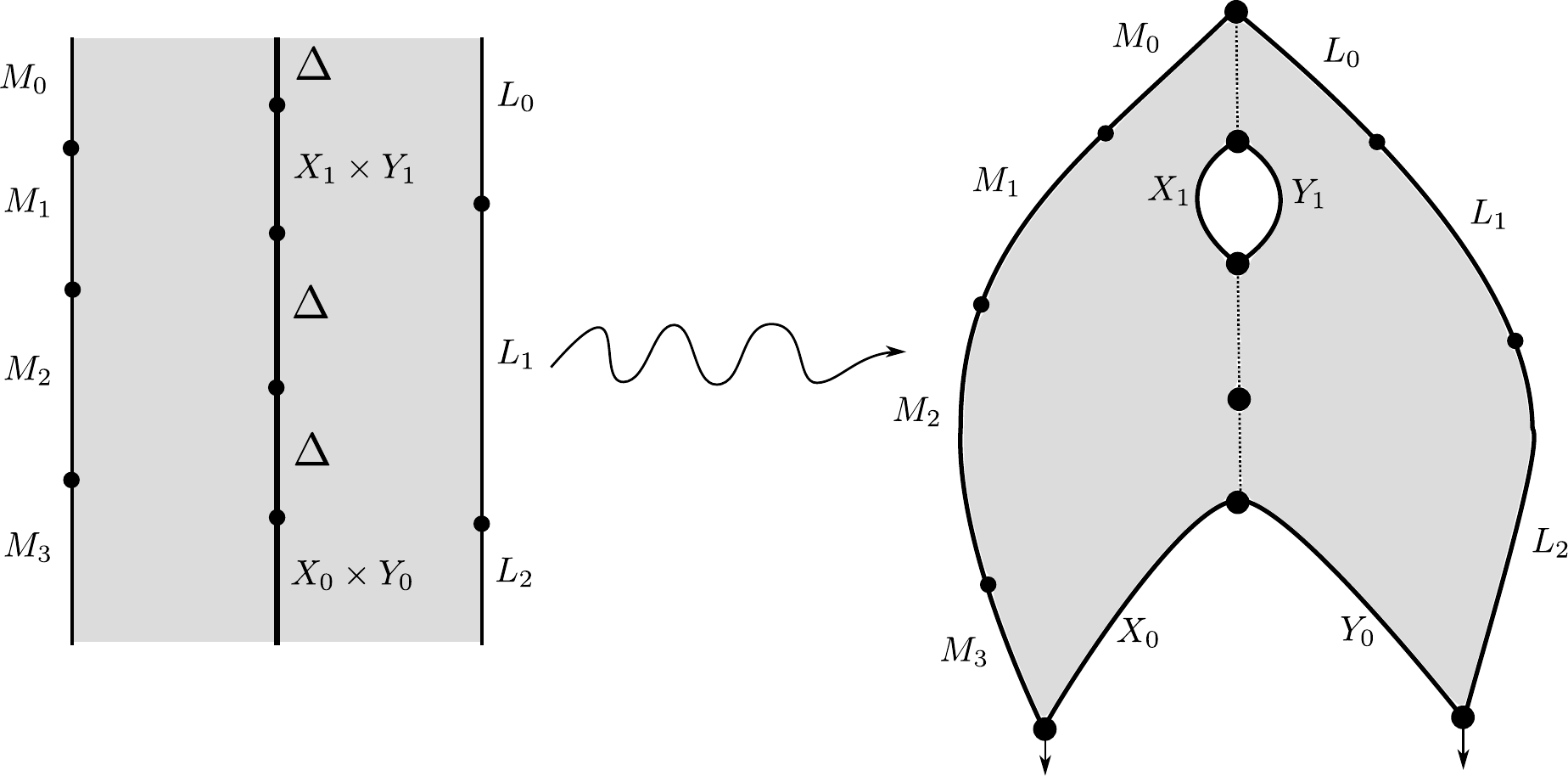}
\end{figure}

\subsection{The \texorpdfstring{$\ainf$}{A-infinity} functor}
Using the above embeddings of labeled moduli spaces, we construct an
$\ainf$ functor 
\begin{equation}
    \mathbf{M} : \mc{W}^2 \longrightarrow \w\bimod\w.
\end{equation}
On an object $X \in \mc{W}^2$, the bimodule $\mathbf{M}(X)$ is specified by the following data:
\begin{itemize}
    \item for pairs of objects $A, B \in \ob \w$,
        $\mathbf{M}(X)(A,B)$ is generated as a graded vector space by
        $\chi_{M^2}(A\times B,X)$, which we recall to be: 
        \begin{equation}
            \begin{cases}
                \chi(L_1,A) \times \chi(B,L_2) & X = L_1 \times L_2 \\
                \chi(B,A) & X = \Delta
            \end{cases}
        \end{equation}
    \item differential \begin{equation}
            \mu^{0|1|0}_{\mathbf{M}}: \mathbf{M}(X)(L,L') \lra   \mathbf{M}(X)(L,L')
        \end{equation}
        which is exactly the differential $\mu^1_{\w^2}$ on
        $\hom_{\w^2}(L\times L',X)$, counting pairs of strips modulo
        simultaneous automorphisms.  
    \item for objects $(A_0, \ldots, A_r, B_0,
        \ldots, B_s)$, higher bimodule
        structure maps 
        \begin{equation} \label{bimodstructurequilt}
        \begin{split}
            \mu^{r|1|s}_{\mathbf{M}}:&\hom_{\w}(A_{r-1},A_r) \times \cdots
            \hom_{\w}(A_0,A_1) \times \mathbf{M}(X)(A_0,B_0) \times \\
            &\times\hom_{\w}(B_1,B_0) \times \cdots \times \hom_{\w}(B_{s},B_{s-1})
            \longrightarrow \mathbf{M}(X)(A_r,B_s).
        \end{split}
        \end{equation}
        These maps are the ones determined by the moduli space 
        \begin{equation}
            \overline{\mathbf{\Psi}_{\hat{\vec{L}}}(Q(3,(r,0,s))_{\hat{\vec{L}}})}
        \end{equation}
        in the sense of equations (\ref{operationlabeled1}) and
        (\ref{operation2}), where 
        \begin{equation}
            \hat{\vec{L}} = ( (A_0, \ldots, A_r), (X), (B_0, \ldots, B_s) ),
        \end{equation}
        using existing choices of Floer data and sign twisting datum
        \begin{equation}
            \vec{t} = (1, 2, \ldots, s, s, s+1, \ldots, s+r)
        \end{equation}
        with respect to the reverse ordering of inputs in
        (\ref{bimodstructurequilt}) (considering $\mathbf{M}(X)(A_0,B_0)$ as a
        single input).
\end{itemize}
The consistency condition imposed on the choice of Floer data pairs of glued
discs and the codimension-1 strata (\ref{boundaryquilt}) imply that
\begin{prop}
    $\mathbf{M}(X)$ is an $\ainf$ bimodule.
\end{prop}
\begin{proof}
    We look at the boundary of the associated one-dimensional moduli spaces.
    The resulting pair of glued discs has a sequential point identification
    $\mf{T} = \{(1,1),(2,2)\}$. We have already examined the boundary strata of
    $_{\mf{S}(\hat{\vec{L}}),\mf{T}} \mc{R}_{r+1+1,r_3+1}$ in Proposition
    \ref{sequentialpointidentifications}. The composed operations corresponding
    to stratas (\ref{bubblingoutofP1}) - (\ref{bubblingLR}) vanish by
    Proposition \ref{extraauto}. The strata (\ref{bubblingoverlapsP}),
    (\ref{bubblingoutofPright}), and (\ref{bubblingoutofPend}) correspond
    exactly to equations of the form 
    \begin{equation}
        \begin{split}
            \mu_{\mathbf{M}(X)}( \cdots \mu_{\mathbf{M}(X)}(\cdots, \mathbf{b}, \cdots), \cdots),\\
        \mu_{\mathbf{M}(X)}(\cdots, \mathbf{b}, \cdots \mu_{\w}(\cdots), \cdots),\\
        \mu_{\mathbf{M}(X)}(\cdots, \mu_{\w}(\cdots), \cdots, \mathbf{b}, \cdots, \cdots)
    \end{split}
    \end{equation}
    respectively, which together comprise the terms of the $\ainf$ bimodule
    relations. There are final terms coming from strip-breaking, corresponding
    to allowing ourselves to apply $\mu^1$ or $\mu^{0|1|0}$ before or after
    applying $\mu_{\mathbf{M}(X)}$. Verification of signs is as in Appendix
    \ref{orientationsection}.
\end{proof}
Given objects $X_0, \ldots, X_d \in \w^2$, the higher terms of the functor are maps
\begin{equation}
    \mathbf{M}^d: \hom_{\w}(X_{d-1},X_d) \otimes \cdots \hom_{\w}(X_0, X_1)
    \lra \hom_{\w\bimod\w}(\mathbf{M}(X_0),\mathbf{M}(X_d)) 
\end{equation}
sending 
\begin{equation}
    x_d \otimes \cdots \otimes x_1 \longmapsto \mathbf{m}_{(x_d, \ldots, x_1)} \in \hom_{\w\bimod\w}(\mathbf{M}(X_0),\mathbf{M}(X_d)).
\end{equation}
The bimodule homomorphism $\mathbf{m}_{(x_d, \ldots, x_1)}$ consists of, for
objects $(A_0, \ldots, A_r, B_0, \ldots, B_s)$ in $\w$, maps:
\begin{equation} \label{quilthigherorderbimod}
    \begin{split}
    \mathbf{m}_{(x_d, \ldots, x_0)}^{r|1|s} : &\hom_{\w}(A_{r-1},A_r) \times \cdots
            \hom_{\w}(A_0,A_1) \times \mathbf{M}(X_0)(A_0,B_0) \times \\
            &\times\hom_{\w}(B_1,B_0) \times \cdots \times \hom_{\w}(B_{s},B_{s-1})
            \longrightarrow \mathbf{M}(X_d)(A_r,B_s)
\end{split}
\end{equation}
Letting $\hat{\vec{L}} = ( (A_0, \ldots, A_r), (X_0, \ldots, X_d), (B_0,
\ldots, B_s))$, we define the above operation to be the one controlled in the
sense of Equations (\ref{operationlabeled1}) and (\ref{operation2}) by the
unfolded image 
\begin{equation}
    \overline{\mathbf{\Psi}_{\hat{\vec{L}}}(Q(3,(r,d,s))}_{\hat{\vec{L}}})
\end{equation}
with sign twisting datum
\begin{equation}
    \vec{t} = (1, \ldots, d, 1, \ldots, r, r, r+1, \ldots, r+s)
\end{equation}
with respect to the ordering of inputs given by $x_1, \ldots, x_d$ followed by
the reverse of the order of inputs in (\ref{quilthigherorderbimod}) (as before,
this means that we twist the image of the inputs after unfolding by these
quantities).  The consistency condition for Floer data for open-closed strings
and pairs of discs, along with the codimension 1 boundary of quilted strips
(\ref{boundaryquilt}) imply
\begin{prop}\label{quiltainffunctor}
    The data $\mathbf{M}^d$ as defined above gives an $\ainf$ functor
    \begin{equation}
        \mathbf{M}: \mc{W}^2 \longrightarrow \w\bimod\w.
    \end{equation}
\end{prop}
\begin{proof}
We need to verify the $\ainf$ functor equation, which (as $\w\bimod\w$ is a dg category),
takes the form: 
\begin{equation}
    \mu^1_{\w\!-\!\w} \circ \mathbf{M}^d + \sum_{i_1+i_2=d}  (\mathbf{M}^{i_1} ``\circ" \mathbf{M}^{i_2}) = \mathbf{M} \circ \hat{\mu}_{\w}.
\end{equation}
We examine the boundary strata $_{\mf{S}(\hat{\vec{L}}),\mf{T}}
\mc{R}_{r+d+1,d + r_3+1}$ computed in Proposition
\ref{sequentialpointidentifications}, for $\mf{T}= \{(1,1), \ldots, (d+1,
d+1)\}$.  The first term, $\mu^1_{\w\bimod\w} (\mathbf{M}^d) = \mu_{\mathbf{M}(X_d)} \circ \hat{\mathbf{M}}^d \mp \mathbf{M}^d \circ \hat{\mu}_{\mathbf{M}(X_0)}$, matches up exactly with the strata 
(\ref{bubblingoutofPright}), (\ref{bubblingoutofPend}), and
(\ref{bubblingoverlapsP}) (in the case that one of $\mc{P}$' or $\mc{P}''$ has
size $d+1$). The cases of (\ref{bubblingoverlapsP}) in which neither $\mc{P}'$
or $\mc{P}''$ are maximal give exactly $(\mathbf{M}^{i_1} ``\circ"
\mathbf{M}^{i_2})$, and $\mathbf{M} \circ \hat{\mu_{\w}}$ is given by
(\ref{bubblinginP}). Finally, there is strip-breaking of the geometric moduli
spaces, giving the $\mu^1_{\w}$ portions of the equations, and the remaining
boundary strata (\ref{bubblingoutofP1}) - (\ref{bubblingLR}) vanish by
Proposition \ref{extraauto}. Once more, details on how to fill in the sign
verification are discussed in Appendix \ref{orientationsection}.  
\end{proof}

\subsection{Relation to existing maps}
We observe that the functor $\mathbf{M}$ geometrically packages together a
number of existing algebraic and geometric maps that have been discussed. This
is a bimodule variant of observations made by Abouzaid-Smith
\cite{Abouzaid:2010vn}.

First, examine the functor $\mathbf{M}$ on split Lagrangians.  
\begin{prop}\label{quiltyonedabimodule}
    The bimodule $\mathbf{M}(L_i \times L_j)$ is exactly the tensor product of
    Yoneda modules 
    \begin{equation}
        \mc{Y}^l_{L_i} \otimes_k \mc{Y}^r_{L_j}.
    \end{equation}
\end{prop}
\begin{proof}
    For objects $(A,B)$ in $\w$, $\mathbf{M}(L_i \times L_j)(A,B)$ and
    $\mc{Y}^l_{L_i}(A) \otimes \mc{Y}^r_{L_j}(B)$ are identical as chain
    complexes. The bimodule maps
    $\mu^{r|1|s}_{\mathbf{M}(L_i \times L_j)}$ are zero if $r,s > 0$, by
    Proposition \ref{extraauto}. If $r = 0$ or $s=0$, by Proposition
    \ref{yonedaquilt}, the operations are: \begin{equation}
        \begin{split}
        \mu^{0|1|s}_{\mathbf{M}(L_i \times L_j)} &= id \otimes \mu^s \\
        \mu^{r|1|0}_{\mathbf{M}(L_i \times L_j)} &= \mu^r \otimes id, 
    \end{split}
    \end{equation}
    concluding the proof.
\end{proof}

\begin{prop}\label{quiltkunneth}
    $\mathbf{M}$ is full and faithful on the subcategory generated by objects
    of the form $L_i\times L_j$.  
\end{prop}
\begin{proof}
    The first-order map
    \begin{equation}
        \mathbf{M}^1: \hom_{\w^2}(L_i \times L_j,L_i'\times L_j') \ra
        \hom_{\w\bimod\w}(\mathbf{M}(L_i \times L_j),\mathbf{M}(L_i' \times
        L_j')).
    \end{equation}
    is the operation $(\mathbf{M}^1(\alpha \otimes \beta))^{r|1|s}$ controlled
    by the embeddings
    \begin{equation}
        \mathbf{\Psi}_{(\vec{L}^1,\vec{L}^2,\vec{L}^3)}(Q(3,(r,1,s))) \subset (\mc{R}_{r+2,s+2})_{\hat{L}'}
    \end{equation}
    where $\vec{L}^2 = (L_i \times L_j, L_i' \times L_j')$. On the level of
    unlabeled surfaces, this map takes a 3-quilted line with one marked point
    on the interior line and associates a pair of discs $S_1$, $S_2$, with $r +
    2$ and $s+2$ positive marked points respectively, such that 3 of the marked
    points of $S_1$ (corresponding to $\pm \infty$ and the marked point on the
    interior line) are coincident with 3 of the marked points of $S_2$. 
    By Proposition \ref{threecoincide}, the corresponding operation is
    $\mu^{r+2} \otimes \mu^{s+2}$.
    
    This implies that $\mathbf{M}^1$
    is exactly the first order Yoneda map, followed by the inclusion in
    Proposition \ref{bimodulekunneth}: 
    \begin{equation}
        \begin{split}
        CW^*(L_i, L_i') \otimes CW^*(L_j',L_j) &\stackrel{(\mathbf{Y}^l)^1 \otimes 
        (\mathbf{Y}^r)^1}{\longrightarrow} \hom_{\w\mod}(\mc{Y}^l_{L_i},\mc{Y}^l_{L_i'}) \otimes 
        \hom_{\rmod\w}(\mc{Y}^r_{L_j},\mc{Y}^r_{L_j'}) \\
        &\longhookrightarrow \hom_{\w\bimod\w}(\mc{Y}^l_{L_i} \otimes 
        \mc{Y}^r_{L_j},\mc{Y}^l_{L_i'} \otimes \mc{Y}^r_{L_j'})
    \end{split}
    \end{equation}
     Fullness follows immediately from the fullness of the Yoneda embedding and
     Proposition \ref{bimodulekunneth}.
\end{proof}
Proposition \ref{quiltkunneth} may be regarded as an $\ainf$ version of the
K\"{u}nneth decomposition for Floer homology.  Now, we examine $\mathbf{M}$ and
$\mathbf{M}^1$ for the remaining
object of $\w^2$: $\Delta$.  
\begin{prop}\label{quiltdelta}
    $\mathbf{M}(\Delta)$ is the diagonal bimodule $\w_{\Delta}$.
\end{prop}
\begin{proof}
    Consider the unfolding map $\mathbf{\Psi}$ when the middle strip is
    labeled $\Delta$. The space of quilted strips $Q(3,(r_1,0,r_2))$ is sent to
    the associahedron $\mc{R}^{r_1 + 1 + r_2}$ with a distinguished input
    marked point corresponding to the intersection point at $+\infty$ in the
    quilt. These are exactly the structure maps corresponding to the diagonal
    bimodule.
\end{proof}
\begin{prop}\label{m1equalsco}
    There is a chain-level identification between
    \begin{equation}
        \mathbf{M}^1: \hom_{\w^2}(\Delta, \Delta) \lra \hom_{\w\bimod\w}(\mathbf{M}(\Delta),\mathbf{M}(\Delta))
    \end{equation}
    and 
    \begin{equation}
        {_2}\co : SH^*(M) \lra \hom_{\w\bimod\w}(\w_{\Delta},\w_{\Delta}).
    \end{equation}
\end{prop}
\begin{proof}
    In this case, the relevant space of quilted strips is $Q(3,(r_1,1,r_2))$
    with middle strip Lagrangian labels both $\Delta$.  The relevant boundary
    identification datum is $\mf{S} = \{(1,1),(2,2)\}$. Proposition
    \ref{cofromglue} shows that the operation corresponding to the moduli space
    $_{\mf{S},p(\mf{S})} \overline{\mc{R}}_{k,l}$ is exactly ${_2}\co^{r_1,r_2}$.
    See also Figure \ref{shhhquilts}.
\end{proof}

\begin{figure}[h] 
\caption{The unfolding of $\mathbf{M}^1: \hom_{\w^2}(\Delta, \Delta) \ra
\hom_{\w\bimod\w}(\mathbf{M}(\Delta),\mathbf{M}(\Delta))$ to give the glued pair
of discs corresponding to ${_2}\co$. \label{shhhquilts}} \centering
    \includegraphics[scale=0.7]{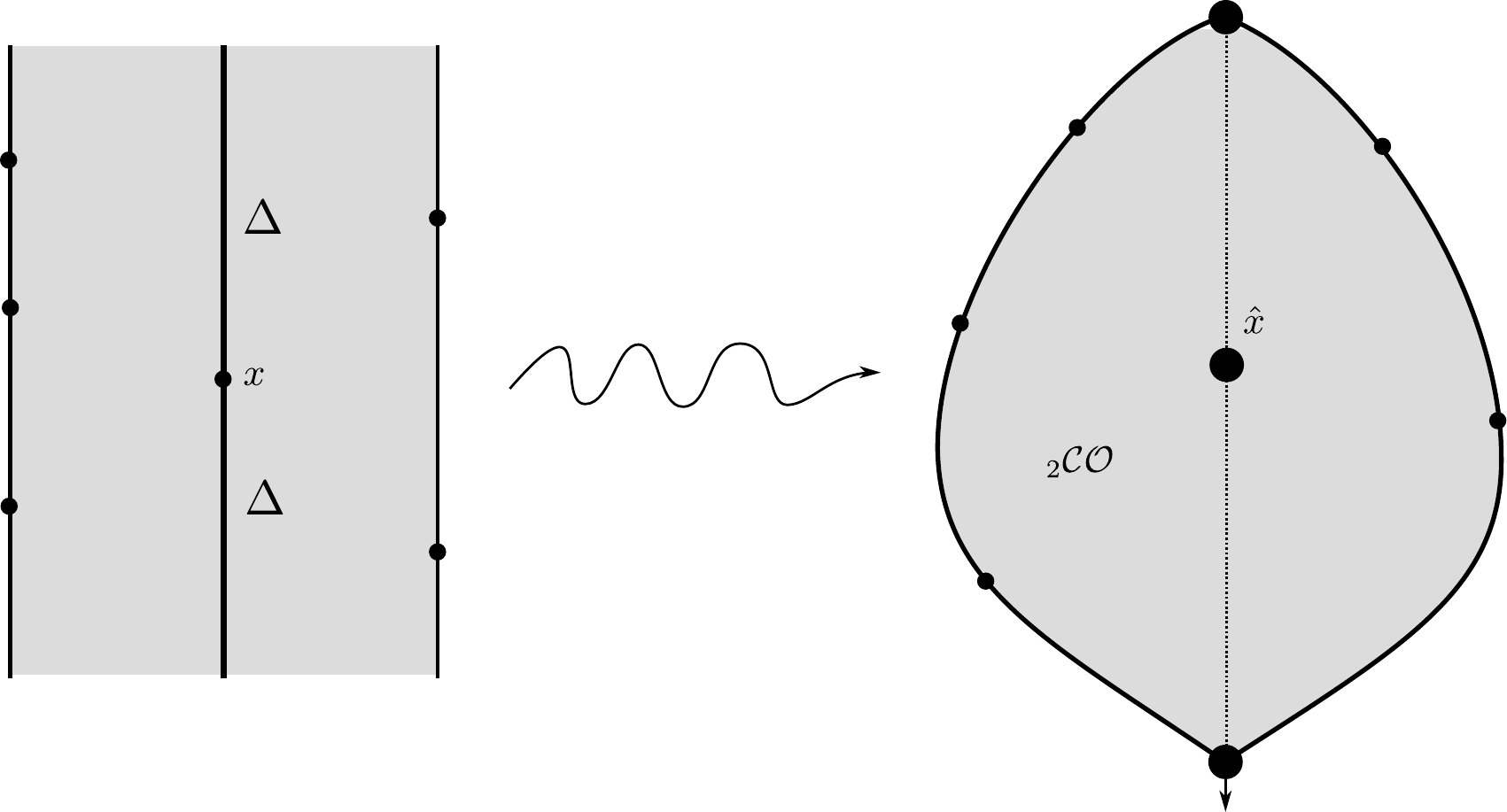}
\end{figure}
Thus, as Proposition \ref{twopointhomotopy1} implies that ${_2\co}$ is homotopic
to $\co$, we see that an isomorphism 
\begin{equation}
    SH^*(M) \stackrel{\co}{\lra} \r{HH}^*(\w,\w)
\end{equation}
is implied by the statement that $\mathbf{M}$ is full on $\Delta$. In turn,
this statement may be reduced by the following proposition to showing that in
the category $\w^2$, $\Delta$ is split-generated by objects of the form $L_i
\times L_j$.

\section{Forgotten points and homotopy units}\label{homotopyunits}
In this section, we introduce an important technical tool used in our result:
homotopy units for glued pairs of discs. We can motivate the need and/or
application of such a tool as follows:

Suppose for a moment that we are in an idealized setting of Lagrangian Floer
theory for a single Lagrangian $L \subset M$, in which we may ignore all issues
of perturbations, transversality of moduli spaces, and obstructedness of Floer
groups.  Let us also for a moment reason using the conceptually intuitive
singular chain variant of Floer theory as developed by \cite{Fukaya:2009ve}. In
this framework, generators of the Floer chain complex $CF(L,L)$ are given by
{\it equivalence classes} of geometric (singular) cycles in $L$. Given cycles $\mf{b}_1,
\ldots, \mf{b}_k$, we 
define the $\ainf$ structure map $\mu^k$ to be:
\begin{equation} \label{evaluation}
    \mu^k(\mf{b}_k, \ldots, \mf{b}_1) := (ev_0)_* [\mathcal{M}^k (\mf{b}_k, \ldots, \mf{b}_1)]
\end{equation}
Here $[\mathcal{M}^k(\mf{b}_k, \ldots, \mf{b}_1)]$ is a ``virtual fundamental chain'' for
the moduli space of holomorphic maps 
\[
u: (D, \bd D, z_1^+, \ldots, z_k^+,z_0^-) \ra (M,L,\mf{b}_1, \ldots, \mf{b}_k,\cdot)
\]
with positive boundary marked points $z_i^+$ constrained to lie on the cycles
$\mf{b}_i$, and negative boundary marked point $z_0$ unconstrained. The
notation from (\ref{evaluation}) simply means that we take as result the cycle
``swept out'' by the marked point $z_0^-$ in this moduli space. 

In this (unfortunately imaginary) setting, there is a canonical choice of {\it
strict unit} for the $\ainf$ algebra $CF(L,L)$: the fundamental class $[L]$.
This cycle satisfies the fundamental property that for $u: (D, \partial D) \ra (M,L)$, 
{\it the condition that $z_i \in \partial D$ lies on the cycle $[L]$ is an empty constraint}.

Let us very informally show that this property gives $[L]$ the structure of a
strict unit. First, work in the stable range $d \geq 3$. There is a projection
map
\begin{equation}
    \pi_j: \mc{R}^d \ra \mc{R}^{d-1},
\end{equation}
forgetting the $j$th marked point. In the above setting, $\pi_j$ extends to a
map between moduli spaces of stable maps:
\begin{equation}
    (\pi_j)_*: \mathcal{M}^d(\mf{b}_1, \ldots, \mf{b}_k) \ra
    \mathcal{M}^{d-1}(\mf{b}_1, \ldots, \mf{b}_{j-1}, \mf{b}_{j+1}, \ldots,
    \mf{b}_k) \end{equation}
Suppose $\mf{b}_j = [L]$, an empty constraint on the marked point $z_j$. This
implies that $(\pi_j)_*$ is a submersion with one-dimensional fibers,
corresponding to the location of the $j$th marked point. In particular, 
\[
\dim \mathcal{M}^d(\mf{b}_1, \ldots, \mf{b}_k) = \dim
\mathcal{M}^{d-1}(\mf{b}_1, \ldots, \mf{b}_{j-1}, \ldots, \mf{b}_{j+1}, \ldots,
\mf{b}_k) + 1,
\] 
which implies that $(ev_0)_* \mathcal{M}^d(\mf{b}_1, \ldots, \mf{b}_k)$ is a
degenerate chain and thus zero on homology. Hence
\begin{equation}\label{fooounit}
\mu^d(\ldots, [L], \ldots) = 0.
\end{equation}
When $d=2$, we leave it as a pictorial exercise to the interested reader to
``prove'' that 
\[
\mu^2([L],x) = \pm \mu^2(x,[L]) = \pm x.
\]
Even in this setting, there are a number of issues:
\begin{itemize}
    \item in order to obtain transversality, one needs to coherently perturb
        the holomorphic curve equations in a domain-dependent manner and there
        is no known way to make the forgetful map compatible with these
        perturbations. These perturbations occur in the setting of {\it
        Kuranishi structures}, making them even less likely to be compatible
        with the forgetful map.
    \item strictly speaking, this moral argument only
        proved the equality
        (\ref{fooounit}) modulo degenerate chains. To move to an $\ainf$
        structure on $H^*(L)$, a host of additional arguments are required,
        including homological perturbation theory. The pay-off is that after
        some additional work one obtains a strictly unital structure on
        $H^*(L)$.
\end{itemize}
The reader is referred to \cite{Fukaya:2009qf}*{Ch. 7, \S 31} for details
on Fukaya-Oh-Ohta-Ono's approach to these problems.  

In our setting, generators are time-1 Hamiltonian chords, so we have an
additional issue:
\begin{itemize}
    \item even if transversality were not an issue, we have no time-1 chord(s)
        $x$ with the property that imposing an asymptotic condition to $x$ is a
        forgetful map.  
\end{itemize}
The remedy that seems to have been used in the literature most is this:
construct a homology level unit geometrically, and then apply algebraic results
of Seidel to obtain a quasi-isomorphic $\ainf$ algebra that is strictly unital. 

However, we are in a setting where we do not just care about algebraic
properties of strictly unital $\ainf$ categories. We would like to
carefully analyze certain operations on $\w^2$ controlled by forgetful maps
applied to submanifolds of moduli spaces of open-closed surfaces and pairs of
discs. To be able to use such operations in $\w^2$, we will need them to be
homotopic to existing operations.

If the reader wished to skip most of this section, the eventual
punchline is this: {\it Given some operations controlled by a submanifold
$\mc{Q}$ of open-closed strings, the construction of homotopy units gives us a
quasi-isomorphic category with additional elements $e^+_j \in \chi(L_j,L_j)$
such that the operation $\mc{Q}(\cdots e^+_j \cdots)$ is controlled by the
submanifold $\pi_j(\mc{Q})$.}

\subsection{Forgotten marked points}\label{forgottensubsection}
We begin with a notion of what it means to have forgotten a boundary marked
point in Floer-theoretic operations. Since the construction is identical for
discs and pairs of glued discs, we initiate them in parallel. Strictly
speaking, we do not need the single-disc construction in our paper, but it is
no additional work and may be foundationally useful. Also, we only consider
forgotten points on pairs of {\it identical discs modulo simultaneous
automorphism}, the only case that arises for us.
\begin{defn}
    The {\bf moduli space of discs with $d$ marked points and $F
    \subseteq \{1, \ldots, d\}$ forgotten marked points}, denoted 
    \begin{equation}
        \mc{R}^{d,F}
    \end{equation}
    is exactly the moduli space of discs $\mc{R}^{d}$ with 
    marked points labeled as belonging to $F$. 
\end{defn}
\begin{defn}
    The {\bf moduli space of $\mf{S}$-glued pairs of discs} with {\bf $(k,l)$ marked points}, {\bf $\mf{T}$ point identifications,} and
    \begin{equation} 
        (F_1, F_2) \subseteq (\{1, \ldots, k\}, \{1, \ldots, l\}),
    \end{equation}
    {\bf forgotten points}, denoted 
    \begin{equation}
        _{\mf{S},\mf{T}}\mc{R}_{k,l}^{(F_1,F_2)}
    \end{equation} 
    is the image of the diagonal associahedron in the moduli space of
    glued pairs of discs $\mc{R}_{k,k,\mf{S}}$ with positive
    marked points on each disc corresponding to $F_1$ and $F_2$ labeled as {\bf
    forgotten points}. Crucially, $F_1$ and $F_2$ must satisfy the following
    conditions: 
    \begin{itemize}
        \item $F_1$ and $F_2$ are subsets of the left and right identified
            points respectively. Namely, 
            \begin{equation}
                F_i \subset \pi_i(\mf{T}), 
            \end{equation} 
            where $\pi_i$ is projection onto the $i$th component. 
        \item $F_1$ and $F_2$ are not associated to a boundary identification, i.e.
            \begin{equation}
                F_i \cap \pi_i(p(\mf{S})) = \emptyset
            \end{equation}
        \item $F_1$ and $F_2$ do not contain both the left and right points of
            any identification, i.e.  \begin{equation}
                (F_1 \times F_2) \cap \mf{T} = \emptyset
            \end{equation}
    \end{itemize}
\end{defn}
\begin{rem}
    Put another way, the conditions $F_1$ and $F_2$ must satisfy correspond to
    the following from the viewpoint of tricolored discs developed in Section
    \ref{pairs}: $F_1$ and $F_2$ correspond to disjoint subsets of the points
    colored $LR$, such that neither $F_1$ or $F_2$ is adjacent to a boundary
    component labeled as identified.
\end{rem}
For the purpose of solving Floer's equations, we will be putting the marked
points labeled by $F$, $F_i$ back in. Such points should be thought of as {\it
markers} rather than punctures.
\begin{defn}
    Let $I \subseteq F$. The {\bf $I$-forgetful map}
    \begin{equation}
        \mc{F}_I: \mc{R}^{d,F} \lra \mc{R}^{d-|I|,F'}
    \end{equation}
    associates to any $S$ the surface obtained by putting the points of $I$
    back in and forgetting them. $F'$ in the equation above is the set of
    forgetful points $F-I$, re-indexed appropriately.  
\end{defn}
There is a similar forgetful map for pairs of glued discs, 
\begin{equation}
    \mc{F}_{I_1,I_2} : _{\mf{S},\mf{T}}\mc{R}_{k,l}^{F_1,F_2} \lra 
    _{\mf{S}',\mf{T}'}\mc{R}_{k-|I_1|,l-|I_2|}^{F_1',F_2'}
\end{equation}

We need a notion that corresponds to stability of the underlying disc once we
have forgotten points.  
\begin{defn}
    A disc with $d$ marked points and $F$ forgotten points is {\bf f-stable} or
    {\bf f-semistable} if $d-|F| \geq 2$ or $d - |F| = 1$ respectively.
A pair of discs with $(k,l)$ marked points and $F_1, F_2$ forgotten points is
{\bf f-stable} if $k-|F_1|$, $l-|F_2|$ are both greater than or equal to 1 and one is
greater than or equal to 2. It is {\bf f-semistable} if both of these quantities are
equal to 1.
\end{defn}

In the f-stable range, there are maximally forgetful maps, collectively denoted
$\mc{F}_{max}$:
\begin{align}
        \mc{F}_{max} = \mc{F}_F : \mc{R}^{d,F} &\lra \mc{R}^{d-|F|} \\
        \mc{F}_{max} = \mc{F}_{F_1,F_2} : {_{\mf{S},\mf{T}}}\mc{R}_{k,l}^{F_1,F_2} &\lra
        _{\mf{S}',\mf{T}'}\mc{R}_{k-|F_1|,l-|F_2|}\
\end{align}

The (Deligne-Mumford) compactifications
\begin{align} 
    \overline{\mc{R}}^{d,F}\\
    _{\mf{S},\mf{T}}\overline{\mc{R}}_{k,l}^{F_1,F_2} 
\end{align}
are exactly the usual Deligne-Mumford compactifications, along with the data of
{\it forgotten} labels for the relevant boundary marked points.
Interior positive nodes inherit the label of {\it forgotten} in the following
fashion: 
\begin{defn}
    An interior positive node of a stable representative $S$ of a disc
     or pair of glued discs 
    is said to be a {\bf forgotten node} if and only
    if every boundary marked point in every component above $p$ is a forgotten
    marked point and there are no interior marked points in any component above
    $p$.  
\end{defn}

In the $f$-stable range, stable discs with forgotten marked points have
underlying stable representatives with forgotten points removed.  
\begin{defn}
    A component of a stable representative $S$ of a disc or a pair of glued
    discs is said to be {\bf forgettable} if all of its positive boundary
    marked points (including nodal ones) are forgotten points and it has no
    interior marked points.
\end{defn}
Using the above definitions, one can extend the maximally forgetful map to
compactifications.
\begin{defn}
    Let $S$ be a nodal bordered f-stable surface with forgotten marked
    points. The {\bf associated reduced surface} $\hat{S}$ is the nodal surface
    obtained by
    \begin{itemize}
        \item eliminating all forgettable components
        \item putting back in all forgotten boundary points and forgetting them
        \item if in the $f$-stable range, eliminating any {\it non-main}
            component with only one non-forgotten marked point $p$,  and
            labeling the positive marked point below this component by $p$.
    \end{itemize}
    Define the {\bf induced marked points} of $\hat{S}$ to be the boundary
    marked points that survive this procedure.
\end{defn}
In other words, the nodal surface $\hat{S}$ is obtained from the nodal surface
$S$ by forgetting the points with an $F$ label and then stabilizing the
resulting bubble tree.  
\begin{defn}
    The {\bf maximally forgetful map} $\mc{F}_{max}$, defined for any nodal
    $f$-semistable disc or pair of glued discs is defined to be the map that
    associate to a nodal surface with forgotten marked points $S$ the
    associated reduced surface $\hat{S}$. 
\end{defn}

\begin{defn} \label{floerforgotten}
A {\bf Floer datum} for a stable, $f$-semistable disc or pair of glued discs
with forgotten marked points consists of 
a Floer datum for the associated reduced surface $\hat{S} = \mc{F}_{max}(S)$,
in the sense of Definition \ref{floeropenclosed} or Definition
\ref{floerglued}, satisfying the following conditions:
\begin{itemize}
    \item in the {\bf f-stable range}, it is identical to our previously
chosen Floer datum for $\hat{S}$ thought of as an open-closed string. 
\item in the {\bf f-semistable range}, it is given by the unique
    translation-invariant Floer datum on the strip $\hat{S}$.
\end{itemize}
This implies in particular that the Floer datum only depends on the point
$\mf{F}_{max}(S)$.
\end{defn}

\begin{figure}[h] 
    \caption{Two drawings of a disc with forgotten points (denoted by hollow points). The drawing on the right emphasizes the choice of strip-like ends.\label{forgottenpointex}}
    \centering
    \includegraphics[scale=0.7]{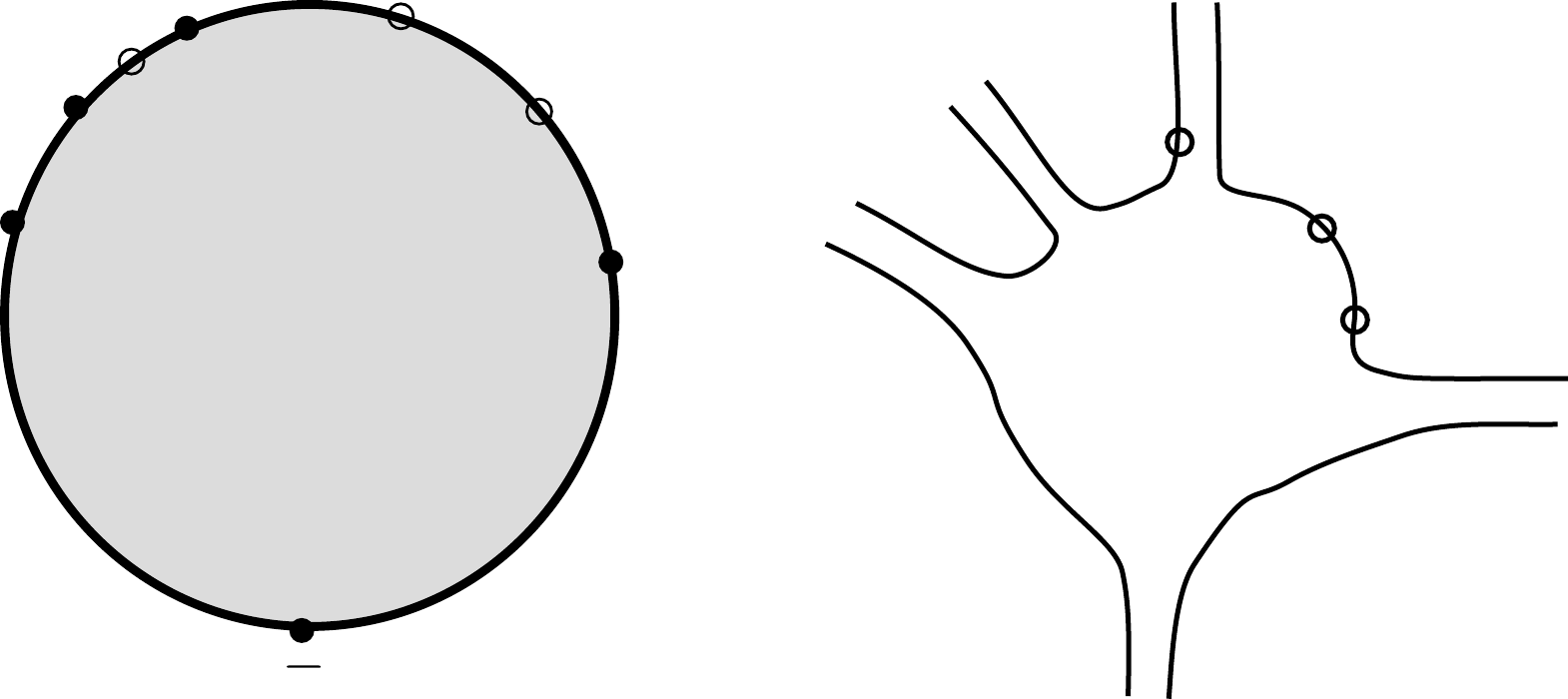}
\end{figure}

\begin{rem} 
    By the above definition, a Floer datum for a pair of discs $P$ with
    $\mf{S}$ boundary identifications, $\mf{T}$ point identifications, and
    $F_1, F_2$ forgotten points is a Floer datum for the open-closed string
    obtained by forgetting the marked points corresponding to $F_1$ and $F_2$,
    stabilizing, and gluing the resulting pairs of discs via $\pi_{\mf{S}}$.
\end{rem}

Because we have chosen our Floer data to be the one we have already chosen for
the underlying reduced open-closed string, we immediately obtain: 
\begin{prop}
    There exists a universal and consistent choice of Floer data for discs or
    pairs of glued discs with forgotten marked points.
\end{prop}

\begin{defn}
    An {\bf admissible Lagrangian labeling} for a surface $S$ with forgotten
    marked points is a choice of Lagrangian labeling that descends to a
    well-defined labeling on the associated reduced surface $\mc{F}_{max}(S)$.
    Namely, if $p$ is any forgotten boundary marked point of $S$, then the
    labels before and after $p$ must coincide. The {\bf reduced labeling} is
    the corresponding labeling on the underlying reduced surface.
\end{defn}
Now, suppose we have fixed a universal and consistent choice of Floer data for
discs. Consider a compact submanifold with corners of dimension $d$ 
\begin{equation}
    \overline{\mc{Z}}^d \longhookrightarrow {_{\mf{S},\mf{T}}} \mc{R}_{k,l}^{F_1,F_2}.
\end{equation}
with an admissible Lagrangian labeling $\vec{L}$. In the usual fashion, fix
input and output chords $\vec{x}_{in}, \vec{x}_{out}$ and orbits $\vec{y}_{in}, \vec{y}_{out}$
for the induced marked points of the gluing $\pi_{\mf{S}}(\mc{F}_{max}
(\overline{\mc{Z}}^d))$, which forgets all points labeled as forgotten and glues
along the boundary components $\mf{S}$. Define
\begin{equation}
    \overline{\mc{Z}}^d(\vec{x}_{out},\vec{y}_{out};\vec{x}_{in},\vec{y}_{in})
\end{equation}
to be the space of maps
\begin{equation}
    \{u: \pi_{\mf{S}}(\mc{F}_{max}(S)) \lra M: S \in \overline{\mc{Z}}^d\}
\end{equation}
satisfying Floer's equation with respect to the Floer datum and asymptotic and
boundary conditions specified by the Lagrangian labeling $\vec{L}$ and
asymptotic conditions $(\vec{x}_{out},\vec{y}_{out}, \vec{x}_{in},
\vec{y}_{in})$.

As before, $h(\mf{S},k,l)$ denote the number of boundary components of any
resulting surface obtained from the gluing.
\begin{lem}
    The moduli spaces
    $\overline{\mc{Z}}^d(\vec{x}_{out},\vec{y}_{out};\vec{x}_{in},\vec{y}_{in})$
    are compact, and empty for all but finitely many
    $(\vec{x}_{out},\vec{y}_{out})$ given fixed inputs
    $(\vec{x}_{in},\vec{y}_{in})$. For generically chosen Floer data, they form
    smooth manifolds of dimension
    \begin{equation}
        \begin{split}
    \dim \overline{\mc{Z}}^d(\vec{x}_{out},\vec{y}_{out};\vec{x}_{in},\vec{y}_{in}):=
            \sum_{x_-\in \vec{x}_{out}} \deg(x_-) + &\sum_{y_- \in \vec{y}_{out}}\deg(y_-) \\
            + (2 - h(\mf{S},k,l) - |\vec{x}_{out}|- 2 |\vec{y}_{out}|)n
        &+d -\sum_{x_+\in \vec{x}_{int}} \deg(x_+) - \sum_{y_+ \in \vec{y}_{in}} \deg(y_+).
    \end{split}
\end{equation}
\end{lem}
In the usual fashion, when the dimension of the spaces
$\overline{\mc{Z}}^d(\vec{x}_{out},\vec{y}_{out};\vec{x}_{in},\vec{y}_{in})$
are zero, we can use natural isomorphisms of orientation lines to count 
(with signs) the number of points in such spaces, and associate operations
\begin{equation}
    (-1)^{\vec{t}}\mathbf{H}_{\overline{\mc{Z}}^d} 
\end{equation} 
from the tensor product of wrapped Floer complexes and symplectic cochain
complexes where $\vec{x}_{in}$, $\vec{y}_{in}$ reside to the tensor product of
the complexes where $\vec{x}_{out}$, $\vec{y}_{out}$ reside. 
We can specify certain submanifolds of the space of forgotten marked points by
applying forgotten labels to various boundary points on spaces of open-closed
strings.
\begin{defn}
    The {\bf forget map}
    \begin{align}
        \mf{f}_{F}: \mc{R}^d &\lra \mc{R}^{d,F}\\
        \mf{f}_{F_1,F_2}: {_{\mf{S},\mf{T}}} \mc{R}_{k,l} &\lra {_{\mf{S},\mf{T}}} \mc{R}_{k,l}^{F_1,F_2}
    \end{align}
    simply marks boundary points with indices in $F$ (or $(F_1,F_2)$) as
    forgotten.
\end{defn}

\subsection{Operations with forgotten points}
Our main application is of course to think of forgotten points as formal units,
either for a disc or pair of discs. It is thus illustrative to see how
operations with forgotten marked points either vanish or reduce to other known
operations.
\begin{prop}
    Let $F \subset \{1,\ldots, d\}$ be a non-empty subset of size $0 < |F| <d$.
    Then the operation associated to $\overline{\mc{R}}^{d,F}$ is zero if $d >
    2$ and the identity operation $I(\cdot)$ (up to a sign) when $d=2$.
\end{prop}
\begin{proof}
    Suppose first that $d>2$, and let $u$ be any solution to Floer's equation
    over the space $\mc{R}^{d,F}$ with domain $S$. Let $p\in F$ be the last element of $F$. Since the Floer data on $S$
    only depends on $\mc{F}_{p}(S)$, we see that maps from $S'$ with $S' \in
    \mc{F}_p^{-1}(\mc{F}_{p}(S))$ also give solutions to Floer's equation with
    the same asymptotics. Moreover, the fibers of the map $\mc{F}_p$ are
    one-dimensional, implying that $u$ cannot be rigid, and thus the associated
    operation is zero.

    Now suppose that $d = 2$, and without loss of generality $F = \{1\}$. Then
    the forgetful map associates to the single point $[S] \in \mc{R}^{2,F}$ the
    unstable strip with its translation invariant Floer datum. We conclude,
    based on Section \ref{stripidentity}, that the resulting operation is the
    identity.
\end{proof}
\begin{rem}
    Actually, one would like this operation to be zero when $|F| = d$ as well.
    However, we have not defined an operation with $|F|=d$, due to the
    instability of the underlying reduced surface. Our solution will be to
    declare this operation to be zero, and check that our declaration is
    compatible with the behavior of boundaries of one dimensional moduli
    spaces.
\end{rem}

\begin{prop}\label{shuffleid1}
    Let $\Delta_{d} \subset \mc{R}_{d,d}$ be the diagonal associahedron.  Let
    $[d]$ denote the set $\{1, \ldots, d\}$, and $(k,l)$ such that pairs of
    discs with $k$,$l$ marked points are stable. Then, the operation given by
    the disjoint union
\begin{equation}
    \coprod_{I \subset [k+l]| |I| = k} (\mf{f}_{I,[k+l]-I})
    (\Delta_{k+l})
\end{equation}
with appropriate orientations is identical to the operation given by
$\mc{R}_{k,l}$. In other words, it is equal to zero when $k$, $l$ $\geq 1$ and
one is $\geq 2$.  
\end{prop}
\begin{proof}
    On the open locus $ _0 \mc{R}_{k,l}$ where none of the $k$ points on the
    first disc and the $l$ points on the second disc are in identical
    positions, we can consider the {\bf overlay} map:
    \begin{equation}
        _0 \mc{R}_{k,l} \lra \coprod_{I \subset [k+l]| |I| = k} (\mf{f}_{I,[k+l] - I}) (\Delta_{k+l}(\mc{R}^{k+1}))
    \end{equation}
    given by marking the $l$ positive marked points on $S_2$ as extra forgotten
    points on $S_1$ and vice versa. On the level of tricolored discs, the
    overlay makes $L$ points $LR$ points with the $R$ component marked as
    forgotten, and makes $R$ points $LR$ points with the $L$ component marked
    as forgotten.  By construction, this map is compatible with Floer data, and
    covers the entire interior of the target.  Since after a perturbation
    zero-dimensional solutions to Floer's equation come from a representative
    on the interior of any source abstract moduli space, we conclude that the
    two operations in the Proposition are identical, modulo sign. See Figure
    \ref{overlay_forget1} for an example of this overlay map.
\end{proof}

\begin{figure}[h] 
    \caption{An example of the overlay map from $_0\mc{R}_{3,2}$ to $\mc{R}_{5,5}^{\{2,4\}, \{1,3,5\}}$. \label{overlay_forget1}}
    \centering
    \includegraphics[scale=0.7]{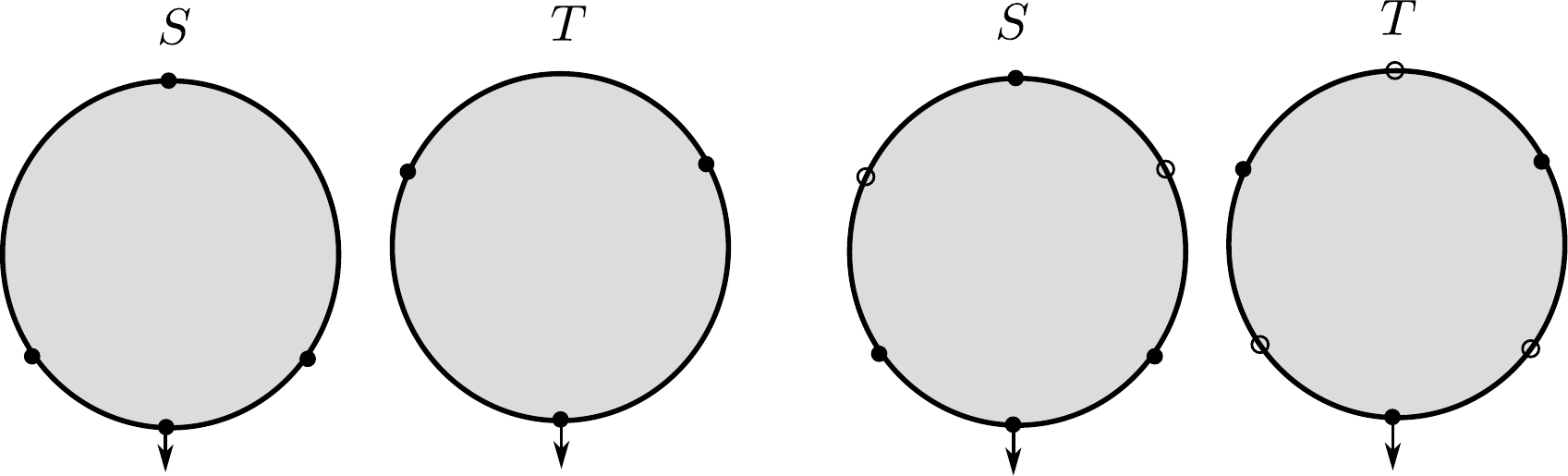}
\end{figure}


\begin{prop} \label{shuffleid2}
    Take boundary identification $\mf{S} = \{(1,1)\}$ and maximal point
    identification $\mf{T}_{max} = \{(1,1), \ldots, (k+l,k+l)\}$. Then, letting $S = \{2, \ldots, k+l-2\}$, the
    operation corresponding to 
    \begin{equation}
        \coprod_{I \subset S| |I|=k-1} \mf{f}_{I,S-I}
({_{\mf{S},\mf{T}_{max}}} \mc{R}_{k,l})
\end{equation}
is $\mu^{k+l-1}$ (with suitable sign twisting datum).
\end{prop}
\begin{proof}
    On the open locus of $_{(1,1),(1,1)} \mc{R}_{k,l}$ where the only
    coincident points are $(1,1)$ and no other points coincide, denoted
    \begin{equation}
    _{(1,1),(1,1)} \mc{R}_{k,l}^0
    \end{equation}
    there is once more an {\it overlay map}
\begin{equation}
    _{(1,1),(1,1)} \mc{R}_{k,l}^0 \lra 
    \coprod_{I \subset S| |I|=k-1} \mf{f}_{I,S-I}
({_{\mf{S},\mf{T}_{max}}} \mc{R}_{k,l})
\end{equation}
given by superimposing left and right discs, and marking points from the right
as forgotten points on the left and vice versa. This gives an isomorphism of
spaces with Floer data on the open locus, so we conclude by applying
Proposition \ref{mufromglue} to calculate the operation associated to $_{(1,1),(1,1)}\mc{R}_{k,l}$.
\end{proof}

\begin{prop} \label{shuffleid3}
Take boundary identification $\mf{S} = \{(k+l,k+l)\}$ and maximal point
identification $\mf{T}_{max} = \{(1,1), \ldots, (k+l,k+l)\}$. Then, letting $S
= \{1, \ldots, k+l-1\}$, the
    operation corresponding to 
    \begin{equation}
        \coprod_{I \subset S| |I|=k-1} \mf{f}_{I,S-I}
({_{\mf{S},\mf{T}_{max}}} \mc{R}_{k,l})
\end{equation}
is $\mu^{k+l-1}$ (with suitable sign twisting datum).
\end{prop}
\begin{proof}
    The proof is identical to the above case, using an overlay map and reducing to Proposition \ref{mufromglue}.
\end{proof}

\begin{prop} \label{shuffleid4}
Take boundary identification $\mf{S} = \{(1,1),(k+l,k+l)\}$ and maximal point
identification $\mf{T}_{max} = \{(1,1), \ldots, (k+l,k+l)\}$. Then, letting $S
= \{2, \ldots, k+l-1\}$, the operation corresponding to 
    \begin{equation}
        \coprod_{I \subset S| |I|=k-2} \mf{f}_{I,S-I}
({_{\mf{S},\mf{T}_{max}}}\mc{R}_{k,l})
\end{equation}
is $_2 \oc^{k-2,l-2}$ (with suitable sign twisting datum).
\end{prop}

\begin{figure}[h] 
    \caption{The overlay map from $_{\mf{S},p(\mf{S})}\mc{R}_{5,4}$ to 
    $_{\mf{S}',p(\mf{S})'} \mc{R}_{7,7}^{\{3,5\},\{2,4,6\}}$, 
    where $\mf{S} = \{ (1,1), (5,4) \} $ 
    and 
    $\mf{S'} = \{(1,1),(7,7)\}$. 
    Forgotten points are marked with rings. \label{ocoverlay}
    }
    \centering
    \includegraphics[scale=0.8]{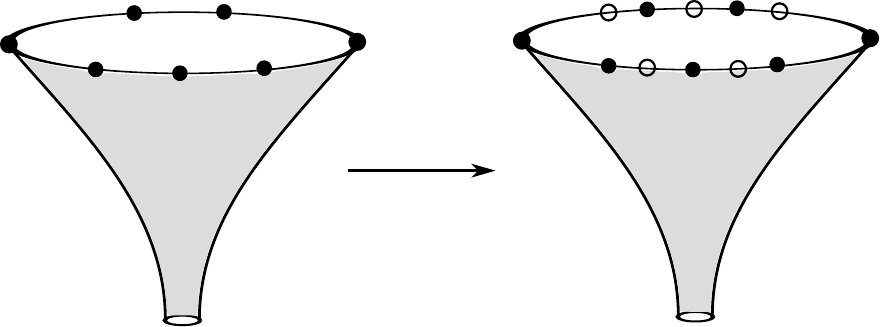}
\end{figure}

\begin{proof}
    The same arguments using overlay maps as before apply, only now we compare to $_{\{(1,1),(k,l)\},\{(1,1),(k,l)\} }\mc{R}_{k,l}$. 
The associated operation is, by Proposition \ref{ocstringsfromgluing}, $_2
\oc^{k-2,l-2}$.  See Figure \ref{ocoverlay} for an example of this particular
overlay map.  \end{proof}

\def\H{ {\mathbb H}}
\subsection{A local model}
Our definition of forgetful operations and homotopy units is based
upon the following local model. Let $\H$ denote the upper half plane and $\H^o$
the upper half plane with the origin removed. Viewing $\H$ as a disc with a
point removed, and $\H^o$ as a
disc with two points removed, there is the natural ``forgetful'' map
\begin{equation}
    F: \H^o \longhookrightarrow \H
\end{equation}
which forgets the special point 0.  Consider the
following negative strip-like end around $\infty$:
\begin{align}\label{eh}
    \e_\H: (-\infty,0] \times [0,1]  &\longrightarrow \H \\
    (s,t) &\longmapsto \exp(-\pi (s + it))
\end{align}
which has image $\{(r,\theta) | r \geq 1\} \subset \H$.
For $\H^o$, define the following basic positive strip-like end around 0:
\begin{align}
    \e_{\H^o}: [0,\infty) \times [0,1] &\longrightarrow \H^0 \\
    (s,t) &\longmapsto  2 \cdot \exp(-\pi (s+it)).
\end{align}
This end has image $\{(r,\theta)| 0 < r \leq 2\} \subset \H^o.$ 
With these special choices of strip-like ends, we observe that 0-connect sum
is exactly the forgetful map $F$. The precise statement is this:
\def\T{ {\mathbb T}}
\begin{prop} \label{forget}
Let $\T$ be the associated thick part in $\H^0$ of the 0 connect sum
\[
C: = \H^o \#^0_{(\e_{\H^o},0),(\e_\H,\infty)} \H
\]
and let $C^0$ denote the complement of $0 \in \H$ in the connect sum $C$. Then,
there is a commutative diagram 
\begin{equation}
\xymatrix{
\T \ar@{^{(}->}[r] \ar@{^{(}->}[rd] & C^0 \ar[d] \ar@{^{(}->}[r] & C\\
& \H^0 \ar@{^{(}->}[r] & \H
}
\end{equation}
\end{prop}
\begin{proof}
This is obvious via viewing each of these regions and the connect sum itself as
subsets of $\H$.  
\end{proof}
\subsection{Revisiting the unit}
We revisit our choice of Floer datum for the explicit geometric unit map,
defined in Section \ref{theunitelement}. Let $\Sigma_0 = \H$ once more denote the
upper half plane, and fix an outgoing striplike end at $\infty$ given by
$\e_{\H}$, defined in (\ref{eh}).

Let $\psi: [0,\infty) \lra [0,1]$ be a smooth function equaling 0 in a
neighborhood of 0 and 1 in a neighborhood of $[1,\infty)$.

Given a weight $w=1$, and a Hamiltonian $H$, fix the following Floer
datum on $\H$:
\begin{itemize}
    \item one form $\a_{\H}$ given by $-\frac{1}{\pi} \cdot \psi(r) d \theta$
    \item rescaling function $a_{\H}$ equal to 1. 
    \item any primary Hamiltonian, $H_\H$ that is compatible with the strip-like end $\e_\H$
    \item any almost complex structure that is compatible with $\e_\H$.
    \item some constant perturbation term $F_\H$.
\end{itemize}
Define the Floer datum for arbitrary weight $w$ to be the $w$ conformal
rescaling of the above Floer datum. It has the following properties:
\begin{itemize}
    \item one form $\a_{\H}^w$ given by $-\frac{1}{\pi} \cdot w \psi(r) d \theta$
    \item rescaling function $a_{\H}^w$ equal to $w$. 
    \item primary Hamiltonian $H_{\H}^w$ given by $\frac{H_\H \circ \psi^{w}}{w^2}$
    \item almost complex structure $J_{\H}^w$ given by $(\psi^{w})^* J_\H$
    \item perturbation term $F_\H^w$ given by $w \cdot F_\H$, another constant.
\end{itemize}
Call the above datum a {\bf standard unit datum of type $w$}.  By design,
$\e_{\H}^*(\a_{\H}^w) = w dt$.

\subsection{Damped connect sums}
We describe a local model, depending on a time parameter
\begin{equation}\tau\in[0,1],\end{equation} that gives a
homotopy relating a ``formal unit,'' or forgotten marked point, to the
geometric unit described in Section \ref{unstableoperations} and again above.
We would like such a homotopy, which we call a {\bf $\tau$-damped connect sum},
to have the following properties in a neighborhood of a given forgotten point
$p$ on a surface $S$.
\begin{itemize}
    \item at time $\tau = 0$, the Floer datum is essentially unconstrained in a
        neighborhood of $p$, agreeing with whatever Floer datum we obtained by
        forgetting $p$ and compactifying.  
    \item at intermediate time $\tau$, the Floer datum is
        modeled on a growing connect
        sum of a neighborhood of $p$ with a disc with one output, thought of as
        $\H$ with output at $\infty$. 
    \item the $\tau = 1$ limit is the nodal connect sum $\H \#^{1}_{p} S$. The
        Floer datum on the $\H$ component should agree with the Floer datum on
        the geometric unit, and the Floer data on the $p$ side should agree
        with a standard, previously chosen Floer datum.  
\end{itemize}
Readers who wish to skip this section should treat the {\bf $\tau$-damped
connect sum} along a boundary point $p$ as a formal operation on surfaces with
Floer data, satisfying the property that at $\tau=0$, one has the forgetful
map, and at $\tau=1$, one has nodally glued on an $\H$.  

In reality, we will need to construct such an operation in two steps:
\begin{itemize}
    \item for $\tau \in (0,\frac{1}{2})$, the Floer datum on $S$ goes from
        arbitrary with respect to $p$ to (partially) compatible with respect to a
        strip-like end around $p$.
    \item for $\tau \in (\frac{1}{2},1)$, the datum is modeled as a growing
        connect sum as before.  
\end{itemize}

\begin{figure}[h] 
    \caption{A schematic of a damped connect sum (though in reality, the conformal structure of $S$ will stay the same). \label{dampedconnectfig}}
    \centering
    \includegraphics[scale=0.7]{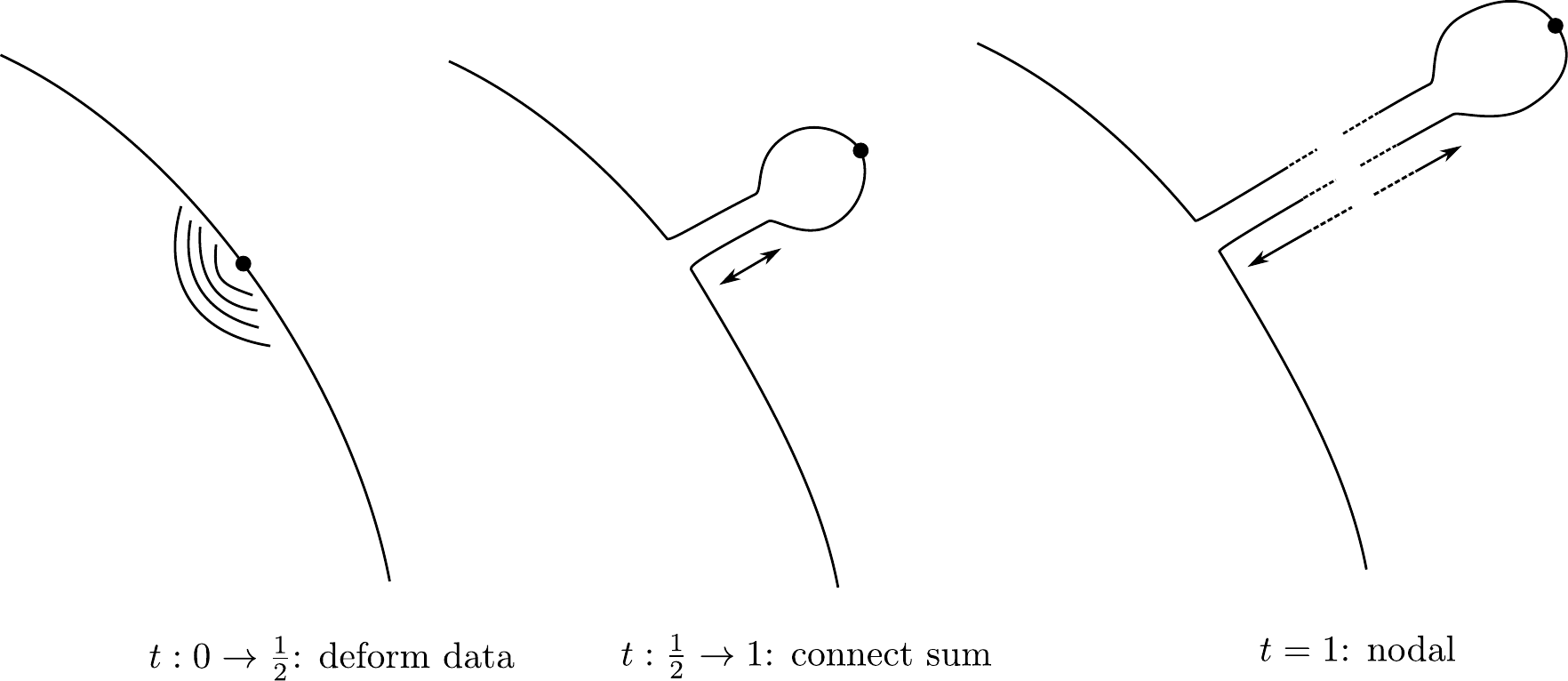}
\end{figure}

The basic setup is as follows: Let $S$ be a Riemann surface with boundary, with
some boundary marked points removed.  Fix one such positive boundary marked
point $z$, with strip-like end around $z$
\begin{equation}
    \e_z: [0,\infty) \times [0,1] \lra S.
\end{equation}
\begin{defn}
    Let $\hat{S}_z$ be $S$ with the point $z$ filled back in. Call the
    strip-like end $\e_z$ {\bf rational} 
    if it extends to a holomorphic map
$\bar{\e}_z: ([0,1] \times [0,\infty)) \cup \{\infty\} \lra \hat{S}_z$. 
\end{defn}
\begin{rem}
    Working with rational strip-like ends does not impose any additional
    trouble in choosing Floer data. We can implicitly choose all of our
    strip-like ends $\e: Z_+ \ra S$ to be rational.
    See e.g. \cite{Abouzaid:2010ly}*{Addendum 2.3}.
\end{rem}
Now, let $\e_z$ be any rational striplike end. Let $D_2^0$ denote the punctured
upper half radius two disc $\{0 < |z| \leq 2\} \subset \C$ and $D_2$ the domain
arising from $D_2^0$ by filling in the origin, i.e. $D_2 =  \{0 \leq |z| \leq
2\}$. By precomposing with the standard map
\begin{equation}
    \begin{split}
    \e_{\H_0}^{-1}: D_2^0 &\lra [0,\infty) \times [0,1]\\
z &\longmapsto (\ln(\frac{2}{|z|}), 1 - \frac{\arg(z)}{\pi})
\end{split}
\end{equation}
we may equivalently suppose $\e_z$ is a map $\tilde{\e}_z$ from $D_1^0$ to $S$
that extends to a map from $D_2$ to $\hat{S}_z$. Call $\tilde{\e}_z$ the
associated {\bf disc-like end} of $\e_z$, and let $\bar{\tilde{\e}}_z$ be the
associated map from $D_2$ to $\hat{S}_z$.

Now, fix a time-parameter $\tau \in [0,1]$. In a manner depending on $\tau$, we
weaken the notion of compatibility with respect to the strip-like end $\e_z$.
\begin{defn}\label{taucompatible}
    A Floer datum ($\a_S$, $a_S$, $J_S$, $H_S$, $F_S$) is said to be {\bf 
    $\tau$-partially compatible} with a strip-like end $(z,\epsilon)$, for
    $\tau \in [0,1]$, if the datum extends to one on the compactification
    $\hat{S}_z$ when $\tau \in [0,\frac{1}{2}]$ and the conditions
\begin{equation}
    \begin{split}
        \e^* a_S &= w_1(\tau)\\
        \e^* \a_S &= w_2(\tau) dt \\
\e^* H_S &= \frac{H \circ \psi^{w_2(\tau)}}{w_2(\tau)^2}
    \end{split}
\end{equation}
    hold when $\tau \geq \frac{1}{2}$. Furthermore, we require that at
    $\tau=1$, $\tau$-compatibility is genuine compatibility; in other words, 
    \begin{equation}
        w_1(1) = w_2(1).\end{equation}
\end{defn}
\begin{rem}
    In the limit $\tau = 0$, $\tau$-partial compatibility is an empty condition
    for the Floer data on $S$. Note that in contrast to normal Floer data, we
    are using two potentially different functions $w_1(\tau)$ and $w_2(\tau)$.
    In other words, we are not requiring the value of the one-form $\a_S$ or
    the amount flowed by the Hamiltonian to always be the same as the amount of
    rescaling or time-shifting performed by the almost complex structure or
    Lagrangian boundary. This is sensible---on a boundary point that is not a
    priori a striplike end, the one form $\a_S$ is asymptotically 0, but $a_S$
    is always non-zero.
\end{rem}

\begin{defn} \label{taustructure}
    Let $S$ have rational strip-like end $\e_z$ around $z$ with associated
    disc-like end $\tilde{\e}_z$, and suppose we have chosen a Floer datum
    $\mathbf{D}^\tau$ that is $\tau$-compatible with $\tilde{\e}_z$. An
    associated {\bf $\tau$-structure} on $\H$, denoted
    $\mathbf{D}_{\H}^\tau(z)$ consists of the following Floer datum on
    $\H$, depending on $\tau$:
    \begin{itemize}
        \item for $\tau \in [0,\frac{1}{2}]$, the Floer datum $\mathbf{D}$
            extends to the compactification $\hat{S}_z$. The pullback
            $(\bar{\tilde{e}}_z)^*D$ gives some Floer datum on $D_2$. Define
            the Floer datum on $\H$ to be any datum extending this one to all
            of $\H$.
        \item For $\tau \in [\frac{1}{2},1]$, the Floer datum is defined as
            follows: 
            \begin{itemize}
\item one-form $\a_\H^\tau(\mathbf{D})$ given by $-\frac{1}{\pi} w_2(\tau) \cdot \psi(r) d\theta$.
                \item any primary Hamiltonian $H^\tau(\mathbf{D})$ equal to $\frac{H \circ \psi^{w_2(\tau)}}{w_2(\tau)^2}$ on the striplike end $\e_\H$.
                \item any rescaling function $a_S^\tau(\mathbf{D})$ equal to $w_1(\tau)$ when restricted to $\e_\H$.
                \item any complex structure equal to $(\psi^{a_S^\tau(\mathbf{D})})^* J_t$ on $\e_\H$.
            \end{itemize}
    \end{itemize}
    Furthermore, we mandate that when $\tau = 1$, the Floer datum on $\H$ must
    be the {\bf standard unit datum of type} $w_1(1) = w_2(1)$.
\end{defn}
Pick a smooth non-decreasing function $\kappa: [0,1] \lra [0,1]$ that is 0 in a
neighborhood of $[0,\frac{1}{2}]$ and 1 exactly at 1.
\begin{defn}
    Let $S$ have rational strip-like end $\e_z$ around $z$ with associated
    disc-like end $\tilde{\e}_z$, and suppose we have chosen a Floer datum
    $\mathbf{D}^\tau$ that is $\tau$-compatible with $\tilde{\e}_z$, and an
    associated {\bf $\tau$-structure} on $\H$,
    $\mathbf{D}_\H^\tau(z)$.  Define the {\bf $\tau$-damped
    connect sum}
    \begin{equation}
        S \sharp^\tau_z \H,
    \end{equation}
    to be the surface
    \begin{equation}
        S \#^{\kappa(\tau)}_{z,\infty} \H
    \end{equation}
    equipped with the Floer datum $\mathbf{D}^{\tau}$ on $S - \e_z([0,\infty) \times [0,1])$ and
    the Floer datum $\mathbf{D}_\H^\tau(z)$ on $\H$ elsewhere. 
\end{defn}
By construction, this is a smooth Floer datum on $S \sharp^\tau_z \H$, and
satisfies the following properties:
\begin{itemize}
    \item For $\tau \in [0,\frac{1}{2}]$ it agrees with the compactified Floer
        datum $\hat{\mathbf{D}}^\tau$ on $\hat{S}_z$.
    \item For $\tau = 1$, it is the nodal connect sum 
        \begin{equation}
            S \#^1 \H
        \end{equation}
        where $\H$ is equipped with the {\bf standard unit datum} of type
        $w_1(\tau) = w_2(\tau)$, and $S$ has some Floer datum $\mathbf{D}^1$
        that is genuinely compatible with $S,z,\e_z$ in the usual sense.
\end{itemize}

\begin{rem}
    We should note that any intermediate damped connect sum with a copy of $\H$
    for our choices of standard strip-like ends (\ref{eh}) is conformally
    equivalent to the forgetful map. All that changes as the damped connect sum
    parameter approaches 1 is that the standard unit Floer datum is rescaled
    and shrunk into a smaller and smaller neighborhood of the marked point,
    until eventually at time 1 it is forced to break off. Despite this, it is
    useful sometimes to visualize the process as a topological connect sum.
\end{rem}

\subsection{Abstract moduli spaces and operations}

\begin{defn}
    The {\bf moduli space of discs} with $d$ {\bf marked points}, $F\subset
    [d]$ {\bf forgotten points}, and $H  \subset [d] - F$ {\bf homotopy units} 
    \begin{equation}
        \mf{H}^{d,F,H}
    \end{equation}
    is exactly the moduli space of discs $\mc{R}^d$, with points in $F$ or $H$
    labeled as belonging to $F$ or $H$ times a copy of $[0,1]$ for each element
    of $H$:
    \begin{equation}
        \mf{H}^{d,F,H} \simeq \mc{R}^d \times [0,1]^{|H|}.
    \end{equation}
    When $H = \emptyset$, we define $\mf{H}^{d,F,\emptyset} = \mc{R}^{d,F}$.
\end{defn}
We think of a point in this moduli space as a pair $(S,\vec{v} = (v_1, \ldots,
v_{|H|}))$.  We associate the $i$th copy of the interval to the $i$th ordered
point in $H$ in the following sense: Suppose $H$ is ordered $\{p_{n_1}, \ldots,
p_{n_{|H|}}\}$ Then, for each element of $H$, there are {\bf endpoint maps}
\begin{align}
    \pi_{p_{n_i}}^1: \mf{H}^{d,F,H}|_{v_i=1} &\lra \mf{H}^{d,F,H - \{p_{n_i}\}}\\
    \pi_{p_{n_i}}^0: \mf{H}^{d,F,H}|_{v_i=0} &\lra \mf{H}^{d,F + \{p_{n_i}\}, H - \{p_{n_i}\}}
\end{align}
defined as follows: given an element $(S,\vec{v})$ $\pi^1_{p_{n_i}}$
removes the label $H$ from the point $p_{n_i}$ in $S$, and projects $\vec{v}$
away from the $i$th component (which is 1). $\pi^0_{p_{n_i}}$ removes the label
of $H$ but assigns the label of $F$ to $p_{n_i}$, and projects $\vec{v}$ away
from the $i$th component (which is 0).

\begin{defn} The {\bf moduli space of $\mf{S}$-glued pairs of discs} with
    $(k,l)$ {\bf marked points,} $\mf{T}$ {\bf point identifications}, $F_1,
    F_2 \subset ([k],[l])$ {\bf forgotten points}, and $H_1, H_2 \subset
    ([k],[l])$ {\bf homotopy units}, denoted
    \begin{equation}
        _{\mf{S},\mf{T}} \mf{H}_{k,l}^{F_1,F_2,H_1,H_2}
    \end{equation}
    is exactly the moduli space $_{\mf{S},\mf{T}} \mc{R}_{k,l}$ with points in
    $F_1, F_2, H_1, H_2$ labeled accordingly, times a copy of $[0,1]$ for each
    element of $H_1$ or $H_2$:
    \begin{equation}
        _{\mf{S},\mf{T}} \mf{H}_{k,l}^{F_1,F_2,H_1,H_2} \simeq _{\mf{S},\mf{T}} \mc{R}_{k,l} \times [0,1]^{|H_1|} \times [0,1]^{|H_2|}.
    \end{equation} As with forgotten marked points, we have the following constraints:
    \begin{itemize}
        \item $F_1, H_1$ and $F_2, H_2$ are disjoint subsets of the left and right identified
            points respectively. Namely, 
            \begin{equation}
                F_i, H_i \subset \pi_i(\mf{T}), F_i \cap H_i = \emptyset
            \end{equation} 
            where $\pi_i$ is projection onto the $i$th component. 
        \item $F_1,H_1$ and $F_2,H_2$ are not associated to a boundary identification, i.e.
            \begin{equation}
                (F_i\cup H_i) \cap \pi_i(p(\mf{S})) = \emptyset
            \end{equation}
        \item $F_1,H_1$ and $F_2,H_2$ do not contain both the left and right points of
            any identification, i.e.  \begin{equation}
                ((F_1\cup H_1) \times (F_2\cup H_2)) \cap \mf{T} = \emptyset
            \end{equation}
    \end{itemize}
\end{defn}
We think of a point of $_{\mf{S},\mf{T}} \mf{H}_{k,l}^{F_1,F_2,H_1,H_2}$ as a
tuple \begin{equation}(P,\vec{v},\vec{w}).\end{equation} 
Suppose $H_1, H_2= \{p_{n_1}, \ldots,
p_{n_{|H_1|}}\}, \{p_{m_1}, \ldots, p_{m_{|H_2|}}\}$.
For any point $p_{n_i} \in H_1$ or $p_{m_i} \in H_2$, there are analogously
defined {\bf endpoint maps} 
\begin{align}
    \pi_{L,p_{n_i}}^1: _{\mf{S},\mf{T}} \mf{H}_{k,l}^{F_1,F_2,H_1,H_2} &\lra _{\mf{S},\mf{T}} \mf{H}_{k,l}^{F_1,F_2,H_1 - \{p_{n_i}\},H_2}\\
    \pi_{L,p_{n_i}}^0: _{\mf{S},\mf{T}} \mf{H}_{k,l}^{F_1,F_2,H_1,H_2} &\lra _{\mf{S},\mf{T}} \mf{H}_{k,l}^{F_1 + \{p_{n_i}\},F_2,H_1 - \{p_{n_i}\},H_2}\\
    \pi_{R,p_{m_i}}^1:  _{\mf{S},\mf{T}} \mf{H}_{k,l}^{F_1,F_2,H_1,H_2} &\lra _{\mf{S},\mf{T}} \mf{H}_{k,l}^{F_1 ,F_2+ \{p_{m_i}\},H_1,H_2 - \{p_{m_i}\}}\\
    \pi_{R,p_{m_i}}^0: _{\mf{S},\mf{T}} \mf{H}_{k,l}^{F_1,F_2,H_1,H_2} &\lra _{\mf{S},\mf{T}} \mf{H}_{k,l}^{F_1,F_2,H_1,H_2 - \{p_{m_i}\}},
\end{align}
which change the labellings of $P$, and apply a projection map to $(\vec{v},
\vec{w})$ in the following way: for $\pi_{L,p_{n_i}}^b$: given a point
$(P,\vec{v},\vec{w})$, remove the point $p_{n_i}$ from the set $H_1$, and add
it to $F_1$ if $b=0$. Also, project $\vec{v}$ away from the $i$th factor and do
nothing to $\vec{w}$. For $\pi_{R,p_{m_j}}^b$: given a point
$(P,\vec{v},\vec{w})$, remove the point $p_{m_j}$ from the set $H_2$, and add
it to $F_2$ if $b=0$. Also, project $\vec{w}$ away from the $j$th factor and do
nothing to $\vec{v}$.

As before, there are {\bf forgetful maps}
\begin{align}
    \mf{F}_{I}: \mf{H}^{d,F,H} &\lra \mf{H}^{d-|I|,F',H'}\\
    \mc{F}_{I_1,I_2} : _{\mf{S},\mf{T}}\mf{H}_{k,l}^{F_1,F_2,H_1,H_2} &\lra 
    _{\mf{S}',\mf{T}'}\mf{H}_{k-|I_1|,l-|I_2|}^{F_1',F_2',H_1',H_2'}
\end{align}
for $I \subset F$ or $I_1, I_2 \subset F_1, F_2$. $F_1'$ and $F_2'$ are $F_1$
and $F_2$ sans $I_1$ and $I_2$, reindexed appropriately, and $H_1'$ and $H_2'$
are just $H_1$ and $H_2$ reindexed. On the $[0,1]^{|H_i|}$ components, the
forgetful maps are the identity.
\begin{defn}
    Fix some very small number $\e \ll 1$.
    let $(S,\vec{v})$ denote an element of the moduli space $\mf{H}^{d,F,H}$. This element is said to be {\bf h($\e$)-semistable} if 
    \begin{equation} \label{hsemistable}
        d - |F| - |H| + \#\{j | v_j > \e\} = 1.
    \end{equation}
    It is said to be {\bf h($\e$)-stable} if the equality above is replaced by the strict inequality $>$.

    Similarly, let $(P, \vec{v}, \vec{w})$ denote an element of the moduli space $_{\mf{S},\mf{T}} \mf{H}_{k,l}^{F_1,F_2,H_1,H_2}$. This element is said to be {\bf h($\e$)-semistable} if 
    \begin{equation}
        \begin{split}
            k- |F_1| - |H_1| + \#\{j | v_j > \e\} &= 1\\
            l- |F_2| - |H_2| + \#\{k | w_k > \e\} &= 1,
        \end{split}
    \end{equation}
    and {\bf h($\e$)-stable} if the equalities above are replaced by
    inequalities $\geq$, with one of the inequalities being strict.
\end{defn}

The Deligne-Mumford compactifications
\begin{align}
    \overline{\mf{H}}^{d,F,H}\\
    _{\mf{S},\mf{T}} \overline{\mf{H}}_{k,l}^{F_1,F_2,H_1,H_2}
\end{align}
exist, equal as abstract spaces to the product of the compactifications
\begin{align}
    \overline{\mc{R}}^{d,F} \times [0,1]^{|H|}\\
    _{\mf{S},\mf{T}}\overline{\mc{R}}_{k,l}^{F_1,F_2} \times [0,1]^{|H_1|+|H_2|}
\end{align}
respectively.  The codimension 1 boundaries of these spaces are given by the
codimension 1 boundary of the various underlying spaces of discs, along with
restrictions to various endpoints.
\begin{align}
    \partial^1 \overline{\mf{H}}^{d,F,H} &= (\partial^1 \overline{\mc{R}}^{d,F}
    ) \times [0,1]^{H} \cup \coprod \overline{\mc{R}}^{d,F} \times [0,1]^i
    \times \{0,1\} \times [0,1]^{|H|-i-1}\\
    \partial^1{_{\mf{S},\mf{T}}}
    \overline{\mf{H}}_{k,l}^{F_1,F_2,H_1,H_2} &=
    \partial^1(_{\mf{S},\mf{T}}\overline{\mc{R}}_{k,l}^{F_1,F_2} ) \times
    [0,1]^{|H_1|+|H_2|}  \\
    \nonumber &\cup \coprod
    (_{\mf{S},\mf{T}}\overline{\mc{R}}_{k,l}^{F_1,F_2}) \times [0,1]^i \times
    \{0,1\} \times [0,1]^{|H_1|+|H_2|-i-1} 
\end{align}

In a manner identical to the previous section, in the {\bf f-stable} range
(which is independent of $H$ or $H_1, H_2$), the maximal forgetful map extends
to a map on compactifications:
\begin{align}
    \mf{F}_{max}: \overline{\mf{H}}^{d,F,H} &\lra \overline{\mf{H}}^{d-|F|,\emptyset,H'}\\
    \mc{F}_{max} : _{\mf{S},\mf{T}}\overline{\mf{H}}_{k,l}^{F_1,F_2,H_1,H_2} &\lra 
    _{\mf{S}',\mf{T}'}\overline{\mf{H}}_{k-|F_1|,l-|F_2|}^{\emptyset,\emptyset,H_1',H_2'}
\end{align}

In what follows, we will only construct Floer data for glued pairs of
discs---though the case for a single disc is identical (and in fact simpler).  
\begin{defn}
    A {\bf Floer datum} for a pair of glued discs with homotopy units and
    forgotten points $(P,\vec{v},\vec{w})$ is a Floer datum for the reduced
    gluing $\pi_{\mf{S}}(\mc{F}_{max}(P,\vec{v},\vec{w}))$ in the usual sense,
    with the following exceptions:
    \begin{itemize}
        \item For boundary point $p_{n_i} \in H_1,$ thought of as a point in
            $\pi_{\mf{S}}(P)$, the Floer datum only needs to be {\bf
            $v_i$-partially compatible} with the associated strip-like end
            $\e_{p_{n_i}}$, in the sense of Definition \ref{taucompatible}.
        \item Similarly, for boundary point $p_{m_j} \in H_2$, thought of as a
            point in the gluing $\pi_{\mf{S}}(P)$, the Floer datum only needs
            to be {\bf $w_j$-partially compatible} with the striplike end
            $\e_{p_{m_j}}$.
    \end{itemize}
    We additionally fix, for each element of $H_1$ and $H_2$, a copy
    $(\H,\e_{\H})$. Call $\H_{p_{n_i}}$ and $\H_{p_{m_j}}$ the copies of $\H$
    corresponding to points $p_{n_i} \in H_1$ and $p_{m_j} \in H_2$
    respectively.  Then, a Floer datum also consists of a choice of {\bf
    associated $v_i$ and $w_j$ structures} on $\H_{p_{n_i}}$ and $\H_{p_{m_j}}$
    for $p_{n_i}$ and $p_{m_j}$ respectively, in the sense of Definition
    \ref{taustructure}.
\end{defn}

\begin{figure}[h] 
    \caption{A single disc with forgotten points (marked with hollow circles) and homotopy units (marked with stars and dotted connect sums). The connect sums should be thought of simply as a schematic picture; really the conformal structure on the disc stays the same. \label{homotopyunitpic1}}
    \centering
    \includegraphics[scale=0.7]{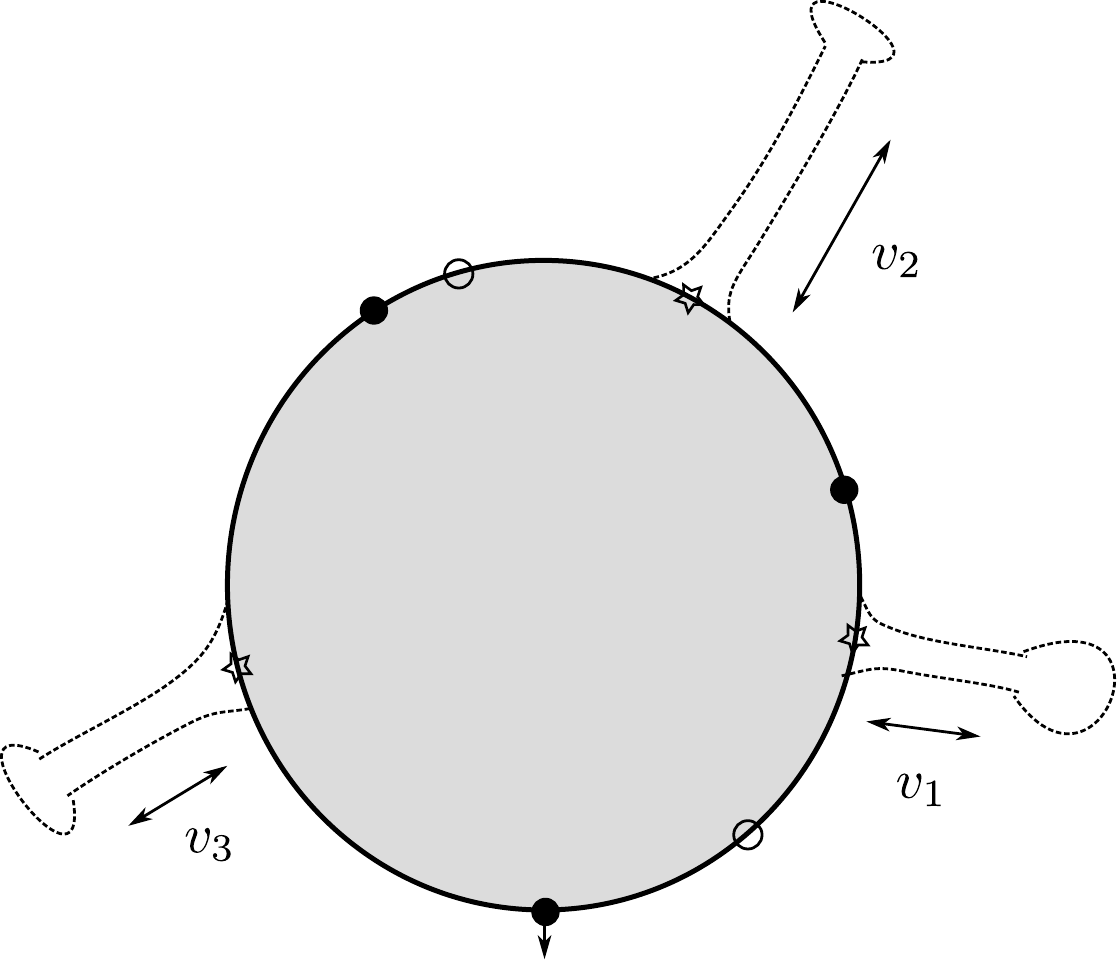}
\end{figure}

\begin{defn}
    A {\bf universal and conformally consistent choice of Floer data for glued pairs
    of discs with homotopy units} is a choice
    $\mathbf{D}_{(P,\vec{v},\vec{w})}$, for every boundary identification
    $\mf{S}$ and compatible sequential point identification $\mf{T}$, and every
    representative $(P,\vec{v},\vec{w})$, $_{\mf{S},\mf{T}}
    \overline{\mf{H}}_{k,l}^{F_1,F_2,H_1,H_2}$, varying smoothly over this
    space, whose restriction to a boundary stratum is conformally equivalent
    to a Floer datum coming from lower dimensional moduli spaces. Moreover,
    Floer data agree to infinite order at the boundary stratum with the Floer
    datum obtained by gluing. Finally, we require that 
    \begin{itemize}
        \item ({\bf forgotten points are forgettable}) 
            In the {\bf h($\e$)-stable} range, the choice of Floer datum only
            depends on the reduced surface $\mc{F}_{max}(P,\vec{v},\vec{w})$. 
            In the {\bf h($\e$)-semistable} range, the Floer datum agrees
            with the translation-invariant Floer datum on the strip.
        \item ({\bf 0 endpoint is
            forgetting}) In the {\bf h($\e$)-stable range}, if $v_i=0$ or
            $w_j=0$, then after forgetting the copy of $\H$
            corresponding to $p_{n_i}$ or $p_{m_j}$ respectively, the Floer
            datum should be isomorphic to the Floer datum on
            $\pi_{L,p_{n_i}}^0(S,\vec{v},\vec{w})$ or
            $\pi_{R,p_{m_j}}^0(S,\vec{v},\vec{w})$ respectively. 
            In the h($\e$)-semistable case, the Floer datum should be
            isomorphic to the translation invariant Floer datum on the
            respective surface.
        \item ({\bf 1 endpoint is gluing in a unit}) if $v_i=1$ or $w_j=1$,
            then $\H_{p_{n_i}}$ or $\H_{p_{m_j}}$ should have the standard unit
            datum Floer data, and the Floer datum on the main component should
            be isomorphic to a Floer datum on $\pi_{L,p_{n_i}}^1(P,\vec{v},\vec{w})$
            or $\pi_{R,p_{m_j}}^1(P,\vec{v},\vec{w})$ respectively.
    \end{itemize}
\end{defn}
\begin{prop}
    There exists a universal and conformally consistent choice of Floer data for glued pairs of discs with homotopy units.
\end{prop}
\begin{proof}
One proceeds inductively on the number of homotopy units. Suppose that we have
universally and conformally consistently chosen Floer data for $|H_1| + |H_2|
\leq k$ and Floer data for glued pairs of discs with at least $r+s$ marked
points, with homotopy units such that $|H_1|+|H_2| = k+1$. Using the endpoint
constraints described above, we have already described constraints on our Floer
data on the endpoints, and codimension-1 boundary strata, so we pick some Floer
datum extending these cases. Recall that this is possible because all of the
spaces of choices are contractible.
\end{proof}

\begin{rem} 
    Notice that the notion of {\bf h($\e$)-stability} depends in some cases on the
    chosen point in the moduli space. For example, an element $(S,\vec{v})$ of
    $\mf{H}^{d,\{1, \ldots, d\}, \emptyset}$ is only h($\e$)-stable if at least
    two of the components of $\vec{v}$ are greater than $\e$. It will not be
    possible to consistently inherit Floer data at zero-endpoints from the
    forgetful map when all (or all but one) $\vec{v}$ equal to zero. Thus we
    are forced to turn off stability in a neighborhood of this case.
\end{rem}

\begin{defn}
    Let $(P,\vec{v},\vec{w})$ be a pair of glued discs with $H_1, H_2$ homotopy
    units, and suppose we have fixed a Floer datum $\mathbf{D}$ for
    $(P,\vec{v},\vec{w})$.  Then the {\bf associated homotopy-unit surface},
    denoted \begin{equation}
        \mf{h}(P),\end{equation} 
    is the {\bf iterated damped connect sum}
    \begin{equation}
        \mf{h}(P):= \pi_{\mf{S}}(\mf{F}_{max}(P)) \sharp^{v_1}_{p_{n_1}} \H_{p_{n_1}}\cdots \sharp^{v_{|H_1|}}_{p_{n_{|H_1|}}} \H_{p_{n_{|H_1|}}} \sharp^{w_1}_{p_{m_1}} \H_{p_{m_1}} \cdots \sharp^{w_{|H_2}}_{p_{m_{|H_2|}}} \H_{p_{m_{|H_2|}}}.
    \end{equation}
    This is a (potentially nodal) surface with associated Floer data.
\end{defn}

\begin{defn}
    An {\bf admissible Lagrangian labeling} for pairs of glued discs with
    homotopy units and forgotten point is a labeling in the usual sense,
    satisfying the conditions that labellings before and after H points and F
    points must coincide.  
\end{defn}
The admissibility condition implies that there is an {\bf induced labeling}
on the {\bf associated homotopy-unit surface}.

Now, suppose we have fixed a universal and consistent choice of Floer data for
homotopy units. Consider a compact submanifold with corners of dimension $d$ 
\begin{equation}
    \overline{\mc{E}}^d \longhookrightarrow {_{\mf{S},\mf{T}}} \overline{\mc{H}}_{k,l}^{F_1,F_2,H_1,H_2}.
\end{equation}
with an admissible Lagrangian labeling $\vec{L}$. In the usual fashion, fix
input and output chords $\vec{x}_{in}, \vec{x}_{out}$ and orbits
$\vec{y}_{in}$, $\vec{y}_{out}$ for the induced marked points of the associated
homotopy-unit surface $\mf{h}(P,\vec{v},\vec{w})$.
Define
\begin{equation}
    \overline{\mc{E}}^d(\vec{x}_{out},\vec{y}_{out};\vec{x}_{in},\vec{y}_{in})
\end{equation}
to be the space of maps
\begin{equation}
    \{u: \mf{h}(P,\vec{v},\vec{w})) \lra M: S \in \overline{\mc{E}}^d\}
\end{equation}
satisfying Floer's equation with respect to the Floer datum and asymptotic and
boundary conditions specified by the Lagrangian labeling $\vec{L}$ and
asymptotic conditions $(\vec{x}_{out},\vec{y}_{out}, \vec{x}_{in},
\vec{y}_{in})$.

As before, $h(\mf{S},k,l)$ denote the number of boundary components of any
resulting surface $\mf{h}(P)$.
\begin{lem}
    The moduli spaces
    $\overline{\mc{E}}^d(\vec{x}_{out},\vec{y}_{out};\vec{x}_{in},\vec{y}_{in})$
    are compact, and empty for all but finitely many
    $(\vec{x}_{out},\vec{y}_{out})$ given fixed inputs
    $(\vec{x}_{in},\vec{y}_{in})$. For generically chosen Floer data, they form
    smooth manifolds of dimension
    \begin{equation}
        \begin{split}
    \dim \overline{\mc{E}}^d(\vec{x}_{out},\vec{y}_{out};\vec{x}_{in},\vec{y}_{in}):=
            \sum_{x_-\in \vec{x}_{out}} \deg(x_-) + &\sum_{y_- \in \vec{y}_{out}}\deg(y_-) \\
            + (2 - h(\mf{S},k,l) - |\vec{x}_{out}|- 2 |\vec{y}_{out}|)n 
        &+d -\sum_{x_+\in \vec{x}_{in}} \deg(x_+) - \sum_{y_+ \in \vec{y}_{in}} \deg(y_+).
    \end{split}
\end{equation}
\end{lem}
\begin{proof}
    The usual transversality arguments, dimension calculation, and compactness results apply.
\end{proof}
In the usual fashion, when the dimension of the spaces
$\overline{\mc{E}}^d(\vec{x}_{out},\vec{y}_{out};\vec{x}_{in},\vec{y}_{in})$
are zero, we use orientation lines to count (with signs) the number of points
in such spaces, and associate operations
\begin{equation}
    (-1)^{\vec{t}}\mathbf{I}_{\overline{\mc{E}}^d} 
\end{equation} 
from the tensor product of wrapped Floer complexes and symplectic cochain
complexes where $\vec{x}_{in}$, $\vec{y}_{in}$ reside to the tensor product of
the complexes where $\vec{x}_{out}$, $\vec{y}_{out}$ reside, where $\vec{t}$ is
a chosen sign twisting datum.

An interesting source of submanifolds for operations comes from the entire
moduli spaces
\begin{equation}
    _{\mf{S},\mf{T}} \mf{H}_{k,l}^{F_1,F_2,H_1,H_2}
\end{equation}
for an initial sequential point identification $\mf{T}$.

\subsection{New operations}

Up until now, we have been somewhat imprecise when specifying correspondences
between inputs and asymptotic boundary conditions on moduli spaces associated
with operations. Let us fix some notation for a specific class of moduli
spaces.

Let $\mf{S}$ be a boundary identification, and let $\mf{T}$ be an initial
sequential boundary identification that is compatible with $\mf{S}$; say it is
\begin{equation}
    \mf{T} = \{(1,1), (2,2), \ldots, (r,r)\}
\end{equation}
We previously defined an operation $\mathbf{G}_{\mf{S},\mf{T}}$ corresponding
to the entire moduli space 
\begin{equation}
    _{\mf{S},\mf{T}} \overline{\mc{R}}_{k,l}.
\end{equation}

Let us be precise about inputs. Given {\it boundary marked points} $z_1,
\ldots, z_k$, $z_1', \ldots, z_l'$ on each factor of our pair of discs, 
if $i \leq r$, define
\begin{equation}
    g_{\mf{S}}(z_i,z_i')
\end{equation}
to be the image of the pair of identified points under the gluing
$\pi_{\mf{S}}$. The possibilities are
\begin{itemize}
    \item a {\bf pair of boundary input points} ($\tilde{z}_i, \tilde{z}'_i)$ if
        $z_i$, $z_i'$ were not adjacent to a boundary identification;
    \item a {\bf single boundary input point} $\tilde{z}_{i,i'}$ if $z_i$, $z_i'$
        were adjacent to a single boundary identification; or
    \item a {\bf single interior input point} $\tilde{y}_{i,i'}$ if $z_i$, $z_i'$
        were adjacent to two boundary identifications.
\end{itemize}
Denote by 
\begin{equation}
    g_{\mf{S}}(z_j), g_{\mf{S}}(z_j')
\end{equation}
the images of non-identified points under the gluing $\pi_{\mf{S}}$.
Then, the associated operation takes the form
\begin{equation}
    \mathbf{G}_{\mf{S},\mf{T}}( (\bar{x}_1, \ldots, \bar{x}_r), (x_{r+1}, \ldots, x_k), (x_{r+1}', \ldots, x_l')),
\end{equation}
where $\bar{x}_i$ is an asymptotic condition of the same basic type as
$g_{\mf{S}}(z_i,z_i')$, $x_j$ is a boundary asymptotic
condition corresponding to $g_{\mf{S}}(z_j)$, and $x_j'$ is a boundary
asymptotic condition associated  to $g_{\mf{S}}(z_j')$. This operation returns
a sum of boundary asymptotic condition of the same type as
$g_{\mf{S}}(z_{out},z_{out}')$, the gluing of the outputs.

To be even a bit more precise, let us move now to the operations of the above
form arising in $\w^2$. Given a tuple of Lagrangians $\vec{X} = X_1, \ldots,
X_d$ in $\ob \w^2$,  and morphisms 
\begin{equation}
    x_i \in \hom(X_i,X_{i+1}),
\end{equation}
identified via the correspondence
\begin{equation}
    x_i \leftrightarrow \hat{x}_i
\end{equation}
with a boundary asymptotic condition, pair of boundary asymptotic conditions,
or interior asymptotic condition respectively, discussed in Proposition
\ref{gradings}, $\mu^d(x_d, \ldots, x_1)$ is by definition the labeled operation
\begin{equation}
    \mathbf{G}_{\mf{S}(\vec{X}),\mf{T}}(\hat{x}_1, \ldots, \hat{x}_d)
\end{equation}
in the sense of above, where we are implicitly composing with the reverse
identification 
\begin{equation}
    \hat{x} \leftrightarrow x
\end{equation}
to obtain the correct output, and using the usual sequential sign twisting
datum $\vec{t}_d = (1, \ldots, d)$. This is sensible because the boundary
asymptotic type of the input $\hat{x}_i$ is compatible with the type of the
glued marked point $g_{\mf{S}(\vec{X})}(z_i,z_i')$ by construction.

With this notation in place, let us now incorporate homotopy units and
forgotten points.  Define the $\ainf$ category 
\begin{equation}
    \tilde{\w}^2
\end{equation} 
to the have the same objects as $\w^2$. Its morphisms will be identical to
$\w^2$ as graded vector spaces, except for each $L$, it also contains the
following formal generators:
\begin{align}\label{factorunits}
    f_{L} \otimes x,\  e^+_{L} \otimes x \in \hom_{\tilde{\w}^2}(L \times L_j, L \times L_k) \mathrm{\ for\ all\ }x \in CW^*(L_k,L_j)\\
    x \otimes f_{L},\ x \otimes e^+_{L} \in \hom_{\tilde{\w}^2}(L_j \times L, L_k \times L) \mathrm{\ for\ all\ }x \in CW^*(L_j,L_k).
\end{align}
The degrees of these generators are
\begin{align}
    \label{formaldegree1}    \deg(f_L \otimes x) = \deg(x \otimes f_L) &= \deg(x) - 1.\\
    \label{formaldegree2} \deg(e_L^+ \otimes x) = \deg (x \otimes e_L^+) &= \deg(x),
\end{align}
i.e. $f_L$ and $e_L^+$ should be thought of as having degrees -1 and 0
respectively. Denote the generators of the morphism space between $X$ and $X'$
in $\tilde{\w}^2$ by $\tilde{\chi}(X,X')$.  The operations on $\tilde{\w}^2$
are as follows: Fix a label-set $\vec{X} = X_0, \ldots, X_d$. As in Section
\ref{unfoldingproduct}, there is an associated
boundary identification
\begin{equation}
    \mf{S}(\vec{X}) = \{(i,i) | X_i = \Delta. \}
\end{equation}
Now, let $x_1, \ldots, x_d$ be a sequence of asymptotic boundary conditions,
i.e.  $x_i \in \tilde{\chi}(X_{i-1},X_i)$. Let 
\begin{equation}
    F_1, F_2, H_1, H_2 \subset \{1, \ldots, d\}
\end{equation}
denote the subset of these of the form $e^+_L \otimes x$, 
\begin{equation}
    \begin{split}
    (F_1 \cup H_1)\cap (F_2 \cup H_2) &= \emptyset\\
    F_i \cup H_i &= \emptyset.
\end{split}
\end{equation}

Then, define
\begin{equation}
    \mu^d(x_d, \ldots, x_1)
\end{equation}
to be the operation controlled by the moduli space
\begin{equation}
    _{\mf{S}(\vec{X}),\mf{T}} \overline{\mf{H}}_{d,d}^{F_1, F_2, H_1, H_2},
\end{equation}
with labeling induced by the labeling of $\vec{X}$ as in Section
\ref{unfoldingproduct} as follows: If the $k$th Lagrangian $X_k$  was
labeled $L_i \times L_j$, then in the gluing $\pi_{\mf{S}}(P)$, the left image
of
$\partial^k S$ will be labeled $L_i$ and the right image $\bar{\partial}^k S$
will be labeled $L_j$. If $\partial_k S$ was labeled $\Delta$, then it
disappears under gluing so there is nothing to label. This induces a labeling
for the associated homotopy-unit surface $\mf{h}(P,\vec{v},\vec{w})$: since our
labeling was by choice {\it admissible}, any boundary point which we forget or
take damped connect sum is adjacent to boundary components with the same label.

The asymptotic conditions in the gluing
\begin{equation}
    \mf{h}(P)
\end{equation}
are as follows: in the glued surface $\pi_{\mf{S}}(P)$, let
$g_{\mf{S}}(z_i,z_i')$ be the resulting inputs (or pair of inputs) obtained by
the gluing. Then if $x_i$ is not a formal element, one requires these inputs to
be asymptotic to the associated $\hat{x}_i$ as before. If $x_i$ is a formal
element of any form, then 
$g_{\mf{S}}(z_i, z_i')$ is a pair of boundary marked points $(\tilde{z}_i, \tilde{z}_i')$.
\begin{itemize}
    \item if $x$ is of the form $f_L \otimes x$, then $\tilde{z}_i$ is marked as one of the $H_1$ points, and disappears under the damped connect sum operation. We require the other point $\tilde{z}_i'$ to be asymptotic to $x$.
    \item if $x$ is of the form $x\otimes f_L$, then $\tilde{z}_i'$ is marked as one of the $H_2$ points, and disappears under the damped connect sum operation. We require the other point $\tilde{z}_i'$ to be asymptotic to $x$.
    \item if $x$ is of the form $e_L^+ \otimes x$, then $\tilde{z}_i$ is marked as one of the $F_1$ points, and disappears under the forgetful map. We require the other point $\tilde{z}_i'$ to be asymptotic to $x$.
    \item if $x$ is of the form $x \otimes e_L^+ $, then $\tilde{z}_i'$ is marked as one of the $F_2$ points, and disappears under the forgetful map. We require the other point $\tilde{z}_i$ to be asymptotic to $x$.
\end{itemize}

This gives rise to a well-defined operation
\begin{equation}
    \mu^d(x_d, \ldots, x_1)
\end{equation}
where $x_1, \ldots, x_d$ are allowed to be formal elements---implicitly again,
we are taking the output of this operation, and composing under the reverse
association 
\begin{equation}
    \hat{x} \leftrightarrow x.
\end{equation}
In this case, we use sign twisting datum $\vec{t}_d = (1, \ldots, d)$,
including the degrees and presence of formal elements.

One can check that degree of the associated operation is $2-d$, under the
choice of gradings of the formal elements (\ref{formaldegree1}) and
(\ref{formaldegree2}), for the following reason: there are no Maslov type
contributions of the form $\deg(f_L)$, but this is compensated for by any
additional factor of the interval $[0,1]$ in the source abstract moduli space.

As we have constructed it, this operation is only well defined for $(d,d,F_1,
F_2, H_1, H_2)$ in the {\bf f-semistable range}. Hand-declare the following
operations, corresponding to the {\bf f-unstable range}:
\begin{align}
\mu^k_{\w^2}(x_1 \otimes e^+_L, \cdots, x_k \otimes e^+_L)  &:= \mu^k_{\w^{op}}(x_1, \ldots, x_k) \otimes e^+_L \\
\nonumber &= (-1)^* \mu^k_{\w}(x_k, \ldots, x_1) \otimes e^+_L \\
    \mu^k_{\w^2}(e^+_L \otimes x_k, \cdots, e^+_L \otimes x_1 )  &:= e^+_L \otimes \mu^k_{\w}(x_k, \ldots, x_1) \\
    \mu^1_{\w^2}(x \otimes e^+_L) &:= \mu^1_\w (x) \otimes e^+_L\\
    \mu^1 (e^+_L \otimes x) &:=  e^+_L \otimes \mu^1_{\w}(x)\\
    \mu^1_{\w^2}(f_L\otimes x) &:= (e^+_L - e_L) \otimes x \pm f_L \otimes \mu^1_{\w}(x)\\
    \mu^1_{\w^2}(x \otimes f_L) &:= x \otimes (e^+_L - e_L) \pm \mu^1_{\w}(x) \otimes f_L.
\end{align}

\begin{prop} 
    The resulting category $\tilde{\mc{W}}^2$ is an $\ainf$ category.
\end{prop}
\begin{proof}
We need to verify the $\ainf$ equations hold on sequences of morphisms that
include the formal elements $x \otimes e^+$, $e^+ \otimes x$, $f \otimes x$, $x
\otimes f$. This is mostly a consequence of the codimension-1 boundary of
moduli spaces of homotopy-unit maps, although some cases (corresponding to
bubbling of $f$-unstable components) will need to be checked by hand.  Without
loss of generality, we can assume that our original category contained just one
Lagrangian $L$, so $\ob \w^2 = \{L\times L, \Delta\}$; the multi-Lagrangians
case is identical but slightly more notationally complex. The codimension 1 boundary of the abstract moduli space 
\begin{equation}
    {_{\mf{S},\mf{T}_{max}}}
        \overline{\mf{H}}_{k,l}^{F_1,F_2,H_1,H_2}
\end{equation}
is covered by the following strata: 
\begin{itemize}
\item {\bf 0 and 1 endpoints}
        \begin{equation}
            {_{\mf{S},\mf{T}_{max}}}
            \overline{\mf{H}}_{d,d}^{F_1,F_2,H_1,H_2}|_{v_i\in\{0,1\}}, 
            {_{\mf{S},\mf{T}_{max}}}
        \overline{\mf{H}}_{d,d}^{F_1,F_2,H_1,H_2}|_{w_j\in \{0,1\}}, 
    \end{equation}
\item {\bf nodal degenerations}:
    \begin{equation}
        {_{\mf{S}',\mf{T}_{max}}}
        \overline{\mf{H}}_{d',d'}^{F_1',F_2',H_1',H_2'} \times 
        {_{\mf{S}'',\mf{T}_{max}}}
        \overline{\mf{H}}_{d-d'+1,d-d'+1}^{F_1'',F_2'',H_1'',H_2''}
    \end{equation}
\end{itemize}
Here, the boundary marked points in
${_{\mf{S}',\mf{T}_{max}}}
\overline{\mf{H}}_{d',d'}^{F_1',F_2',H_1',H_2'}$ consist of some
subsequence of length $d'$ of $(z_1, z_1'), \ldots, (z_d, z_d')$ along with inherited $F$/$H$ labels, and 
the boundary marked points of $ {_{\mf{S}'',\mf{T}_{max}}}
\overline{\mf{H}}_{d-d'+1,d-d'+1}^{F_1'',F_2'',H_1'',H_2''}$ consist of the
sequence $(z_1,z_1'), \ldots, (z_d,z_d')$ where the chosen subsequence
is replaced by a single new point $(z_{new},z_{new}')$ (again with inherited $F$/$H$ labels).

This implies that the boundary of the one-dimensional space of maps will
consist of compositions of operations coming from these strata as well as
various strip-breaking operations, corresponding to pre and post-composing with
$\mu^1$ in all possible ways.

By the choices we have made in our Floer datum, the 0 endpoint for a point
$p_{n_i} \in H_1$ correspond to the operation of forgetting the point
$p_{n_i}$, which changes the formal asymptotic condition from $f_L$ to $e_L^+$.
The 1 end point corresponds to gluing in a geometric unit to an existing
$\ainf$ operation, i.e. the formal condition $f_L$ is replaced by an actual asymptotic condition $e_L$.
In conjunction we see that the endpoint strata account for the occurrences of
$\mu^1$ for the $f_L$ as formally defined above.  

The nodal degenerations strata ensure that the associated operation is a
genuine composition of the form $\mu^{d-d'+1}( \cdots \mu^{d'} (\cdots)
\cdots)$ when both components of the strata are {\bf $f$-semistable}. Let us
without loss of generality suppose an $f$-unstable component bubbles off,
consisting of a subsequence of the form $(z_{i+1}, z_{i+1}'), \ldots,
(z_{i+d'},z_{i+d'}')$, with all of the right factored pointed labeled as
forgotten, with adjacent boundary components labeled by $L$. By construction
such a sequence corresponds to inputs $x_{i+1} \otimes e^+_L, \ldots, x_{i+d'}
\otimes e^{+}_L$. In the induced forgetful/gluing map, the right disc consists
entirely of points labeled forgotten and is thus deleted by $f$-stabilization.
Moreover, the right input of the lower disc $z_{new}'$ is marked as forgotten.
The left disc survives, contributing a $(\mu^{d'}_{\w})^{op}$. We conclude that
the operation associated to the top stratum is $(-1)^*\mu^{d'}(x_{i+1'}, \ldots,
x_{i+d'}) \otimes e^+_L$, which equals $\mu^{d'}(x_{i+d'} \otimes e^+_L, \ldots,
x_{i+1} \otimes e^+_L)$ as desired.
\end{proof}

We call the data that we have just constructed the structure of {\bf one-sided
homotopy units} for the category $\w^2$.

\begin{prop}
    The modified category $\tilde{\mc{W}}^2$ is quasi-equivalent to $\mc{W}^2$.
\end{prop}
\begin{proof}
    By construction, the inclusion 
    \begin{equation}
        \mc{I}: \mc{W}^2 \longhookrightarrow \tilde{\mc{W}}^2
    \end{equation}
is the desired quasi-isomorphism. If $\mu^1_{\w}(x) = 0$, the elements $e^+_L
\otimes x, x \otimes e^+_L$, which are the only potentially new elements of
cohomology, are homologous to $e_L \otimes x$, $x \otimes e_L$.  
\end{proof}

In order to simplify notation, define the {\bf total homotopy unit}
\begin{equation}
    e^+ := \sum_{L \in \ob \w} e^+_L,
\end{equation} 
thought of as an element in the semi-simple ring version of $\w$. The
corresponding elements in $\w^2$ are the {\bf total one-sided units}
\begin{equation}
    \begin{split}
    e^+ \otimes x := \sum_{L \in \ob \w} e^+_L \otimes x\\
     x \otimes e^+ := \sum_{L \in \ob \w} x \otimes e^+_L
 \end{split}
 \end{equation}

\subsection{Shuffle identities}

The technology we have introduced, and the analyses of the previous section
give some morphisms involving the $e^+$ desirable properties. To state them, we
first recall the combinatorial notion of a {\bf shuffle}:

\begin{defn}
    Let $V$ be a graded vector space. The {\bf $(k,l)$ shuffle} of an ordered
    collections of elements $\{a_1, \ldots, a_k\}$ and $\{b_1, \ldots, b_l\}$
    is defined to be following element in the tensor algebra $TV$:
    \begin{equation}
        \mc{S}_{k,l}(\{a_i\},\{b_j\}) := \sum_{\sigma \in \mathrm{shuff}(\{a_i\},\{b_j\})} (-1)^{sgn(\sigma)} \sigma (a_1 \otimes \cdots \otimes a_k \otimes b_1 \otimes \cdots b_l).
    \end{equation}
    Above, $\mathrm{shuff}(\{a_i\},\{b_j\})$ is the collection of permutations of the
    set $\{a_1, \ldots, a_k,b_1, \ldots, b_l\}$ that preserve the relative
    orderings of the $a_i$ and $b_j$, $\sigma$ is the corresponding permutation on
    the tensor algebra, and the sign $sgn(\sigma)$ is the sign of the graded
    permutation, i.e. the ordinary sign of the permutation plus a sign of
    parity the sums of degrees of elements that have been permuted past one
    another.  
\end{defn}

The following Proposition, essential for our forthcoming argument, is the main
consequence of the technology of one-sided homotopy units.  
\begin{prop}\label{shuffleidentities}
    We have the following identities in $\tilde{\w}^2$:
    \begin{align}
        \label{split}
        \mu^{k+l}_{\tilde{\w}^2}(\mc{S}_{k,l}(\{x_i \otimes e^+\}_{i=k}^1;\{e^+\otimes y_j\}_{j=1}^l)) &= 0, \mathrm{\ for\ }k,\ l > 0 \\
        \label{leftdiag}\mu^{k+l+1}_{\tilde{\w}^2}(\hat{\mathbf{a}},\mc{S}_{k,l}(\{x_i \otimes e^+\}_{i=k}^1;\{e^+\otimes y_j\}_{j=1}^l)) &= 
        \mu^{k+l+1}_{\w}(x_1, \ldots, x_k, \mathbf{a}, y_1, \ldots, y_l)\\
\label{rightdiag}\mu^{k+l+1}_{\tilde{\w}^2}(\mc{S}_{k,l}(\{x_i \otimes e^+\}_{i=k}^1;\{e^+\otimes
y_j\}_{j=1}^l),\hat{\mathbf{b}}) &= 
        \mu^{k+l+1}_{\w}(y_1, \ldots, y_l, \mathbf{b}, x_1, \ldots, x_k) \\
        \label{leftrightdiag}\mu^{k+l+2}_{\tilde{\w}^2}(\hat{\mathbf{a}},\mc{S}_{k,l}(\{x_i \otimes e^+\}_{i=k}^1;\{e^+\otimes
        y_j\}_{j=1}^l, \hat{\mathbf{b}}) &= 
        \mathrm{ }_2 \oc(\mathbf{a}, y_1, \ldots, y_l, \mathbf{b}, x_1, \ldots, x_k)
    \end{align}
    where $\mathbf{b} \in \hom(\Delta,L_i \times L_j)$ and $\mathbf{a} \in
    \hom(L_i \times L_j, \Delta)$ respectively.  
\end{prop}
\begin{proof}
    This is the content of Propositions \ref{shuffleid1}, \ref{shuffleid2},
    \ref{shuffleid3}, and \ref{shuffleid4} except for the case of
    (\ref{split}) when $k=l=1$. In that case, we have that
    \begin{equation}
        \begin{split}
            \mu^2_{\tilde{\w}^2}(\mc{S}_{1,1}(\{x \otimes e^+\}; \{e^+ \otimes y\}) &=
            \mu^{2}_{\tilde{\w}^2}(x \otimes e^+, e^+ \otimes y) - \mu^2_{\tilde{\w}^2}(e^+ \otimes y, x \otimes e^+) \\
            &= x\otimes y - x\otimes y\\
            & = 0.
    \end{split}
\end{equation}
\end{proof}

\section{Split-resolving the diagonal}\label{splitgendeltasection}
In this section, we prove the following theorem.
\begin{thm} \label{splitgen}
    If $M$ is non-degenerate, the product Lagrangians $\{L_i \times L_j\}$
    split-generate $\Delta$ in the category $\w^2$.  
\end{thm}

The proof uses a criterion for split generation discussed in Section
\ref{splitgensection}, which we now recall. Let $\w^2_{split}$ be the
full sub-category of $\w^2$ with objects given by the product Lagrangians
$\{L_i \times L_j\}$. There is a natural bar complex
\begin{equation}
    \mc{Y}^r_{\Delta} \otimes_{\w^2_{split}} \mc{Y}^l_{\Delta}
\end{equation}
and collapse map
\begin{equation}
    \mu: 
    \mc{Y}^r_{\Delta} \otimes_{\w^2_{split}} \mc{Y}^l_{\Delta} \lra \hom_{\w^2}(\Delta,\Delta).
\end{equation}
If $H^*(\mu)$ hits the unit element $[e] \in \hom_{\w^2}(\Delta,\Delta) =
SH^*(M)$, then we can conclude that the product Lagrangians split-generate
$\Delta$. 

Because split-generation is invariant under quasi-isomorphisms, it will suffice
to establish the above claim in the category 
\begin{equation} 
    \tilde{\w}^2
\end{equation}
which is quasi-isomorphic to $\w$.

Define a map
\begin{equation}
    \Gamma :\ _2 \r{CC}_*(\w,\w) \lra \mc{Y}^r_{\Delta} \otimes_{\tilde{W}^2_{split}} \mc{Y}^l_{\Delta}
\end{equation}
as follows: 
\begin{equation}
    \begin{split}
    \Gamma: \mathbf{a} \otimes b_1 \otimes \ldots \otimes & b_l \otimes 
    \mathbf{b}
    \otimes a_1 \otimes \cdots \otimes a_k \longmapsto \\ 
    &(-1)^{\blacksquare}\hat{\mathbf{a}} \otimes
    \mc{S}_{k,l}( (a_k \otimes e^+, \ldots, a_1 \otimes e^+) ; 
    (e^+ \otimes b_1, \ldots, e^+ \otimes b_l)) 
    \otimes \hat{\mathbf{b}}.  
\end{split}
\end{equation}
where $\mc{S}_{k,l}$ is the {\bf $(k,l)$ shuffle product} defined in the
previous section, and $\hat{\mathbf{a}}$ refers to $\mathbf{a}$ thought of as
an element of $\hom(L_i\times L_j, \Delta)$ instead of $\hom(L_j,L_i)$, and
similarly for $\hat{\mathbf{b}}$ under the usual correspondence from
Proposition \ref{gradings}. The Koszul sign 
\begin{equation}
    \blacksquare:= \sum_{j=k}^1 \bigg( ||a_j|| \cdot \big( \sum_{i=j-1}^1 ||a_i|| + |\mathbf{b}| \big) \bigg)
\end{equation}
can be thought of as arising from rearranging the substrings of the Hochschild
chain $\mathbf{a}, b_1, \ldots, b_l$ and $\mathbf{b}, a_1, \ldots, a_k$ so that
they are superimposed, with the latter sequence in reverse order.
\begin{prop}
    $\Gamma$ is a chain map.
\end{prop}
\def\Sh{\mc{S}}
\begin{proof}
    We verify this proposition up to sign. Using Proposition
    \ref{shuffleidentities}, we must show that $\Gamma$ intertwines the
    two-pointed Hochschild differential with the bar complex
    differential on $\tilde{\w}^2$. Abbreviate the shuffle product
    \begin{equation}
        \mc{S}_{i,j}(\{a_s \otimes e^+\}_{s=r+i}^{r+1} ; \{e^+ \otimes b_t\}_{t=n+1}^{n+j})
    \end{equation}
    by
    \begin{equation}
        \Sh(a_{r+i\ra r+1};  b_{n+1\ra n+j})
    \end{equation}
    
    The bar differential applied to
    \begin{equation}
       \Gamma( \mathbf{a} \otimes b_1 \otimes \cdots \otimes b_l \otimes \mathbf{b} \otimes a_1 \otimes  \cdots\otimes  a_k)
    \end{equation}
    is the sum of the following terms (with Koszul signs described in
    (\ref{yrlbardifferential}) that are omitted): 
    \begin{align}
        &\label{shuffleleft}\sum_{i\geq 0,j\geq 0} \mu_{\tilde{\w}^2} (\hat{\mathbf{a}}, \mc{S}(a_{k\ra k-i+1}; b_{1\ra j})) \otimes \mc{S}(a_{k-i \ra 1}; b_{j \ra l})
        \otimes \hat{\mathbf{b}}\ \ \ \textrm{(collapse on left)} \\
&\label{shuffleright}\sum_{i\geq 0,j \geq 0} \hat{\mathbf{a}} \otimes 
\mc{S}(a_{k\ra k-i+1}; b_{1 \ra j}) \otimes
\mu_{\tilde{\w}^2}( \mc{S}(a_{k-i \ra 1}; b_{j+1\ra l}), 
\hat{\mathbf{b}})\ \ \ \textrm{(collapse on right)}\\
&\label{shufflemiddle} \sum_{i_0, i_1, j_0, j_1} \hat{\mathbf{a}} \otimes \mc{S}(a_{k\ra k-i_0+1}; b_{1 \ra j_0}) \otimes \mu_{\tilde{\w}^2}(\mc{S}(a_{k-i_0 \ra k-i_0 - i_1+1}; b_{j_0+1 \ra j_0 + j_1})) \\
\nonumber &\ \ \ \ \ \ \ \ \ \ \ \ \  \otimes \mc{S}(a_{k-i_0-i_1\ra 1}; b_{j_0+j_1+1\ra l}) \otimes \hat{\mathbf{b}} \ \ \ \textrm{(collapse in middle)}.
    \end{align}

    By Proposition \ref{shuffleidentities}, 
    \begin{align}
\mu_{\tilde{\w}^2}(\hat{\mathbf{a}},\mc{S}(a_{k\ra k-i+1}; b_{1 \ra j})) &= \mu^{i+j+1}_{\w}(a_{k-i+1}, \ldots, a_k, \mathbf{a}, b_1, \ldots, b_j)\\
\mu_{\tilde{\w}^2}(\mc{S}(a_{k-i\ra 1}; b_{j+1 \ra l}),\hat{\mathbf{b}}) &= \mu^{(k-i) + (l-j) + 1}_{\w}(b_{j+1}, \ldots, b_l, \mathbf{b}, a_1, \ldots, a_{k-i})
\end{align}
and
\begin{equation}
    \begin{split}
    \mu_{\tilde{\w}^2}(\mc{S}(a_{k-i_0 \ra k-i_0-i_1+1};& b_{j_0+1 \ra j_0 + j_1})) = \\
    &\begin{cases}
        0 & i_1 \geq 1\textrm{ and }j_1 \geq 1\\
        \mu^{i_1}_\w(a_{k-i_0-i_1}, \ldots, a_{k-i_0}) \otimes e^+ & j_1 = 0\\
        e^+ \otimes \mu^{j_1}(b_{j_0+1}, \ldots, b_{j_0+j_1}) & i_1 = 0
    \end{cases}
\end{split}
\end{equation}
Putting this all together, we see that the non-zero terms above comprise
exactly the terms in 
\begin{equation}
    \Gamma \circ d_{ {_2}\r{CC}}(\mathbf{a}\otimes a_1\otimes \cdots\otimes a_k\otimes \mathbf{b}\otimes b_1\otimes \cdots\otimes b_l).
\end{equation}
\end{proof}

The following Proposition completes the proof of Theorem \ref{splitgen}.
\begin{prop}
There is a commutative diagram of chain complexes
\begin{equation}\label{shuffleproductdiagram}
    \xymatrix{_2\r{CC}_*(\w,\w) \ar[r]^{\Gamma\ \ \ }\ar[d]^{_2\oc} & \y^r_\d \otimes_{\tilde{\w}^2_{split}}\y^l_\d \ar[d]^{\mu} \\
    CH^*(M) \ar[r]^{D} & \hom_{\tilde{\w}^2}(\Delta,\Delta)}
\end{equation}
where $D$ is the identity map (by our definition of $\hom(\Delta,\Delta)$).
\end{prop}
\begin{proof}
    This is also a corollary of Proposition \ref{shuffleidentities}. Namely, we
    showed there that 
    \begin{equation}
        \begin{split}
            \mu(\hat{\mathbf{a}}, \mc{S}_{k,l} (a_k \otimes e^+, \ldots, a_1 \otimes e^+); (e^+ \otimes b_1, &\ldots, e^+ \otimes b_l),\hat{\mathbf{b}}) =\\
            &{_2}\oc (\hat{\mathbf{a}}, b_1, \ldots, b_l, \hat{\mathbf{b}}, a_1, \ldots, a_k),
\end{split}
\end{equation}
a restatement of (\ref{shuffleproductdiagram}).
\end{proof}
\begin{proof}[Proof of Theorem \ref{splitgen}]
    If the map $[\co]: \r{HH}_*(\w,\w) \lra SH^*(M)$ hits $[e]$, we conclude
    first that the representing chain-map ${_2}\co: {_2}\r{CC}_*(\w,\w)\lra
    CH^*(M)$ hits a representative $[e]$.  Thus by the existence of the diagram
    (\ref{shuffleproductdiagram}), $H^*(\mu)$ hits $[e] \in
    HW^*(\Delta,\Delta)$, verifying the split-generation criterion
    (\ref{splitgencriterion}) for $\Delta$.  
\end{proof}
\begin{cor}[$\mathbf{M}$ is full on $\w^2$] \label{fulldelta}
    Assuming non-degeneracy, $\mathbf{M}$ is full on $\w^2$.
\end{cor}
\begin{proof}
    We have shown that if $M$ is non-degenerate, then $\Delta$ is
    split-generated by the product Lagrangians $\{L_i \times L_j\}$. We have
    also constructed an $\ainf$ functor,
    \begin{equation}
        \mathbf{M}: \w^2 \lra \w\bimod\w
    \end{equation}
    and we have shown that
    $\mathbf{M}$ takes product Lagrangians $\{L_i \times L_j\}$ to
    Yoneda bimodules $\mc{Y}^l_{L_i} \otimes \mc{Y}^r_{L_j}$, and is
    full on these objects.
    This is the content of Propositions \ref{quiltainffunctor},
    \ref{quiltyonedabimodule}, and \ref{quiltkunneth}.  Thus, by Proposition
    \ref{fullsplitgen}, we conclude $\mathbf{M}$ is full on $\{\Delta,
    \{L_i\times L_j\}\}$.
\end{proof}
\begin{proof}[Proof of Theorem \ref{smoothness}]
    We showed in Proposition \ref{quiltdelta} that $\mathbf{M}$ sends $\Delta$
    to the diagonal bimodule $\w_{\Delta}$. Thus, by Corollary \ref{fulldelta},
    we conclude that $\w_{\Delta}$ is split-generated by Yoneda bimodules, the
    definition of homological smoothness.
\end{proof}
\begin{cor}\label{shiso}
    The maps 
    \begin{equation}
        \begin{split}
            \co: SH^*(M) &\lra \r{HH}^*(\w,\w)\\
            _2 \co: SH^*(M) &\lra {_2}\r{HH}^*(\w,\w)\\
        \end{split}
    \end{equation}
    are isomorphisms.
\end{cor}
\begin{proof}
    By Proposition \ref{m1equalsco}, the map $_2 \co$ is exactly the first order map 
    \begin{equation}
        \mathbf{M}^1: HW^*(\Delta,\Delta) \lra \hom_{\w\!-\!\w}(\w_{\Delta}, \w_{\Delta}),
    \end{equation}
    and $\mathbf{M}$ is full on $\Delta$.
\end{proof}

\begin{rem}
In fact, there is also a commutative diagram of the form
\begin{equation}
    \xymatrix{\r{CC}_*(\w^{op},\w^{op}) \otimes_{\K} \r{CC}_*(\w,\w) \ar[r]^{\ \ \ \ \ \ \ \Omega}\ar[d]^{\oc \otimes \oc} & \r{CC}_*(\tilde{\w}^2_{split},\tilde{\w}^2_{split}) \ar[d]^{\oc^2} \\
CH^*(M^-) \otimes_\K CH^*(M) \ar[r]^{\ \ \ =} & CH^*(M^- \times M)},
\end{equation}
where $\oc^2$ is the open-closed map on the product. The map $\Omega$ is given
by sending a pair of Hochschild chains, the first in reverse order, to the
shuffle of the chains:
\begin{equation}
    (a_0 \otimes \cdots \otimes a_k) \otimes (b_l \otimes \cdots \otimes b_0) \longmapsto \mc{S}_{k+1,l+1}(a_k \otimes e^+, \ldots, a_0 \otimes e^+; e^+ \otimes b_l, \ldots, e^+ \otimes b_0).
\end{equation}
This diagram exists on $M\times M$ with the symplectic form $(\omega, \omega)$
as well---one simply stops reversing the order of the left sequence.  The
conjectural implication is that if $M$ is non-degenerate, then any product $M^k
\times (M^-)^l$ is also non-degenerate, with essential Lagrangians given by
products of the essential Lagrangians in $M$.
The reason we have not adopted such a more general approach is that our current
construction of $\w^2$ is somewhat ad hoc, only allowing the use of split
Hamiltonians. As a result, most Lagrangians we might like to consider are
inadmissible. Modulo this technical detail, which has already been resolved for
symplectic cohomology \cite{Oancea:2006oq}, our argument applies.
\end{rem}
\begin{rem}
    The map $\Omega$ and $\Gamma$ that we have described are generalizations of
    a natural product structure on the Hochschild homology of associative
    algebras  \cite{Loday:1992fk}*{\S 4.2}. In the setting of unital
    associative algebras, a version of the Eilenberg-Zilberg theorem says that
    shuffle product induces an isomorphism
    \begin{equation}
        sh_*: \r{HH}_*(A) \otimes \r{HH}_*(A') \lra \r{HH}_*(A \otimes A').
    \end{equation} 
    We have described an $\ainf$ version of the above morphism, which requires
    the strictification of units to carry through.  There is a well defined
    quasi-inverse that always exists in the associative unital setting, and we
    conjecture that such quasi-inverses exist in the $\ainf$ setting as well.
    Constructing them may involve deforming the diagonal associahedron
    $\Delta_d \subset \mc{R}^{d}\times \mc{R}^{d}$ onto various copies of
    products of strata in order to obtain formulas for the tensor product of
    $\ainf$ algebras---see e.g., \cite{Saneblidze:2004fk}.
\end{rem}

\section{The non-compact Calabi-Yau structure}\label{wtowshrieksection}

Assuming $\mathbf{M}$ is non-degenerate, we have shown that $\w$ is
homologically smooth.  Thus from the results in Section \ref{moduleduality},
the {\bf inverse dualizing bimodule},
\begin{equation}
    \w^!:= \hom_{\w\!-\!\w}(\w_{\Delta},\w_{\Delta} \otimes_{\mathbb K} \w_{\Delta})
\end{equation}
is a perfect bimodule that represents Hochschild cohomology
\begin{equation}
    \w^! \otimes_{\w\!-\!\w} \mc{B} \simeq \r{HH}^*(\w,\mc{B}).
\end{equation}
        
In this section, we describe a geometric morphism of bimodules.
\begin{equation}
    \mc{CY}: \w_{\Delta} \lra \w^![n].
\end{equation}
The construction involves operations arising from discs with two negative
punctures and arbitrary numbers of positive punctures. We require that there be
a {\it distinguished} positive puncture on each component of the boundary of
the disc minus negative punctures; namely, we require there to be at least two
inputs.
Then, we interpret one of the distinguished positive punctures as belonging to
$\w$ and the remaining distinguished input and two outputs as belonging to
$\w^!$.
\begin{defn}
    The {\bf moduli space of discs with two negative punctures, two positive
    punctures, and $(k,l; s,t)$ positive marked points} 
    \begin{equation}
        \mc{R}_2^{k,l;s,t} 
    \end{equation}
    is the abstract moduli space of discs with \begin{itemize}
        \item two distinguished negative
    marked points $z_-^1$, $z_-^2$, 
    \item two distinguished positive marked points $z_+^1$,
        $z_+^2$, one removed from each boundary component cut out by $z_-^1$ and $z_-^2$, 
        \item $k$ positive marked points $a_1, \ldots, a_k$ between $z_-^1$ and $z_+^1$
        \item $l$ positive marked points $b_1, \ldots, b_l$ between $z_+^1$ and $z_-^2$
        \item $s$ positive marked points $c_1, \ldots, c_s$ between $z_-^2$ and $z_+^2$; and 
        \item $t$ positive marked points $d_1, \ldots, d_t$ between $z_+^2$ and $z_-^1$.
\end{itemize}
    Moreover, the distinguished points $z_-^1$, $z_-^2$, $z_+^1$, and $z_+^2$
    are constrained to lie (after automorphism) at $1$, $-1$, $i$ and $-i$
    respectively. Namely, we fix the cross-ratios of these 4 points.
\end{defn}
The boundary strata of the Deligne-Mumford compactification 
\begin{equation}
    \overline{\mc{R}}_2^{k,l;s,t}
\end{equation}
is covered by the images of natural inclusions of the following products:
\begin{align}
    \label{cycodim1start}\overline{\mc{R}}^{k' + 1 + l'} &\times_{1+} \overline{\mc{R}}_2^{k-k',l-l';s,t} \\
    \label{cycodim12}\overline{\mc{R}}^{k'} &\times^a_{n+1} \overline{\mc{R}}_2^{k-k'+1,l;s,t},\ \ \ 0 \leq n < k-k'+1\\
    \label{cybimodfinish}\overline{\mc{R}}^{l'} &\times^b_{n+1} \overline{\mc{R}}_2^{k,l-l'+1;s,t},\ \ \ 0 \leq n < l-l'+1\\
    \label{wshriek010}\overline{\mc{R}}^{s' + 1 + t'} &\times_{2+} \overline{\mc{R}}_2^{k,l,s-s',t-t'}\\
    \overline{\mc{R}}^{s'} &\times^c_{n+1} \overline{\mc{R}}_2^{k,l;s-s'+1,t},\ \ \ 0 \leq n < s-s'+1\\
    \label{wshriek010end}\overline{\mc{R}}^{t'} &\times^d_{n+1} \overline{\mc{R}}_2^{k,l;s,t-t'+1},\ \ \ 0 \leq n < t-t'+1\\
    \label{wshriekhigher}\overline{\mc{R}}_2^{k-k',l;s',t-t'} &\times_{(1,k'+1)} \overline{\mc{R}}^{k' + 1 + t'}\\
    \label{cycodim1end}\overline{\mc{R}}_2^{k,l-l';s-s',t} &\times_{(2,s'+1)} \overline{\mc{R}}^{l' + 1 + s'}.
\end{align}
Above, the notation $\times_{j+}$ in (\ref{cycodim1start}) and
(\ref{wshriek010}) indicates that the output of the first component is glued to
the special input $z_j^+$ of the second component, $\times^a_j$, $\times^b_j$,
$\times^c_j$ and $\times^d_j$ indicate gluing to the input $a_j$, $b_j$, $c_j$,
and $d_j$ respectively, and $\times_{(i,j)}$ in (\ref{wshriekhigher}) and
(\ref{cycodim1end}) indicate gluing the $i$th output of the first component to
the $j$th input of the second. Also, in (\ref{cycodim1start}) and
(\ref{wshriek010}), the $k'+1$st and $s'+1$st input points of the first
component become the special points $z^1_+$ and $z^2_+$ after gluing
respectively.

\begin{figure}[h] 
    \caption{A schematic of the moduli space $\mc{R}^{2,3;3,2}_2$. All non-signed marked points are inputs.\label{wwshriek}}
    \centering
    \includegraphics[scale=0.6]{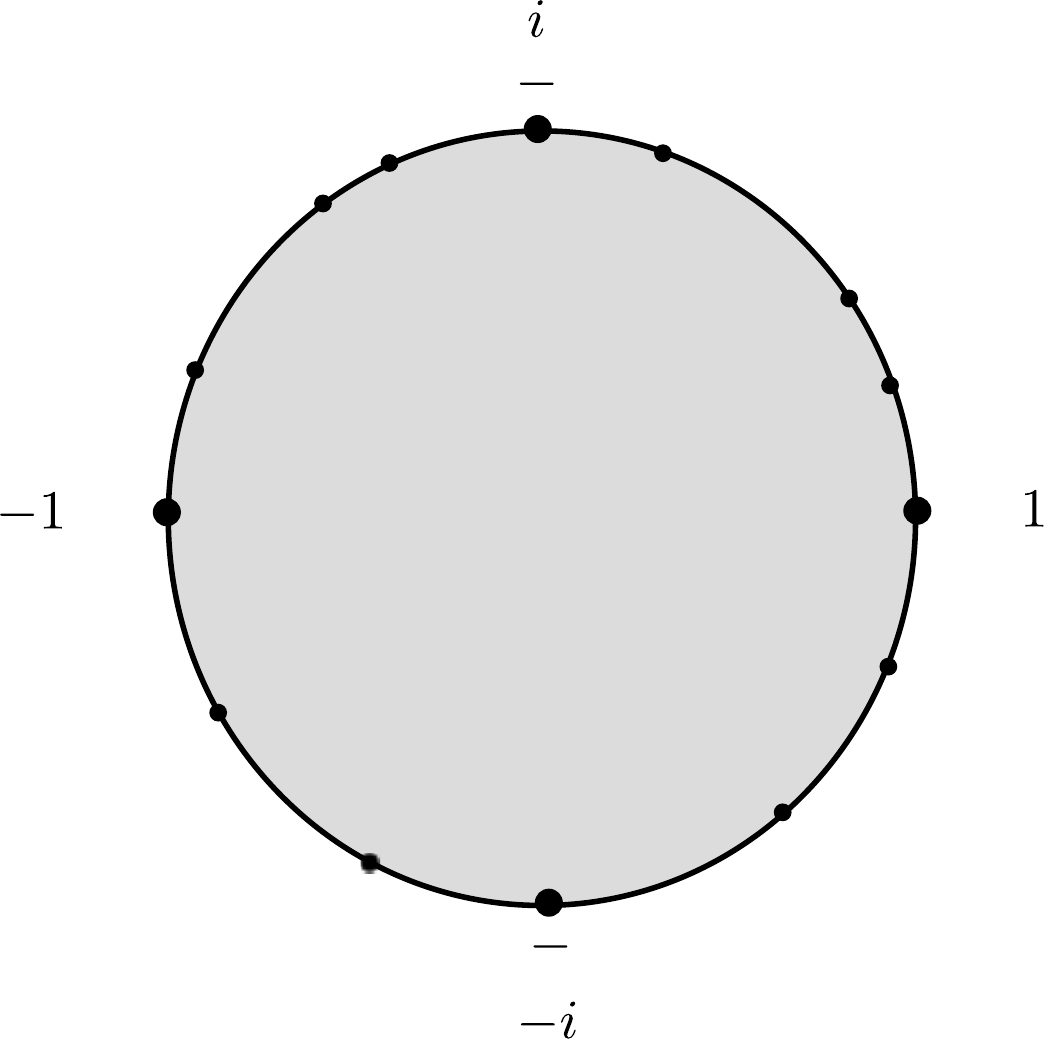}
\end{figure}
\begin{defn}
    A {\bf Floer datum} for a disc $S$ with two positive, two negative, and
    $(k,l; s,t)$ positive boundary marked points is a Floer datum of $S$
    thought of as an open-closed string.
\end{defn}

Fix a sequence of Lagrangians
\begin{equation} \label{twooutputlabels}
    A_0, \ldots, A_k, B_0, \ldots, B_l, C_0, \ldots, C_s, D_0, \ldots, D_t,
\end{equation}
corresponding to a labeling of the boundary of an element of
$\mc{R}_2^{k,l;s,t}$ by specifying that $a_i$ be the intersection point between
$A_{i-1}$ and $A_i$, and so on for $b_i$, $c_i$, and $d_i$. 
In the manner described in (\ref{signtwistoperation}), 
the space $\overline{\mc{R}}_2^{k,l;s,t}$, along with the sign twisting datum
\begin{equation}
    \vec{t}_{\mc{CY}_{k,l,s,t}} := (1,2, \ldots, k, k, k+1, \ldots, k+l, 1, 2, \ldots, s, s, s+1, \ldots, s + t )
\end{equation}
corresponding to inputs $(a_1, \ldots, a_k, z_1^+, b_1, \ldots, b_l, c_1, \ldots, c_s, z_2^+, d_1, \ldots, d_t)$,
determines an operation
\begin{equation} \label{cyoperationklst}
    \begin{split}
        \mathbf{CY}_{l,k,t,s}: 
        \big( \hom(B_{l-1}, B_l) &\otimes \cdots  \otimes \hom(B_{0},B_{1}) \\
&\otimes  \mathbf{hom}(A_k,B_0) \otimes \hom(A_{k-1}, A_k) \otimes \cdots 
\otimes \hom(A_{0},A_{1}) \big) \\
\otimes \big(\hom(D_{t-1}, D_t) &\otimes \cdots \otimes \hom(D_{0},D_{1})\\
        &\otimes \mathbf{hom}(C_s, D_0) \otimes 
        \hom(C_{s-1}, C_s) \otimes \cdots \otimes \hom(C_{0},C_{1}) \big) \\
            &\longrightarrow \mathbf{hom}(A_0,D_t) \otimes \mathbf{hom}(C_0,B_l).
\end{split}
    \end{equation}

\begin{defn}
    The {\bf Calabi-Yau morphism}
    \begin{equation}
        \mc{CY}: \w_{\Delta} \lra \w^![n]
    \end{equation}
    is given by the following data:
    \begin{itemize}
        \item For objects $(X,Y)$, a map 
            \begin{equation}
                \begin{split}
                    \mc{CY}^{0|1|0}: \w_{\Delta}(X,Y) &\lra \hom_{\w\bimod\w}(\w_{\Delta}, \mc{Y}^l_X
                \otimes \mc{Y}^r_Y) \\
                \mathbf{a} &\longmapsto \phi_{\mathbf{a}}
            \end{split}
            \end{equation}
            where $\phi_{\mathbf{a}}$ is the morphism whose $t|1|s$ term is
            \begin{equation}
                \begin{split}
                \phi_{\mathbf{a}}^{t|1|s}(&d_t, \ldots, d_1, \mathbf{b}, c_s, \ldots, c_1):= \\
                &\mc{CY}_{0,0,t,s}(\mathbf{a}, d_t, \ldots, d_1, \mathbf{b}, c_s, \ldots, c_1)
            \end{split}
            \end{equation}

        \item Higher morphisms
            \begin{equation}
                \begin{split}
                    \mc{CY}^{l|1|k}: \hom_{\w}(Y_{k-1},Y_k) \otimes \cdots \otimes &\hom_\w(Y_0,Y_1) \otimes \w_{\Delta}(X_0,Y_0) \\
                    \otimes \hom_\w(X_1, X_0) \otimes \cdots \otimes \hom_\w(X_{l},X_{l-1}) &\lra \hom_{\w\bimod\w}(\w_{\Delta}, \mc{Y}^l_{X_k} \otimes \mc{Y}^r_{Y_l})\\
            (b_l, \ldots, b_1, \mathbf{a}, a_k, \ldots, a_1) &\longmapsto \psi_{b_l, \ldots, b_1, \mathbf{a}, a_k, \ldots, a_1}
        \end{split}
    \end{equation}
    where 
    $\psi = \psi_{b_l, \ldots, b_1, \mathbf{a}, a_k, \ldots, a_1}$ 
    is the morphism whose $t|1|s$ term is
            \begin{equation}
                \begin{split}
            \phi^{t|1|s}(d_t,& \ldots, d_1, \mathbf{b}, c_s, \ldots, c_1) := \\
                &\mc{CY}_{l,k,t,s}(b_l, \ldots, b_1, \mathbf{a}, a_k, \ldots, a_1; d_t, \ldots, d_1, \mathbf{b}, c_s, \ldots, c_1).
            \end{split}
            \end{equation}
        \end{itemize}
        Put another way, we can in a single breath say that
        \begin{equation}
            \begin{split}
            (\mc{CY}^{l|1|k}(b_l, \ldots, b_1, \mathbf{a}, a_k, \ldots, &a_1))^{t|1|s}(d_t, \ldots, d_1, \mathbf{b}, c_s, \ldots, c_1) \\
            &:= \mc{CY}_{l,k,t,s}(b_l, \ldots, b_1, \mathbf{a}, a_k, \ldots, a_1; d_t, \ldots, d_1, \mathbf{b}, c_s, \ldots, c_1).
        \end{split}
    \end{equation}
\end{defn}
The Gromov bordification
$\overline{\mc{R}}_2^{k,l;s,t}(\vec{x}_{in},\vec{x}_{out})$ has boundary
covered by the images of the Gromov bordifications of spaces of maps
from the nodal domains (\ref{cycodim1start}) - (\ref{cycodim1end}), along with
standard strip breaking, which put together implies that: 
\begin{prop} \label{cycloseddegn}
    $\mc{CY}$ is a closed morphism of $\ainf$ bimodules of degree $n$.
\end{prop}
\begin{proof}
    We will briefly indicate how to convert the strata
    (\ref{cycodim1start})-(\ref{cycodim1end}) to the equation 
    \begin{equation}
        \delta \mc{CY} = \mc{CY} \circ \hat{\mu}_{\w_{\Delta}} - \mu_{\w^!}
        \circ \hat{\mc{CY}} = 0.  \end{equation} 
The strata (\ref{cycodim1start}) - (\ref{cybimodfinish}) 
correspond to $\mc{CY}$ composed with various $\ainf$ bimodule
differentials for $\w_{\Delta}$. The strata (\ref{wshriek010}) -
(\ref{wshriek010end}) all correspond to the internal differential
$\mu^{0|1|0}_{\w^!}$, which itself involves various pieces of the $\w_{\Delta}$ $\ainf$
bimodule differentials for the second string of inputs. Finally, the strata
(\ref{wshriekhigher}) - (\ref{cycodim1end} ) for fixed $k'$, $l'$, and varying
over all $s'$, $t'$
correspond to the terms of the form $\mu^{k'|1|0}_{\w^!} \circ \mc{CY}$ and
$\mu^{0|1|l'} \circ \mc{CY}$. The ingredients to verify signs are discussed in Section \ref{orientationsection}.
\end{proof}

We now observe that to first order, the operation $\mc{CY}$ is controlled by a
moduli space identical to one appearing in our definition of quilts:
\begin{prop}\label{mcy}
     For any $A,B \in \ob \w$, there is an equality
     \begin{equation}
         \mc{CY}^{0|1|0}_{A,B} = \mathbf{M}^1_{\Delta,A \times B}.
    \end{equation}
\end{prop}
\begin{proof}
The maps
\begin{equation}
    \mc{CY}^{0|1|0}: \hom_\w (A,B) \lra \w^!(A,B) := \hom_{\w\!-\!\w}(\w_{\Delta},\mc{Y}^l_A \otimes \mc{Y}^r_B)[n]
\end{equation}
and
\begin{equation}
    \mathbf{M}^1: \hom_{\w^2}(\Delta,A \times B) := \hom_\w(A,B)[-n] \lra \hom_{\w\!-\!\w}(\w_{\Delta},\mc{Y}^l_A \otimes \mc{Y}^r_B)
\end{equation}
have the same source and targets, so we need to verify that the spaces
controlling the Floer equations are the same. The unfolding map,
$\mathbf{\Psi}_{\hat{\vec{L}}}$, defined in (\ref{unfoldingquiltdefinition}),
when applied to a quilted strip with $(r,1,s)$ marked points, with middle label
sequence $(\Delta, A \times B)$, produces a surface with two output marked
points, two distinguished input marked points, a distinguished input marked
point between them corresponding to $\hom(A,B)$, and then $r+1+s$ input marked
points (the $r+1$st of which is distinguished) around the two outputs. This is
exactly the definition of the operation given by $\mc{CY}_{0,0; r,s}$. See also
Figure \ref{wrapcyimg} for a picture of this unfolding.
\end{proof}

\begin{figure}[h] 
    \caption{The equality between the quilted strip controlling $\mathbf{M}^1_{\Delta, L_i \times L_j}$ and the first-order map $\mc{CY}^{0|1|0}_{L_i, L_j}$. \label{wrapcyimg}}
    \includegraphics[scale=0.7]{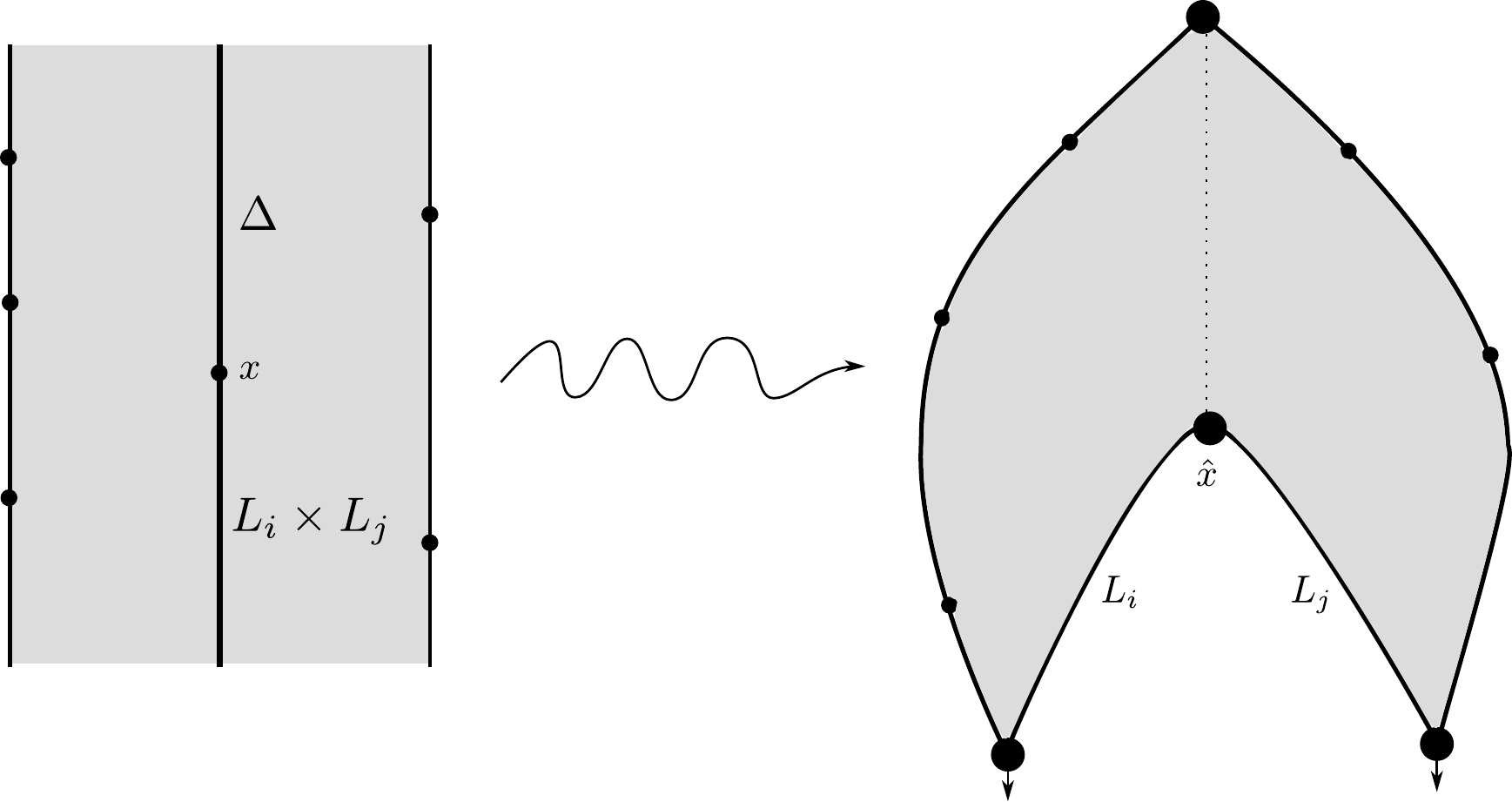}
\end{figure}

As an immediate corollary, 
\begin{cor}[The wrapped Fukaya category is Calabi-Yau]
    \label{calabiyau} Assuming non-degeneracy, 
    $\mc{CY}$ is a quasi-isomorphism.
\end{cor}
\begin{proof}
    We have shown that under the above hypothesis $\mathbf{M}$ is full
    (Corollary \ref{fulldelta}); hence it induces isomorphisms on homology.
    Thus by Proposition \ref{mcy} so does $\mc{CY}$.
\end{proof}
This completes the proof of Theorem \ref{wrapcy}.
\begin{cor}
    For any perfect bimodule $\mc{B}$, there is a natural quasi-isomorphism 
    \begin{equation}
        \r{HH}_{*-n}(\w,\mc{B}) \ra \r{HH}^*(\w,\mc{B}).
    \end{equation}
\end{cor}
\begin{proof}
    The isomorphism is the composition of two maps, which are both
    quasi-isomorphisms by Corollary \ref{calabiyau} and Corollary
    \ref{representhh}.  Using two-pointed complexes for Hochschild homology and
    cohomology, these maps are: 
    \begin{equation}\label{cymaps}
        \w_{\Delta} \otimes_{\w\!-\!\w} \mc{B} \stackrel{\mc{CY}_\#}{\lra} \w^!
        \otimes_{\w\!-\!\w} \mc{B} \stackrel{\bar{\mu}}{\lra}
        \hom_{\w\!-\!\w}(\w_{\Delta},\mc{B}).
    \end{equation}
\end{proof}

In the next section, we use Corollary \ref{calabiyau} to deduce that $\oc:
\r{HH}_*(\w,\w) \ra SH^*(M)$ is an isomorphism.

\section{The Cardy condition} \label{cardyconditionsection}

\subsection{A geometric bimodule quasi-isomorphism}
In Section \ref{diagonalbimodulesection}, we gave a construction of a
quasi-isomorphism of bimodules 
\begin{equation}
    \mc{F}_{\Delta, left, right}: \cc_{\Delta} \otimes_\cc \mc{B} \otimes_\cc \cc_{\Delta} \stackrel{\sim}{\lra} \mc{B}.
\end{equation}
where $\cc$ was an arbitrary $\ainf$ category, $\cc_{\Delta}$ the diagonal
bimodule, and $\mc{B}$ a $\cc\!-\!\cc$ bimodule.  The morphism involved
collapsing on the right followed by collapsing on the left by the bimodule
structure maps $\mu_{\B}$. The order of collapsing is of course immaterial; we
have just picked one.

Let us now suppose that $\cc = \w$ and $\mc{B} = \w_{\Delta}$. We would like to
give a direct geometric quasi-isomorphism
\begin{equation}
    \mu_{LR}: \w_{\Delta} \otimes_\w \w_{\Delta} \otimes_\w \w_{\Delta} \lra \w_{\Delta},
\end{equation}
homotopic to $\mc{F}_{\Delta, left, right}$, but not involving counts of
degenerate surfaces.  
\begin{defn}
    The {\bf moduli space of discs with four special points of type
    $(r,k,l,s)$} 
    \begin{equation}
        \mc{R}^{r,k,l,s}
    \end{equation}
    is the abstract moduli space of discs with $r+k+l+s+3$ positive boundary
    marked points and one negative boundary marked point labeled in counterclockwise order from the negative point as  $(z_{out}^-, z_1, \ldots, z_r, \bar{z}^1, z_1^1, \ldots, z_k^1, \bar{z}^2, z_1^2, \ldots, z_l^2, \bar{z}^3, z_1^3, \ldots, z_s^3)$, such that
    \begin{equation}
        \begin{split}
        &\textrm{after automorphism, the points }z_{out}^-,\ \bar{z}^1,\ \bar{z}^2,\ \bar{z}^3\textrm{ lie at } \\
        &-i, -1, i,\textrm{ and } 1\textrm{ respectively}.
    \end{split}
    \end{equation}
\end{defn}

\begin{figure}[h] 
    \caption{A schematic of the moduli space of discs with four special points of type $(2,3,3,2)$. All non-signed marked points are inputs.\label{muLRimg}}
    \centering
    \includegraphics[scale=0.6]{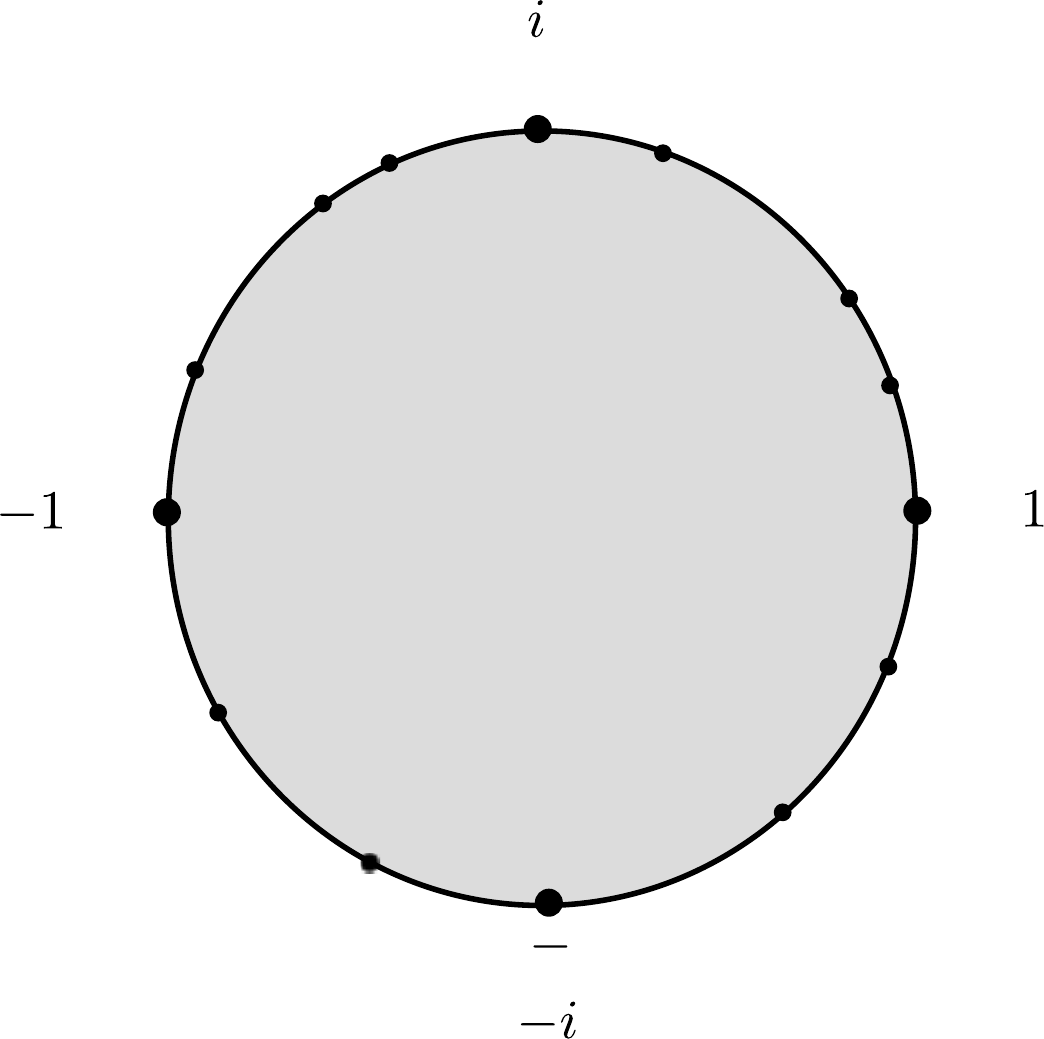}
\end{figure}

The associated Floer theoretic operation to the space
$\overline{\mc{R}}^{r,k,l,s}$ with sign twisting datum
\begin{equation}
    \begin{split}
    \vec{t}_{LR,r,k,l,s} = &(1, 2, \ldots, r, r , r + 1, \ldots, r + k, r + k,\\
    &r + k + 1, \ldots, r + k + l, r + k + l, r + k + 1 + 1, \ldots,\\
    &r + k + l + s).
\end{split}
\end{equation}
is 
\begin{equation}
    \begin{split}
        (\mu_{LR})_{r,k,l,s}:= (-1)^{\vec{t}_{LR,r,k,l,s}}&\mathbf{F}_{\mc{R}^{r,k,l,s}}: \\
        &\w^{\otimes r} \otimes \w_{\Delta} \otimes \w^{\otimes k} \otimes \w_{\Delta} \otimes \w^{\otimes l} \otimes \w_{\Delta} \otimes \w^{\otimes s} \lra \w_{\Delta}
\end{split}
\end{equation}
Then, define the morphism
\begin{equation}
    \mu_{LR}^{r|1|s} := \bigoplus_{k\geq 0, l\geq 0} (\mu_{LR})_{r,k,l,s}: \w^{\otimes r} \otimes (\w_{\Delta} \otimes T\w \otimes \w_{\Delta} \otimes T\w \otimes \w_{\Delta}) \otimes \w^{\otimes s} \lra \w_{\Delta}.
\end{equation}
One can calculate that the morphism is degree zero, as desired.
\begin{prop}
    The pre-morphism of bimodules
    \begin{equation}
    \mu_{LR} \in \hom_{\w\!-\!\w} (\w_{\Delta} \otimes_\w \w_{\Delta} \otimes_\w \otimes \w_{\Delta},\w_{\Delta})
\end{equation}
is {\bf closed}, i.e. 
\begin{equation}
    \delta \mu_{LR} = 0.
\end{equation}
\end{prop}
\begin{proof}
    We leave this mostly as an exercise, but this follows from analyzing the
    equations arising from the boundary of the one-dimensional space of maps
    with domain $\overline{\mc{R}}^{r,k,l,s}$. The relevant codimension 1
    boundary components involve strip-breaking and the codimension 1 boundary
    strata of the abstract moduli space $\overline{\mc{R}}^{r,k,l,s}$, which is
    covered by
    \begin{align}
        \label{fourpoint1}  \overline{\mc{R}}^{r'} &\times_{n+1} \overline{\mc{R}}^{r-r'+1,k,l,s},\ \ \ 0 \leq n < r - r'+1 \\
        \label{fourpoint2}  \overline{\mc{R}}^{k'} &\times^1_{n+1} \overline{\mc{R}}^{r,k-k'+1,l,s},\ \ \ 0 \leq n < k-k'+1 \\
        \label{fourpoint3} \overline{\mc{R}}^{l'} &\times^2_{n+1} \overline{\mc{R}}^{r,k,l-l'+1,s}, \ \ \ 0 \leq n < l-l'+1 \\
        \label{fourpoint4} \overline{\mc{R}}^{s'} &\times^3_{n+1} \overline{\mc{R}}^{r,k,l,s-s'+1}, \ \ \ 0 \leq n < s - s'+1 \\
        \label{fourpoint5} \overline{\mc{R}}^{r'+1 + k'} &\times_{1+} \overline{\mc{R}}^{r-r',k-k',l,s} \\
        \label{fourpoint6} \overline{\mc{R}}^{k'+1+l'} &\times_{2+} \overline{\mc{R}}^{r,k-k',l-l',s} \\
        \label{fourpoint7} \overline{\mc{R}}^{l'+1+s'} &\times_{3+} \overline{\mc{R}}^{r,k,l-l',s-s'} \\
        \label{fourpoint8} \overline{\mc{R}}^{r-r',k,l,s-s'} &\times_{r'+1} \mc{R}^{r'+1+s'}.
    \end{align}
Above, the notation $\times^j_{k}$ means the output of the first component is
glued to the input point $z^j_k$ of the second component, and the notation
$\times_{i+}$ means the output of the first component is glued to the special
point $\bar{z}^i$. Also, in (\ref{fourpoint5}), (\ref{fourpoint6}), and
(\ref{fourpoint7}), the $r'+1$st, $k'+1$st, and $l'+1$st inputs on the first
component become the distinguished point $\bar{z}^1$, $\bar{z}^2$, and
$\bar{z}^3$ respectively after gluing.
\end{proof}

Now, we show that $\mu_{LR}$ was in fact homotopic to
$\mc{F}_{\Delta,left,right}$. We construct a geometric homotopy using
\begin{defn}
The {\bf moduli space}
    \begin{equation}
        \mc{S}^{r,k,l,s}
    \end{equation}
    is the abstract moduli space of discs with $r+k+l+s+3$ positive boundary
    marked points and one negative boundary marked point labeled in counterclockwise
    order from the negative point as  $(z_{out}^-, z_1, \ldots, z_r, \bar{z}^1,
    z_1^1, \ldots, z_k^1, \bar{z}^2, z_1^2, \ldots, z_l^2, \bar{z}^3, z_1^3,
    \ldots, z_s^3)$, such that, for any $t \in (0,1)$,
    \begin{equation}
        \begin{split}
        &\textrm{after automorphism, the points }z_{out}^-,\ \bar{z}^1,\ \bar{z}^2,\ \bar{z}^3\textrm{ lie at } \\
        &-i, -1, \exp(i\frac{\pi}{2}(1-t)),\textrm{ and } 1\textrm{ respectively}.
    \end{split}
    \end{equation}
\end{defn}
$\mc{S}^{r,k,l,s}$ fibers over the open interval $(0,1)$ given by the value of
$t$, and thus has dimension one greater than $\mc{R}^{r,k,l,s}$. Compactifying,
we see that $\overline{\mc{S}}^{r,k,l,s}$
submerses over $[0,1]$ and has codimension one boundary covered by the natural
inclusions of the following strata, the first two of which correspond to fibers
at the endpoints $0$ and
$1$, and the remainder of which lie over the entire interval:
\begin{align}
    \label{flrstratum}    \overline{\mc{R}}^{r' +k  + l' + 2} &\times_{r-r'+1} \mc{R}^{(r-r') + s + (l-l') + 1} \ \ \textrm{(}t=1\textrm{ fiber)}\\
    \label{mulrstratum}& \overline{\mc{R}}^{r,k,l,s} \ \ \ \textrm{(}t=0\textrm{ fiber)}\\
    \label{beginintervalstrata}\overline{\mc{R}}^{r'} &\times_{n+1} \overline{\mc{S}}^{r-r'+1,k,l,s}, \ \ \ 0 \leq n < r-r'+1 \\
    \label{fphom2}    \overline{\mc{R}}^{k'} &\times_{n+1}^1 \overline{\mc{S}}^{r,k-k'+1,l,s},\ \ \ 0 \leq n < k-k'+1 \\
    \label{fphom3}   \overline{\mc{R}}^{l'} &\times_{n+1}^2 \overline{\mc{S}}^{r,k,l-l'+1,s},\ \ \ 0 \leq n < l-l'+1 \\
    \label{fphom4}   \overline{\mc{R}}^{s'} &\times_{n+1}^3 \overline{\mc{S}}^{r,k,l,s-s'+1},\ \ \ 0 \leq n < s - s' + 1 \\
    \label{fphom5}  \overline{\mc{R}}^{r'+1 + k'} &\times_{1+} \overline{\mc{S}}^{r-r',k-k',l,s} \\
    \label{fphom6} \overline{\mc{R}}^{k'+1+l'} &\times_{2+} \overline{\mc{S}}^{r,k-k',l-l',s} \\
    \label{fphom7} \overline{\mc{R}}^{l'+1+s'} &\times_{3+} \overline{\mc{S}}^{r,k,l-l',s-s'} \\
    \label{endintervalstrata}\overline{\mc{S}}^{r-r',k,l,s-s'} &\times_{r'+1} \mc{R}^{r'+1+s'}.
\end{align}
Above, in (\ref{flrstratum}), the $r'+1$st and $r' + k + 2$nd inputs of the
first component and the $(r-r') + (l-l') + 2$nd input of the second component
become the three special points $\bar{z}^1$, $\bar{z}^2$, and $\bar{z}^3$
respectively after gluing. Also, the notation for strata
(\ref{beginintervalstrata})-(\ref{endintervalstrata}) exactly mirrors the
notation in (\ref{fourpoint1})-(\ref{fourpoint8}).

\begin{figure}[h] 
    \caption{The moduli space $\mc{S}^{r,k,l,s}$ and its $t\ra\{0,1\}$ degenerations. \label{muLRdegeneration}}
    \centering
    \includegraphics[scale=0.7]{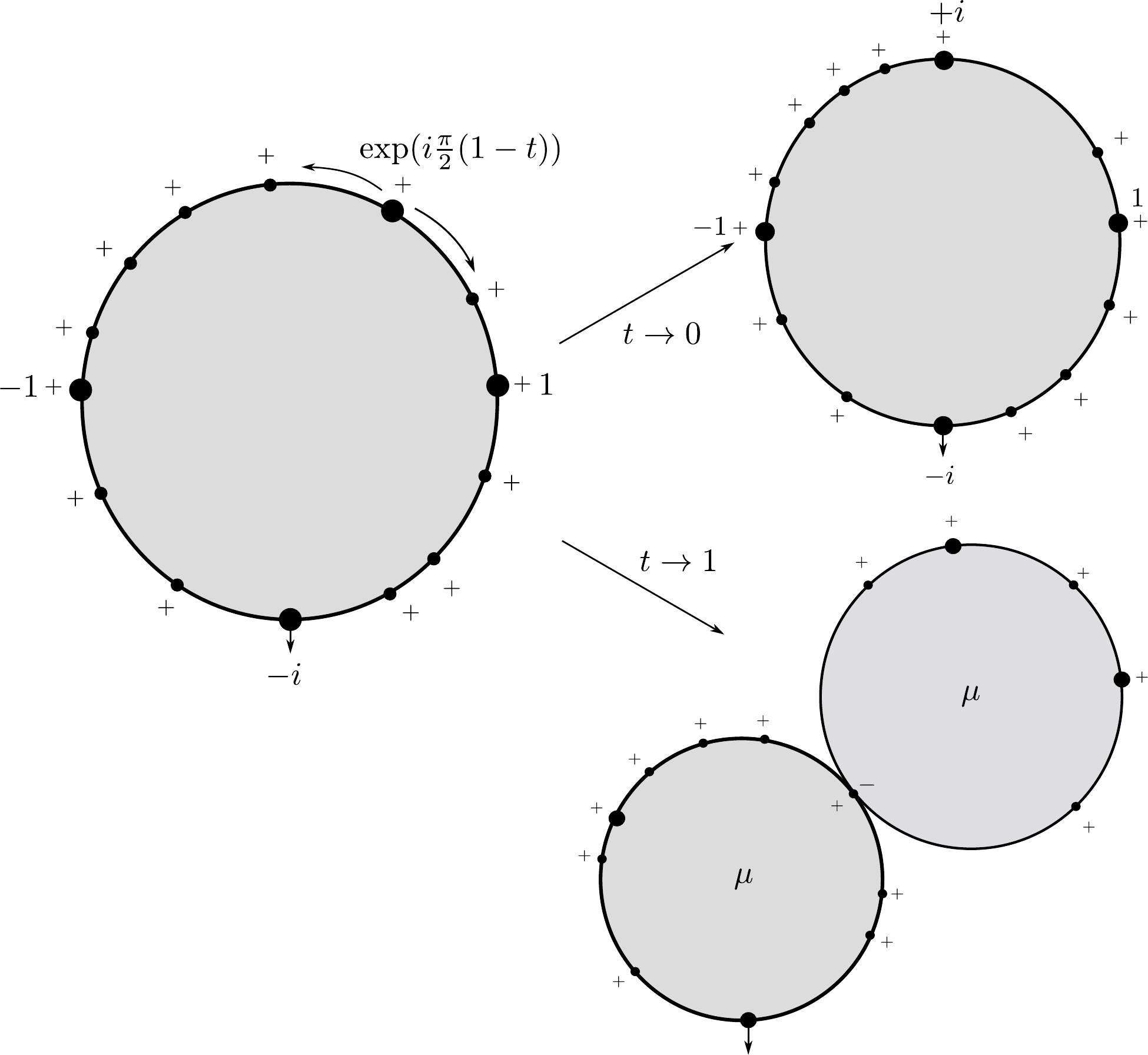}
\end{figure}

There is an associated Floer operation
\begin{equation}
    \mc{H}_{r,k,l,s} = \mathbf{F}_{\overline{S}^{r,k,l,s}},
\end{equation}
and we can thus define a morphism of bimodules, of degree -1
\begin{equation}
    \mc{H} \in \hom_{\w\!-\! \w} (\w_{\Delta} \otimes_\w \w_{\Delta} \otimes_\w \otimes \w_{\Delta},\w_{\Delta}),
\end{equation}
defined by
\begin{equation}
    \mc{H}^{r|1|s} = \bigoplus_{k,l} \mc{H}_{r,k,l,s}.
\end{equation}
An analysis of the boundaries of the one-dimensional moduli space of maps given
by $\overline{\mc{S}}_{r,k,l,s}$ reveals
\begin{prop} \label{chainhomotopylr}
    $\mc{H}$ is a chain homotopy between $\mc{F}_{\Delta,left,right}$ and $\mu_{LR}$.
\end{prop}
\begin{proof}
    The $t=1$ strata (\ref{flrstratum}) correspond to $\mc{F}_{\Delta,left,right}$,
    the $t=0$ strata (\ref{mulrstratum}) correspond to $\mu_{LR}$, and the
    other strata (\ref{beginintervalstrata})-(\ref{endintervalstrata})
    correspond to the chain homotopy terms $\mc{H} \circ d - d \circ \mc{H}$.
\end{proof}

\subsection{A family of annuli}
Recall that the morphism $\mc{CY}$ induces functorial maps
\begin{equation}
    \mc{CY}_\#: \w_{\Delta} \otimes_{\w\!-\w} \w_{\Delta} \lra
    \w_{\Delta} \otimes_{\w\!-\!\w} \w^!,
\end{equation}
which, composed with the map
\begin{equation}
    \bar{\mu}: \w_{\Delta} \otimes_{\w\!-\!\w} \w^! \lra \hom_{\w\!-\!\w}(\w_{\Delta},\w_{\Delta})
\end{equation}
defined in (\ref{explicitquasiiso}), gives a map from (two-pointed) Hochschild
homology to Hochschild cohomology. We now prove that this composed map is in
fact homotopic to the map from Hochschild homology to cohomology passing
through $SH^*(M)$.  
\begin{thm}[Generalized Cardy Condition]\label{cardythm1}
    There is a (homotopy)-commutative diagram
    \begin{equation}\label{cardy1}
        \xymatrix{\w_{\Delta} \otimes_{\w\!-\!\w} \w_{\Delta} \ar[r]^{\mc{CY}_\#} \ar[d]^{_2\oc} & \w_{\Delta} \otimes_{\w\!-\!\w} \w^! \ar[d]^{\bar{\mu}} \\
        SH^*(M) \ar[r]^{_2\co\ \ \ \ \ }& _2 \r{CC}^*(\w,\w)}
\end{equation}
\end{thm}
We first how this result completes the proof of Theorem \ref{shhh}.
\begin{proof}[Proof of Theorem \ref{shhh}]
    By Corollary \ref{shiso} and Theorem \ref{wrapcy}, the maps $_2 \co$ and
    $\mc{CY}_\#$ are isomorphisms. Moreover, so is $\bar{\mu}$, by Proposition
    \ref{representhh} (see also (\ref{explicitquasiiso}) for the explicit for
    of $\bar{\mu}$). Thus by the diagram (\ref{cardy1}), so is $_2 \oc$. By
    Proposition \ref{twopointhomotopy1}, $_2
    \co$ and $_2 \oc$ are homotopic to the usual $\co$ and $\oc$, implying the
    result. 

    Alternatively, let us note that we didn't need to a priori know that $_2
    \co$ was an isomorphism to conclude the proof. Surjectivity of $_2 \oc$ and
    injectivity of $_2 \co$, the contents of Proposition \ref{surjectivity},
    suffice.
\end{proof}

To construct the homotopy, we introduce some auxiliary moduli spaces of annuli.

\begin{defn}
The moduli space 
\begin{equation}
    \mc{A}_1
\end{equation}
consists of annuli with two positive punctures on the inner
boundary, one positive puncture on the outer boundary, and one negative
puncture on the outer boundary. The codimension 3 subspace 
\begin{equation}
    \mc{A}_1^-
\end{equation}
consists of those annuli that are conformally equivalent to 
\begin{equation}\label{conformalannulus}
\{z | 1 \leq |z| \leq R\} \subset \C,
\end{equation}
for any (varying) $R$, with inner positive marked points at $\pm i$, outer
positive marked point at $Ri$, and outer negative marked point at $-Ri$.
\end{defn}
\begin{defn}
Define
\begin{equation}
    \mc{A}_{k,l;s,t}
\end{equation}
to be the moduli space of annuli with 
\begin{itemize}
    \item $k + l + 2$ positive marked points on
the inner boundary, labeled $a_0, a_1, \ldots, a_k$, $a_0', a_1', \ldots, a_l'$ in
counterclockwise order, 
\item one negative marked point on the outer boundary, labeled $z_{out}$, and
\item $s + t + 1$ positive marked points on the outer boundary, labeled
    counterclockwise from $z_{out}$ as $b_1, \ldots, b_s, b_{0}', b_1', \ldots, b_t'$.
\end{itemize}
\end{defn}
\noindent There is a map
\begin{equation}
    \pi: \mc{A}_{k,l;s,t} \lra \mc{A}_1
\end{equation}
given by forgetting all of the marked points except for $a_0$, $a_0'$,
$b_{0}'$, and $z_{out}$. 
\begin{defn}
Define
\begin{equation}
    \mc{A}_{k,l;s,t}^-
\end{equation}
to be the pre-image of $\mc{A}_1^-$ under $\pi$.
\end{defn}
Via the map
\begin{equation}
    \mc{A}_1^- \lra (0,1)
\end{equation}
which associates to any annulus the scaling parameter $\frac{R}{1+R}$, the
space $\mc{A}_{k,l;s,t}^-$ also fibers over $(0,1)$.  Compactifying, we see
that $\overline{\mc{A}}_{k,l;s,t}^-$ submerses over $[0,1]$, and has boundary
stratum covered by the natural images of the inclusions of the following
codimension 1 strata
\begin{align}
    \label{annulusstrata1} \mc{R}^1_{k,l} &\times \mc{R}^{1,1}_{s,t} \ \ \ \ \ \ \ \ \ \ \ \ \ \ \ \ \ \ \textrm{(fiber over 1)}\\
    \label{annulus2} \mc{R}_2^{s',t',l',k'} &\times_{(1,1+),(2,3+)} \mc{R}^{s-s',k-k',l-l',t-t'} \ \ \ \textrm{(fiber over 0)}\\
    \label{firsthomologydiff}\overline{\mc{R}}^{k'} &\times^a_{n+1} \overline{\mc{A}}_{k-k'+1,l;s,t}^-,\ \ \ 0 \leq n < k - k' + 1\\
    \overline{\mc{R}}^{l'} \times^a_{(n+1)'} &\times \overline{\mc{A}}_{k,l-l'+1;s,t}^- ,\ \ \ 0 \leq n < l-l'+1\\
    \label{thirdhomologydiff}\overline{\mc{R}}^{k'+l' + 1} &\times^a_{0} \overline{\mc{A}}_{k-k',l-l';s,t}^-\\
    \label{lasthomologydiff}\overline{\mc{R}}^{l'+k' + 1} &\times^a_{0'} \overline{\mc{A}}_{k-k',l-l';s,t}^-\\
    \label{firstcohomologydiff}\overline{\mc{R}}^{s'} &\times^b_{n+1} \overline{\mc{A}}_{k,l;s-s'+1,t}^-,\ \ \ 0 \leq n < s-s'+1 \\
    \overline{\mc{R}}^{t'} &\times^b_{(n+1)'} \overline{\mc{A}}_{k,l;s,t-t'+1}^-,\ \ \ 0 \leq n < t-t'+1\\
    \label{middlecohomologydiff}\overline{\mc{R}}^{s'+t'+1} &\times^b_{0'} \overline{\mc{A}}_{k,l;s-s',t-t'}^-\\
    \label{lastcohomologydiff}\overline{\mc{A}}_{k,l;s-s',t-t'}^- &\times_{s'+1} \overline{\mc{R}}^{t'+s'+1}.
\end{align}
Here, the notation $\times^a_j$ means the output of the first stratum is glued
to the input point $a_j$, $\times^a_{j'}$ means glue to $a_j'$, and
$\times^b_j$, $\times^b_{j'}$ mean the same for $b_j$, $b_j'$. Also, in
(\ref{annulus2}), the two special inputs of the first factor and the second
special input of the second factor become the special input points on the
annulus after gluing. Also, in (\ref{thirdhomologydiff}) and
(\ref{lasthomologydiff}), the $k'+1$st and $l'+1$st marked points of the first
component become the special point ($a_0$ or $a_0'$ respectively) after gluing.

\begin{figure}[h] 
    \caption{The space of annuli $\mc{A}_{k,l;s,t}^-$ and its degenerations associated to the endpoints $\{0,1\}$. \label{annulus1}}
    \centering
    \includegraphics[scale=0.7]{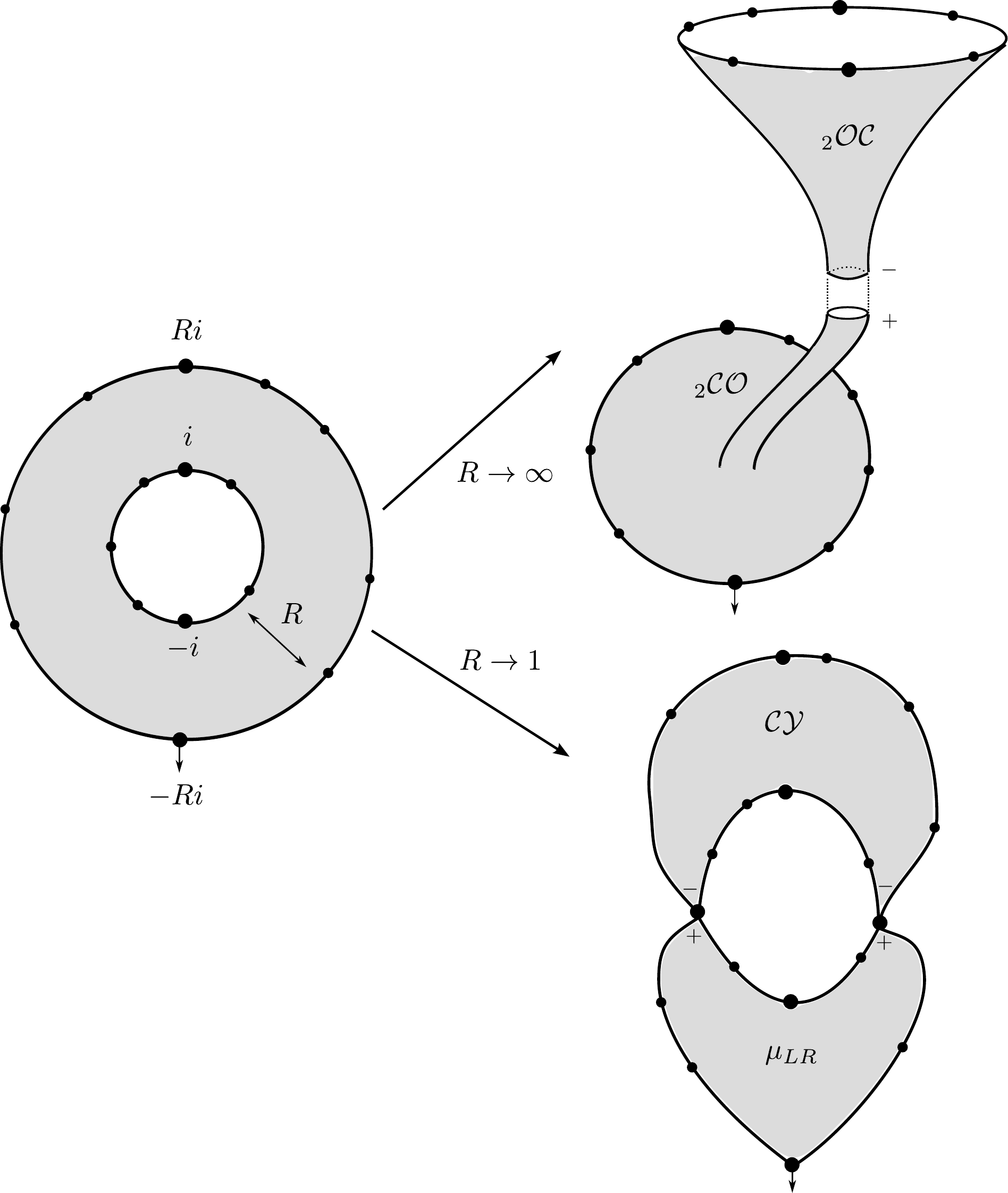}
\end{figure}

Given sign twisting datum
\begin{equation}
    \vec{t}_{\mathbf{A},k,l,s,t} = \{(1, \ldots, k, k, k +1, \ldots, k + l, k + l, 1, \ldots, s, s, s+1, \ldots, s + t, s + t\}
\end{equation}
with respect to the ordering of boundary inputs 
\begin{equation}
    a_1, \ldots, a_k, a_0', a_1', \ldots, a_l', a_0, b_1, \ldots, b_s, z_{in}, b_1', \ldots, b_t'
\end{equation}
there are associated Floer operations
\begin{equation}
    \begin{split}
    \mathbf{A}_{k,l;s,t}& := (-1)^{\vec{t}_{\mathbf{A},k,l,s,t}}\mathbf{F}_{\overline{\mc{A}}_{k,l;s,t}^-}:\\ 
    &(\w_{\Delta} \otimes \w^{\otimes l} \otimes \w_{\Delta} \otimes \w^{\otimes k})^{diag} \otimes \w^{\otimes t} \otimes \w_{\Delta} \otimes \w^{\otimes s} \lra \w_{\Delta}
\end{split}
\end{equation}
where we have indicated the inputs corresponding to the special points  $a_0$,
$a_0'$, and $b_{0}'$ by the first, second, and third $\w_{\Delta}$ input
factor, and the output  $z_{out}$ by the output $\w_{\Delta}$ factor. As usual,
the $diag$ superscript indicates that the first set
of $k+l+2$ inputs must be cyclically composable.
Using these operations, define a map
\begin{equation}
    \mathbf{A}: {_2}\r{CC}_*(\w,\w) \lra {_2}\r{CC}^*(\w,\w)
\end{equation}
by 
\begin{equation}
    \mathbf{A}:
    (\mathbf{x} \otimes x_r \otimes \cdots \otimes x_1 \otimes \mathbf{y} \otimes y_q \otimes \cdots \otimes y_1) \longmapsto \Phi
\end{equation}
where $\Phi$ is the 2-Hochschild co-chain given by
\begin{equation}
    \begin{split}
    \Phi(c_m, \ldots, c_1, \mathbf{c}, d_n, \ldots, d_1) = \mathbf{A}_{q,r;n,m}(&\mathbf{x} \otimes x_r \otimes \cdots \otimes x_1 \otimes \mathbf{y} \\
    &\otimes y_q \otimes \cdots \otimes y_q; c_m, \ldots, c_1, \mathbf{c}, d_n, \ldots, d_1).
\end{split}
\end{equation}
A dimension computation shows that the operation $\mathbf{A}$ has degree
$n-1$ as a map from Hochschild homology to Hochschild cohomology.  
An analysis of the boundary of the one-dimensional moduli spaces of maps with
source domain the various $\mc{A}_{k,l;s,t}^-$ reveals: 
\begin{prop}
    $\mathbf{A}$ gives a chain homotopy between ${_2}\co \circ {_2}\oc$
    and $\mu_{LR} \circ \mc{CY}_{\#}$.
\end{prop}
\begin{proof}
    The strata over the endpoints of the interval $\{0,1\}$ correspond exactly
    to the operations ${_2}\co \circ {_2}\oc$ and $\mu_{LR} \circ
    \mc{CY}_{\#}$.  The various intermediate strata give terms corresponding to
    $d_{\r{CC}^*} \circ \mathbf{A} \pm \mathbf{A} \circ d_{\r{CC}_*}$. In
    Appendix \ref{orientationsection} we discuss the ingredients necessary to
    check the signs of this equation.
\end{proof}

By postcomposing with the chain homotopy in Proposition \ref{chainhomotopylr}
between $\bar{\mu}$ and $\mu_{LR}$, Theorem \ref{cardythm1} follows.

\begin{rem}
    If bimodules $\mc{B}_0,\mc{B}_1$ come from any Lagrangian in the product
    for which we are able to define the quilt functor, 
    \begin{equation} 
        \begin{split}
        \mc{B}_0 &= \mathbf{M}(L_0) \\
        \mc{B}_1 &=\mathbf{M}(L_1)
    \end{split}
    \end{equation} 
    then there is an analogue of the Cardy condition, 
    which looks like (\ref{cardy1}):
    \begin{equation}
        \xymatrix{\mc{B}_0 \otimes_{\w\!-\!\w} \mc{B}_1 \ar[r]^{\mc{CY}_\#^{\mc{B}_0}} \ar[d]^{\oc} & \mc{B}_0^! \otimes_{\w\!-\!\w} \mc{B}_1 \ar[d]^{\mu} \\
        \hom_{\w^2}(L_0,L_1) \ar[r]^{\co\ }& \hom_{\w\!-\!\w}(\mc{B}_0,\mc{B}_1)}
\end{equation}
Here, $\mc{B}_0^!$ is the bimodule dual of $\mc{B}_0$, as defined in Section
\ref{moduleduality}, and $\mc{CY}^{\mc{B}_0}$ is a generalization of our
$\mc{CY}$ morphism.
The moduli space controlling the relevant commutative diagram is a quilted
generalization of the annulus.  The only obstacle to the existence of this
diagram for arbitrary pairs of Lagrangians in $M^2$ 
is our current inability to define the functor $\mathbf{M}$ in complete
generality, due to issues of admissibility and compactness of moduli spaces in
$M^2$.  
\end{rem}

\section{Some consequences} \label{consequencessection}
\subsection{A converse result}\label{converseresult}
In Section \ref{splitgendeltasection}, we proved that if $M$ is non-degenerate,
then the product Lagrangians $L_i \times L_j$ split-generate $\Delta$ in
$\w^2$. The proof went via analyzing a homotopy commutative diagram
\begin{equation}
    \xymatrix{_2\r{CC}_*(\w,\w) \ar[r]^{\Gamma\ \ \ }\ar[d]^{_2\oc} & \y^r_\d \otimes_{\tilde{\w}^2_{split}}\y^l_\d \ar[d]^{\mu} \\
    CH^*(M) \ar[r]^{D} & \hom_{\tilde{\w}^2}(\Delta,\Delta)},
\end{equation} 
where $\tilde{\w}_{split}^2$ was a category quasi-isomorphic to $\w^2_{split}$,
the full subcategory of product Lagrangians in $\w^2$.
\begin{cor}
    Under the same hypotheses, $\Gamma$ is a quasi-isomorphism.
\end{cor}
\begin{proof}
    The map $H^*(\mu)$ hits the unit, so by Proposition \ref{splitgeniso}
    $H^*(\mu)$ is an isomorphism. $D$ is an isomorphism by definition, and we
    then note, thanks to Theorem \ref{cardythm}, that $_2 \oc$ is also a
    quasi-isomorphism. Hence, $\Gamma$ is a quasi-isomorphism.
\end{proof}
\begin{rem} 
    It seems believable that the map $\Gamma$ is always an isomorphism,
    via the existence of an explicit quasi-inverse instead of such circuitous arguments. At
    the time of writing we have not come up with a simple proof.  
\end{rem}

With the technology we have established, we can also see that non-degeneracy is
in fact {\it equivalent} to split-generation of the diagonal: 
\begin{prop}
    If $\Delta$ is split-generated by product Lagrangians in $\w^2$, then
    $M$ is non-degenerate.
\end{prop}
\begin{proof} 
    If so, then since the quilt functor $\mathbf{M}$ is full on
    product Lagrangians (Proposition \ref{quiltkunneth}), we conclude that
    $\mathbf{M}$ is full on $\Delta$. This implies that the
    maps $\mc{CY}_{\#}$, $\bar{\mu}$, and 
    $\co$ 
    are isomorphisms in the Cardy
    Condition diagram (\ref{cardy}) (this is the content of Corollary
    \ref{calabiyau}, Corollary \ref{representhh}, and Proposition
    \ref{m1equalsco}). Hence $\oc$ is an isomorphism as well; in particular, it
    hits the unit.
\end{proof}
In particular, $\Gamma$ is once more an isomorphism.

\subsection{The fundamental class}\label{fundclass}
Assume $M$ is non-degenerate, and let $\sigma \in \r{CC}_*(\w,\w)$ be any
pre-image of $[e]\in SH^*(M)$ under the map $\oc$. Because $\oc$ is of degree
$n$, $\sigma$ is a degree $-n$ element. Following terminology from the
introduction, call $\sigma$ a {\bf fundamental class} for the wrapped Fukaya
category. Our reason for this terminology is that
\begin{cor}
    Cap product with $\sigma$ induces an isomorphism
    \begin{equation}
        \cdot \cap \sigma: \r{HH}^*(\w,\w) \stackrel{\sim}{\lra} \r{HH}_{*-n}(\w,\w)
    \end{equation}
    that is quasi-inverse to the geometric morphisms 
    $\co \circ \oc$. Thus, by Theorem \ref{cardythm}, it is also quasi-inverse
    to $\mu_{LR} \circ \mc{CY}_\#$.
\end{cor}
\begin{proof}
    We note that by the module structure compatibility of $\oc$ the
    following holds on the level of homology: 
    \begin{equation}
        \begin{split}
        \oc((\co \circ \oc(x)) \cap \sigma) &= \oc(x) \cdot \oc(\sigma)\\
        &= \oc(x) \cdot [e]\\
        &= \oc(x).
    \end{split}
\end{equation}
Since $\oc$ is a homology-level isomorphism, we conclude that
\begin{equation}
    (\co \circ \oc(x)) \cap \sigma = x
\end{equation}
as desired.
\end{proof}

\subsection{A ring structure on Hochschild homology}\label{ringstructure}
We can pull back the ring structure from Hochschild cohomology to Hochschild
homology. Thanks to Theorem \ref{cardythm}, this can be done without passing
through symplectic cohomology.
\begin{cor}
    Let $\sigma$ be the pre-image of the unit, and let $\a$, $\beta$ be two
    classes in $\r{HH}_*(\w,\w)$ that map to elements $a$ and $b$ of symplectic
    cohomology via $\oc$. Then, the following Hochschild homology classes are equal
    in homology and map to $a \cdot b$: 
    \begin{align}
        \label{nonexplicitproduct} \alpha \star^1 \beta &:= (\bar{\mu} \circ \mc{CY}_{\#}(\a)) \cap \beta \\
        \alpha \star^2 \beta &:= \a \cap (\bar{\mu} \circ \mc{CY}_{\#}(\beta)) \\
        \alpha \star^3 \beta &:= 
        \left((\bar{\mu} \circ \mc{CY}_{\#})(\a) * (\bar{\mu}\circ
        \mc{CY}_{\#})(\beta)\right) \cap \sigma.
    \end{align}
\end{cor}

It is illustrative to write down an explicit expression for
(\ref{nonexplicitproduct}) in terms of operations. First, we note that for
two-pointed complexes, cap-product has a very simple form
\begin{equation}
    \begin{split}
        {_2}\r{CC}_*(\w,\w) \times {_2}\r{CC}^*(\w,\w) &\lra {_2}\r{CC}_*(\w,\w)\\
    (\alpha , \mc{F}) &\longmapsto \mc{F}_\# (\alpha)
\end{split}
\end{equation}
where $\mc{F}_\#$ is the pushforward operation on the tensor product
$\w_{\Delta} \otimes_{\w\!-\!\w} \w_{\Delta}$ that acts by collapsing terms
around and including the first factor of $\w_{\Delta}$ (with usual Koszul
reordering signs).

Now, if $\beta$ is the Hochschild class represented by
\begin{equation}
    \beta = \mathbf{a} \otimes b_1 \otimes \cdots \otimes b_t \otimes
    \mathbf{b} \otimes a_1 \otimes \cdots \otimes a_s 
\end{equation}
and $\a$ is represented by
\begin{equation}
    \alpha = \mathbf{c} \otimes c_1 \otimes \cdots \otimes c_v \otimes \mathbf{d} \otimes d_1\otimes \cdots \otimes d_w
\end{equation}
then the formula (\ref{nonexplicitproduct}) is, up to sign:
\begin{equation}
    \begin{split}
        \alpha \star^1 \beta:= \sum  \bigg(\bigg(\mu_\w(&a_{s-k''+1}, \ldots, a_{s-k'}, \cdot, c_{r'+1}, \ldots, c_{r''-1}, \\
        &\mu_\w(  c_{r''}, \ldots, c_v \mathbf{d}, d_1, \ldots, d_{w-q'}, \cdot, b_{l'+1}, \ldots, b_{l'+l''}),\\
        & \ \ \ \ \ \ b_{l'+l''+1}, \ldots, b_{l'+l''+l'''} \bigg) \\
        &\circ  \mc{CY}_{k',l',q',r'}(d_{w-q'+1}, \ldots, d_w, \mathbf{c}, c_1, \ldots, c_r'; \\
        &\ \ \ \ \ \ \ a_{s-k'+1}, \ldots, a_s, \mathbf{a}, b_1, \ldots, b_l')\bigg)\\
        &\otimes b_{l'+l''+l'''+1} \otimes \cdots \otimes b_t \otimes \mathbf{b} \otimes a_1 \otimes \cdots a_{s-k''}.
    \end{split}
\end{equation}

This formula is formally analogous to Bourgeois, Ekholm, and Eliashberg's
surgery formula for the symplectic cohomology product on a Weinstein manifold
\cite{BEE2}, which is expressed as sums over operations similar to $\mc{CY}$ and
$\mu_{\w}$ applied to Legendrian symplectic field theory invariants of
attaching cores. A detailed comparison between these formulae, in the case that
our essential collection of Lagrangians are the ascending co-cores of a
Weinstein manifold, will appear elsewhere.

\appendix
\section{Action, energy and compactness} \label{actionsection}

The goal of this appendix is to prove a compactness result for Floer-theoretic
operations controlled by bordered Riemann surfaces mapping into a Liouville
manifold, under some assumptions about the almost complex structure and
Hamiltonian perturbation terms.  There are several existing compactness results
for the wrapped Fukaya category and some open-closed maps, e.g. 
\cite{Abouzaid:2010kx}*{\S B}, which are unfortunately not directly applicable for our
choices of Hamiltonian perturbations.
The problems occur because we use time-dependent perturbations of a standard
Hamiltonian, which we cannot guarantee will vanish at infinitely many levels of
the cylindrical coordinate $r$ (this is an essential assumption in
\cite{Abouzaid:2010kx}*{\S B}).
Solutions to Floer's equation for such perturbed Hamiltonians will fail to satisfy
a maximum principle, but if the complex structure has been carefully chosen and
the time-dependent perturbations are sufficiently small, this failure can be
controlled. We make use in an essential way of a delicate technique for
obtaining a-priori $C^0$ bounds on such solutions due to Floer-Hofer and
Cieliebak \cite{Floer:1994uq} \cite{Cieliebak:1994fk}. This technique has also
been used by Oancea \cite{Oancea:2008fk}, whose work we draw upon.

\begin{rem}
    Our situation is a little different from \cite{Floer:1994uq}
    \cite{Cieliebak:1994fk} \cite{Oancea:2008fk} in that we need a variant of
    their compactness result for potentially finite cylindrical regions in a
    larger Riemann surface. This, and differing conventions regarding
    Hamiltonians (quadratic versus linear) and complex structures (contact type
    versus rescaled contact type) prevent us from citing any of these papers
    directly.
\end{rem}

Our setup is as follows: Let $W$ be a Liouville manifold with cylindrical end
\begin{equation}
    W = \bar{W} \cup_{\bd W} \bd W \times [1,\infty)_r.
\end{equation}
The coordinate on the
end is given by a function $\pi_r$ on $W - \bar{W}$, which we extend over the
interior of $W$ to a function 
\begin{equation}
\pi_r: W \lra [0,\infty)
\end{equation}
such that
\begin{equation}
\bar{W} =\pi^{-1}([0,1]).
\end{equation}
\noindent Let $S$ be a bordered surface with boundary $\partial S$, and equip $S$
with a Floer datum in the sense of Definition
\ref{floeropenclosed}, namely:
\begin{itemize}
    \item a collection of {\bf $\delta$-bounded weighted strip and cylinder data},
    \item a {\bf sub-closed one form $\a_S$}, compatible with the weighted strip and
        cylinder data
    \item a {\bf primary Hamiltonian} $H_S : S \ra \mc{H}(M)$ that is $H$ compatible with the weighted strip and cylinder data
    \item an adapted {\bf rescaling function $a_S$}, constant and equal to the weights on each strip and cylinder region,
    \item an {\bf almost complex structure} $J_S$ that is adapted to the weighted strip and cylinder data, the rescaling function $a_S$, and some fixed $J_t$, and
    \item an {\bf $S^1$ perturbation} $F_S$ adapted to $(F_T, \phi_\e)$ for some $F_T$, $\phi_\e$ as in the definition.
\end{itemize} 
\begin{rem}\label{oneoutput}
    The requirement that $S$ carry a sub-closed one form with asymptotics
    specified above implies via Stokes theorem that $S$ must have at least one
    output marked point. Although the use of the sub-closed one form may seem
    like a technical artifice of the proofs of compactness here and elsewhere,
    this constraint on $S$ is an essential one. Indeed, symplectic cohomology
    and morphism spaces in the wrapped category are often infinite dimension,
    and thus cannot admit the Poincar\'{e} duality type pairings that arise
    from surfaces without outputs.  
\end{rem}
Fix a Lagrangian labeling $\vec{L}$ for the boundary
components of $S$ and a
compatible choice of input and output chords and orbits corresponding to the
positive and negative marked points on $S$
\begin{equation}
    \begin{split}
        \vec{x}_{in}, \vec{x}_{out}, \vec{y}_{in}, \vec{y}_{out}.
    \end{split}
\end{equation}
We study maps $u: S \ra W$ satisfying Floer's equation for this datum, namely
\begin{equation}
    (du - X_S \otimes \gamma)^{0,1} = 0
\end{equation}
with asymptotic/boundary conditions
\begin{equation}
    \nonumber \begin{cases}\lim_{p \ra z^{\pm}_k} u = y_k^\pm \\
     \lim_{p \ra b^{\pm}_j} u = x_j^\pm \\
    \mathrm{for\ } p\in \bd^n S,& u(p) \in \psi^{\rho(p)}(L_n)
\end{cases}.
\end{equation}
Here $X_S$ is the Hamiltonian vector field corresponding to the {\bf total
Hamiltonian} 
\begin{equation}
    H^{tot}_S = H_S + F_S.
\end{equation}
The compactness result we need is:
\begin{thm} \label{c0bounds}
Given such a map $u: S \ra W$, there is a constant $C$ depending only on $F_t$,
and $\phi_\e$, $\vec{x}_{in}$, $\vec{x}_{out}$, $\vec{y}_{in}$, $\vec{y}_{out}$
such that \begin{equation}
    (\pi_r \circ u) \leq C.
\end{equation}  
Moreover, given any set of $\vec{y}_{in}$
$\vec{x}_{in}$, there are a finite number of collections $\vec{y}_{out}$,
$\vec{x}_{out}$ for which the relevant moduli spaces are non-empty.
\end{thm}
\noindent First, we define appropriate notions of action and energy.
Suppose we have fixed a Hamiltonian $H$ and a time-dependent perturbation
$F_t$, and have picked a surface $S$ with compatible Floer data. Let $x \in
\chi(L_i, L_j)$ be the asymptotic condition at strip-like end $\e^k$ with
corresponding {\it weight} $w_k$. Moreover, suppose the perturbation term $F_S$
is equal to the constant $C_k$ on this strip-like end (this can be chosen to be
zero if there are no interior marked points in $S$).
    
\begin{defn} The {\bf action} of $x$ is defined to be the quantity 
    \begin{equation}
        \mc{A}(x) : = - \int_{0}^{1} (\tilde{x}_{w_k})^* \theta + \int_0^1 w_k \cdot H^{w_k}(x(t))dt + f_{L_j}(x(1)) - f_{L_i}(x(0)) + w_k C_k.
    \end{equation}
    where the $f_{L_i}$ are the chosen fixed primitives of the Lagrangians
    $L_i$, 
    \begin{equation}
        H^{w_k}:= \frac{(\psi^{w_k})^* H}{w_k^2},
    \end{equation}
    and 
    \begin{equation}
        \tilde{x}_{w_k}
    \end{equation}
    is the chord $x$ thought of as a time-1 chord of $w_k \cdot H^{w_k}$, under
    the rescaling correspondence (\ref{rescaling}).  
\end{defn}
Similarly, let $y \in \mc{O}$ be the asymptotic condition at cylindrical end
$\delta^l$ with corresponding weight $v_l$. 
\begin{defn} The {\bf action} of $y$ is the
    quantity 
    \begin{equation}
        \mc{A}(y) : = -\int_{S^1} (\tilde{y}_{v_l})^*\theta + \int_0^1 v_l\cdot H^{v_l} (x(t)) dt +
        \int_0^1 v_l\cdot  F_t^{v_l}(t,x(t))dt 
    \end{equation}
    where $H^{v_l}$ is as before and $F_t^{v_l}$ is defined as
    \begin{equation}
        F_t^v := \frac{(\psi^{v_l})^* F_t}{v_l^2}.
    \end{equation}
\end{defn}
\begin{lem}\label{negativeaction}
    The action of a Hamiltonian chord or orbit becomes arbitrarily negative
    as $r \ra \infty$.
\end{lem}
\begin{proof}
This Lemma is a variant of one in \cite[\S B.2]{Abouzaid:2010kx}.  The first
observation is, by Lemma \ref{rescalinglemma}, that 
\begin{equation}
    H^{w_k} =  H = r^2
\end{equation} 
away from a compact subset. Also, for a Hamiltonian chord $\tilde{x}_{w_k}$
\begin{equation}
        \label{quadraticprop}
    \begin{split}
        \tilde{x}_{w_k}^* \theta &=  \theta(X_{w_k \cdot H^{w_k}}) dt\\
        &= w_k\cdot \omega(Z,X_{H^{w_k}}) dt \\
        &= w_k \cdot dH^{w_k}(Z) dt \\
        &= 2w_k \cdot r dr(r \partial_r) dt \\
        &= 2w_k r^2 dt.
    \end{split}
    \end{equation}
Thus, for a chord $x \in \chi(L_i,L_j)$ away from a compact
set (so $f_{L_j}$ and $f_{L_i}$ are zero): 
    \begin{align}
        \mc{A}(x) &= -\int_0^1 {\tilde{x}_{w_k}}^*\theta + \int_0^1 w_k H^{w_k}(x(t)) dt + w_k C_k\\
        &= -\int_0^1 w_k \cdot r^2 dt + w_k C_k,
    \end{align}
which satisfies the Lemma. Similarly for an orbit $y \in \mc{O}$,
    \begin{align}
        \mc{A}(y) &= -\int_0^1 \tilde{y}_{v_l}^*\theta + \int_0^1 v_l \cdot H_S(t,y(t)) dt\\
        &= -\int_0^1 v_l\cdot (\theta(X_{H^{v_l}}) + \theta(X_{F^{v_l}}) - H^{v_l} - F^{v_l}) dt\\
        &= -\int_0^1 v_l \cdot (H^{v_l} + \theta(X_{F^{v_l}}) - F^{v_l}) dt.
    \end{align}
     The above expression also satisfies the Lemma, as $H^{v_l}$ dominates
     $\theta(X_{F^{v_l}})$ and $F^{v_l}$ away from a compact set.
\end{proof}
Following \cite{Abouzaid:2010ly}, given a map $u$ satisfying Floer's
equation, we define two notions of energy. The {\bf geometric energy} of $u$ is
defined as 
\begin{equation}
    E_{geo}(u) := \int_S ||du - X \otimes \gamma||^2.
\end{equation}
where the norm $||\cdot ||$ comes from the complex structure $J_S$. Picking
local coordinates $z = s + it$ for $S$, we see that for a solution $u$ to
Floer's equation,
\begin{equation} \label{energyidentityfloer}
    \begin{split}
E_{geo}(u)     &= \int \omega( (du - X \otimes \gamma) (\partial_s ), J_S (du - X \otimes \gamma) (\partial_s)ds dt\\
        &= \int \omega (du(\partial_s) - X\otimes \gamma(\partial_s), (du - X \otimes \gamma) \circ j (\partial_s)) ds dt \\
        &= \int_S \left(\omega(\partial_s u, \partial_t u) -  \omega(X \otimes \gamma(\partial_s), \partial_t u) - \omega(\partial_s u,X \otimes \gamma(\partial_t)) \right)dsdt\\
        &= \int_S \left( u^*\omega - (dH(\partial_s u) \gamma(\partial_t) - dH(\partial_t u) \gamma (\partial_s)) dsdt \right) \\
        &= \int_S u^*\omega - d(u^*H) \gamma,
    \end{split}
\end{equation}
a version of the energy identity for $J$-holomorphic curves. The {\bf
topological energy} of $u$ is defined as \begin{equation}
    E_{top}(u) := \int_S u^*\omega - \int_S d(u^*(H_S) \cdot \gamma). \label{topologicalenergy}
\end{equation}
Since $\gamma$ is sub-closed and $H_S$ is positive,
\begin{equation}\label{energyinequality}
    0 \leq E_{geo}(u) \leq E_{top}(u).
\end{equation}
Noting that
\begin{equation}
    \e_k^*(u^*H_S \gamma) = w_k \cdot (H^{w_k} + C_k) dt
\end{equation}
on strip-like ends and
\begin{equation}
    \delta_l^*(u^*H_S \gamma) = (w_k \cdot H^{v_l} + F_T^{v_l}) dt
\end{equation}
on cylindrical ends, we apply Stokes' theorem to (\ref{topologicalenergy}) to conclude:
\begin{equation} \label{stokesenergy}
        E_{top}(u) = (\sum_{y \in \vec{y}_{out}}
        \mc{A}(y) - \sum_{y \in \vec{y}_{in}} \mc{A}(y)) + (\sum_{x \in
        \vec{x}_{out}} \mc{A}(x) - \sum_{x \in \vec{x}_{in}} \mc{A}(x)).
\end{equation}
\begin{prop}\label{finitecorollary}
    For any $x_{in}$, $y_{in}$, there are finitely many choices of $x_{out}$,
    $y_{out}$ such that there is a solution $u$ to the relevant Floer's
    equation.  
\end{prop}
\begin{proof}
    By Lemma \ref{negativeaction} and (\ref{stokesenergy}),
    for fixed inputs $\vec{x}_{in}$, $\vec{y}_{in}$, for all but a finite selection
    of $\vec{x}_{out}$, $\vec{y}_{out}$, any $u$ satisfying Floer's equation
    with asymptotic conditions
    $(\vec{x}_{in},\vec{y}_{in},\vec{x}_{out},\vec{y}_{out})$ has negative
    $E_{geo}(u)$, which is impossible.
\end{proof}

Given a map 
\begin{equation}
    u: S \lra W
\end{equation}
as above, there is a notion of {\it intermediate action} of $u$ along a loop
$S^1 \hookrightarrow S$.  
\begin{defn}
    Define the {\bf intermediate action} of an embedded oriented loop $\mc{L} :
    S^1 \ra S$ to be: 
\begin{equation}
    \mc{A}(\mc{L}) : = -\int_{S^1} \mc{L}^*u^*\theta + \int_{S^1} \mc{L}^* (u^* (H_S\cdot \gamma))
\end{equation}
\end{defn}
\noindent A useful fact about intermediate action is:
\begin{lem} \label{intermediateactionbound}
There are constants $c_1$, $c_2$ depending on the input and output chords and
orbits such that 
\begin{equation}
    \mc{A}(\mc{L}) \in [c_1,c_2]
\end{equation}
for any embedded oriented loop $\mc{L}$.
\end{lem}
\begin{proof}
    As $S$ is genus 0, any embedded loop $\mc{L}$ is separates $S$ into two
    regions, one $S_{in}$ that has outgoing boundary on $\mc{L}$ and one
    $S_{out}$ that has incoming boundary on $\mc{L}$. We note that topological
    energy $E_{top}(u)$ is non-negative on any sub-region of $S$ and additive,
    i.e.
    \begin{equation}
        E_{top}(u) = E_{top}(u|_{S_{in}}) + E_{top}(u|_{S_{out}}),\ E_{top}(u|_{S_{in}}), E_{top}(u|_{S_{out}}) \geq 0. 
    \end{equation}
    Hence
    \begin{equation}
        \begin{split}
            E_{top}(u|_{S_{out}}) &\leq E_{top}(u)\\
            E_{top}(u|_{S_{in}}) &\leq E_{top}(u).
        \end{split}
    \end{equation}
    Now each of $E_{top}(u|_{S_{out}})$, $E_{top}(u|_{S_{in}})$ can be
    expressed via Stokes theorem as the positive or negative action of $\mc{L}$
    respectively plus/minus actions of inputs and outputs on each of
    $S_{out}$/$S_{in}$, so we obtain upper and lower bounds for
    $\mc{A}(\mc{L})$ in terms of actions of inputs and outputs.
\end{proof}

Now, given a fixed Floer datum on $S$, we view $S$ as the union of two regions
with different Hamiltonian behavior. On the first region, there is a non-zero
time-dependent perturbation: 
\begin{defn}
    Define the {\bf cylindrical perturbed regions} $S^c$ of $S$ to be the union of the images of the 
    \begin{itemize}
        \item cylindrical ends 
            \[
            \delta_{\pm}^j: A_{\pm} \lra S,
            \] and
        \item finite cylinders
            \[
            \delta^r: [a_r,b_r] \times S^1 \lra S\]
    \end{itemize}
    in the strip and cylinder data chosen for $S$.
\end{defn}

Recall that the Floer datum on $S$ consists of $\delta$-bounded cylinder data,
and an $S^1$ perturbation adapted to $(F_T,\phi_\e)$, for some chosen $\delta
\ll 1$ and $\e \ll 1$, the sense of Definitions \ref{stripcylinderdata} and
\ref{adaptedperturbation}. This implies that on the $\bar{\delta}$-collar (see
Definition \ref{deltacollar}) 
\begin{equation}
    S^{\bar{\delta}}
\end{equation}
of $S$, with
\begin{equation}
    \bar{\delta} = \delta \cdot \epsilon,
\end{equation}
the perturbation $F_T$ is locally constant. 

\begin{defn}
    Define the {\bf unperturbed region} $S^u$ of $S$ to be the union of the
    complement of the cylindrical perturbed region $S \backslash S^c$ with the
    $\bar{\delta}$-collar $S^{\bar{\delta}}$. This is the region where the
    perturbation term $F_T$ is locally constant.
\end{defn}
\noindent The intersection of the two regions $S^u$ and $S^c$ is exactly the
$\bar{\delta}$ collar $S^{\bar{\delta}}$.

We now examine the function
\begin{equation}
    \rho = \pi_r \circ u
\end{equation}
on each of the two regions of $S$, $S^u$ and $S^c$. In the next two
sections, we prove the following claims:
\begin{prop}\label{unperturbedprop}
    A maximum principle holds for $\rho$ on the unperturbed region $S^u$.
\end{prop}

\begin{prop}\label{perturbedprop}
    On the portion of the cylindrical perturbed region $S^c$ outside the
    $\bar{\delta}$-collar
    \begin{equation}
        S^c \backslash S^{\bar{\delta}}
    \end{equation}
     there is an upper bound on $\rho$ in terms of the Floer data and
     asymptotic conditions.
\end{prop}
\begin{rem}
Note that we are only able to directly establish direct upper bounds for $\rho$
outside of a collar of $S^c$, and are forced to rely on the maximum principle
to deduce bounds for $\rho$ on the collar. This has to do with the technique
used to establish such bounds.  
\end{rem}

These results, along with Proposition \ref{finitecorollary} imply Theorem
\ref{c0bounds}.

\subsection{The unperturbed region}
In a portion of the unperturbed region $S^u$ mapping to the conical end of $W$,
the total Hamiltonian is given by $H_S$, a quadratic function, plus a locally
constant function $F_S$. Thus, the Hamiltonian vector field $X_S$ is equal to
\begin{equation}
    X_S = 2r \cdot R,
\end{equation} 
where $R$ is the Reeb flow on $\partial W$.  
In particular, 
\begin{equation}
    dr(X) = 0.
\end{equation}
Recall that on the conical end, our (surface-dependent) almost complex
structure $J_S$ satisfies 
\begin{equation}
    dr \circ J = \frac{a_S}{r} \cdot \theta
\end{equation}
for some positive rescaling function $a_S: S \ra [1,\infty)$. Namely, setting 
\begin{equation}
    \mc{S} := \frac{1}{2} r^2
\end{equation}
we have that
\begin{equation}
    d\mc{S} \circ J =  r dr \circ J = a_s \cdot \theta.
\end{equation}
Now, consider Floer's equation on the conical end
\begin{equation}
    J \circ (du - X \otimes \gamma) = (du - X\otimes \gamma) \circ j
\end{equation}
\noindent and apply $d\mc{S}$ to both sides, with $\xi = \mc{S} \circ u = \frac{1}{2}\rho^2$ to obtain
\begin{equation}
    \label{dcrhogamma}
    d^c \xi = -a_S\cdot ( \theta \circ (du - X \otimes \gamma))
\end{equation}
\noindent Differentiating once more, 
\begin{equation}
    \label{laprhogamma}
    \begin{split}
    d d^c \xi 
    =-a_S (u^* \omega - d (\theta(X) \cdot \gamma))  
    -da_S \wedge (\theta \circ (du - X \otimes \gamma)).
  \end{split}
\end{equation}
\noindent Substituting (\ref{dcrhogamma}) into (\ref{laprhogamma}), we finally
obtain the following second-order differential equation for $\xi$:
\begin{equation}
    d d^c \xi = -a_S (u^* \omega - d(\theta(X) \cdot \gamma)) +
    \frac{da_S}{a_S} \wedge d^c \xi
\end{equation}
On the cylindrical end, we have that up to a locally constant function
\begin{equation}
\theta(X)= 2r^2 = 2 H;
\end{equation}
hence
\begin{equation}
    \begin{split}
    u^*\omega - d(\theta(X) \cdot \gamma) &= u^* \omega - d(\theta(X)) \wedge \gamma - \theta(X) d\gamma \\
    &= u^* \omega - 2 d(u^*H)\wedge \gamma - \theta(X) d\gamma \\
    &= (u^* \omega - d(u^*H)\wedge \gamma) +  \theta(X) d\gamma - d(u^*H)\wedge \gamma \\
    &= (u^* \omega - d(u^*H) \wedge \gamma) - \theta(X) d\gamma - d (2\xi) \wedge \gamma.
\end{split}
\end{equation}
Thus, $\xi$ satisfies
\begin{equation}
    d d^c \xi + 2 a_S d \xi \wedge \gamma - \frac{da_S}{a_S}\wedge d^c \xi = -a_S ( u ^* \omega - d(u^*H) \wedge \gamma) + a_S \theta(X) d\gamma. 
\end{equation}
Note that by the energy identity (\ref{energyidentityfloer}) and the fact that
$\gamma$ is sub-closed,
\begin{equation}
    -a_S (u^* \omega - d(u^*H) \wedge \gamma) + a_S \theta(X) d\gamma \leq 0.
\end{equation}
Thus, $\xi$ satisfies an equation of the form
\begin{equation}
d d^c \xi + 2 a_S d \xi \wedge \gamma - \frac{da_S}{a_S}\wedge d^c \xi  \leq 0
\end{equation}
which in local coordinates $z = s+it$ looks like
\begin{equation}
    \Delta \xi + v(s,t) \partial_s \xi + w(s,t) \partial_t \xi \geq 0.
\end{equation}
for some functions $v$, $w$.  Such equations are known to satisfy the maximum
principle; see e.g.
\cite{Evans:2010fk}. 
To finally establish Proposition \ref{unperturbedprop}, we
must show that maxima of $\xi$ achieved along portions of the Lagrangian
boundary $\partial S|_{S^{u}}$ mapping to the cylindrical end also satisfy
\begin{equation}
    d \xi = 0
\end{equation}
hence are subject to the usual maximum principle. Pick local coordinates $z =
s+ it$ near a boundary point $p$ with boundary locally modeled by $\{t = 0\}$. 
For a boundary maximum,
\begin{equation}
    \partial_s \xi = 0.
\end{equation}
Using (\ref{dcrhogamma}) to calculate $\partial_t \xi$, we see that
\begin{equation}
    \begin{split}
     \partial_t \xi &= d^c \xi (\partial_s) \\
    &= -a_S \theta \circ (du - X \otimes \gamma) (\partial_s).
\end{split}
\end{equation}
Since $\gamma$ was chosen to equal zero on the boundary of $S$, we have that
$X \otimes \gamma(\partial_s) = 0$. Similarly, at our point $p$, $\partial_s u$
lies in the tangent space of an exact Lagrangian $L$ with chosen primitive $f_L$
vanishing on the cylindrical end. Thus $\theta(\partial_s u) = 0$. Putting
these together, 
\begin{equation}
    \partial_t \xi = 0
\end{equation}
as desired.

\subsection{A convexity argument for the unperturbed region}
Below we present an alternate convexity argument for the unperturbed region.
This section is not strictly necessary, but we have included it for its
potential usefulness. 

Let $C$ be the overall bound for $\rho = \pi_r \circ u$ on the perturbed region
and consider 
\begin{equation}
    \tilde{S} := \rho^{-1}([C,\infty)) \subset S^{u}
\end{equation}
$\tilde{S}$ splits as a disjoint union of surfaces $\bar{S}$ on which the total
Hamiltonian is equal to a quadratic Hamiltonian $r^2$ plus a constant term
$K$ (different surfaces have different constant term). On any such
region $\bar{S}$, note that there is a refinement of the basic
geometric/topological energy inequality as follows:
\begin{equation}\label{refinement}
    \begin{split}
        E_{top}(u) &\geq E_{geo}(u) +  \int_{\bar{S}} u^* H (-d\gamma) \\
        &\geq E_{geo}(u) + (C^2 + K) \int_{\bar{S}} (-d\gamma) \\
    &\geq (C^2 + K) \int_{\bar{S}} (-d\gamma),
\end{split}
\end{equation}
with equality if and only if $E_{geo}(u) = 0$.
The boundary of $\bar{S}$ splits as $\partial^n \bar{S}$, the portion mapping
via $u$ to $\partial W \times \{C\}$, $\partial^l \bar{S}$, the portion with
Lagrangian boundary, and $\partial^p \bar{S}$, the punctures. Suppose there are
boundary punctures $\{x_i\}$ in
$\partial^l \bar{S}$.
We calculate, 
\begin{equation}
        E_{top}(u|_{\bar{S}}) = \int_{\partial^n \bar{S}} (u^*\theta - u^*H \gamma) + \int_{\partial^l \bar{S}} (u^*\theta - u^*H \gamma) + \sum_i \mc{A}(x_i).
\end{equation}
Note first that $\theta$ restricted to the cylindrical end of any Lagrangian is
zero, and similarly for $\gamma$; implying the second term above vanishes.
Moreover, 
\begin{equation}
    \theta(X) = 2r^2 = 2(u^*H - K)
\end{equation}
and 
\begin{equation}
    u^*H|_{\partial^n \bar{S}} = C^2 + K;
\end{equation}
hence
\begin{equation}
    \begin{split}
    E_{top}(u|_{\bar{S}}) &= \int_{\partial^n \bar{S}} \theta \circ (du - X \otimes \gamma) + \int_{\partial^n \bar{S}} (u^*H - 2K) \gamma + \sum_i \mc{A}(x_i)\\
    &= \int_{\partial^n \bar{S}} \theta \circ (du - X \otimes \gamma) + \int_{\partial^n \bar{S}} (C^2 - K) \gamma + \sum_i \mc{A}(x_i).
\end{split}
\end{equation}
Following action arguments in \cite{Abouzaid:2010kx}*{Appendix B} and
\cite{Abouzaid:2010ly}*{Lemma 7.2}, we rewrite
\begin{equation}
    \begin{split}
        \int_{\partial^n \bar{S}} \theta \circ (du - X \otimes \gamma) &= \int_{\partial_n \bar{S}} \theta \circ (-J) (du - X \otimes \gamma) \circ j \\
        &= \int_{\partial^n \bar{S}} \frac{r}{a_S} dr \circ (du - X \otimes \gamma) \circ j \\
        &= \int_{\partial^n \bar{S}} -\frac{r}{a_S} dr \circ (du) \circ j.
    \end{split}
\end{equation} 
as $dr \circ X = dr \circ (2r \cdot R) = 0$ on $\bar{S}$. As in
\cite{Abouzaid:2010ly}*{Lemma 7.2}, note now that for a vector $\xi$ tangent to
$\partial^n \bar{S}$ with the positive boundary orientation, $j \xi$ points
inward. Apply $du$ and note that in order for $j \xi$ to point inward, $du
\circ (j \xi)$ must not decrease the $r$-coordinate. Namely $dr \circ du \circ
j(\xi) \geq 0$, and
\begin{equation}
    \int_{\partial^n \bar{S}} \theta \circ (du - X \otimes \gamma) \leq 0.
\end{equation}
Thus,
\begin{equation}
    E_{top}(u|_{\bar{S}}) \leq \int_{\partial^n \bar{S}} (C^2 - K) \gamma + \sum_i \mc{A}(x_i).
\end{equation}
Outside a sufficiently large compact set, actions are negative, so (increasing
$C$ if necessary)
\begin{equation} 
    \begin{split}
    E_{top}(u|_{\bar{S}}) &\leq \int_{\partial^n \bar{S}} (C^2 - K) \gamma \\
    & \leq (C^2-K)(\int_{\bar{S}} d \gamma - \int_{\partial\bar{S}^p} \gamma)\\
    & \leq (C^2-K) \int_{\bar{S}} (d \gamma- \sum_i w_i)\\
    &\leq (C^2 +K) \int_{\bar{S}} (-d \gamma).
\end{split}
\end{equation}
Along with the opposite inequality (\ref{refinement}), this implies that
$E_{geo}(u|_{\bar{S}}) = 0$. So $du$ must be a constant multiple of the Reeb
flow, which is possible only if the image of $u|_{\bar{S}}$ is contained in a
single level $\partial W \times \{C\}$.

\subsection{The perturbed cylindrical regions}

The starting point for this case is the following classical refinement of the
maximum principle for uniformly elliptic second order linear differential
operators associated with Dirichlet problems in bounded domains:
\begin{prop}[Compare \cite{Cieliebak:1994fk}*{Prop. 5.1}] \label{dirichlet}
    Let $L$ be a strongly positive second-order elliptic differential operator
    on a domain $\Omega$ (such as $-\Delta$ on a finite cylinder $[a,b]\times
    S^1$). Let $\overline{\lambda}$ be the
    smallest eigenvalue of $L$.  Then $\overline{\lambda}$ is positive.
    Moreover, for any positive $\lambda < \overline{\lambda}$, if $f: \Omega
    \ra \R$ is smooth and satisfies the following properties:
    \begin{equation}
        \begin{cases}
            Lf \geq \lambda f &\mathrm{in\ }\Omega\\
            f = 0  &\mathrm{on\ }\bd \Omega
        \end{cases}
    \end{equation}
    then $f \geq 0$ on $\Omega$.
\end{prop}
\begin{proof}
    The proof combines a theorem of Krein-Rutman with the maximum principle,
    and can be found in Theorems 4.3 and 4.4 of \cite{Amann:1976fk}.
\end{proof}
\noindent Using this Proposition we prove a variant of a result of
Floer-Hofer \cite{Floer:1994uq}*{Prop. 8}: 
\begin{prop}\label{floerhoferprop8}
Let $Z$ be a cylinder of the form $[c,d]_s \times S^1_t$, with $c\in
[-\infty,\infty)$, $d\in (\infty,\infty]$.
Let \begin{equation}
    g: Z \lra \R
\end{equation}
be a function satisfying the following two properties: 
\begin{enumerate}
    \item 
        For any $\eta \ll 1$, there exists a $c_\eta$ with the following
        property: On any $\eta$-width sub-cylinder 
        \begin{equation}(s,s+\eta) \times S^1 \subseteq Z \end{equation}
            there is a loop $\{s'\} \times S^1$ satisfying
\begin{equation} \label{actionloopbound}
            \sup_t [g(s',t)] < c_{\eta}.
        \end{equation}

    \item For some $\lambda > 0$ and (possibly negative) constant $A$, $g$ satisfies the following equation
        \begin{equation}
            \Delta g + \lambda g \geq A.
        \end{equation}
\end{enumerate}
Then for sufficiently small $\eta$, there is a constant $C(\lambda, A,
\eta)$ such that 
\begin{equation} \label{c0eqn}
    g(s,t) < C 
\end{equation}
everywhere except possibly outside a $2\eta$-collar of $Z$.
\end{prop}
\begin{proof}
Starting from one end, partition $Z$ into a maximal collection of adjacent
$\eta$-sized cylinders $[s,s+\eta] \times S^1$---this covers all of $Z$
except a at most a portion of $Z$'s collar of width at most $\eta$ (when $Z$ is finite).
To each such sub-cylinder $Z_k = [s_k,s_{k}+\eta] \times S^1$ of $Z$
associate a number
\begin{equation}
    s_k' \in (s_k,s_k+\eta)
\end{equation}
satisfying (\ref{actionloopbound}). An adjacent pair of such $s_k'$, $s_{k+1}'$
satisfies 
\begin{equation}
    s_{k+1}' - s_k' < 2 \eta
\end{equation}
and
\begin{equation}
    \begin{split}
        g(s_k',t) &< c_{\eta} \\
        g(s_{k+1}',t) &< c_{\eta}.
    \end{split}
\end{equation}
We will now examine the new regions 
\begin{equation}
    Z_k' = [s_k',s_{k+1}'] \times S^1,
\end{equation}
which cover all of $Z$ except a portion of $Z$'s collar of width at most $2
\eta$. Let
\begin{equation}
    \e_k = \frac{1}{2}(s_{k+1}' - s_k')
\end{equation}
and consider the function
\begin{equation}
    h(s,t) := (\lambda c_{\eta} + |A|) (\e_k^2 - (s - s_k' - \e_k)^2) + c_{\eta}.
\end{equation}
$h(s,t)$ satisfies the following properties on $[s_k',s_{k+1}']$:
\begin{align}
    h(s,t) &\geq c_{\eta}.\\
    h(s,t) &\leq (\lambda c_{\eta} + |A|)\e_k^2 + c_{\eta} \leq (\lambda c_{\eta} + |A|)\eta^2 + c_{\eta}. 
\end{align}
Moreover, for $\eta$ chosen sufficiently small $( < \sqrt{\frac{1}{\lambda}})$,
\begin{equation}
    \begin{split}
    (-\Delta - \lambda) h(s,t) &= 2 (\lambda c_{\eta} + |A|) - \lambda h(s,t)\\
    &\geq (2 (\lambda c_{\eta} + |A|)) - \lambda [(\lambda c_{\eta} + |A|)\eta^2 + c_{\eta}] \\
    & \geq |A| (2 - \lambda \eta^2) + \lambda c_{\eta} (2 - \lambda \eta^2 - 1)\\
    & \geq |A|.
\end{split}
\end{equation}
Therefore, the function 
\begin{equation}
    -g(s,t) + h(s,t)
\end{equation}
satisfies the following two properties:  
\begin{align}
    -g(s,t) + h(s,t) &\geq 0 \textrm{ on } \partial C_k', \\
    (-\Delta - \lambda) (-g(s,t) + h(s,t)) & \geq 0 \textrm{ on } C_k'.
\end{align}
Now, the smallest eigenvalue of $-\Delta$ on $Z_k'$ subject to the boundary
condition of 0 on $\partial Z_k'$ can be explicitly calculated via Fourier
series to be 
\begin{equation}
    \bar{\lambda} = \frac{\pi^2}{4\e_k^2} \geq \frac{\pi^2}{4 \eta^2}.
\end{equation}
so for $\eta$ sufficiently small, $\lambda$ is smaller than $\bar{\lambda}$.
Thus Proposition \ref{dirichlet} applies, and we conclude that on all of
$Z_k'$:
\begin{equation}
    -g(s,t) + h(s,t) \geq 0;
\end{equation}
namely on $Z_k'$
\begin{equation}
    g(s,t) \leq h(s,t) \leq (\lambda c_{\eta} + |A|)\eta^2 + c_{\eta}.
\end{equation}
This final bound holds on every new cylinder $Z_k'$ and is independent of $k$. Since the cylinders $Z_k'$ cover all but a $2\eta$ collar of $Z$, we conclude the result.
\end{proof}

Returning to our main argument, let us recall that pulled back to a particular
cylinder $[c,d] \times S^1$ in the cylindrical region $S^{c}$ with associated
{\bf weight} $v$, our chosen Floer datum has the following form:
\begin{itemize}
    \item The {\bf sub-closed one-form} $\gamma$ is actually closed and equal
        to $v dt$.  
    \item The {\bf main Hamiltonian} $H_S$ is equal to $\frac{H
        \circ
        \psi^{v}}{v^2}$ 
    \item The {\bf rescaling function} $a_S$ is equal to the constant $v$.
    \item The {\bf almost-complex structure} $J_S = (\psi^{v})^* J_t$ is
        {\bf $v$-rescaled contact type}, e.g. on the conical end, 
        \begin{equation}
            dr \circ J = \frac{v}{r} \cdot \theta
        \end{equation}
    \item The {\bf perturbation term} $F_T$ is {\bf monotonic} in $s$, i.e.
        \begin{equation}
            \partial_s F_T \leq 0,
        \end{equation}
        and has norm and all derivatives bounded by constants independent of
        the particular cylinder.
\end{itemize}
In particular, we note that on any such region, the {\bf total Hamiltonian}
\begin{equation}
    H^{tot}_S: = H_S+F_S
\end{equation} is monotonic in $s$. Also, $u$ pulled back to any such region satisfies the
usual form of Floer's equation:
\begin{equation}\label{floercoordinates}
    \begin{split}
        (du - X \otimes (v \cdot dt))^{0,1} &= 0,\textrm{ i.e.}\\
        \partial_s u + J_t (\partial_t u -  X_{v\cdot (H^{tot}_S)}) &= 0.
    \end{split}
\end{equation}
Below, we will frequently suppress the weight $v$, building it into the total
Hamiltonian $\tilde{H}^{tot}_S = v\cdot H^{tot}_S$.
On a cylindrical region of $S$, we examine the function
\begin{equation}
    \xi := \frac{1}{2} \rho^2 = \frac{1}{2} (\pi_r \circ u)^2 : [c,d] \times S^1 \lra \R.
\end{equation}
Proposition \ref{perturbedprop} can now be refined as follows:
\begin{prop}
$\xi$ (almost) satisfies Conditions 1 and 2 from Proposition
\ref{floerhoferprop8}. Namely, there is a replacement $\tilde{\xi}$ satisfying
Conditions 1 and 2, with
\begin{equation} \label{xitilde}
    \begin{split}
    \xi &< \tilde{\xi} + C\\
    \tilde{\xi} &< \xi + C'.
\end{split}
\end{equation}
Choosing $\eta$ smaller than $\frac{\bar{\delta}}{2}$, we see that on all but a
$\bar{\delta}$ collar of $S$, $\tilde{\xi}$ and hence $\xi$ and $\rho$ are
absolutely bounded in terms of the Floer data and asymptotic conditions.
\end{prop}
In order to better examine intermediate actions along loops $\mc{L}: S^1
\hookrightarrow S^c$ in the cylindrical region, we will define define relevant
function spaces for maps $x: S^1 \ra W$, following \cite{Oancea:2008fk}*{\S 4}.
First, define the following continuous projection to the compact region: 
\begin{align}
    \pi_{in}: W &\lra \bar{W}\\
    p &\longmapsto 
    \begin{cases} 
        p & p \in \bar{W} \\ 
        \bar{p} & p = (\bar{p},r) \in \bd \bar{W} \times [1,\infty).
    \end{cases}
\end{align}
Given $x: S^1 \ra W$, denote by $\bar{x}$ the composition 
\begin{equation}
    \bar{x}:= \pi_{in} \circ x.
\end{equation}
\begin{defn}
Define the following function spaces:
\begin{align}
    L^2(S^1,W) &:= \{x: S^1 \ra W\mathrm{\ measurable:\ } \pi_r \circ x \in L^2(S^1,\R)\}\\
    \label{h1norm} H^1(S^1,W) &:= \{x \in L^2(S^1,W)\mathrm{\ :\ } \dot{\bar{x}} \in
    L^2(x^*T\bar{W}), (\pi_r \circ x)' \in L^2(S^1,\R)\}.
\end{align}
Here, measurability is with regards to some metric $g = \omega (\cdot, J
\cdot)$, and independent of choices.  Also, for a smooth map $x$,
$\dot{\bar{x}}$ is well-defined as a distribution given an embedding of
$\bar{W}$ into Euclidean space. The requirement that it be $L^2(x^*T\bar{W})$ is
independent of embedding, though we declare that for some fixed embedding that
$||\dot{\bar{x}}||^2_{L^2}$ be the restriction of the usual Euclidean $L^2$
norm. We define the norms associated to these spaces as follows:
\begin{align}
    ||x||_{L^2}^2 &:= ||\pi_r \circ x||^2_{L^2} \\
    ||x||_{H^1}^2 &:= ||\pi_r \circ x||^2_{H^1} + ||\dot{\bar{x}}||^2_{L^2}.
\end{align}
\end{defn}
\begin{defn}
Define the normed space
\begin{align}
    C^0(S^1,W) 
\end{align}
of continuous functions from $S^1$ to $W$ as follows: The norm of a continuous
map $x$ is given by by choosing an embedding of $\bar{W}$ into Euclidean space
and restricting the standard Euclidean sup norm on $\bar{x}$, plus the sup norm
of $\pi_r \circ x$. This space is independent of embedding of $\bar{W}$.  
\end{defn}
\begin{lem}[Sobolev embedding] There is a compact embedding 
    \begin{equation} \label{sobolevembedding}
    H^1(S^1,W) \subset C^0(S^1,W).
    \end{equation}
\end{lem}
\begin{proof} See \cite{Oancea:2008fk}*{Lemma 4.7}. The main point is to
    leverage the known compact embedding $H^1(S^1,\R) \subset C^0(S^1,\R)$ to
    ensure that any sequence $f_k$ bounded in $H^1(S^1,W)$ takes values in a
    compact set. Thus, one can apply Sobolev embedding for maps from $S^1$ to a
    compact target manifold.
\end{proof}
\noindent By definition, we have that
\begin{equation} \label{c0normbound}
    \textrm{if } f \in C^0(S^1,W)\textrm{ then } \pi_r \circ f \textrm{ is bounded.}
\end{equation}
Given an almost complex structure $J$, we recall the associated metric
\begin{equation}
\langle X,Y \rangle_J = \omega(X,JY).
\end{equation}
For a smooth map $x: S^1 \lra W$ and a $S^1$ dependent complex structure $J_t$
we use this metric to define the following $L^2$ norm
\begin{equation}
    ||\dot{x}||^2_{L^2} := \int_{S^1} \omega(\dot{x}(t),J_t \dot{x}(t)) dt.
\end{equation}
The relation to the function spaces defined earlier is as follows:
\begin{claim} \label{l2relations}
    For $J_t$ of rescaled contact type, the $L^2$ norm $||\dot{x}||^2_{L^2}$
    bounds $||\dot{\bar{x}}||^2_{L^2}$ and $||(\pi_r \circ x)'||^2_{L^2}$.
\end{claim}
\begin{proof} 
    Suppose first that $x$ maps entirely to the cylindrical end of $W$. Then
    $\bar{x}$ is smooth and the norm of $\dot{\bar{x}}$ given by a choice of
    embedding into Euclidean space is equivalent to the one coming from
    $\omega(\cdot, J\cdot)$ on $\bar{x}$, for any $J$. Now, for any $J$ of
    rescaled contact type, $\partial_r$ and $T\partial W$ are $\langle\cdot,
    \cdot\rangle_J$ orthogonal, implying that
\begin{equation}
    \begin{split}
        |\dot{x}|^2 &= |\left(\dot{\bar{x}},(\pi_r \circ x)'\right)|^2 \\
        &= |\dot{\bar{x}}|_{x(t)}^2 + |(\pi_r \circ x)'|^2.
\end{split}
\end{equation} 
Above, the notation $|\dot{\bar{x}}|_{x(t)}^2$ refers to the fact that we are
taking the norm of $\dot{\bar{x}}$ with respect to the metric at level $\pi_r
\circ x$. The norm $\langle \cdot,\cdot \rangle_J$ behaves in the following manner with
regards to level on the cylindrical portion $\partial W \times [1,\infty)$: For
$R$ and $\partial_r$, 
\begin{equation}
    \begin{split}
        \langle R, R \rangle_J &= v\\
        \langle \partial_r, \partial_r \rangle_J &= \frac{1}{v},
\end{split}
\end{equation}
independent of level $r$ (here $v$ is the rescaling constant). For vectors in
the orthogonal complement of $R$, $\partial_r$, the norm grows linearly in $r$.
In particular,
\begin{equation}
    |\dot{\bar{x}}|_{\bar{x}(t)}^2 \leq |\dot{\bar{x}}|_{x(t)}^2;
\end{equation}
thus both $|\dot{\bar{x}}|$ and $|(\pi_r \circ x)'|$ are bounded by a constant
multiple of $|\dot{x}|$.  Now extend this bound to arbitrary $x$ as follows:
suppose that $\pi_r \circ x \leq 1$, e.g. $x \subset \bar{W}$. Then, since
$d\pi_r$ is an operator with bounded norm on $\bar{W}$, 
\begin{equation}
    (\pi_r \circ x)^2 = (d\pi_r \circ \dot{x})^2 \leq (\mathrm{const})\cdot |\dot{x}|^2.
\end{equation}
Similarly, when $x \subset \bar{W}$, $\dot{\bar{x}} = \dot{x}$, so
$|\dot{\bar{x}}|$ is trivially bounded by $|\dot{x}|$.
\end{proof}

\begin{lem}[Condition 1]
For any $\eta > 0$ there is a $c_{\eta}$ such that
on any sub-cylinder $[s,s+\eta] \times S^1$ of the cylindrical region $S^{c}$,
\begin{equation}
    \sup_t [\pi_r \circ u(s',t)] \leq c_{\eta}
\end{equation}
for some $s' \in [s,s+\eta]$.
\end{lem}
\begin{proof}
Let
\begin{equation}
    \bar{Z} = (s_0,s_0+\eta) \times S^1
\end{equation}
be a given sub-cylinder of the cylindrical region $S^c$. Let
\begin{equation}
    \mc{A}_{\bar{Z}}(s)
\end{equation}
denote the intermediate action of the loop $\{s\} \times S^1 \subset \bar{Z}$,
and let $u(s,t)$ denote the restriction of $u$ to $\bar{Z}$.
By the positivity of topological energy,
\begin{equation}
    \mc{A}_{\bar{Z}}(s_0)-\mc{A}_{\bar{Z}}(s_0+\eta) = E_{top}(u|_{\bar{Z}})
    \leq K, 
\end{equation}
where $K$ is the topological energy of $u$.  
The mean-value theorem therefore implies the existence of
\begin{equation}
    s' \in (s_0,s_0+\eta)
\end{equation}
satisfying
\begin{equation} \label{derivactionbound}
|\partial_s (\mc{A}_{\bar{Z}}(s))||_{s=s'} \leq K/\eta.
\end{equation}
Moreover, we know from Lemma \ref{intermediateactionbound} that
\begin{equation} \label{abound}
    (-\mc{A}_{\bar{Z}}(s')) \leq M
\end{equation}
for some constant $M$ depending only on the asymptotics of $S$. 
We claim that the equations (\ref{derivactionbound}) and (\ref{abound}) give a
bound for the loop $u(s',\cdot)$ in the $H^1$ norm (\ref{h1norm}), establishing
the Lemma for $s'$ (using the Sobolev embedding (\ref{sobolevembedding})).

Recalling the special form of our Floer data on $\bar{Z}$, and abbreviating
$H:= H^{tot}_S$, $J := J_S$, $X = X_{H^{tot}_S}$ the derivative $\partial_s
\mc{A}_{\bar{Z}}(s)$ can be expressed as:
\begin{equation} \label{phibound}
    \begin{split}
        \partial_s \mc{A}_{\bar{Z}}(s) &= -\int_{\{s\} \times S^1} \omega(\partial_s u, \partial_t u)dt + \int_{\{s\} \times S^1} \partial_s (u^* (H)) dt \\
        &= -\int_{s \times S^1} \omega (\partial_t u, J (\partial_t u - X))dt + \int_{s \times S^1} (dH \circ \partial_s u + \partial_s H) dt \\
        &= -\int_{s \times S^1} \omega (\partial_t u, J (\partial_t u - X))dt + \int_{s \times S^1} \omega(X,\partial_s u)dt + \int_{s\times S^1} \partial_s H dt \\
        &= -\int_{s \times S^1} \omega(\partial_t u, J (\partial_t u - X))dt - \int_{s \times S^1} \omega(X,J_S (\partial_t u - X))dt + \int_{s \times S^1} \partial_s H dt \\
        &= -\int_{s \times S^1} \omega(\partial_t u - X, J(\partial_t u - X))dt + \int_{s \times S^1} \partial_s H dt \\
        &= -||\partial_t u(s,\cdot) - X||^2 + \int_{S^1} \partial_s H(s,\cdot) dt,
    \end{split}
\end{equation}
where above we have twice used that $u$ satisfies Floer's equation. Abbreviate
\begin{equation}
    x_s(t) := u(s,t),
\end{equation}
and note that since $H^{tot}_S$ is monotonically decreasing, (\ref{phibound})
and (\ref{derivactionbound}) imply that 
\begin{equation} \label{phibound2}
    ||\partial_t x_{s'} - X||^2 \leq \frac{K}{\eta}.
\end{equation}

For a rescaled-standard complex structure $J$, and a Hamiltonian vector field
$X_S$ equal to  $2r \cdot R$ plus a bounded term,
\begin{equation}
    |X_S|^2 \leq  C(\omega(2r\cdot R, J(2r \cdot R)) + 1) = C (r^2 + 1),
\end{equation}
i.e.
\begin{equation}
    |X_S| \leq \tilde{C}(|r| + 1).
\end{equation}
for some constants $C, \tilde{C}$.  Thus by (\ref{phibound2}), 
\begin{equation}
    \label{rbound}
    ||\partial_t x_{s'}|| \leq || \partial_t x_s - X|| + ||X|| \leq \tilde{C}(1 + ||\pi_r \circ x_{s'}||_{L^2}).
\end{equation}
By Claim \ref{l2relations}, a bound on $||\partial_t x_{s'}||$ is as good as a
bound on $||(\pi_r \circ x_{s'})'||$ and $||\dot{\bar{x}}_{s'}||$. 
The equation (\ref{rbound}) implies that a bound for $||\pi_r \circ
x_{s'}||_{L^2}$ suffices to establish the desired $H^1$ bound on $x_{s'}$.

We know by hypothesis that for any $s$ the action of the loop $x_s$ is bounded below:
\begin{equation}
    \label{actionbound}
    -\mc{A}(x_s) = \int_{S^1} x_s^*\theta - \int x_s^*H \leq M
\end{equation}
for some $M$. We rewrite the first term of (\ref{actionbound}) as
\begin{equation} \begin{split}
    \int_{S^1} x_{s'}^*\theta &= \int \omega (Z,\dot{x}_{s'}) dt \\
    &= \int \omega(Z,(\partial_t x_{s'} - X(x_{s'})) + X(x_{s'})) \\
    &= \int \langle JZ, \partial_t x_{s'} - X\rangle_J + \int_{x_{s'}(S^1)} dH(Z) dt.
\end{split} \end{equation}
Thus by Cauchy-Schwarz
\begin{equation}
    \int_{S^1} x_{s'}^*\theta \geq -||Z||\cdot ||\partial_t x_{s'} - X|| + \int_{x_{s'}(S^1)}  dH(Z)
\end{equation}
Substituting into (\ref{actionbound}) we see that
\begin{equation}
    M \geq \int_{x_{s'}(S^1)} ( dH(Z) - H) - ||Z|| \cdot ||\partial_t x_{s'} - X||,
\end{equation} 
e.g. 
\begin{equation}
    \int_{x_{s'}(S^1)} (dH(Z) - H) \leq M + ||Z|| \cdot ||\partial_t
    x_{s'} - X||.
\end{equation}
By (\ref{phibound2}), $||\partial_t x_{s'} - X||$ is bounded.
Moreover, $Z = r \cdot \partial_r$ has norm equal to $r$ times a constant. Thus,
\begin{equation}
    \int_{x_{s'}(S^1)} (dH(Z) - H) \leq \tilde{M} + C_0 ||\pi_r \circ x_{s'}||
\end{equation}
for some new constant $\tilde{M}$.  For $x_{s'}$ mapping entirely to the
cylindrical end, we see that
\begin{equation}
    \begin{split}
    dH(Z) - H &= r\cdot dH(\partial_r)- H \\
    &=  r^2 + dF_S(Z) - F_S.
\end{split}
\end{equation}
The last two terms are totally bounded by assumption, so 
\begin{equation}
   \int_{x_{s'}(S^1)} r^2 =  ||\pi_r \circ x_{s'}||^2  \leq N + C_0 || \pi_r \circ x_{s'}||,
\end{equation}
for constants $N, C_0$, implying a bound for $||\pi_r \circ x_{s'}||$.  We
extend to the general case by noting that whenever $x_{s'}$ maps to $\bar{W}$,
$|\pi_r \circ x_{s'}|$ is bounded by 1.
\end{proof}
We have just proven that Condition 1 holds for 
\begin{equation}
    \rho = \pi_r \circ u
\end{equation}
This implies that it holds for 
\begin{equation}
    \xi := \frac{1}{2} \rho^2
\end{equation}
as well as any $\tilde{\xi}$ satisfying (\ref{xitilde}).

\begin{lem}[Condition 2]
    On any cylindrical part of $S$ there exists a $\tilde{\xi}$, satisfying 
    \begin{equation}
        \begin{split}
        \xi &< \tilde{\xi} + C\\
        \tilde{\xi} &< \xi + C'\\
    \end{split}
\end{equation}
such that 
\begin{equation} \label{laplaceeqn}
    \Delta \tilde{\xi} + \lambda \tilde{\xi} \geq -A
\end{equation}
for constants $C$, $C'$ depending only on the asymptotic conditions of $S$ and the Floer
data.  
\end{lem}
\begin{proof} 
    Actually, we will prove that $\xi$ itself satisfies (\ref{laplaceeqn}) on
    the cylindrical end of $W$; we will then perform the replacement
    $\tilde{\xi}$ to extend the validity of (\ref{laplaceeqn}) to the compact
    region $\bar{W}$. 

    Letting \begin{equation}
        \mc{S} = \frac{1}{2} r^2,\end{equation} 
and $\xi = \mc{S} \circ u$ we calculate the
Laplacian $\Delta\xi$ on the cylindrical end of $W$.  Begin with Floer's equation
\begin{equation}
J \circ (du - X\otimes dt) = (du - X\otimes dt) \circ j, \end{equation}
apply $d\mc{S}$ to both sides. Since $J$ is $v$-rescaled-contact type
\begin{equation}
    d\mc{S} \circ J = r dr \circ J = -vr \bar{\theta} = -v \theta,
\end{equation} we have that
\begin{equation}
    d^c\xi = d\xi \circ j = v(-u^* \theta + \theta(X) dt) + d\mc{S}(X) (dt \circ j)
\end{equation}
Differentiating once more, we see that
\begin{equation}
    d d^c \xi = v (-u^*\omega + (\partial_s \theta(X))\mbox{ } ds \wedge dt) - \partial_t (d\mc{S}(X))\mbox{ }ds\wedge dt.
\end{equation}
Since $dd^c \xi = -\Delta \xi \mbox{ } ds \wedge dt$, we see that
\begin{equation}\label{laprho}
    \Delta \xi = v (\omega (\partial_s u, \partial_t u) - \partial_s (\theta(X))) + \partial_t (d\mc{S}(X)).
\end{equation}
Now, as in \cite{Cieliebak:1995fk} and \cite{Oancea:2008fk}, rewrite
$\omega(\partial_s u, \partial_t u)$ as: 
\begin{equation}
    \begin{split}
    \omega(\partial_s u, \partial_t u) &= \frac{1}{2} \omega (\partial_s u, \partial_t u) + \frac{1}{2} \omega (\partial_s u, \partial_t u)\\
    &= \frac{1}{2} \left( \omega(\partial_s u, J \partial_s u + X) + \omega (-J \partial_t u + JX, \partial_t u) \right)\\
    &= \frac{1}{2}\left( |\p_s u|^2 + |\p_t u|^2 + \omega(\partial_s u, X) - \omega (\partial_t u, JX) \right)
\end{split}
\end{equation} 
The other terms in (\ref{laprho}) can be expanded as follows, where
$\bar{\theta}$ is the contact form on $\partial W$ (as in (\ref{contactform})):
\begin{equation}\partial_t
    (d\mc{S}(X)) = d\mc{S} \circ \partial_t X + d\mc{S}
    \circ \nabla_{u_t} X,
\end{equation} and 
\begin{equation}\begin{split}
    \partial_s (\theta(X)) &= \partial_s (r \bar{\theta}(X)) \\
    &= (d r \circ u_s) \bar{\theta}(X) + r  \bar{\theta} (\partial_s X) + r \bar{\theta} (\nabla_{u_s} X)
\end{split} \end{equation}
Putting these together, we have that
\begin{equation}\begin{split}\label{expandedlaprho}
    \Delta \xi =& \frac{v}{2} |\p_s u|^2 + \frac{v}{2} |\p_t u|^2 + \frac{v}{2}\omega(\p_s u,X) + \frac{v}{2} \omega(\partial_t u, JX) + d\mc{S} \circ \p_t X \\
    & + d\mc{S} \circ \nabla_{u_t} X + v(dr \circ u_s) \bar{\theta}(X) + vr\bar{\theta} (\p_s X) + vr\bar{\theta} (\nabla_{u_s} X).
\end{split}
\end{equation}
When $J$ is of $v$-rescaled contact type, we recall that as linear operators
\begin{align}
    &\bar{\theta}\textrm{ has constant norm; and}\\
    &dr = d(\sqrt{2\mc{S}}) = \frac{d\mc{S}}{\sqrt{2\mc{S}}}\textrm{ has constant norm, hence }\\
    &d\mc{S}\textrm{ has norm } O(\sqrt{\mc{S}}).
\end{align}
Moreover, we have the following inequalities:
\begin{equation}\label{fundinequalities} 
\begin{split} 
    |X| &\leq C(1 + \sqrt{\mc{S}}), \\
     |\nabla_Y X| &\leq C |Y|, \mathrm{\ for\ any\ vector\ field\ }Y,\\
     |\partial_s X| &\leq C (1 + \sqrt{\mc{S}})
\end{split}
\end{equation}
for some (possibly different) constants $C$ depending on the rescaling constant
$v$ and the time-dependent perturbation term of our total Hamiltonian.  We use
these inequalities to estimate the terms in
(\ref{expandedlaprho}):
\begin{align}
    |\omega(\p_s u,X)| \leq |\p_s u| |X| &\leq C (1 + \sqrt{\xi})|\p_s u|\\
    |\omega (\p_t u, JX)| \leq |\p_t u| |X| &\leq C (1 + \sqrt{\xi})|\p_t u|\\
    |d\mc{S} \circ \p_t X| &\leq C (1 + \sqrt{\xi}) ^2\\
    |d\mc{S} \circ \nabla_{u_t} X| &\leq C(1 + \sqrt{\xi})|\p_t u|\\
    |(dr \circ u_s) \bar{\theta}(X)| &\leq C (1 + \sqrt{\xi})|\p_s u|\\
    |r \bar{\theta}(\p_s X)| \leq |r||\p_s X| &\leq C (1 + \sqrt{\xi})^2\\
    |r\bar{\theta} \nabla_{u_s} X| \leq |r| |\p_s u| &\leq C (1 + \sqrt{\xi})|\p_s u|
\end{align}
again for potentially different constants $C$ depending on those in
(\ref{fundinequalities}).  Putting these
together, there exists constants $c_1$, $c_2$ and $c_3$ such that 
$\xi$ satisfies an equation
of the following form: \begin{equation}
    \begin{split}
    \Delta \xi \geq & \frac{1}{2}|\p_s u|^2 + \frac{1}{2} |\p_t u|^2 - c_1 (1 + \sqrt{\xi}) |\p_s u| \\
    &- c_2 (1 + \sqrt{\xi})|\p_t u| - c_3 (1 + \sqrt{\xi})^2.
\end{split}
\end{equation}
which implies that 
\begin{equation}
    \Delta \xi + \lambda \xi \geq -A
\end{equation} 
for obvious constants $\lambda, A$ depending on $c_1, c_2,  c_3$.  This holds
on the cylindrical end $r \geq 1$, where the estimates above apply.

We extend as follows, directly following an argument in
\cite{Oancea:2008fk}*{Thm. 4.6}. Let 
\begin{equation}
    \varphi: \R_+ \lra \R+
\end{equation}
be a smooth function such that $\varphi(r) = 0$ for $r \leq 1$, $\varphi'(r) =
1$ for $r \geq 2$, and $\varphi''(r) > 0$ for $1 \leq r \leq 2$.
Clearly
\begin{equation}
    \begin{split}
    r &< \varphi(r) + C\\
    \varphi(r) &< r + C'
\end{split}
\end{equation}
Thus, the modification 
\begin{equation}
    \tilde{\xi} := \varphi \circ \xi
\end{equation}
satisfies (\ref{xitilde}) as required. Moreover, note that
\begin{equation}
    \begin{split}
        \Delta \tilde{\xi} &= \partial_s (\varphi'(\xi(s,t)) \cdot \partial_s \xi) + \partial_t (\varphi'(\xi(s,t)) \cdot \partial_t \xi)\\
        &= \varphi''(\xi(s,t)) (|\partial_s \xi|^2 + |\partial_t \xi|^2) + \varphi'(\xi(s,t)) \cdot (\Delta \xi) \\
        & \geq \varphi'(\xi(s,t)) ( - A - \lambda \xi) \geq (- A - \lambda \xi) \geq (-A - \lambda \tilde{\xi})
    \end{split}
\end{equation}
as desired. 
\end{proof}

\section{Orientations and signs} \label{orientationsection}
In this appendix, we recall the ingredients necessary to orient various moduli
spaces of maps, thereby obtaining operations defined over the integers (or over
a field of arbitrary characteristic).  The relevant theory was first developed
in \cite{Floer:1993fk} and adapted to the Lagrangian case in
\cite{Fukaya:2009qf}*{\S 8}. We will proceed as follows:

\begin{itemize}
    \item In Section \ref{orientationlines}, we associate, to every time-1
        chord $x \in \chi(L_i,L_j)$ or orbit $y\in \mc{O}$, real
        one-dimensional vector spaces called {\bf orientation lines}
        \begin{equation}
            o_x,o_y,
        \end{equation}
        coming from the linearization of Floer's equation on any strip-like or
        cylindrical end.
   Then we recall from general theory how orientation lines, orientations of
   our Lagrangians and orientations of abstract moduli spaces determine
   canonical orientations of moduli spaces of maps. In the semi-stable case,
   this orientation is canonical up to a choice of trivialization of the
   natural $\R$ action. 

   \item In Section \ref{comparingsigns}, we give a recipe for computing the
       sign of terms in the expression arising from the codimension 1
       boundary components of a moduli space of maps.

    \item In Section \ref{abstractorientations}, we choose orientations for the
        top strata of various abstract moduli spaces of open-closed
        strings/glued pairs of discs. 

    \item Finally, in Section \ref{verification}, we will use all of the
        ingredients discussed to carefully verify the signs arising in a single
        case.

\end{itemize}

We will draw heavily from the discussion in \cite{Seidel:2008zr}*{(11l)} (which
discusses surfaces without interior punctures), along with the extension to
general open-closed strings in \cite{Abouzaid:2010kx}*{\S C}.  Our notation
will primarily follow \cite{Abouzaid:2010kx}*{\S C}. 

\subsection{Orientation Lines and Moduli Spaces of Maps}\label{orientationlines} 
Given $y \in \mc{O}$, there is a unique homotopy class of trivializations of
the pullback of $TM$ to $S^1$  that is compatible with our chosen
trivialization of $\Lambda_\C^n TM$. The linearization of Floer's equation
(\ref{floersequation2}) on a cylindrical end $[1,\infty) \times S^1$ with
respect to such a trivialization exponentially converges to an operator of the
form
\begin{equation} \label{linearizedfloerorbit}
    Y \longmapsto \partial_s Y - J_t \partial_t Y - A(+\infty,t)Y,
\end{equation}
where $J_t - A(+\infty,t)$ is a self-adjoint operator (see \cite{Floer:1993fk}
for more details).  Thus, to define an orientation line, we once and for all
fix a local operator
\begin{equation} \label{linearizedorbitoperator}
    \mathrm{D}_y: H^1(\C,\C^n) \lra L^2(\C,\C^n)
\end{equation}
extending the asymptotics (\ref{linearizedfloerorbit}) in the following
fashion.  Endow $\C$ with a negative strip-like end around $\infty$ of the form
\begin{equation}
    \begin{split}
    \e: (-\infty,0] \times \R/\Z &\lra \C\\
    s,t &\longmapsto \exp(-2\pi(s+it))
\end{split}
\end{equation}
and consider extensions of $J_t$ and $A(-\infty,t)$ to families $J_\C$ of
complex structures and $A_\C$ of endomorphisms of $\C^n$. Using these families,
we define the operator $D_y$ to be as in
(\ref{linearizedfloerorbit}) using the extended families $J_\C$, $A_\C$.
\begin{defn}
    The {\bf orientation line} $o_y$ is the determinant line $\det(\mathrm{D}_y)$.
\end{defn}

In a similar fashion, given $x \in \chi(L_i,L_j)$, after applying a canonical
up to homotopy trivialization of the pullback of $TM$ to $[0,1]$, $x$ can be
thought of as a path between two Lagrangian subspaces $\Lambda_i$ and
$\Lambda_j$ of $\C^n$, and the linearized operator corresponding to Floer's
equation (\ref{floersequation}) asymptotically takes the same form as
(\ref{linearizedfloerorbit}).  Now, choose a negative striplike end around
$\infty$ in the upper half plane
\begin{equation}
    \begin{split}
    \e: (-\infty,0] \times [0,1] &\lra \H\\
    s,t &\longmapsto \exp(\pi i -\pi(s+it)).
\end{split}
\end{equation} 
Also choose some family of Lagrangian subspaces $F_z$, $z \in \R \subset \H$
such that $F_{ \e(s \times \{0\}) } = \Lambda_i$, $F_{\e(s \times \{1\})} =
\Lambda_j$, 
and choose extensions of $A$ and $J$ to all of $\H$ as before. 
One thus obtains an operator
\begin{equation}
    \mathrm{D}_x: H^1( {\mathbb H}, \C^n, F) \lra L^2( {\mathbb H}, \C^n)
\end{equation}
\begin{defn}
    The {\bf orientation line} $o_x$ is the determinant line $\det(\mathrm{D}_x)$.
\end{defn}

\begin{rem} 
    We have omitted from discussion the grading structure, but we should remark
    that in reality, when working with graded Lagrangians, trivialization gives
    us {\bf graded Lagrangian spaces} $\Lambda_i^\#$, $\Lambda_j^\#$ of $\C^n$
    (thought of as living in the universal cover of the Lagrangian Grassmannian
    of $\C^n$). Instead of giving such a discussion now, we simply note that
    the family $F$ of Lagrangian subspaces chosen above must lift to a
    family of graded Lagrangian subspaces interpolating between the lifts
    $\Lambda_i^\#$ and $\Lambda_j^\#$. 
\end{rem}
The reader is referred to \cite{Seidel:2008zr}*{(11g)} for a more explicit
spectral flow description of these determinant lines and indices.

By definition, orientation lines are naturally graded by the indices of the
operators we have constructed above, meaning that
\begin{equation}
    o_{x_1} \otimes o_{x_2} = (-1)^{|x_1|\cdot |x_2|} o_{x_2} \otimes o_{x_1}
\end{equation}
where $|x|$ is the degree of the chord (or orbit) $x$. Also, there are natural
pairings
\begin{equation} \label{orientationpairing}
    o_{x}^{\vee} \otimes o_{x} \lra \R.
\end{equation}
Given a vector $\vec{x}$ of chords or orbits, abbreviate the tensor product of
orientation lines in $\vec{x}$ as
\begin{equation}
    o_{\vec{x}} := \bigotimes_{x \in \vec{x}} o_{x}
\end{equation}
The application of orientation lines to our setup is this: Standard gluing
theory tells us that given a regular point $u$ of some moduli space of maps
with asymptotic conditions, orientation lines for the asymptotic conditions and
an orientation for the abstract domain moduli space canonically determine an
orientation of the tangent space at $u$. 

To elaborate, let $\mc{M}$ be some abstract moduli space of Riemann surfaces
$\Sigma$ with boundary $S$. Denote by $(\bar{\Sigma},\bar{S})$ the surface
obtained by compactifying, i.e. filling in the boundary and interior punctures.
Given a collection of asymptotic conditions $(\vec{x}_{out}, \vec{y}_{out},
\vec{x}_{in}, \vec{y}_{in})$ one can form the moduli space of maps
\begin{equation}
\mc{M}(\vec{x}_{out}, \vec{y}_{out}; \vec{x}_{in}, \vec{y}_{in})
\end{equation}
as described, for example, in Section \ref{ocfloersection}. Suppose we have
chosen an orientation of $\mc{M}$ and orientation lines $o_x, o_y$. Then:
\begin{lem} \label{orientationlem}
    Let $\bar{C}_1, \ldots, \bar{C}_k$ be the components of the
    boundary of $\bar{S}$, and $e_j$ denote the number of negative ends of
    $\bar{C}_j$. If we fix a marked point $z_j \in C_j$ mapping to a Lagrangian
    $L_j$ then, assuming the moduli space $\mc{M}(\vec{x}_{out}, \vec{y}_{out};
    \vec{x}_{in}, \vec{y}_{in})$ is regular at a point $u$, we have a canonical
    isomorphism
    \begin{equation}
        \lambda(\mc{M}(\vec{x}_{out}, \vec{y}_{out}; \vec{x}_{in}, \vec{y}_{in}))   \cong \lambda (\mc{M}) \otimes 
        \bigotimes_j  
        \lambda(T|_{u(z_j)} L_j)^{\otimes 1 - e_j} 
        \otimes o_{\vec{x}_{in}}^\vee  
        \otimes o_{\vec{y}_{in}}^\vee 
        \otimes o_{\vec{x}_{out}} \otimes o_{\vec{y}_{out}},
    \end{equation}
    where $\lambda$ denotes top exterior power.
\end{lem}
\begin{proof} 
    The version of this Lemma in the absence of interior punctures can be found
    in \cite{Seidel:2008zr}*{Prop.  11.13}. The minor generalization of
    including interior punctures is discussed in \cite{Abouzaid:2010kx}*{Lem.
    C.4}.
\end{proof}
In particular, given fixed orientation of $\mc{M}$, when the moduli space
$\mc{M}(\vec{x}_{out}, \vec{y}_{out}; \vec{x}_{in}, \vec{y}_{in})$ is rigid,
we obtain, at any regular point $u \in \mc{M}(\vec{x}_{out},\vec{y}_{out};
\vec{x}_{in}, \vec{y}_{in})$, an isomorphism
\begin{equation} \label{orientisos}
    \mc{M}_u: 
o_{\vec{x}_{in}} \otimes o_{\vec{y}_{in}}
    \lra 
    \bigotimes_j \bigg( \lambda(T|_{u(z_j)}
    L_j)^{\otimes 1 - e_j} \bigg) \otimes 
    o_{\vec{x}_{out}} \otimes o_{\vec{y}_{out}}.
\end{equation}
If we fix orientations for the $L_j$, then we obtain an isomorphism of the form
(\ref{orientisos}) without the $L_j$ factors. But the $L_j$ factors will
continue to have relevance in sign comparison arguments.

\begin{rem}[The semistable case] \label{trivializingR}
The moduli spaces $\mc{M}(y_0; y_1)$, $\mc{R}(x_0; x_1)$ arise as a further
quotient of the non-rigid elements of $\tilde{\mc{M}}(y_0; y_1)$ and
$\tilde{R}^1(x_0; x_1)$ by the natural $\R$ actions. 
Thus, at rigid points $u \in \mc{M}(y_0; y_1)$, $v \in \mc{R}(x_0; x_1)$
one obtains trivializations of $\lambda (\tilde{\mc{M}}(y_0; y_1))$,
$\lambda(\tilde{\mc{R}}^1(x_0; x_1))$
and hence isomorphisms 
\begin{equation}
    \begin{split}
        o_{y_1} &\lra o_{y_0}\\
        o_{x_1} &\lra o_{x_0}
\end{split}
\end{equation} 
by choosing a trivialization of the $\R$ actions. In both cases, following the conventions in  
\cite{Seidel:2008zr}*{(12f)} and \cite{Abouzaid:2010kx}*{\S C.6},
choose
$\partial_s$ to be the vector field inducing the
trivialization.  
\end{rem}

\subsection{Comparing Signs}\label{comparingsigns}

Let $\overline{\mc{Q}}$ be some abstract compact moduli space, and suppose its
codimension one boundary has a component covered by a product of lower
dimensional moduli spaces (either of which may also decompose as a product).
\begin{equation}
    \overline{\mc{A}} \times_{\vec{v}, \vec{w}} \overline{\mc{B}}.
\end{equation}
We should first elaborate upon the notation $\times_{\vec{v},\vec{w}}$. We
suppose first that we have fixed separate orderings of the input and output
boundary and interior marked points for $\mc{A}$, $\mc{B}$, $\mc{Q}$; such
orderings will be specified case by case.
\begin{defn}
    Given vectors of the form 
    \begin{equation}
        \begin{split}
            \vec{v} &= \{(v_1^{-},v_1^+),\ldots, (v_k^-,v_k^+)\}\\
            \vec{w} &= \{(w_1^-,w_2^+), \ldots, (w_l^-, w_l^+)\},
        \end{split}
    \end{equation}
    the notation 
    \begin{equation}
        \overline{\mc{A}} \times_{\vec{v},\vec{w}} \overline{\mc{B}}
    \end{equation}
    refers to the product of abstract moduli spaces $\overline{\mc{A}}$ with
    $\overline{\mc{B}}$ in which
    \begin{itemize}
        \item the $v_i^-$th boundary output of $\overline{\mc{A}}$ is (nodally)
            glued to the $v_i^+$th boundary input of $\overline{\mc{B}}$, for
            $1 \leq i \leq k$, and 
        \item the $w_j^-$th interior output of $\overline{\mc{A}}$ is (nodally)
            glued to the $w_j^+$th interior input of $\overline{\mc{B}}$, for
            $1 \leq j \leq l$.
    \end{itemize}
    We refer to such $(\vec{v},\vec{w})$ as a {\bf nodal gluing datum}.
\end{defn}
Now, suppose we had fixed orientations for $\mc{A}$, $\mc{B}$, and $\mc{Q}$. 
The associated space of maps 
\begin{equation} \label{spaceofmaps1}
    \overline{\mc{Q}}(\vec{x}_{in}, \vec{y}_{in}; \vec{x}_{out}, \vec{y}_{out})
\end{equation}
inherits an orientation from Lemma \ref{orientationlem} and has as a
codimension-1 boundary component the product of moduli spaces
\begin{equation} \label{productspaceofmaps}
    \mc{A}(\vec{x}_{in}^1, \vec{y}_{in}^1; \vec{x}_{out}^1, \vec{y}_{out}^1) \times_{\vec{v},\vec{w}} \mc{B}(\vec{x}_{in}^2, \vec{y}_{in}^2; \vec{x}_{out}^2, \vec{y}_{out}^2)
\end{equation}
for suitable input and output vectors $\vec{x}_{in}^i,\vec{y}_{in}^i;
\vec{x}_{out}^i, \vec{y}_{out}^i.$ Thus, the product (\ref{productspaceofmaps})
inherits a boundary orientation from (\ref{spaceofmaps1}). However, Lemma \ref{orientationlem} and our chosen orientations for $\mc{A}$ and $\mc{B}$ also give   (\ref{productspaceofmaps}) a canonical {\it product orientation}
The question of relevance to us is
\begin{equation}\label{signdifferenceques}
    \begin{split}
        &\textrm{What is the sign difference between the product orientation and }\\
    &\textrm{boundary orientation of (\ref{productspaceofmaps})?}
\end{split}
\end{equation}
Actually, we will also equip $\mc{A}$, $\mc{B}$, and $\mc{Q}$ with sign
twisting data $\vec{t}_{\mc{A}}$, $\vec{t}_{\mc{B}}$ and $\vec{t}_{\mc{Q}}$ and
calculate the sign difference with these twistings incorporated. But we can
just add them at the end.

Abbreviate
\begin{equation}
    o_{\vec{x}_{in},\vec{y}_{in}; \vec{x}_{out}, \vec{y}_{out}} := o_{\vec{x}_{out}} \otimes o_{\vec{y}_{out}} \otimes o_{\vec{x}_{in}}^\vee \otimes o_{\vec{y}_{in}}^\vee
\end{equation}
and
\begin{equation}
    \vec{xy}^i: = (\vec{x}_{in}^i, \vec{y}_{in}^i; \vec{x}_{out}^i, \vec{y}_{out}^i).
\end{equation}
Then, (\ref{signdifferenceques}) can be rephrased as: what is the sign
difference in the failure of commutativity of the following diagram?
\begin{equation}
    \xymatrix{\lambda (\mc{Q}(\vec{xy}))
    \ar[r]^{ } \ar[d]^{} & 
    \lambda(\mc{Q}) \otimes \mc{L}_{\mc{Q}} 
    \otimes 
o_{\vec{xy}} 
    \ar[d]^{} \\
    \lambda(\mc{B}(\vec{xy}^2)) \otimes \lambda(\mc{A}(\vec{xy}^1)) 
    \ar[r]^{ }& 
    \lambda(\mc{B}) \otimes \mc{L}_{\mc{B}} \otimes 
o_{\vec{xy}^2} 
    \otimes 
    \lambda(\mc{A}) \otimes \mc{L}_{\mc{A}} \otimes 
o_{\vec{xy}^1} 
    }
\end{equation}
Here $\mc{L}_{\mc{Q}}$, $\mc{L}_{\mc{A}}$, and $\mc{L}_{\mc{B}}$ are the powers
of orientation of a fixed boundary Lagrangians, one for each boundary component
of representatives of the moduli spaces, appearing in Lemma
\ref{orientationlem}; these satisfy $\mc{L}_{\mc{B}} \otimes
\mc{L}_{\mc{A}} = \mc{L}_{\mc{Q}}$ (up to an even power of the top exterior
power of Lagrangians, which is trivial). The top and bottom horizontal arrows
are the ones given by Lemma \ref{orientationlem}. The reversal of $\mc{B}$ and
$\mc{A}$ above comes from the fact that we originally listed the boundary
strata of $\mc{A}$, $\mc{B}$ in the reverse order of composition.
\begin{prop}[Sign Comparison] \label{signcompprop}
The sign difference between the product and boundary orientations is the sum of
four contributions: 
\begin{itemize}
    \item {\bf Koszul signs} from {\bf reordering} $\lambda(A)$  past $o_{\vec{xy}^2}$ and $\mc{L}_{\mc{B}}$.
    \item {\bf Koszul signs} from {\bf reordering} $\mc{L}_\mc{A}$ past $o_{\vec{xy}^2}$.
    \item {\bf Koszul signs} from {\bf reordering} $o_{\vec{xy}^2} \otimes o_{\vec{xy}^1}$ to become $o_{\vec{xy}}$ (using the natural pairings (\ref{orientationpairing}) on elements coming from the gluing $(\vec{v},\vec{w})$).
    \item Comparing the {\bf product versus boundary orientation} on {\bf
        abstract moduli spaces} $\mc{A} \times_{\vec{v},\vec{w}} \mc{B}$.
\end{itemize}
\end{prop}
Since all of the operations we construct involve {\it sign twisting data}
(Definition \ref{signtwistingdatum}), we add back in said data to obtain the
right signs.
\begin{cor}\label{signtwistcor}
    The sign of the composed term $(-1)^{\vec{t}_1} \mathbf{F}_{\mc{B}}
    \circ_{\vec{v},\vec{w}} (-1)^{\vec{t}_2} \mathbf{F}_{\mc{A}}$ in the
    expression arising from the codimension 1 boundary principle (Lemma
    \ref{bordificationlem}) applied to $\mc{Q}$ 
    is the sum of: 
\begin{itemize}
    \item all terms from Proposition \ref{signcompprop}; and 
    \item contributions from the sign twisting data $\vec{t}_1$ and
        $\vec{t}_2$, in the sense of
        (\ref{twistoperation1}).
\end{itemize}
\end{cor}

\subsection{Abstract Moduli Spaces and their orientations} \label{abstractorientations}
By Lemma \ref{orientationlem}, we must choose orientations of the various
abstract moduli spaces we consider in order to orient the associated
operations. In this section, we do precisely that. We will also, in a
sample case, compute explicitly the sign difference between the induced and
chosen orientation on boundary strata.

\subsubsection{$\mc{R}^d$}\label{ainforientation}
Fix a slice of $\mc{R}^d$ in which the first three boundary marked points
$z_0^-$, $z_1^+$, and $z_2^+$ are fixed, and consider the positions of the
remaining points $(z_3, \ldots, z_d)$ with respect to the counterclockwise
boundary orientation as a local chart. With respect to this chart, orient
$\mc{R}^d$ by the top form
\begin{equation}
    dz_3 \wedge \cdots \wedge dz_d.
\end{equation}
This agrees with the conventions in \cite{Seidel:2008zr} and
\cite{Abouzaid:2010kx}, so we will not discuss this case or its signs further.

\subsubsection{$\mc{R}_d^1$}\label{oc1orientation}
Take a slice of $\mc{R}_d^1$ in which the interior point $y_{out}$ is fixed, as
is the distinguished marked point $z_d$. With respect to the induced
coordinates $(z_1, \ldots, z_d)$ induced by the positions of the remaining
marked points, pick orientation form 
\begin{equation}
- dz_1 \wedge \cdots \wedge dz_{d-1}.
\end{equation}
This agrees with the choice made in \cite{Abouzaid:2010kx}*{C.3}.

\subsubsection{$\mc{R}_d^{1,1}$} \label{co1orientation}
Take a slice of $\mc{R}_{d}^{1,1}$ in which the interior point $y_{in}$ and
outgoing boundary point $z_0^-$, are fixed, and using the positions of the
remaining coordinates $(z_1, \ldots, z_d)$ as the local chart, again pick
orientation form
\begin{equation}
    -dz_1 \wedge \cdots \wedge dz_{d}.
\end{equation}

\subsubsection{$\mc{R}_{d_1,d_2}^1$}
Fix a slice of $\mc{R}_{d_1,d_2}^1$ (see Definition \ref{twopointedocmod}) in
which $z_0$, $z_0'$, and $y_{out}$ are fixed at $-i$, $i$, and $0$
respectively, and consider the positions of the remaining points in the slice
$(z_1, \ldots, z_{d_1}, z_1', \ldots, z_{d_2}')$ as the local chart. With
respect to these coordinates, pick orientation form
\begin{equation}
    -dz_1 \wedge \cdots \wedge dz_{d_1} \wedge dz'_{1} \wedge \cdots \wedge dz'_{d_2}.
\end{equation}

\begin{prop} \label{2ocabstractboundary}
    With respect to the strata listed in
    (\ref{tocspecialpoint0})-(\ref{tocspecialpoint3}), two of the differences
    in sign between chosen and induced orientations on boundary
    strata are as follows:\\
\begin{center}
    \begin{tabular}{|c|c|}
        \hline{\bf stratum} & {\bf sign difference}\\
        \hline
        (\ref{tocspecialpoint0}) & $1 + n + k' + k'(l+k - n)$ \\
        (\ref{tocspecialpoint3}) & $1 + l(l'+1) + k(l-l')$\\ 
        \hline
    \end{tabular}
    \end{center}
\end{prop}
\begin{proof}
We have not listed the sign differences for the strata (\ref{tocspecialpoint1})
and (\ref{tocspecialpoint2}), which follow from identical calculations.  For
(\ref{tocspecialpoint0}), with respect to the local charts
    $(z_{n+3}, \ldots, z_{n+k'})$ on $\mc{R}^{k'}$ and 
    $(z_1, \ldots, z_n, \tilde{z}, z_{n+k'+1}, \ldots, z_{k}, z_1', \ldots,
    z_l')$ on $\mc{R}^1_{k-k'+1,l}$, the gluing map 
    \begin{equation}
        \rho: [0,1) \times \mc{R}^{k'} \times \mc{R}^1_{k-k'+1,l} \lra \mc{R}^1_{k,l}
    \end{equation}
    has the approximate form
    \begin{equation}
        \begin{split}
            t, (z_{n+3}, \ldots, z_{n+k'}), &(z_1, \ldots, z_n, \tilde{z}, z_{n+k'+1}, \ldots, z_{k}, \cdots) \longmapsto \\
            & (z_1, \ldots, z_n, \tilde{z}, \tilde{z} + t, \tilde{z} + tz_{n+3}, \ldots, \tilde{z} + tz_{n+k'}, z_{n+k'+1}, \ldots).
    \end{split}
\end{equation}
Thus, the pullback under $\rho$ of the top form
\begin{equation} \label{topform2oc}
    -dz_1 \wedge \cdots \wedge dz_{k} \wedge dz'_{1} \wedge \cdots \wedge dz'_{l}.
\end{equation}
is, modulo positive rescaling,
\begin{equation}
    dz_1 \wedge \cdots \wedge dz_n \wedge d\tilde{z} \wedge dt \wedge dz_{n+3} \wedge \cdots \wedge dz_{n+k'} \wedge dz_{n+k'+1} \wedge \cdots \wedge dz_k \wedge dz_1' \wedge \cdots dz_l'
\end{equation}
which visibly differs from the product orientation (using the outward pointing
vector $-dt$)
\begin{equation}
    (-dt) \wedge (-dz_1 \wedge dz_n \wedge d\tilde{z} \wedge dz_{n+k'+1} \wedge \cdots) \wedge (dz_{n+3} \wedge \cdots dz_{n+k'})
\end{equation}
by a sign of parity
\begin{equation}
    n+1 + (k'-2)\cdot (l + k - n - k') = 1 + n + k' + k'(l+k-n).
\end{equation}
For (\ref{tocspecialpoint3}), the gluing map 
\begin{equation}
    \tilde{\rho}: [0,1) \times \mc{R}^{l'+ k' + 1} \times \mc{R}^1_{k-k', l-l'} \lra \mc{R}^1_{k,l}
\end{equation}
takes the following approximate form:
\begin{equation}
    \begin{split}
    t, (z_{k-k'+3}, &\ldots, z_k, z_0', z_1', \ldots, z_{l'}'), (z_1, \ldots, z_{k-k'}, z_{l'+1}', \ldots, z_l') \longmapsto\\
    (z_1, &\ldots, z_{k-k'}, i - t(z_0'-a), i- t(z_0'-b), i - t z_{k-k'+3}, \ldots, i - tz_{k}, i + tz_1', \\
    &\ldots, i + tz_{l'}' , z_{l'+1}', \ldots, z_{l}')
\end{split}
\end{equation}
for some constants $b > a > 0$.  Thus, the pull back of the top form
(\ref{topform2oc}) is, up to positive rescaling
\begin{equation}
    (-1)^{k'-2} dz_1 \wedge \cdots dz_{k-k'} \wedge dt \wedge dz_0' \wedge dz_{k-k'+3} \wedge \cdots \wedge dz_k \wedge dz_1' \wedge \cdots dz_{l'}' \wedge \cdots \wedge dz_l',
\end{equation}
which differs from the product orientation
\begin{equation}
    \begin{split}
    (-dt) \wedge (dz_1 \wedge &\cdots \wedge dz_{k-k'} \wedge dz_{l'+1}' \wedge \cdots \wedge dz_{l'}) \\
    &\wedge (dz_{k-k'+3} \wedge \cdots dz_k \wedge dz_0' \wedge \cdots \wedge dz_{l'}')
\end{split}
\end{equation}
by a sign of parity
\begin{equation}
    1 + (k'-2) + (k-k') + (k'-2) + (k' + l' -1) \cdot (l -l') = 1 + l (l'+1) + k' (l - l') \ \ (\mathrm{mod}\ 2).
\end{equation}

\end{proof}

\subsubsection{$Q(3,\mathbf{r})$}
Strictly speaking, we embed the open locus of quilted strips into glued discs
with sequential point identifications, but since this embedding is an
isomorphism on the open locus, it will suffice to write down a top form on the
level of $Q(3,\mathbf{r})$. There are three cases:
\begin{itemize}
    \item If $r_2 > 0$, then picking a slice of the $\R$ action for which the
        highest marked point on the middle strip $z^{r_2}_{2}$ is fixed, we
        obtain coordinates
        \begin{equation}
            (z^{1}_3, \ldots, z^{r_3}_3, z^{r_1}_1, \ldots, z^{1}_1, z^{r_2-1}_2, \ldots, z^1_2).
        \end{equation}
    \item If $r_2 =0$ but $r_3 >0$, pick a slice for which $z^{1}_3$ is fixed
        to obtain the chart
        \begin{equation}
            (z^2_3, \ldots, z^{r_3}_3, z^{r_1}_1, \ldots, z^1_1).
        \end{equation}
    \item Lastly, if $r_2 = r_3 = 0$, picking a slice for which $z^{r_1}_1$ is fixed, we obtain the chart
        \begin{equation}
            (z^{r_1-1}_1, \ldots, z^1_1)
        \end{equation}
\end{itemize}
In all three cases, pick orientation form the top exterior power of these
coordinates of the chart in the orders specified above.
\subsubsection{$\mc{R}_2^{k,l,s,t}$} Fixing a slice of the action for which
$z_-^1$, $z_-^2$, $z_+^1$, $z_+^2$ are fixed at $i$, $-i$, $1$ and $-1$ and
using the positions of the remaining coordinates 
\begin{equation}(z_1, \ldots,
z_k, z_1^1, \ldots, z^1_l, z^2_1, \ldots, z^2_s, z^3_1, \ldots,
z^3_t)\end{equation} 
as a chart, pick orientation form 
\begin{equation}
    \begin{split} -dz_1 \wedge \cdots \wedge dz_k &\wedge dz_1^1 \wedge
        \cdots dz_l^1 \wedge \\ & dz^2_1 \wedge \cdots \wedge  dz^2_s
        \wedge dz^3_1 \wedge \cdots \wedge dz^3_t.  \end{split}
\end{equation}

\subsubsection{$\mc{A}_{k,l;s,t}$}
In a similar fashion, fix a slice of the action in which $a_0$, $a_0'$,
$b_{0}'$, and $z_{out}$ are at $\pm i$, $Ri$ and $-Ri$ respectively. The
remaining coordinates include the positions of the remaining boundary points
$a_1, \ldots, a_k, a_1', \ldots, a_l'$, $b_1, \ldots, b_s$, $b_1', \ldots,
b_t'$, and the radial parameter $r = \frac{R}{R+1}$. With respect to these
coordinates, choose orientation form
\begin{equation}
    \begin{split}
    -dr \wedge da_1 \wedge \cdots \wedge da_k &\wedge da_1' \wedge \cdots da_l' \wedge \\
    & db_1 \wedge \cdots \wedge  db_s \wedge db_1' \wedge \cdots \wedge db_t'.
\end{split}
\end{equation}

\subsection{Sign Verification}\label{verification}
In this section, we use all of the ingredients above to verify the signs of
equations in one case. 
Namely, we will (partly) show that
\begin{prop}[Corollary \ref{2occhain} with signs] \label{2ocsigns}
    ${_2}\oc$ is a chain map (with the right signs).
\end{prop}

\begin{proof}[Proof of Prop. \ref{2ocsigns}]
    We need to establish that the boundary strata
    (\ref{tocspecialpoint0})-(\ref{tocspecialpoint3}) along with strip-breaking
    and our chosen sign twisting data, determine the equation
    \begin{equation}
       {_2}\oc \circ d_{ {_2}\r{CC}_*} - d_{CH} \circ {_2}\oc  = 0
    \end{equation}
    up to an overall sign. As all the cases are analogous, we will simply show
    that some of the terms in 
    \begin{equation}
        {_2}\oc \circ d_{ {_2}\r{CC}_*}
    \end{equation}
    appear with the correct sign (up to the overall sign); in particular we
    focus upon all of the terms in the strata (\ref{tocspecialpoint0}), which
    should contribute to the terms
    \begin{equation}
        \sum_{m,k} (-1)^{\maltese_1^n}{_2}\oc (\mathbf{a}, x_l, 
        \ldots, x_1, \mathbf{b}, y_s, \ldots, y_{m+k+1}, 
        \mu^{k}(y_{m+k}, \ldots, y_{m+1}), y_m, \ldots, y_1),
    \end{equation}
    where $\maltese_1^m$ is the sign
    \begin{equation}
        \sum_{j=1}^m ||y_i||.
    \end{equation}
    So, fix a set of asymptotic inputs $(y_1, \ldots, y_s, \mathbf{b}, x_1,
    \ldots, x_l, \mathbf{a})$.
    The strata (\ref{tocspecialpoint0}) are, for $k < s$ and $0 \leq m < s-k'+1$,
    \begin{equation}
        \mc{R}^k \times_{m+1} \mc{R}^1_{s-k+1,l}.
    \end{equation}
    Abbreviating 
    \begin{equation}
        \begin{split}
            x_{i\ra j} &:= \{x_i, x_{i+1}, \ldots, x_j\}\\
            y_{i \ra j} &:= \{y_i, y_{i+1}, \ldots, y_j\}
        \end{split}
    \end{equation}
    the corresponding moduli spaces are 
    \begin{equation}
        \mc{R}^k(\tilde{y}, y_{m+1\ra m+k}) \times \mc{R}^1_{s-k'+1,1}(z; y_{1\ra m}, \tilde{y}, y_{m+k+1 \ra s}, \mathbf{b}, x_{1\ra l}, \mathbf{a})
    \end{equation}
    in reverse order of composition, where $z$ is an output orbit, and
    $\tilde{y}$ ranges over all possible admissible asymptotic conditions.
Abbreviating $\lambda_{\mu}:= \lambda (\mc{R}^k)$, and $\lambda_{2\oc} :=
\lambda (\mc{R}^1_{s-k+1,l})$, Lemma \ref{orientationlem} tells us that the
natural product orientation form is isomorphic to 
\begin{equation}
        \begin{split}
            \big(\lambda_{2\oc} \otimes \lambda(L_0) \otimes o_z \otimes o_{y_{1\ra m}}^\vee &\otimes o_{\tilde{y}}^\vee \otimes o_{y_{m+k+1\ra s}}^\vee \otimes o_{\mathbf{b}}^\vee \otimes o_{x_{1 \ra l}}^\vee \otimes o_{\mathbf{a}}^\vee \big)\\
        &\otimes \big(\lambda_{\mu} \otimes o_{\tilde{y}} \otimes o_{y_{m+1\ra m+k}}^\vee\big).
    \end{split}
    \end{equation}
    where, as before, we've abbreviated $o_{y_{1\ra k}}^{\vee} = o_{y_1}^\vee
    \otimes \cdots o_{y_k}^\vee$ and so on, and we've abbreviated
    $\lambda(L_0):= \lambda (T|_{u(z_i)}L_0)$ for one of the Lagrangian
    boundary conditions $L_0$ of the moduli space $\mc{R}^1_{s-k'+1,1}(z; y_{1\ra m}, \tilde{y}, y_{m+k+1 \ra s}, \mathbf{b}, x_{1\ra l}, \mathbf{a})$ .  From the above
    description, we can immediately calculate some of the sign
    contributions in Proposition \ref{signcompprop}:
    \begin{itemize}
        \item $\mc{R}^k$ has dimension $(k-2)$, so the sign for reordering
            $\lambda_\mu$ to be next to $\lambda_{2\oc}$ has parity
            \begin{equation}\label{2ockoszulsign1}
              \star_1 := (k-2)(n - l - s + k + 1 + n) =  k(l+s) \ (\mathrm{mod}\ 2).
            \end{equation}
        \item there are no Lagrangian terms $\lambda(T|_{u(z_j)} L_j)$ in
            orientation form of the moduli space associated to $\mc{R}^k$, so
            the associated signs of this sort are zero,
        \item the sign for reordering the orientation lines $o_{\tilde{y}}
            \otimes o_{y_{m+1 \ra m+k}}^\vee$ to be immediately to the right of
            $o_{\tilde{y}}^\vee$ (allowing one also to pair and cancel the
            $o_{\tilde{y}}^\vee$, $o_{\tilde{y}}$) has parity
            \begin{equation}
                \begin{split}
                \star_2 &:= (2-k)(|\mathbf{a}|+ \sum_{i=m+k+1}^s |y_i| + |\mathbf{b}| + \sum_{i=1}^l |x_i|).\\
                &= k(|\mathbf{a}|+ \sum_{i=m+k+1}^s |y_i| + |\mathbf{b}| + \sum_{i=1}^l |x_i|)\ \ (\mathrm{mod}\ 2).
            \end{split}
            \end{equation}
        \item the sign difference between boundary and product orientations on
            the moduli space $\mc{R}^k \times \mc{R}^1_{s-k+1,l}$ was computed
            in Proposition \ref{2ocabstractboundary} to have parity
            \begin{equation}
                \star_3 := 1 + m + k + k(l + s - m).
            \end{equation}
    \end{itemize}

    Finally, we can add in the sign twist contributions mentioned in Corollary
    \ref{signtwistcor}, corresponding to the operations ${_2}\oc_{l,s-k+1}$,
    and $\mu^k$:
    \begin{itemize}
        \item The sign twist contribution from $\mu^k$ has parity
        \begin{equation}
           \S_2:= (1, \ldots, k) \cdot (|y_{m+1}|, \ldots, |y_{m+k}|) = \sum_{i=m+1}^{m+k} (i-m) |y_i|.
        \end{equation}

    \item The sign twist contribution from ${_2}\oc_{l,s-k+1}$ has parity
        \begin{equation}
            \begin{split}
            \S_1:= \sum_{i=1}^m i |y_i| + (m+1) |\tilde{y}| &+ \sum_{i=m+k+1}^s (i-k+1) |y_i| + (s-k+1) |\mathbf{b}| \\
            &+ \sum_{i=1}^l (s-k+1 + i) |x_i| + (s-k+1+l)|\mathbf{a}|
        \end{split}
        \end{equation}
        where 
        \begin{equation}
            |\tilde{y}| = 2 - k + \sum_{i=m+1}^{m+k} |y_i|.
        \end{equation}

    \end{itemize}

    Combining all of these signs, we compute that 
    \begin{equation}
        \S_1 + \S_2 + \star_1 + \star_2 + \star_3 = \sum_{i=1}^m |y_i| + m + \bigstar_{l,s} \ \ (\mathrm{mod}\ 2),
    \end{equation}
    where 
    \begin{equation}
        \bigstar_{l,s} = \sum_{i=1}^s (i+1) |y_i| + (s+1) |\mathbf{b}| + \sum_{i=1}^l (s + 1 + i )|x_i| + (s + 1 + l) |\mathbf{a}|
    \end{equation}
    is independent of $k, m$, and 
    \begin{equation}
        \sum_{i=1}^m |y_i| + m = \maltese_{1}^m \ \ (\mathrm{mod}\ 2)
    \end{equation}
    as desired. This calculation extends formally to the semi-stable case $k=1$
    as well. The only extra ingredient, following Remark \ref{trivializingR},
    is an extra sign of parity $0$ or $1$ coming from determining whether the
    vector $\partial_s$ after gluing is inward pointing ($1$) or outward
    pointing ($0$). In this case, the vector is outward pointing so there is no
    additional sign contribution. Note that when the second component is
    semi-stable instead, the vector will be inward pointing, contributing to
    e.g.,  the $-1$ coefficient in $-d_{CH} \circ {_2} \oc$. See
    \cite{Seidel:2008zr}*{(12f)} for more details.
\end{proof}


\bibliography{/Users/sheelganatra/Dropbox/References/math_bib}

\end{document}